\documentclass{amsart}
\usepackage{amsmath}
\usepackage{amsfonts}
\usepackage{fullpage, verbatim, amssymb, url, graphicx}
\usepackage[toc,page]{appendix}

\numberwithin{equation}{section}
\newtheorem{theorem}{Theorem}[section]

\newtheorem{axiom}[theorem]{Assumption}

\newtheorem{corollary}[theorem]{Corollary}

\newtheorem{example}[theorem]{Example}

\newtheorem{lemma}[theorem]{Lemma}
\newtheorem{notation}[theorem]{Notation}

\newtheorem{proposition}[theorem]{Proposition}
\newtheorem{remark}[theorem]{Remark}

\input{tcilatex}

\begin{document}
\title{On the $p$-adic integration over Igusa towers of Siegel modular
varieties}
\author{M. A. Seveso}
\address{Marco Adamo Seveso: Milano, Italy}
\email{seveso.marco@gmail.com}

\begin{abstract}
We develop an explicit $p$-adic integration theory for Igusa towers of
modular Siegel manifolds, which finds applications to explicit reciprocity
laws.
\end{abstract}

\maketitle
\tableofcontents

\section{Introduction}

The aim of this paper is to develop an explicit $p$-adic integration theory
for Igusa towers of Siegel modular varieties. In the genus $g=1$ case, this
theory is (well known and) elementary. For higher genus it is richer and
much more involved and interesting. Our main motivation in developing this
theory is to provide a key tool for proving explicit reciprocity laws: this
is well known in the case $g=1$ and, when $g=2$, an explicit form of the
theorem stated below will be used in a forthcoming work by F. Andreatta, M.
Bertolini, R. Venerucci and the author.

Suppose that $X$ is a modular curve of prime to $p$ level and that $f_{0}$
is a modular form of weight $k\geq 2$. It is naturally a section of $%
H^{0}\left( X,\mathcal{W}_{k}\right) $, which appears inside the de Rham
cohomology $H_{\mathrm{dR}}^{1}\left( X,\mathcal{L}_{k-2}\right) $ with
coefficients in the automorphic sheaf $\mathcal{L}_{k-2}:=\mathrm{Sym}%
^{k-2}\left( \mathcal{H}_{\mathrm{dR}}^{1}\right) $. Here $\mathcal{H}_{%
\mathrm{dR}}^{1}$ is the relative de Rham cohomology of the universal
(generalized) elliptic curve and, if $\omega _{\mathrm{dR}}$ denotes its
relative Lie algebra, then $\mathcal{W}_{k}:=\omega _{\mathrm{dR}}^{k}$. One
can associate to $f_{0}$ a $p$-depletion $f_{0}^{\left[ p\right] }$\ by
considering the modular form whose $q$-expansion is given by%
\begin{equation*}
f_{0}^{\left[ p\right] }\left( q\right) =\sum\nolimits_{p\nmid n}a_{n}q^{n}%
\text{.}
\end{equation*}

One advantage of the $p$-depletion is that a certain derivation appearing in
the description of the Gauss-Manin connection, the so called Serre theta
operator $\theta $ (which equals $q\frac{d}{dq}$\ at the cusp) can be
inverted. Indeed, these operations of taking the $p$-depletion and
considering the theta operator $\theta $\ make sense, more generally, for
any section $f$ of the Igusa tower $X_{0}$. The effect is that, if we
pull-back the de Rham complex to the Igusa tower and consider the subcomplex
of those forms that are $p$-depleted, then this complex becomes acyclic in
degree one by means of an inductive explicit process of integration obtained
from the theta operators as follows. First of all, since we work over the
Igusa tower, there is a canonical global section $\delta $ of $\omega _{%
\mathrm{dR}}$ and, setting $\eta :=\nabla \left( \theta \right) \delta $
(for the Gauss-Manin connection $\nabla $), we get another global section of 
$\mathcal{H}_{\mathrm{dR}}^{1}$\ trivializing it. It follows that every
section $F$\ of $\mathcal{H}_{\mathrm{dR}}^{1}$ can be written in the form%
\begin{equation*}
F=F_{0}\delta ^{k-2}\eta ^{0}+F_{1}\delta ^{k-3}\eta ^{1}+...+F_{i}\delta
^{k-2-i}\eta ^{i}+...+F_{k-3}\delta ^{1}\eta ^{k-3}+F_{k-2}\delta ^{0}\eta
^{k-2}
\end{equation*}%
for a weight $i$ section $F_{i}$\ of $\mathcal{O}_{X_{0}}$. Furthermore, the
Kodaira-Spencer isomorphism identifies $\omega _{\mathrm{dR}}^{2}$ with $%
\Omega _{X_{0}}^{1}$ and, hence, $\delta ^{2}$ defines a canonical section $%
\omega $ of $\Omega _{X_{0}}^{1}$ (which is $\frac{dq}{q}$ at the cusp).
Using the relations%
\begin{equation*}
\nabla \delta =\eta \otimes _{\mathcal{O}_{X_{0}}}\omega \text{, }\nabla
\eta =0\text{ and }ds=q\frac{ds}{dq}\frac{dq}{q}=\theta \left( s\right)
\otimes _{\mathcal{O}_{X_{0}}}\omega \text{,}
\end{equation*}%
(where $s$ is a global section of $\mathcal{O}_{X_{0}}$) one computes%
\begin{equation*}
\nabla \left( s\delta ^{k-2-i}\eta ^{i}\right) =\left( \theta \left(
s\right) \delta ^{k-2-i}\eta ^{i}+\left( k-2-i\right) s\delta ^{k-3-i}\eta
^{i+1}\right) \otimes _{\mathcal{O}_{X}}\omega \text{,}
\end{equation*}%
from which it follows that%
\begin{eqnarray*}
\nabla F &=&[\theta \left( F_{0}\right) \delta ^{k-2}\eta ^{0}+\left( \theta
\left( F_{1}\right) +\left( k-2\right) F_{0}\right) \delta ^{k-3}\eta
^{1}+...+\left( \theta \left( F_{i}\right) +\left( k-1-i\right)
F_{k-i-1}\right) \delta ^{k-2-i}\eta ^{i} \\
&&+...+\left( \theta \left( F_{k-2}\right) +F_{k-3}\right) \delta ^{0}\eta
^{k-2}]\otimes _{\mathcal{O}_{X_{0}}}\omega \text{.}
\end{eqnarray*}%
The image of $f_{0}\cdot \delta ^{k}\in H^{0}\left( X_{0},\mathcal{W}%
_{k}\right) $ in $H^{0}\left( X_{0},\mathcal{L}_{k-2}\otimes _{\mathcal{O}%
_{X_{0}}}\Omega _{X_{0}}^{1}\right) $ is $f_{0}\cdot \delta ^{k-2}\otimes _{%
\mathcal{O}_{X_{0}}}\omega $ and, consequently, one is lead to consider the
system of differential equations%
\begin{equation*}
\theta \left( F_{0}\right) =f\text{, }\theta \left( F_{1}\right) =-\left(
k-2\right) F_{0},...,\theta \left( F_{i}\right) =-\left( k-1-i\right)
F_{k-i-1},...,\theta \left( F_{k-2}\right) =-F_{k-3}\text{.}
\end{equation*}%
When $f$ is $p$-depleted, as explained, we can consider $F_{0}:=\theta
^{-1}f $ and then, noticing that $F_{0}$ is again $p$-depleted, define $%
F_{1}:=-\left( k-2\right) \theta ^{-1}F_{0}$ and, hence, recursively solve
the equation $\nabla F=f_{0}\delta ^{k-2}\otimes _{\mathcal{O}%
_{X_{0}}}\omega $. Indeed, more generally, a similar process shows that one
can solve the equation $\nabla F=f$ for every $f\in H^{0}\left( X_{0},%
\mathcal{L}_{k-2}\otimes _{\mathcal{O}_{X_{0}}}\Omega _{X_{0}}^{1}\right) $
as long as all the components $f_{i}$ of%
\begin{equation*}
f=\left[ f_{0}\delta ^{k-2}\eta ^{0}+...+f_{i}\delta ^{k-2-i}\eta
^{i}+...+f_{k-2}\delta ^{0}\eta ^{k-2}\right] \otimes _{\mathcal{O}%
_{X_{0}}}\omega
\end{equation*}%
are $p$-depleted. This kind of explicit primitives appear in the proof of
several reciprocity laws.

In order to prepare us for the higher genus case, let us reinterpret the
above calculation from a more representation theoretic perspective. To this
end, let us start by considering the standard representation \textrm{Std}$%
_{2}=\mathbb{Z}^{2}$\ of $\mathbf{G}=\mathbf{GL}_{2}$, where we view the
elements of \textrm{Std}$_{2}$ as column vectors on which $\mathbf{GL}_{2}$
acts from the left and consider the canonical basis given by the elements $%
e_{1}:=\left( 1,0\right) ^{t}$ and $e_{2}:=\left( 0,1\right) ^{t}$. Let us
write $\mathbb{Z}\left[ Y\right] $ for the space of polynomials and, writing 
$\mathbf{Q\subset GL}_{2}$ for the upper triangular Borel subgroup, consider
the $\mathbf{Q}$-action on \textrm{Ind}$_{\mathbf{Q}^{-}}^{\mathbf{G}}\left(
e_{1}^{k-2}\right) \left[ Y\right] :=e_{1}^{k-2}\otimes _{\mathbb{Z}}\mathbb{%
Z}\left[ Y\right] $ defined by the formula $\gamma \left( e_{1}^{k-2}\otimes
_{\mathbb{Z}}f\right) :=a^{k-2}e_{1}^{k-2}\otimes _{\mathbb{Z}}f\left(
a^{-1}\left( b+Yd\right) \right) $ for $\gamma =\left( 
\begin{array}{cc}
a & b \\ 
0 & d%
\end{array}%
\right) $. Then, the $\mathbf{Q}$-module \textrm{Ind}$_{\mathbf{Q}^{-}}^{%
\mathbf{G}}\left( e_{1}^{k-2}\right) \left[ Y\right] $ acquires a $\left( 
\mathfrak{g},\mathbf{Q}\right) $-module structure and a $\left( \mathfrak{g},%
\mathbf{Q}\right) $-equivariant inclusion%
\begin{equation*}
\mathrm{Sym}_{2}^{k-2}:=\mathrm{Sym}^{k-2}\left( \mathrm{Std}_{2}\right)
\hookrightarrow \mathrm{Ind}_{\mathbf{Q}^{-}}^{\mathbf{G}}\left(
e_{1}^{k-2}\right) \left[ Y\right]
\end{equation*}%
by formally setting $Y:=\frac{e_{2}}{e_{1}}$ as follows. First of all, when $%
k=2$, writing $\partial $ for the matrix whose only non-zero entry is in the
lower left position and equals $1$, it can be enriched to a $\left( 
\mathfrak{g},\mathbf{Q}\right) $-module by setting $\partial \left( Y\right)
:=-Y^{2}$. In general, we send $e_{1}^{k-2-i}e_{2}^{i}=e_{1}^{k-2}\left( 
\frac{e_{2}}{e_{1}}\right) ^{i}$ to $e_{1}^{k-2}Y^{i}$ for every $%
i=0,...,k-2 $ and then compute%
\begin{eqnarray*}
\partial \left( e_{1}^{k-2}Y^{i}\right) &=&\left( k-2\right)
e_{1}^{k-3}\partial \left( e_{1}\right) Y^{i}+ie_{1}^{k-2}Y^{i-1}\partial
\left( Y\right) =\left( k-2\right) e_{1}^{k-3}e_{2}Y^{i}-ie_{1}^{k-2}Y^{i+1}
\\
&=&\left( k-2\right) e_{1}^{k-2}Y^{i+1}-ie_{1}^{k-2}Y^{i+1}=\left(
k-2-i\right) e_{1}^{k-2}Y^{i+1}
\end{eqnarray*}%
in order to justify our definition $\partial \left( e_{1}^{k-2}Y^{i}\right)
:=\left( k-2-i\right) e_{1}^{k-2}Y^{i+1}$. In this way, the above inclusion
of $\mathrm{Sym}_{2}^{k-2}$ in $\mathrm{Ind}_{\mathbf{Q}^{-}}^{\mathbf{G}%
}\left( e_{1}^{k-2}\right) \left[ Y\right] $ becomes an inclusion of $\left( 
\mathfrak{g},\mathbf{Q}\right) $-modules. Going back to the geometric
picture, the basis $\left\{ \delta ,\eta \right\} $ yields and isomorphism $%
\mathcal{O}_{X_{0}}^{2}\overset{\sim }{\rightarrow }\mathcal{H}_{\mathrm{dR}%
}^{1}$ and, hence, $\mathcal{O}_{X_{0}}\otimes _{\mathbb{Z}_{p}}\mathrm{Sym}%
_{2}^{k-2}\overset{\sim }{\rightarrow }\mathcal{L}_{k}$ sending $%
e_{1}^{k-2-i}e_{2}^{i}=e_{1}^{k-2}Y^{i}$ to $\delta ^{k-2-i}\eta ^{i}$ for
every $i=0,...,k-2$. Hence, we find an inclusion%
\begin{equation*}
\mathcal{L}_{k-2}\hookrightarrow \mathcal{O}_{X_{0}}\otimes _{\mathbb{Z}_{p}}%
\mathrm{Ind}_{\mathbf{Q}^{-}}^{\mathbf{G}}\left( e_{1}^{k-2}\right) \left[ Y%
\right] =:\mathcal{V}_{k-2}
\end{equation*}%
and, from this point of view, we see that, setting $\Theta \left( s\delta
^{k-2-i}\eta ^{i}\right) :=\theta \left( s\right) \delta ^{k-2-i}\eta
^{i}\otimes _{\mathcal{O}_{X_{0}}}\omega $ and $\Delta :=1_{\mathcal{O}%
_{X_{0}}\otimes _{\mathbb{Z}_{p}}}\partial \otimes _{\mathcal{O}%
_{X_{0}}}\omega $, we have%
\begin{equation*}
\nabla \left( s\delta ^{k-2-i}\eta ^{i}\right) =\Theta \left( s\delta
^{k-2-i}\eta ^{i}\right) +\Delta \left( s\omega ^{k-2-i}\eta ^{i}\right) 
\text{.}
\end{equation*}%
In this way, we have given an interpretation of the multiplication by $%
\left( k-2-i\right) $ map appearing in the above expression of the $\nabla $
operator which is uniform for varying $k$'s: it is just given by the action
of $\partial $. Indeed, setting%
\begin{equation*}
\Theta \left( se_{1}^{k-2}Y^{i}\right) :=\theta \left( s\right)
e_{1}^{k-2}Y^{i}\otimes _{\mathcal{O}_{X_{0}}}\omega \text{, }\Delta :=1_{%
\mathcal{O}_{X_{0}}\otimes _{\mathbb{Z}_{p}}}\partial \otimes _{\mathcal{O}%
_{X_{0}}}\omega
\end{equation*}%
and then $\nabla :=\Theta +\Delta $ yields a connection on $\mathcal{V}%
_{k-2} $ and, in this way, one can promote the inclusion of $\mathcal{L}%
_{k-2}$ in $\mathcal{V}_{k-2}$ to an inclusion of modules with connection.
Then, the above integration process obtained by solving the successive
differential equations%
\begin{equation*}
\Theta \left( F_{i}e_{1}^{k-2}Y^{i}\right) =f_{0}e_{1}^{k-2}Y^{i}\otimes _{%
\mathcal{O}_{X_{0}}}\omega ,...,\Theta \left( F_{i}e_{1}^{k-2}Y^{i}\right)
+\Delta \left( F_{i-1}e_{1}^{k-2}Y^{i-2}\right)
=f_{i}e_{1}^{k-2}Y^{i}\otimes _{\mathcal{O}_{X_{0}}}\omega ,...
\end{equation*}%
essentially works in $\left( \mathcal{V}_{k-2},\nabla \right) $ if the
sections $f_{i}$ of $\mathcal{O}_{X_{0}}$ are all $p$-depleted. Here
\textquotedblleft essentially\textquotedblright\ just means that, after
having solved the equations recursively thus getting the infinite family of
sections $F_{i}$'s, we would like to be able to write $F:=\sum%
\nolimits_{i=0}^{+\infty }F_{i}$ in order to get a solution of $\nabla F=f$.

\bigskip

The aim of this paper is to generalize this integration process in the
framework of Siegel modular varieties. Hence, introduce an analogue of the
notion of $p$-depletion and, for any $p$-depleted $m$-form $f$\ (with $m>0$)
of the pull-back to the Igusa tower of de Rham complex with coefficients in
an automorphic sheaf, define an indefinite integral%
\begin{equation*}
F=\int f
\end{equation*}%
that should be an $m-1$-form in the same $p$-depleted de Rham complex. Note
that, in this setting, modular forms are sections of $H^{0}\left( X,\mathcal{%
W}_{\lambda }\right) $ appearing inside the de Rham cohomology $H_{\mathrm{dR%
}}^{d_{g}}\left( X,\mathcal{L}_{\lambda }\right) $\ of middle degree $d_{g}:=%
\frac{g\left( g+1\right) }{2}$ and, hence, one may first try to integrate $%
d_{g}$-forms. Here, $\lambda $ is a dominant weight of $\mathbf{G}:=\mathbf{%
GSp}_{2g}$ (the symplectic group being defined as in Example \ref%
{Representations E GSp}, to fix ideas), that can also be regarded as a
dominant weight of the Levi sugroup $\mathbf{M:=GL}_{g}\times \mathbf{G}_{m}$%
\ of the standard Siegel parabolic subgroup $\mathbf{Q}\subset \mathbf{G}$
with opposite parabolic subgroup $\mathbf{Q}^{-}$, the sheaf $\mathcal{L}%
_{\lambda }$ looks locally (and globally after a pull-back to the Igusa
tower) like the irreducible representation of $\mathbf{G}$ of highest weight 
$\lambda $ and, similarly, $\mathcal{W}_{\lambda }$ looks like the
irreducible representation of $\mathbf{M}$ of highest weight $\lambda $. Let
us now write $X$ for a Siegel modular variety (we admit level and certain
paramodular structures, but assume that $p$ is prime to the level and the
degree of the parametrization - see Remark \ref{Intro R1} below for possible
generalizations) and, again, let us denote by $X_{0}$ the Igusa tower. Let
us write $\mathbf{S}_{g,\mathbb{Z}}$ for the $\mathbb{Z}$-module of
symmetric $g$-by-$g$ matrices with coefficients in $\mathbb{Z}$ and let $%
\mathbf{S}_{g,\mathbb{Z}}^{even}$ be the $\mathbb{Z}$-submodule of those
matrices whose diagonal entries are even. Let us write $\underline{\beta }%
_{ij}$ for the symmetric matrix whose unique upper triangular non-zero entry
is $1$\ at position $\left( i,j\right) $ and, if $\beta \in \mathbf{S}_{g,%
\mathbb{Z}}$ has $\underline{\beta }_{ij}$-component $\beta _{ij}\in \mathbb{%
Z}$ and define $q_{ij}:=q^{\underline{\beta }_{ij}}$. With these notations,
every global sections $f$\ of $\mathcal{O}_{X_{0}}$ admits, at every cusp, a 
$q$-expansion of the form%
\begin{equation*}
f\left( q\right) =\tsum\nolimits_{\beta \in \mathbf{S}_{g,\mathbb{Z}%
}^{even}}a_{f}\left( \beta \right) q^{N^{-1}\beta }\text{ for }%
q^{N^{-1}\beta }:=\tprod\nolimits_{1\leq i\leq j\leq g}q_{ij}^{N^{-1}\beta
_{ij}}
\end{equation*}%
for a suitable $N$. Furthermore, the $q$-expansion principle holds, up to
considering a cusp for each geometrically connected component of the
ordinary locus. (Here we follow the classical convention, in the higher
genus case, of using the parameter $q_{ij}=\mathrm{exp}\left( \pi
iz_{ij}\right) $ at the infinite cusp: using $q_{ii}^{2}=\mathrm{exp}\left(
2\pi iz_{ii}\right) $ when $i=j$ would give back the usual convention in the 
$g=1$ case, see $\left( \text{\ref{q-exp F Def}}\right) $). Then one can
define $d_{g}$ theta operators as follows: for every $1\leq i\leq j\leq g$,
we can consider the unique derivation $\theta _{ij}$ of $\mathcal{O}%
_{X_{0}}\left( X_{0}\right) $\ which on $q$-expansions is given by the
formulas $\theta _{ii}\left( q^{\beta }\right) =\frac{1}{2}\beta
_{ii}q^{\beta }$ and $\theta _{ij}\left( q^{\beta }\right) =\beta
_{ij}q^{\beta }$ for every $i<j$. More generally, choose any polynomial%
\begin{equation*}
P\in \mathbb{Z}_{p}\left[ T_{ij}:1\leq i\leq j\leq g\right]
\end{equation*}%
and, setting $\mathbf{\theta }=\left( \theta _{ij}\right) $, consider the
differential operator%
\begin{equation*}
\theta _{P}:=P\left( \mathbf{\theta }\right) :H^{0}\left( X_{0},\mathcal{O}%
_{X_{0}}\right) \longrightarrow H^{0}\left( X_{0},\mathcal{O}_{X_{0}}\right) 
\text{.}
\end{equation*}%
One checks that, if $P^{\prime }\left( T_{11},...,T_{gg},T_{ij}:i<j\right)
=P\left( 2^{-1}T_{11},...,2^{-1}T_{gg},T_{ij}:i<j\right) $, then%
\begin{equation*}
\theta _{P}\left( f\right) \left( q\right) =\tsum\nolimits_{\beta \in 
\mathbf{S}_{g,\mathbb{Z}}^{even}}a_{f}\left( \beta \right) P^{\prime }\left(
\beta \right) q^{\beta }
\end{equation*}%
(see Lemma \ref{q-exp L KS}) we have). It follows that, if%
\begin{equation*}
f^{\left[ P\right] }\left( q\right) :=\tsum\nolimits_{\beta \in \mathbf{S}%
_{g,\mathbb{Z}}^{even}:p\nmid P^{\prime }\left( \beta \right) }a_{f}\left(
\beta \right) q^{N^{-1}\beta }\text{,}
\end{equation*}%
and we write $H^{0}\left( X_{0},\mathcal{O}_{X_{0}}\right) ^{\left[ P\right]
}$ for the subset of those sections of $H^{0}\left( X_{0},\mathcal{O}%
_{X_{0}}\right) $ that are such that $f=f^{\left[ P\right] }$ (we say that $%
f $ is $P$-depleted, in ths case), then $\theta _{P}$ is a derivation which
is invertible on $H^{0}\left( X_{0},\mathcal{O}_{X_{0}}\right) ^{\left[ P%
\right] }$. So far we have introduced the representations $L_{\lambda }$ and 
$W_{\lambda }$. To be more precise, let us write $L_{\lambda _{/\mathbb{Q}}}$
and $W_{\lambda _{/\mathbb{Q}}}$ for an irreducible representation of
highest weight $\lambda _{/\mathbb{Q}}$\ of $\mathbf{G}_{/\mathbb{Q}}$ and,
respectively, $\mathbf{M_{/Q}}$ (they are unique up to isomorphism and an
isomorphism is uniquely determined up to a non-zero scalar factor). Then we
can define $L_{\lambda }:=Dist\left( \mathbf{G}\right) v_{\lambda _{/\mathbb{%
Q}}}$ and, respectively, $W_{\lambda }:=Dist\left( \mathbf{M}\right)
v_{\lambda _{/\mathbb{Q}}}$, where $Dist\left( \mathbf{G}\right) \subset
U\left( \mathfrak{g}_{/\mathbb{Q}}\right) $ denotes the distribution algebra
of $\mathbf{G}$ over $\mathbb{Z}$ and similarly for $\mathbf{M}$ and where $%
v_{\lambda _{/\mathbb{Q}}}\in W_{\lambda _{/\mathbb{Q}}}\subset L_{\lambda
_{/\mathbb{Q}}}$ is a highest weight vector, the latter inclusion being
canonical if we realize $L_{\lambda _{/\mathbb{Q}}}$ as a parabolic
induction $L_{\lambda _{/\mathbb{Q}}}=\mathrm{Ind}_{\mathbf{Q}^{-}}^{\mathbf{%
G}}\left( W_{\lambda _{/\mathbb{Q}}}\right) $. Then $L_{\lambda }\subset
L_{\lambda _{/\mathbb{Q}}}$ is a $\mathbf{G}$-stable lattice in $L_{\lambda
_{/\mathbb{Q}}}$ and similarly for $W_{\lambda }$. If we work over a ring $%
\Bbbk $ such that the level and the degree of the polarization are
invertible, we can associate to $W_{\lambda }$ and $L_{\lambda }$\ a sheaf $%
\mathcal{W}_{\lambda }$ and, respectively, a module with a connection $%
\left( \mathcal{L}_{\lambda },\nabla \right) $: we will assume that $\Bbbk $
is a Dedekind domain which contains appropriate roots of unity (in order to
have $q$-expansions) and that it contains $\mathbb{Z}_{p}$ when working over 
$X_{0}$. They can be defined over $X$ and, after a pull-back to $X_{0}$,
these sheaves can be trivialized and we can define $P$-depleted sections $%
H^{0}\left( X_{0},\mathcal{L}_{\lambda }\otimes _{\mathcal{O}_{X_{0}}}\Omega
_{X_{0}/\Bbbk }^{p}\right) ^{\left[ P\right] }$ similarly as we have done in
the $g=1$ case. (Here, if one wants to work with compactified objects, one
has really to use the appropriate logarithmic de Rham modules $\Omega
_{X_{0}/\Bbbk }^{p}$). Indeed, one can show that $H^{0}\left( X_{0},\mathcal{%
L}_{\lambda }\otimes _{\mathcal{O}_{X_{0}}}\Omega _{X_{0}/\Bbbk }^{\cdot
}\right) ^{\left[ P\right] }$ is a complex. We will prove the following
result (see Corollary \ref{Primitives CKey}).

\begin{theorem}
For every $P$ such that $P\left( 0\right) =0$, the complex $H^{0}\left(
X_{0},\mathcal{L}_{\lambda }\otimes _{\mathcal{O}_{X_{0}}}\Omega
_{X_{0}/\Bbbk }^{\cdot }\right) ^{\left[ P\right] }$ is acyclic in degree $%
p=1,...,d_{g}$.
\end{theorem}

\bigskip

Let us now sketch a proof of the above result. As suggested by the $g=1$
case, the first step is to replace $\mathcal{L}_{\lambda }$ by larger
sheaves, which can be done as follows. There is a BGG (dual) exact complex
whose first three terms are of the form%
\begin{equation*}
0\longrightarrow L_{\lambda _{/\mathbb{Q}}}\longrightarrow \mathrm{Ind}_{%
\mathbf{Q}_{/\mathbb{Q}}^{-}}^{\mathbf{G}_{/\mathbb{Q}}}\left( W_{\lambda _{/%
\mathbb{Q}}}\right) \left[ \mathbf{Y}_{/\mathbb{Q}}\right] \overset{d}{%
\longrightarrow }\tbigoplus\nolimits_{w\in W^{\mathbf{M}}:l\left( w\right)
=1}\mathrm{Ind}_{\mathbf{Q}_{/\mathbb{Q}}^{-}}^{\mathbf{G}_{/\mathbb{Q}%
}}\left( W_{w\cdot \lambda _{/\mathbb{Q}}}\right) \left[ \mathbf{Y}_{/%
\mathbb{Q}}\right] \text{.}
\end{equation*}%
Here, $\mathrm{Ind}_{\mathbf{Q}_{/\mathbb{Q}}^{-}}^{\mathbf{G}_{/\mathbb{Q}%
}}\left( W_{\mu }\right) \left[ \mathbf{Y}_{/\mathbb{Q}}\right] $ is a $%
\left( \mathfrak{g}_{/\mathbb{Q}},\mathbf{Q}_{/\mathbb{Q}}\right) $-module.
Indeed, we prove that this exact sequence admits a model%
\begin{equation}
0\longrightarrow L_{\lambda }\longrightarrow \mathrm{Ind}_{\mathbf{Q}^{-}}^{%
\mathbf{G}}\left( W_{\lambda }\right) \left[ \mathbf{Y}\right] \overset{d}{%
\longrightarrow }\tbigoplus\nolimits_{w\in W^{\mathbf{M}}:l\left( w\right)
=1}\mathrm{Ind}_{\mathbf{Q}^{-}}^{\mathbf{G}}\left( W_{w\cdot \lambda
}\right) \left[ \mathbf{Y}\right]  \label{Intro F BGG}
\end{equation}%
over $\mathbb{Z}$. For $p$-small weights, this could be probably deduced by
dualizing the results of \cite{PoTi18}, which give models over $\mathbb{Z}$
of the entire (non-dual) weak BGG complex (note, however, that we will need
a model of the truncated strong dual BGG complex). However, such a kind of
condition is not pleasant in our setting because, in the applications to the
calculation of primitives appearing in the proof of reciprocity laws, one
usually thinks of $p$ as being fixed and $\lambda $ as varying. More
precisely, we show in Theorem \ref{Representations T BGG} that these kind of
models exist for every split reductive group $\mathbf{G}$\ over a Dedekind
domain with characteristic zero fraction field relative to any parabolic
subgroup $\mathbf{Q}$, at least if one is satisfied with a model of $%
L_{\lambda _{/\mathbb{Q}}}$ as a $\left( \mathfrak{g},\mathbf{Q}\right) $%
-module and then we verify that $L_{\lambda }$ is actually a $\mathbf{G}$%
-module in our symplectic case (see Proposition \ref{Representations P Model}%
). (Strictly speaking, because as explained just below one can associate to
every $\left( \mathfrak{g},\mathbf{Q}\right) $-module a sheaf with a
connection, for our purposes it is not so important to know that $L_{\lambda
}$ is more than a $\left( \mathfrak{g},\mathbf{Q}\right) $-module). At the
level of underlying $\mathbf{M}$-modules these induced modules $\mathrm{Ind}%
_{\mathbf{Q}^{-}}^{\mathbf{G}}\left( W_{\mu }\right) \left[ \mathbf{Y}\right]
$\ can be identified with a tensor product $W_{\mu }\otimes _{\mathbb{Z}}%
\mathbb{Z}\left[ \mathbf{Y}\right] $ where $\mathbb{Z}\left[ \mathbf{Y}%
\right] $ denotes the free polynomial ring in the $d_{g}$ variables $Y_{i,j}$
for $1\leq i\leq j\leq g$ that one can think of as being parameters for the
unipotent subgroup $\mathbf{U\subset Q}$ (which is isomorphic to the
symmetric $g$-by-$g$ matrices): the action of $\gamma =\left( 
\begin{array}{cc}
a & b \\ 
0 & d%
\end{array}%
\right) $\ (with $a,b,c\in \mathbf{M}_{g}$) on $\mathbb{Z}\left[ \mathbf{Y}%
\right] $ sends a symmetric matrix $Y$ to $a^{-1}\left( b+Yd\right) $.
Thanks to a beautiful construction due to Z. Liu, we can associate to the
above exact sequence of $\left( \mathfrak{g},\mathbf{Q}\right) $-modules\ an
exact sequence of sheaves with a connection%
\begin{equation}
0\longrightarrow \left( \mathcal{L}_{\lambda },\nabla \right)
\longrightarrow \left( \mathcal{V}_{\lambda },\nabla \right) \overset{d}{%
\longrightarrow }\tbigoplus\nolimits_{w\in W^{\mathbf{M}}:l\left( w\right)
=1}\left( \mathcal{V}_{w\cdot \lambda },\nabla \right) \text{.}
\label{Intro F BGG Sheaves}
\end{equation}%
We recall this result in Theorem \ref{Sheaves T1}, briefly explaining why
Liu's construction extends from the field case to the case of a Dedekind
domain. (The main point is that the underlying $\Bbbk $-modules of the
representations of interest to us are flat and, hence, as representations of 
$\mathbf{Q}$ they are a colimit of representations whose underlying $\Bbbk $%
-module is finitely generated and projective and these kind of
representations can be obtained from the standard representation in the
usual way). Using the inclusion appearing in $\left( \text{\ref{Intro F BGG
Sheaves}}\right) $ one is essentially reduced, as in the $g=1$ case, to work
with the sheaves with connection $\left( \mathcal{V}_{\lambda },\nabla
\right) $ which \textquotedblleft looks always the same for varying $\lambda 
$'s\textquotedblright\ (being modelled on $\mathbb{Z}\left[ \mathbf{Y}\right]
$). Here we say \textquotedblleft essentially\textquotedblright\ because, as
we are going to explain, the above sheaves admits filtrations that are split
after a pull-back to the Igusa tower and we would like the first inclusion
appearing in $\left( \text{\ref{Intro F BGG Sheaves}}\right) $ to be
compatible with these splitting without introducing denominators. In the
case $g=1$, this problem does not exist because the natural basis of $%
\mathrm{Sym}_{2}^{k-2}$ is a subset of that of $\mathrm{Ind}_{\mathbf{Q}%
^{-}}^{\mathbf{G}}\left( e_{1}^{k-2}\right) \left[ Y\right] $.

To this this end, we first remark that, writing $\mathbb{Z}\left[ \mathbf{Y}%
\right] _{\leq r}$ (resp. $\mathbb{Z}\left[ \mathbf{Y}\right] _{=r}$) for
the space of polynomials that are of degree $\leq r$ (resp. $=r$), then 
\textrm{Fil}$_{r}\left( W_{\mu }\otimes _{\mathbb{Z}}\mathbb{Z}\left[ 
\mathbf{Y}\right] \right) :=W_{\mu }\otimes _{\mathbb{Z}}\mathbb{Z}\left[ 
\mathbf{Y}\right] _{\leq r}$ is a $\mathbf{Q}$-submodule of $W_{\mu }\otimes
_{\mathbb{Z}}\mathbb{Z}\left[ \mathbf{Y}\right] $. Because one is able to
attach to a $\mathbf{Q}$-module a sheaf, one gets an increasing filtration $%
\left\{ \mathrm{Fil}_{r}\left( \mathcal{V}_{\lambda }\right) \right\} _{r\in 
\mathbb{N}}$ by taking the sheaf associated to \textrm{Fil}$_{r}\left(
W_{\mu }\otimes _{\mathbb{Z}}\mathbb{Z}\left[ \mathbf{Y}\right] \right) $.
(Although not needed, it turns out that the opposite filtration $\mathrm{Fil}%
^{r}\left( \mathcal{V}_{\lambda }\right) :=\mathrm{Fil}_{-r}\left( \mathcal{V%
}_{\lambda }\right) $ intersected with $\mathcal{L}_{\lambda }$ equals the
Hodge filtration, up to a translation). This filtration is split, namely one
has%
\begin{equation}
\mathrm{Fil}_{i}\left( W_{\mu }\otimes _{\mathbb{Z}}\mathbb{Z}\left[ \mathbf{%
Y}\right] \right) =\bigoplus\nolimits_{j\in \mathbb{N}}\mathrm{gr}%
_{i-j}\left( W_{\mu }\otimes _{\mathbb{Z}}\mathbb{Z}\left[ \mathbf{Y}\right]
\right) \text{ with }\mathrm{gr}_{i-j}\left( W_{\mu }\otimes _{\mathbb{Z}}%
\mathbb{Z}\left[ \mathbf{Y}\right] \right) =W_{\mu }\otimes _{\mathbb{Z}}%
\mathbb{Z}\left[ \mathbf{Y}\right] _{=i-j}\text{.}  \label{Intro F DecAlg}
\end{equation}%
When working over the Igusa tower, one can again split $\mathcal{H}_{\mathrm{%
dR}}^{1}$ by means of a canonical basis $\left\{ \delta _{1},...,\delta
_{g},\eta _{1},...,\eta _{g}\right\} $ that one can use in order to
introduce $P$-depleted de Rham complexes by imposing componentisely this
condition. Furthermore, the span $\mathcal{H}_{\mathrm{dR}}^{1,\varphi =1}$
of the $\eta _{i}$'s yields the so called unit root splitting $\mathcal{H}_{%
\mathrm{dR}}^{1}=\omega _{\mathrm{dR}}\oplus \mathcal{H}_{\mathrm{dR}%
}^{1,\varphi =1}$ of $\mathcal{H}_{\mathrm{dR}}^{1}$. Although the
relationship between the standard representation $\mathrm{Std}_{2g}$ and the
representation $\mathrm{J}:=\mathbb{Z}\left[ \mathbf{Y}\right] _{\leq 1}$
viewed as representations of $\mathbf{Q}$\ is not so transparent as in the $%
g=1$ case (where $\left( \mathrm{Std}_{2}\otimes _{\mathbb{Z}}\mathrm{Std}%
_{1}\right) ^{\vee }=\mathrm{J}$ for $\mathrm{Std}_{1}=\mathbb{Z}e_{1}$),
one can identify $\mathrm{J}$ as being obtained from the standard exact
sequence having $\mathrm{Std}_{2g}$ as a middle term by taking a pull-back,
then a push-out and then a dual as $\mathbf{Q}$-representations (see
Proposition \ref{Representations P Spl}): this allow ones to spread out the
unit root splitting from $\mathcal{H}_{\mathrm{dR}}^{1}$ to every $\mathcal{V%
}_{\lambda }$. Now, the unit root splitting of $\mathcal{H}_{\mathrm{dR}%
}^{1} $\ comes from the splitting of $\mathrm{Std}_{2g}$ under the action of
an element $m_{0}\in \mathbf{M}$ which acts on $\mathbb{Z}\left[ \mathbf{Y}%
\right] _{=r}$ as the multiplication by $p^{-r}$: we deduce from this fact
that the unit root splitting of $\mathcal{V}_{\mu }$ is a splitting of the
filtration $\left\{ \mathrm{Fil}_{r}\left( \mathcal{V}_{\mu }\right)
\right\} _{r\in \mathbb{N}}$. Setting $\mathcal{W}_{\mu ,i}:=\mathcal{O}%
_{X_{0}}\otimes _{\mathbb{Z}}W_{\mu }\otimes _{\mathbb{Z}}\mathbb{Z}\left[ 
\mathbf{Y}\right] _{=i}$, we have%
\begin{equation}
\mathrm{Fil}_{i}\left( \mathcal{V}_{\mu }\right) \simeq
\bigoplus\nolimits_{j\in \mathbb{N}}\mathcal{W}_{\mu ,j-i}\text{ and }%
\mathcal{V}_{\mu }\simeq \bigoplus\nolimits_{j\in \mathbb{N}}\mathcal{W}%
_{\mu ,i}  \label{Intro F DecSh}
\end{equation}%
(see Lemma \ref{Primitives L2}). Of course, $\mathcal{L}_{\lambda }$
receives a filtration from $\mathcal{V}_{\lambda }$ by setting $\mathrm{Fil}%
_{r}\left( \mathcal{L}_{\lambda }\right) :=\mathcal{L}_{\lambda }\cap 
\mathrm{Fil}_{i}\left( \mathcal{V}_{\lambda }\right) $ and this filtration
is split, but only a priori at the cost of introducing denominators. In
order to make everything work at an integral level, we realize the degree
filtration in terms of the underlying $\mathbf{Q}$-modules as follows. For
every algebraic representation $M$ of $\mathbf{U}$ (i.e. an $\mathcal{O}%
\left( \mathbf{U}\right) $-comodule), define a filtration by setting $%
\mathrm{Fil}_{0}\left( M\right) :=M^{\mathbf{U}}$ and, by recursion, define $%
\mathrm{Fil}_{r}\left( M\right) $ to be the inverse image of $\left( \frac{M%
}{\mathrm{Fil}_{r-1}\left( M\right) }\right) ^{\mathbf{U}}$ in $M$. Then the
degree filtration is obtained in this way and, because taking the $\mathbf{U}
$-invariant is a left exact operation (indeed exact, for unipotent groups), $%
\left( \text{\ref{Intro F BGG}}\right) $ and, hence, $\left( \text{\ref%
{Intro F BGG Sheaves}}\right) $\ is promoted to an exact sequence of
filtered objects. Because $W_{\mu }\otimes _{\mathbb{Z}}\mathbb{Z}\left[ 
\mathbf{Y}\right] _{=i}$ is realized as an isotypic component under the
action of $m_{0}$ and $\mathcal{W}_{\mu ,i}:=\mathcal{O}_{X_{0}}\otimes _{%
\mathbb{Z}}W_{\mu }\otimes _{\mathbb{Z}}\mathbb{Z}\left[ \mathbf{Y}\right]
_{=i}$, it follows from $\left( \text{\ref{Intro F BGG Sheaves}}\right) $
viewed as a sequence of filtered objects that, because these isotypic
components split the second and the third term of the sequence thanks to $%
\left( \text{\ref{Intro F DecSh}}\right) $, then they also provide a
splitting of the first term. In this way, we get an analogous decomposition
of $\mathcal{L}_{\lambda }$ and, hence, we can really forget about the
sheaves $\left( \mathcal{L}_{\lambda },\nabla \right) $ and work in $\left( 
\mathcal{V}_{\lambda },\nabla \right) $.

The next step consists in realizing that, as in the $g=1$ setting, using the
decomposition $\left( \text{\ref{Intro F DecSh}}\right) $ one can still write%
\begin{equation*}
\nabla ^{p}:H^{0}\left( X_{0},\mathcal{V}_{\lambda }\otimes _{\mathcal{O}%
_{X_{0}}}\Omega _{X_{0}/\Bbbk }^{p}\right) \rightarrow H^{0}\left( X_{0},%
\mathcal{V}_{\lambda }\otimes _{\mathcal{O}_{X_{0}}}\Omega _{X_{0}/\Bbbk
}^{p+1}\right) \text{,}
\end{equation*}%
when restricted to $H^{0}\left( X_{0},\mathcal{W}_{\lambda ,i}\otimes _{%
\mathcal{O}_{X_{0}}}\Omega _{X_{0}/\Bbbk }^{p}\right) $, as a sum of two
contributions%
\begin{equation*}
\nabla ^{p}\left( F\right) =\Theta ^{p}\left( F\right) +\Delta
_{i}^{p}\left( F\right) \in H^{0}\left( X_{0},\mathcal{W}_{\lambda
,i}\otimes _{\mathcal{O}_{X}}\Omega _{X/\Bbbk }^{p+1}\right) \oplus
H^{0}\left( X_{0},\mathcal{W}_{\lambda ,i+1}\otimes _{\mathcal{O}_{X}}\Omega
_{X/\Bbbk }^{p+1}\right)
\end{equation*}%
where $\Theta ^{p}$ gives the \textquotedblleft
horizontal\textquotedblright\ contribution of $\nabla ^{p}$ whereas, as in
the case $g=1$, the \textquotedblleft vertical\textquotedblright\
contribution $\Delta _{i}^{p}$ is obtained by $\mathcal{O}_{X_{0}}$-linear
extension of operators in Lie algebra $\mathfrak{u}^{-}$ of the unipotent
radical $\mathbf{U}^{-}$\ of $\mathbf{Q}^{-}$\ operating on $\mathrm{Ind}_{%
\mathbf{Q}^{-}}^{\mathbf{G}}\left( W_{\lambda }\right) \left[ \mathbf{Y}%
\right] $ (see Lemma \ref{Primitives LKey}). This relies on two ingredients.
First, by Liu's construction, one is reduced to understand the Gauss-Manin
connection. In this case, the claim is equivalent to verifying that the
expression of the Gauss-Manin connection in terms of the basis $\left\{
\delta _{1},...,\delta _{g},\eta _{1},...,\eta _{g}\right\} $ at the cusps
is obtained as follows. Let us write $\partial _{ij}$ for the matrix in $%
\mathfrak{u}^{-}$ whose whose lower left $g$-by-$g$ entry is the symmetric
matrix whose upper triangular part has all zero except a $1$ in $\left(
i,j\right) $-entry. Then%
\begin{equation*}
\nabla \left( \theta _{i,j}\right) \left( \delta _{1},...,\delta _{g},\eta
_{1},...,\eta _{g}\right) ^{t}=\partial _{ij}\left( \delta _{1},...,\delta
_{g},\eta _{1},...,\eta _{g}\right) ^{t}\text{.}
\end{equation*}%
This is proved by T. J. Fonseca\ in \cite[Proposition 5.17, Remark 5.18 and
Theorem 6.4]{Fo23} (although the lower left $g$-by-$g$ entry is essentially
the Kodaira-Spences isomorphism at the cusps). (In Lemma \ref{de Rham L1} we
explain why one can also include a paramodular level).

As we have already said, the sheaves with connection $\left( \mathcal{V}%
_{\lambda },\nabla \right) $ which \textquotedblleft looks always the same
for varying $\lambda $'s\textquotedblright . In order to formalize this
fact, which is an empty statement in the $g=1$ because there is only one $%
\Theta ^{p}=\Theta ^{0}$, we show that $\Theta ^{p}\circ \Theta ^{p-1}=0$
and, furthermore, that there is an isomorphism of complexes%
\begin{equation*}
\left( H^{0}\left( X_{0},\mathcal{W}_{\lambda ,i}\otimes _{\mathcal{O}%
_{X}}\Omega _{X/\Bbbk }^{\cdot }\right) ,\Theta ^{\cdot }\right) \simeq
\left( W_{\lambda ,i}\otimes _{\Bbbk }H^{0}\left( X_{0},\Omega _{X/\Bbbk
}^{\cdot }\right) ,1\otimes _{\Bbbk }d^{\cdot }\right) \text{.}
\end{equation*}%
(see Lemma \ref{Primitives L3}). Everything we have written so far holds by
placing an upperscript $\left( -\right) ^{\left[ P\right] }$ everywhere.

As a final step, thanks to the above discussion, one is reduced to showing
that the complex $H^{0}\left( X_{0},\Omega _{X/\Bbbk }^{\cdot }\right) ^{%
\left[ P\right] }$ is acyclic for every $P$ such that $P\left( 0\right) =0$.
To this end, we interpret $\theta _{P}$ as a kind of \textquotedblleft
Laplace operator\textquotedblright : indeed, for every such polynomial, we
construct an operator $\partial ^{\cdot }$ of degree $-1$ satisfying%
\begin{equation*}
d^{\cdot }\circ \partial ^{\cdot }+\partial ^{\cdot }\circ d^{\cdot }=\theta
_{P}\text{.}
\end{equation*}%
In particular, restricting the above equality to cocycles yields $d^{\cdot
}\circ \partial ^{\cdot }=\theta _{P}$ and, because $\theta _{P}$ is
invertible on $P$-depleted sections, the wanted acyclicity follows. We
remark that this result is deduced in \S \ref{Primitives PKey} by a very
general method.

\begin{remark}
We have so far remained silent on one hypothesis. Because of Proposition \ref%
{Representations P Spl}, we must assume that the prime $p$ is not $2$. If $%
p=2$, the acyclicity result is true a priori only after extending the
scalars to $\mathbb{Q}_{p}$.
\end{remark}

\begin{remark}
\label{Intro R1}The precise compact subgroups of $\mathbf{G}\left( \mathbb{A}%
_{f}\right) $ that we admit are defined at the beginning of \S \ref{S
Sheaves}: this includes the more basic paramodular levels. At the cost of
slightly extending the calculation on $q$-expansions of \S \ref{S q-exp},
one can also include the more general levels of Remark \ref{Sheaves R PLevel}%
. The cases where $p$ is allowed to divide either the level or the degree of
the polarization are much more delicate, as one needs models for which Liu's
construction works (cfr. also Remark \ref{Sheaves R Ex}).
\end{remark}

\begin{remark}
The method described in this paper is rather general and we expect that it
could be extended to cover more general PEL setting, at least as long as the
ordinary locus is non-empty.
\end{remark}

\bigskip

\textbf{Acknowledgments.} It is a pleasure to thank F. Andreatta, M.
Bertolini and R. Venerucci: without them, this work would never have existed.

\section{\label{S Representations}Integral truncated dual BGG complexes and
representations of symplectic groups}

\subsection{\label{S Representations Pre}Preliminary results and notations}

Suppose that $\mathbf{G}$ is an affine and flat group scheme over a ring $%
\Bbbk $ (all schemes in \S \ref{S Representations} will be understood to be
defined over $\Bbbk $, except when differently stated) and that $\mathbf{Q}%
\subset \mathbf{G}$ is a flat subgroup scheme.

We say that a functor $V$\ on $\Bbbk $-algebras is valued in modules if, for
every morphism of $\Bbbk $-algebras $R\rightarrow R^{\prime }$, then $%
V\left( R\right) =V_{R}$ is an $R$-module and $V_{R}\rightarrow V_{R^{\prime
}}$ is a morphism of $R$-modules. Denote by $\underline{\mathrm{Mod}}$ the
category of functors on the $\Bbbk $-algebras $R$ that are module valued and
such that $V_{R}$ defines a sheaf on the affine Zariski site of $\mathfrak{%
Spec}\left( R\right) $\ (i.e. with coverings $\left\{ \mathfrak{Spec}\left(
R_{f}\right) \rightarrow \mathfrak{Spec}\left( R\right) \right\} _{f\in R}$%
). Let $\underline{\mathrm{Mod}}_{alg}$ be the full subcategory of those $%
V\in \underline{\mathrm{Mod}}$ having the property that the canonical
morphism $R^{\prime }\otimes _{R}V_{R}\rightarrow V_{R^{\prime }}$ is an
isomorphism for every $R$-algebra $R^{\prime }$. (Of course, the above
Zariski sheaf condition could be omitted in the definition of $\underline{%
\mathrm{Mod}}_{alg}$). If $V\in \underline{\mathrm{Mod}}_{alg}$, we will
call $V_{\Bbbk }$ the underlying $\Bbbk $-module. Write $\underline{\mathrm{%
Mod}}_{f}$ for the full subcategory of those $V\in \underline{\mathrm{Mod}}%
_{alg}$ with the property that $V_{\Bbbk }$ is finitely generated and
projective. If we need to specify the base $\Bbbk $, we will write $%
\underline{\mathrm{Mod}}_{?}\left( \Bbbk \right) $. The functor taking the
Lie algebra $\mathfrak{g}_{R}:=\mathrm{ker}\left( \mathbf{G}\left( R\left[
\varepsilon \right] \right) \rightarrow \mathbf{G}\left( R\right) \right) $
of $\mathbf{G}$ is an example of functor in $\underline{\mathrm{Mod}}_{alg}$
(see \cite[Proposition 3.4]{Dm70}) and $\mathfrak{g}_{R}$ is naturally a Lie
algebra over $R$ (see \cite[Proposition 3.4, \S 3.5.1, Proposition 3.6,
p.68-70 and Corollaire 1 to Proposition 4.8]{Dm70}, taking into account that 
$\mathbf{G}$, being representable, is a \textquotedblleft bon $R$%
-group\textquotedblright\ by \cite[Example after Definitoin 4.6, based on
Proposition 2.2 and Proposition 3.3]{Dm70}). Writing $\mathfrak{q}$ for the
functors obtained from $\mathbf{Q}$, we have that $\mathfrak{q}\subset 
\mathfrak{g}$ is a subfunctor such that $\mathfrak{q}_{R}\subset \mathfrak{g}%
_{R}$ is a Lie subalgebra. We have a natural notion of exactness in $%
\underline{\mathrm{Mod}}_{alg}$ making it an abelian category obtained by
looking at the underlying $\Bbbk $-module $V_{\Bbbk }$. We note that this
notion of exactness is different, in general, from the notion of exactness
for module valued functors: an exact sequence in $\underline{\mathrm{Mod}}%
_{alg}$ gives rise to an exact sequence of $R$-modules when evaluated at $R$%
-points for every $\Bbbk $-flat algebra $R$ (hence every $R$, when $\Bbbk $
is field); when $\Bbbk $ is not a field, to be a exact as a functor is in
general a stronger condition.

We let $\mathrm{R}$\textrm{ep}$\left( \mathbf{G}\right) $ (resp. $\mathrm{R}$%
\textrm{ep}$\left( \mathfrak{g}\right) $) be the category of couples $\left(
\rho ,V\right) $ (resp. $\left( \delta ,V\right) $), sometimes denoted
simply by $\rho $ or $V$ (resp. $\delta $ or $V$) with the property that $%
\rho :\mathbf{G}\rightarrow Aut\left( V\right) $ (resp. $\delta :\mathfrak{g}%
\rightarrow End\left( V\right) $) is a morphism of functors such that $\rho
_{R}:\mathbf{G}\left( R\right) \rightarrow Aut_{R}\left( V_{R}\right) $ is a
group homomorphism (resp. $\delta _{R}:\mathfrak{g}_{R}\rightarrow
End_{R}\left( V_{R}\right) $ is a morphism of Lie algebras) for every $\Bbbk 
$-algebra $R$. Let us write $\mathrm{R}$\textrm{ep}$_{?}\left( \mathbf{G}%
\right) $ (resp. $\mathrm{R}$\textrm{ep}$_{?}\left( \mathfrak{g}\right) $)
for the full-subcategory consisting of those $\left( \sigma ,V\right) $
(resp. $\left( \delta ,V\right) $) that are such that $V\in \underline{%
\mathrm{Mod}}_{?}$: for example, writing \textrm{Ad}$_{\mathbf{G}}:\mathbf{G}%
\rightarrow Aut\left( \mathfrak{g}\right) $ for the action obtained by
conjugation (via $\mathbf{G}\left( R\left[ \varepsilon \right] \right)
\simeq \mathbf{G}\left( R\right) \ltimes \mathfrak{g}_{R}$), we have $%
\mathfrak{g}\in \mathrm{R}$\textrm{ep}$_{f}\left( \mathbf{G}\right) $. The
action of the objects of $\mathrm{R}$\textrm{ep}$_{alg}\left( \mathbf{G}%
\right) $ can be differentiated: this gives rise to a functor from $\mathrm{R%
}$\textrm{ep}$_{alg}\left( \mathbf{G}\right) $ to $\mathrm{R}$\textrm{ep}$%
_{alg}\left( \mathfrak{g}\right) $ sending $\left( \rho ,V\right) $ to $%
\left( d\rho ,V\right) $. We can refine this operation as follows. Let $%
\mathrm{R}$\textrm{ep}$_{alg}\left( \mathfrak{g},\mathbf{Q}\right) $ be the
category of triples $\left( \delta ,\rho ,V\right) $ such that $V\in 
\underline{\mathrm{Mod}}_{alg}$, $d\rho =\delta _{\mid \mathfrak{u}}$ and $%
\rho _{R}\left( q\right) \delta _{R}\left( X\right) \rho _{R}\left( q\right)
^{-1}=\delta _{R}\left( \mathrm{Ad}_{\mathbf{G},R}\left( q\right) \left(
X\right) \right) $ holds in $End_{R}\left( V_{R}\right) $\ for every $q\in 
\mathbf{Q}\left( R\right) $ and $X\in \mathfrak{g}_{R}$. Then, sending $%
\left( \rho ,V\right) $ to $\left( d\rho ,\rho _{\mid \mathbf{Q}},V\right) $
yields a functor from $\mathrm{R}$\textrm{ep}$_{alg}\left( \mathbf{G}\right) 
$ to $\mathrm{R}$\textrm{ep}$_{alg}\left( \mathfrak{g},\mathbf{Q}\right) $.
Unless otherwise stated, the categories $\mathrm{R}$\textrm{ep}$_{alg}\left(
?\right) $ will be regarded as exact categories by looking at the underlying
object in $\underline{\mathrm{Mod}}_{alg}$, i.e. the underlying $\Bbbk $%
-module: they are abelian because the (co)kernels in $\underline{\mathrm{Mod}%
}_{alg}$ really belongs to $\mathrm{R}$\textrm{ep}$_{alg}\left( ?\right) $,
as we are going to explain. The category $\mathrm{R}$\textrm{ep}$%
_{alg}\left( \mathfrak{g}\right) $ is an abelian category just because the $%
\mathfrak{g}_{\Bbbk }$-action on the underlying $\Bbbk $-modules $V_{\Bbbk }$%
\ of the (co)kernels $V$ of a morphism in $\mathrm{R}$\textrm{ep}$%
_{alg}\left( \mathfrak{g}\right) $\ extends by linearity to an action of $%
\mathfrak{g}_{R}=R\otimes _{\Bbbk }\mathfrak{g}_{\Bbbk }$ on $V_{R}=R\otimes
_{\Bbbk }V_{\Bbbk }$ for every $\Bbbk $-algebra $R$. The category of $%
\mathcal{O}\left( \mathbf{G}\right) $-comodules is identified with $\mathrm{R%
}$\textrm{ep}$_{alg}\left( \mathbf{G}\right) $ (as $\mathbf{G}$ is affine),
which is an abelian category (due to our flatness assumption, see \cite{Se68}%
). Combining these two arguments, we deduce that $\mathrm{R}$\textrm{ep}$%
_{alg}\left( \mathfrak{g},\mathbf{Q}\right) $ is also abelian.

The following result will be crucial in the proof of Theorem \ref{Sheaves T1}
below.

\begin{theorem}
\label{Representations T1}The following facts hold.

\begin{itemize}
\item[$\left( 1\right) $] If $\left( \rho ,V\right) \in \mathrm{Rep}%
_{alg}\left( \mathbf{G}\right) $ is a representation of $\mathbf{G}$ which
whose underlying $\Bbbk $-module is free which is given by a closed immersion%
\begin{equation*}
\rho :\mathbf{G}\longrightarrow \mathbf{GL}_{V}\text{,}
\end{equation*}%
then every $W\in \mathrm{Rep}_{f}\left( \mathbf{G}\right) $ can be
constructed from $V$ by forming tensor products, direct sums, duals, and
subquotients.

\item[$\left( 2\right) $] If $\Bbbk $ is noetherian, then every object of $%
\mathrm{Rep}_{alg}\left( \mathbf{G}\right) $, regarded as an $\mathcal{O}%
\left( \mathbf{G}\right) $-comodule, is a filtered union of its $\mathcal{O}%
\left( \mathbf{G}\right) $-subcomodules that are of finite type over $\Bbbk $%
. In particular, when $\Bbbk $ is a Dedekind domain, then every $\mathrm{Rep}%
_{alg}\left( \mathbf{G}\right) $ whose underlying $\Bbbk $-module is torsion
free is a filtered union of its subobjects that are in $\mathrm{Rep}%
_{f}\left( \mathbf{G}\right) $.
\end{itemize}
\end{theorem}

\begin{proof}
$\left( 1\right) $ This follows from the fact that the proof of \cite[%
Theorem 4.14]{Mi17} holds in this more general setting. Indeed, first we can
assume that $W_{\Bbbk }$ is free of finite rank, by adding a $\Bbbk $-module
promoted to an object of $\mathbf{G}$ by menas of the trivial action. Then 
\cite[Corollary 4.13]{Mi17} holds unchanged, implying that $W$ is an $%
\mathcal{O}\left( \mathbf{G}\right) $-subcomodule of a finite direct sum of
copies of $\mathcal{O}\left( \mathbf{G}\right) :=\left( \mathcal{O}\left( 
\mathbf{G}\right) ,\Delta \right) $ (here is where we use the $\Bbbk $%
-freeness). Then, with the notations of loc.cit. we can replace $W$ by its
image $W_{i}$ in $\mathcal{O}\left( \mathbf{G}\right) $ and, hence, suppose
that $W\subset \mathcal{O}\left( \mathbf{G}\right) $ is still a $\Bbbk $%
-module of finite type (but no longer free). Then the proof is the same.

$\left( 2\right) $ See \cite[\S 1.5 Corollaire]{Se68}.
\end{proof}

We will need to consider the operation of taking $\mathbf{G}$-invariants. On
the one hand, for every $\left( \rho ,V\right) \in \mathrm{R}$\textrm{ep}$%
\left( \mathbf{G}\right) $, we can consider the functor of fixed points $%
\left( \rho ,V\right) ^{h_{\mathbf{G}}}=V^{h_{\mathbf{G}}}$ such that $%
\left( \rho ,V\right) ^{\mathbf{G}}\left( R\right) $ is the set of those $%
v\in V_{R}$ such that $v_{R^{\prime }}\in V_{R^{\prime }}^{\mathbf{G}\left(
R^{\prime }\right) }$ for every $R$-algebra $R^{\prime }$, where $%
v_{R^{\prime }}$ denotes the image of $v$ in $V_{R^{\prime }}$. On the othe
hand, when $\mathbf{G}$ is affine and $\left( \rho ,V\right) \in \mathrm{Rep}%
_{alg}\left( \mathbf{G}\right) $, we can consider $V_{\Bbbk }^{\mathbf{G}%
}\subset V_{\Bbbk }$, defined to be the set of those $v\in V_{\Bbbk }$ such
that $\rho ^{\#}\left( v\right) =v\otimes _{\Bbbk }1$ if $\rho
^{\#}:V_{\Bbbk }\rightarrow V_{\Bbbk }\otimes _{\Bbbk }\mathcal{O}\left( 
\mathbf{G}\right) $ denotes the comodule structure: writing $\mu _{V_{\Bbbk
}^{\mathbf{G}}}:V_{\Bbbk }^{\mathbf{G}}\rightarrow V_{\Bbbk }^{\mathbf{G}%
}\otimes _{\Bbbk }\mathcal{O}\left( \mathbf{G}\right) $ for the canonical
morphism sending $v$ to $v\otimes _{\Bbbk }1$, the inclusion $V_{\Bbbk }^{%
\mathbf{G}}\subset V_{\Bbbk }$ yields a comodule morphism from $\left( \mu
_{V_{\Bbbk }^{\mathbf{G}}},V_{\Bbbk }^{\mathbf{G}}\right) $ to $\left( \rho
,V\right) $. Let us write $\left( \rho ,V\right) ^{\mathbf{G}}=V^{\mathbf{G}%
}\in \mathrm{R}$\textrm{ep}$_{alg}\left( \mathbf{G}\right) $ for the
corresponding functor. There is always a morphism of functors $V^{\mathbf{G}%
}\rightarrow V^{h_{\mathbf{G}}}$ and, for every $\Bbbk $-flat algebra $R$
(hence $R=\Bbbk $ and every $R$, when $\Bbbk $ is field), one has%
\begin{equation}
V^{\mathbf{G}}\left( R\right) \overset{\sim }{\longrightarrow }V^{h_{\mathbf{%
G}}}\left( R\right)  \label{Representations F Inv}
\end{equation}%
(see \cite[$\left( 2.10\right) $]{Jan07} or \cite[proof of Proposition 4.33]%
{Mi17}). We note that, when $\mathbf{G}$ is affine, irreducible and
connected and $\Bbbk $ is a infinite perfect field, then $V_{\Bbbk }^{%
\mathbf{G}}=\left( V_{\Bbbk }\right) ^{\mathbf{G}\left( \Bbbk \right) }$
(see \cite[After Proposition 4.32]{Mi17}).

\bigskip

Suppose that $X$ is a scheme and that $V$ is a module valued functor on the $%
\Bbbk $-algebras $R$ (for example, suppose that $V\in \underline{\mathrm{Mod}%
}$). If $U\subset X_{/R}$ is an open subset, we let $\mathrm{C}^{\mathrm{alg}%
}\left( U,V\right) $ be the module valued functor on $R$-algebras such that $%
\mathrm{C}^{\mathrm{alg}}\left( U,V\right) _{R^{\prime }}$ is the set of
morphisms of functors $f:U_{/R^{\prime }}\rightarrow V_{/R^{\prime }}$,
where $Y_{/R^{\prime }}$ denotes the base change $Y_{/R^{\prime }}\left(
R^{\prime \prime }\right) :=Y\left( R^{\prime \prime }\right) $ on $%
R^{\prime }$-algebras. According to Lemma \ref{Representations L K(X)} below
we have, indeed, that it belongs to $\underline{\mathrm{Mod}}_{alg}$ when $%
U\subset X$ is affine and $V\in \underline{\mathrm{Mod}}_{f}$. In
particular, when $U\subset \mathbf{G}$ is such that $U\mathbf{Q}\subset U$,
it makes sense to consider the left action by right multiplication of $%
\mathbf{Q}$ on $\mathrm{C}^{\mathrm{alg}}\left( U,V\right) $) and we get an
object of $\mathrm{Rep}_{alg}\left( \mathbf{Q}\right) $. Indeed, essentially
due to the fact that derivations localizes, we can enrich the structure of $%
\mathrm{C}^{\mathrm{alg}}\left( U,V\right) $ to that of an object in $%
\mathrm{R}$\textrm{ep}$_{alg}\left( \mathfrak{g},\mathbf{Q}\right) $ by
means of Lemma \ref{Representations L K(G)} below. In order to make this
further structure explicit, it will be convenient to consider the module
valued functor $\mathrm{C}^{\mathrm{alg}}\left( X,V\right) \left[ \eta %
\right] $\ classifying rational functions defined as follows. If $%
U_{1}\subset U_{2}\subset X_{/R}$ is an inclusion of open subsets of a
scheme $X$ and $V$ is a module valued functor, then there is a map $\mathrm{C%
}^{\mathrm{alg}}\left( U_{2},V\right) _{R}\rightarrow \mathrm{C}^{\mathrm{alg%
}}\left( U_{1},V\right) _{R}$ and we define $\mathrm{C}^{\mathrm{alg}}\left(
X,V\right) \left[ \eta \right] $ to be the module valued functor on $\Bbbk $%
-algebras $R$ defined by means of the formula $\mathrm{C}^{\mathrm{alg}%
}\left( X,V\right) \left[ \eta \right] _{R}:=\lim\limits_{U\rightarrow }%
\mathrm{C}^{\mathrm{alg}}\left( U,V\right) _{R}$, the limit being taken over
all dense open subsets $U$\ of $X_{/R}$ (note that the intersection of
finitely many dense open subsets is a dense open subset and the transition
morphisms are injective). Suppose now that $U\subset \mathbf{G}$ is such
that $U\mathbf{Q}\subset U$ and $V\in \underline{\mathrm{Mod}}_{f}$. Then we
have a $\mathbf{G}$-module action on $\mathrm{C}^{\mathrm{alg}}\left( 
\mathbf{G},V\right) \left[ \eta \right] $ (again given by right
multiplication) and a morphism from $\mathrm{C}^{\mathrm{alg}}\left(
U,V\right) $ to $\mathrm{C}^{\mathrm{alg}}\left( \mathbf{G},V\right) \left[
\eta \right] $ which is $\mathbf{Q}$-equivariant (and injective, when $%
U\subset \mathbf{G}$ is dense). Moreover, as remarked above $\mathrm{C}^{%
\mathrm{alg}}\left( U,V\right) $ is in $\mathrm{R}$\textrm{ep}$_{alg}\left( 
\mathfrak{g},\mathbf{Q}\right) $: it would be natural to conjecture that the
inclusion can ba promoted to a morphism of $\left( \mathfrak{g},\mathbf{Q}%
\right) $-modules. However, it make no sense, a priori, to differentate the $%
\mathbf{G}$-action of $\mathrm{C}^{\mathrm{alg}}\left( \mathbf{G},V\right) %
\left[ \eta \right] $. This issue can be fixed by considering an
intermediate category between $\mathrm{R}$\textrm{ep}$_{alg}\left( \mathbf{G}%
\right) $ and $\mathrm{R}$\textrm{ep}$\left( \mathbf{G}\right) $ to which $%
\mathrm{C}^{\mathrm{alg}}\left( \mathbf{G},V\right) \left[ \eta \right] $
belongs and admitting a functor to some category of $\left( \mathfrak{g},%
\mathbf{Q}\right) $-modules defined as follows.

For a $\Bbbk $-algebra $R$ and an $R$-module $M$, we let $R_{M}\left[
\varepsilon \right] :=R\oplus M$ with the unique $R$-algebra structure
making $M$ a square zero ideal and set $R\left[ \varepsilon \right] :=R_{R}%
\left[ \varepsilon \right] $: it comes equipped with a morphism of $R$%
-algebras $i_{R_{M}\left[ \varepsilon \right] }:R\rightarrow R_{M}\left[
\varepsilon \right] $ (resp. $\pi _{R_{M}\left[ \varepsilon \right] }:R_{M}%
\left[ \varepsilon \right] \rightarrow R$) defined by the rule $i_{R_{M}%
\left[ \varepsilon \right] }\left( r\right) :=\left( r,0\right) $ (resp. $%
\pi _{R_{M}\left[ \varepsilon \right] }\left( r,m\right) :=r$) and we have $%
\pi _{R_{M}\left[ \varepsilon \right] }\circ i_{R_{M}\left[ \varepsilon %
\right] }=1_{R}$. In particular, setting $T_{M}\left( \mathbf{G}\right)
\left( R\right) :=\mathbf{G}\left( R\left[ \varepsilon \right] \right) $, we
deduce that $T_{M}\left( \mathbf{G}\right) \left( R\right) \simeq \mathbf{G}%
\left( R\right) \ltimes \mathfrak{g}_{R}$. Let $\underline{\mathrm{Mod}}%
_{gd} $ be the full subcategory of those $V\in \underline{\mathrm{Mod}}$
having the property that the canonical morphism $R_{M}\left[ \varepsilon %
\right] \otimes _{R}V_{R}\rightarrow V_{R_{M}\left[ \varepsilon \right] }$
is an isomorphism for every $R$-module $M$ and every $\Bbbk $-algebra $R$.
The full subcategories $\mathrm{R}$\textrm{ep}$_{gd}\left( \mathbf{G}\right) 
$ (resp. $\mathrm{R}$\textrm{ep}$_{gd}\left( \mathfrak{g}\right) $) of $%
\mathrm{R}$\textrm{ep}$\left( \mathbf{G}\right) $ (resp. $\mathrm{R}$\textrm{%
ep}$\left( \mathfrak{g}\right) $) are then defined by requiring the
underlying $\Bbbk $-module to be in $\underline{\mathrm{Mod}}_{gd}$. The
action of the objects $\left( \sigma ,V\right) $ of $\mathrm{R}$\textrm{ep}$%
_{gd}\left( \mathbf{G}\right) $ can be differentiated to give an object $%
\left( d\sigma ,V\right) $\ of $\mathrm{R}$\textrm{ep}$_{gd}\left( \mathfrak{%
g}\right) $ as follows. If $X\in \mathfrak{g}_{R}$ is regarded as an element
of $\mathbf{G}\left( R\left[ \varepsilon \right] \right) \simeq \mathbf{G}%
\left( R\right) \ltimes \mathfrak{g}_{R}$, then we suggestively write it as $%
1+\varepsilon X$. We have that $V_{R\left[ \varepsilon \right] }\overset{%
\sim }{\leftarrow }R\left[ \varepsilon \right] \otimes _{R}V_{R}$ and,
because $\pi _{R\left[ \varepsilon \right] }\circ i_{R\left[ \varepsilon %
\right] }=1_{R}$, we see that $V_{R\left[ \varepsilon \right] }\simeq
V_{R}\oplus \varepsilon V_{R}$ (as an $R$-module), the inclusion $%
V_{R}\subset V_{R\left[ \varepsilon \right] }$ being given by $i_{R\left[
\varepsilon \right] }$ whereas $\varepsilon V_{R}\simeq \mathrm{ker}\left(
\pi _{R\left[ \varepsilon \right] }:V_{R\left[ \varepsilon \right]
}\rightarrow V_{R}\right) $. Because $\pi _{R\left[ \varepsilon \right]
}\circ \sigma _{R\left[ \varepsilon \right] }=\sigma _{R}\circ \pi _{R\left[
\varepsilon \right] }$, we also deduce that $\sigma \left( 1+\varepsilon
X\right) \left( i_{R\left[ \varepsilon \right] }\left( v\right) \right) -i_{R%
\left[ \varepsilon \right] }\left( v\right) $ belongs to the kernel $%
\varepsilon V_{R}$ of $\pi _{R\left[ \varepsilon \right] }$. Hence,
identifying $v$ with its image $i_{R\left[ \varepsilon \right] }\left(
v\right) $, the endomorphism $d\sigma _{R}\left( X\right) $ can be defined
by means of the formula%
\begin{eqnarray}
&&\sigma _{R\left[ \varepsilon \right] }\left( 1+\varepsilon X\right)
v=v+\varepsilon d\sigma _{R}\left( X\right) \left( v\right) \text{ for every 
}v\in V_{R}\subset V_{R}\oplus \varepsilon V_{R}\simeq V_{R\left[
\varepsilon \right] }\text{,}  \notag \\
&&\text{i.e. }d\sigma _{R}\left( X\right) \left( v\right) =\frac{\sigma _{R%
\left[ \varepsilon \right] }\left( 1+\varepsilon X\right) v-v}{\varepsilon }%
\text{.}  \label{Representations F dr}
\end{eqnarray}%
It follows from \cite[Corollaire 2 to Proposition 4.8]{Dm70} that the
resulting morphism of functors $d\rho $ yields an element $\left( d\rho
,V\right) \in \mathrm{R}$\textrm{ep}$_{gd}\left( \mathfrak{g}\right) $ (and $%
\left( \text{\ref{Representations F dr}}\right) $ should be quite
convincing). Let us write \textrm{Ad}$_{\mathbf{G}}:\mathbf{G}\rightarrow
Aut\left( \mathfrak{g}\right) $ for the action obtained by conjugation (via $%
\mathbf{G}\left( R\left[ \varepsilon \right] \right) \simeq \mathbf{G}\left(
R\right) \ltimes \mathfrak{g}_{R}$) and let \textrm{ad}$_{\mathbf{G}}:%
\mathfrak{g}\rightarrow End\left( \mathfrak{g}\right) $ be $d$\textrm{Ad}$_{%
\mathbf{G}}$: then $\left[ X,Y\right] :=$\textrm{ad}$_{\mathbf{G}}\left(
X\right) \left( Y\right) $ and the proof that $\mathfrak{g}$ is a Lie
algebra and that $d\rho $ is of Lie algebras is essentially obtained by
functoriality reducing to the case $\mathbf{G}=Aut\left( V\right) $ for
which $\mathfrak{g}=End\left( V\right) $.

\bigskip

We are now going to show that $\mathrm{C}^{\mathrm{alg}}\left( \mathbf{G}%
,V\right) \left[ \eta \right] $ is in $\mathrm{R}$\textrm{ep}$_{gd}\left( 
\mathbf{G}\right) $ when $V\in \underline{\mathrm{Mod}}_{f}$.

\begin{lemma}
\label{Representations L K(X)}Suppose that $V\in \underline{\mathrm{Mod}}%
_{f} $ and that $U\subset X$ is an affine open subset of an affine scheme $X$%
. Then $\mathrm{C}^{\mathrm{alg}}\left( U,V\right) \in \underline{\mathrm{Mod%
}}_{alg}$ and $\mathrm{C}^{\mathrm{alg}}\left( X,V\right) \left[ \eta \right]
\in \underline{\mathrm{Mod}}_{gd}$.
\end{lemma}

\begin{proof}
After noticing that the categories $\underline{\mathrm{Mod}}_{alg}$ and $%
\underline{\mathrm{Mod}}_{gd}$ are closed under taking factors of direct
sums in $\underline{\mathrm{Mod}}$ and using the fact that $V_{\Bbbk }$ is a
finitely generated and projective over $\Bbbk $, we can assume that $V\simeq 
\mathbf{G}_{a}^{d}$ and then that $V=\mathbf{G}_{a}$. In this case, $\mathrm{%
C}^{\mathrm{alg}}\left( U,V\right) _{R}=\mathcal{O}\left( U_{/R}\right) $
and $\mathrm{C}^{\mathrm{alg}}\left( X,V\right) \left[ \eta \right] _{R}$
equals the ring $K\left( X_{/R}\right) $ of rational functions on $X_{/R}$.
The fact that $\mathrm{C}^{\mathrm{alg}}\left( U,V\right) \in \underline{%
\mathrm{Mod}}_{alg}$ follows from the fact that $R^{\prime }\otimes _{R}%
\mathcal{O}\left( U_{/R}\right) \simeq \mathcal{O}\left( U_{/R^{\prime
}}\right) $ holds for every affine scheme. If $Y$ is a scheme with finitely
many irreducible components, it follows from (the proof of) Lemma 31.23.6 $%
\left( 6\right) $ of The Stack Project that $K\left( -\right) $ is a sheaf
on the open subsets of $Y$: this applies to every affine scheme $Y$, as it
is $X_{/R}$. In particular, suppose that $\left\{ \mathfrak{Spec}\left(
R_{f}\right) \rightarrow \mathfrak{Spec}\left( R\right) \right\} _{f\in R}$
is a covering: then $\left\{ X_{/R_{f}}\rightarrow X_{/R}\right\} _{f\in R}$
is an open covering and, hence, the above discussion implies that $K\left(
X_{/-}\right) $ satisfies the sheaf property with respect to $\left\{
U_{i}\rightarrow U\right\} _{i\in I}$, as wanted. When $R^{\prime }=R_{M}%
\left[ \varepsilon \right] $, then $X_{/R}$ and $X_{/R^{\prime }}$ have the
same underlying topological space and, hence, the open subsets of $X_{/R}$
and $X_{/R^{\prime }}$ are the same. Taking into account that $\mathrm{C}^{%
\mathrm{alg}}\left( U,V\right) \in \underline{\mathrm{Mod}}_{alg}\left(
R\right) $ thanks to what we have proved applied to $U\subset X_{/R}$, we
deduce that, if $R^{\prime }=R_{M}\left[ \varepsilon \right] $, then%
\begin{eqnarray*}
R^{\prime }\otimes _{R}\mathrm{C}^{\mathrm{alg}}\left( X,V\right) \left[
\eta \right] _{R} &=&R^{\prime }\otimes _{R}\left( \lim\limits_{U\rightarrow
}\mathrm{C}^{\mathrm{alg}}\left( U,V\right) _{R}\right)
=\lim\limits_{U\rightarrow }\left( R^{\prime }\otimes _{R}\mathrm{C}^{%
\mathrm{alg}}\left( U,V\right) _{R}\right) \\
&=&\lim\limits_{U\rightarrow }\mathrm{C}^{\mathrm{alg}}\left( U,V\right)
_{R^{\prime }}=\mathrm{C}^{\mathrm{alg}}\left( U,V\right) \left[ \eta \right]
_{R^{\prime }}\text{.}
\end{eqnarray*}
\end{proof}

\bigskip

When $X=\mathbf{G}$ and $g\in \mathbf{G}\left( R\right) $, for every open
subset $U\subset X$, the right (resp. left) multiplication by $g$ morphism $%
r_{g}$ (resp. $l_{g}$) sends $U_{/R}g^{-1}$ (resp. $g^{-1}U$) to $U_{/R}$
and, in this way, we get a left (resp. right) action of $\mathbf{G}$ on $%
\mathrm{C}^{\mathrm{alg}}\left( \mathbf{G},V\right) \left[ \eta \right] $
making it an object of $\mathrm{R}$\textrm{ep}$\left( \mathbf{G}\right) $.
It follows from Lemma \ref{Representations L K(X)} that $\mathrm{C}^{\mathrm{%
alg}}\left( \mathbf{G},V\right) \left[ \eta \right] $ is in $\mathrm{Rep}%
_{gd}\left( \mathbf{G}\right) $ and, consequently, it makes sense to regard
it as an object of $\mathrm{Rep}_{gd}\left( \mathfrak{g}\right) $.

\begin{lemma}
\label{Representations L K(G)}Suppose that $V\in \underline{\mathrm{Mod}}%
_{f} $ and that $\mathbf{G}$ is affine. For every affine open subset $%
U\subset \mathbf{G}$, we have that the canonical morphism from $\mathrm{C}^{%
\mathrm{alg}}\left( U,V\right) $ to $\mathrm{C}^{\mathrm{alg}}\left( \mathbf{%
G},V\right) \left[ \eta \right] $ is in $\mathrm{Rep}_{gd}\left( \mathfrak{g}%
\right) $. In particular, when $U$ is dense subset, it realizes $\mathrm{C}^{%
\mathrm{alg}}\left( U,V\right) $ as a subobject of $\mathrm{C}^{\mathrm{alg}%
}\left( \mathbf{G},V\right) \left[ \eta \right] $ in $\mathrm{Rep}%
_{gd}\left( \mathfrak{g}\right) $.
\end{lemma}

\begin{proof}
Again we can assume that $V=\mathbf{G}_{a}$: then $\mathrm{C}^{\mathrm{alg}%
}\left( \mathbf{G},V\right) _{R}=\mathcal{O}\left( \mathbf{G}_{/R}\right) $, 
$\mathrm{C}^{\mathrm{alg}}\left( U,V\right) _{R}=\mathcal{O}\left(
U_{/R}\right) $ and $\mathrm{C}^{\mathrm{alg}}\left( \mathbf{G},V\right) %
\left[ \eta \right] _{R}=K\left( \mathbf{G}_{/R}\right) $. One identifies $%
\mathfrak{g}_{R}$ with the space of $\mathbf{G}$-invariant derivations (for
the left action by right multiplication, sending $1+\varepsilon X=\left( e_{%
\mathbf{G}_{/R}},X\right) \in \mathbf{G}\left( R\right) \ltimes \mathfrak{g}%
_{R}$ to $D_{X}:=\left( 1_{\mathcal{O}\left( \mathbf{G}_{/R}\right) }\otimes
_{R}X\right) \circ \Delta $, where $\Delta $ denotes the colagebra
structure). Then one checks that $d\sigma _{R}\left( X\right) \left(
f\right) =D_{X}\left( f\right) $ if $f\in \mathcal{O}\left( \mathbf{G}%
_{/R}\right) $. In particular, because the derivations localize, the $%
\mathfrak{g}_{R}$-action on $\mathcal{O}\left( \mathbf{G}_{/R}\right) $
uniquely extends to $\mathcal{O}\left( U_{/R}\right) $. Because $K\left( 
\mathbf{G}_{/R}\right) :=\lim\limits_{U\rightarrow }\mathcal{O}\left(
U_{/R}\right) $, we deduce that it also uniquely extends to $K\left( \mathbf{%
G}_{/R}\right) $: hence, the $\mathfrak{g}_{R}$-action on $\mathcal{O}\left(
U_{/R}\right) $ uniquely extends to $K\left( \mathbf{G}_{/R}\right) $. But
because the morphism from $\mathcal{O}\left( \mathbf{G}_{/R}\right) $ to $%
K\left( \mathbf{G}_{/R}\right) $ is of $\mathbf{G}_{/R}$-modules, the $%
\mathfrak{g}_{R}$-action on $K\left( \mathbf{G}_{/R}\right) $ obtained from
the $\mathbf{G}$-module structure of $K\left( \mathbf{G}\right) $ also
extends the $\mathfrak{g}_{R}$-action of $\mathcal{O}\left( \mathbf{G}%
_{/R}\right) $. It follows that these two $\mathfrak{g}_{R}$-actions
coincides and, hence, that the morphism from $\mathcal{O}\left(
U_{/R}\right) $ to $K\left( \mathbf{G}_{/R}\right) $ is in $\mathrm{Rep}%
_{gd}\left( \mathfrak{g}\right) $.
\end{proof}

\bigskip

We will now suppose that $\mathbf{G}$ is affine and we will regard $\mathrm{C%
}^{\mathrm{alg}}\left( \mathbf{G},V\right) \left[ \eta \right] $ as a left $%
\mathbf{G}$-module by right multiplication. Assume that $V\in \underline{%
\mathrm{Mod}}_{f}$ and that $U\subset \mathbf{G}$ is a dense affine open
subset such that $U\mathbf{Q}\subset U$: then Lemma \ref{Representations L
K(X)} and Lemma \ref{Representations L K(G)} together imply that $\mathrm{C}%
^{\mathrm{alg}}\left( \mathbf{G},V\right) \left[ U\right] :=\mathrm{C}^{%
\mathrm{alg}}\left( U,V\right) \in \mathrm{Rep}_{alg}\left( \mathfrak{g},%
\mathbf{Q}\right) $ is a subobject of $\mathrm{C}^{\mathrm{alg}}\left( 
\mathbf{G},V\right) \left[ \eta \right] \in \mathrm{Rep}\left( \mathfrak{g},%
\mathbf{Q}\right) $. Suppose in addition that $\mathbf{Q}^{-}$ is another
affine group scheme that acts on $\mathrm{C}^{\mathrm{alg}}\left( \mathbf{G}%
,V\right) \left[ \eta \right] $ from the right making it an object of $%
\mathrm{Rep}\left( \mathbf{Q}^{-}\right) $ in such a way that the $\left( 
\mathfrak{g},\mathbf{Q}\right) $-action commute with the $\mathbf{Q}^{-}$%
-action and that $\mathrm{C}^{\mathrm{alg}}\left( \mathbf{G},V\right) \left[
U\right] \subset \mathrm{C}^{\mathrm{alg}}\left( \mathbf{G},V\right) \left[
\eta \right] $ is also a subobject in $\mathrm{Rep}\left( \mathbf{Q}%
^{-}\right) $. Applying $\left( -\right) ^{h_{\mathbf{Q}^{-}}}$, we get that 
$\mathrm{C}^{\mathrm{alg}}\left( \mathbf{G},V\right) \left[ U\right] ^{h_{%
\mathbf{Q}^{-}}}$ is a subobject of $\mathrm{C}^{\mathrm{alg}}\left( \mathbf{%
G},V\right) \left[ \eta \right] ^{h_{\mathbf{Q}^{-}}}$ in both $\mathrm{Rep}%
\left( \mathfrak{g}\right) $ and $\mathrm{Rep}\left( \mathbf{Q}\right) $%
\footnote{%
Although, strictly speaking, we can not talk about $\left( \mathfrak{g},%
\mathbf{Q}\right) $-modules because it may not be possible a priori to
differentiate the $\mathbf{Q}$-action, $\mathrm{C}^{\mathrm{alg}}\left( 
\mathbf{G},V\right) \left[ U\right] ^{h_{\mathbf{Q}^{-}}}$ and $\mathrm{C}^{%
\mathrm{alg}}\left( \mathbf{G},V\right) \left[ \eta \right] ^{h_{\mathbf{Q}%
^{-}}}$ are equivariantly for both actions subfunctors of the $\left( 
\mathfrak{g},\mathbf{Q}\right) $-module $\mathrm{C}^{\mathrm{alg}}\left( 
\mathbf{G},V\right) \left[ \eta \right] $.}. Also, the canonical morphism
from $\mathrm{C}^{\mathrm{alg}}\left( \mathbf{G},V\right) \left[ U\right] ^{%
\mathbf{Q}^{-}}$ to $\mathrm{C}^{\mathrm{alg}}\left( \mathbf{G},V\right) %
\left[ U\right] ^{h_{\mathbf{Q}^{-}}}$ is equivariant for both actions of $%
\mathfrak{g}$ and $\mathbf{Q}$. Summarizing, the composition%
\begin{equation}
\mathrm{C}^{\mathrm{alg}}\left( \mathbf{G},V\right) \left[ U\right] ^{%
\mathbf{Q}^{-}}\longrightarrow \mathrm{C}^{\mathrm{alg}}\left( \mathbf{G}%
,V\right) \left[ U\right] ^{h_{\mathbf{Q}^{-}}}\subset \mathrm{C}^{\mathrm{%
alg}}\left( \mathbf{G},V\right) \left[ \eta \right] ^{h_{\mathbf{Q}^{-}}}
\label{Representations F1 (g,Q)}
\end{equation}%
is a morphism both of $\mathfrak{g}$-modules and $\mathbf{Q}$-modules and we
have $\mathrm{C}^{\mathrm{alg}}\left( \mathbf{G},V\right) \left[ U\right] ^{%
\mathbf{Q}^{-}}\in \mathrm{Rep}_{alg}\left( \mathfrak{g},\mathbf{Q}\right) $.

\subsection{\label{S Representations BGG}Induced modules and the dual BGG
complex}

Suppose that $\mathbf{G}$ is connected, split and reductive and that $%
\mathbf{Q}$ is a parabolic subgroup with unipotent radical $\mathbf{U}$,\
Levi component $\mathbf{M}$ which is connected, split and reductive and
opposite parabolic subgroup $\mathbf{Q}^{-}$ having unipotent radical $%
\mathbf{U}^{-}$ (all of them are smooth, the unipotent groups by \cite[%
Corollaty 5.2.5]{Cnr14} and the other by definition, see \cite[Definition
3.1.1 and Definition 5.2.1]{Cnr14}). The projection $\mathbf{Q}%
^{-}\rightarrow \mathbf{M}$ allows us to regard representations of $\mathbf{M%
}$ as representations of $\mathbf{Q}^{-}$: we have a functor%
\begin{equation}
\mathrm{Rep}_{?}\left( \mathbf{M}\right) \longrightarrow \mathrm{Rep}%
_{?}\left( \mathbf{Q}^{-}\right) \text{ for }?\in \left\{ \phi
,alg,f\right\} \text{.}  \label{Representations F Res}
\end{equation}%
Let us now specify relevant $\left( \mathfrak{g},\mathbf{Q}\right) $-modules
associated to representations of $\mathbf{Q}^{-}$ and, via $\left( \text{\ref%
{Representations F Res}}\right) $, of $\mathbf{M}$: to this end, we need a
dense affine open subset subset $U\subset \mathbf{G}$ such that $\mathbf{Q}%
^{-}U\mathbf{Q}\subset U$. Then, if $\rho \in \mathrm{Rep}_{f}\left( \mathbf{%
Q}^{-}\right) $, we define functors%
\begin{equation*}
\mathrm{Ind}_{\mathbf{Q}^{-}}^{\mathbf{G}}\left( \rho \right) \subset 
\mathrm{Ind}_{\mathbf{Q}^{-}}^{\mathbf{G}}\left( \rho \right) \left[ U\right]
\longrightarrow \mathrm{Ind}_{\mathbf{Q}^{-}}^{\mathbf{G}}\left( \rho
\right) \left[ \eta \right]
\end{equation*}%
as follows. We regard the functor $\mathrm{C}^{\mathrm{alg}}\left( \mathbf{G}%
,\rho \right) \left[ U\right] $ (resp. $\mathrm{C}^{\mathrm{alg}}\left( 
\mathbf{G},\rho \right) \left[ \eta \right] $) as a right $\mathbf{Q}^{-}$%
-module by means of the rule $\left( fq\right) \left( x\right) :=\rho \left(
q\right) ^{-1}f\left( qx\right) $ for morphisms of $R$-schemes $%
f:U_{/R}\rightarrow \rho _{/R}$ (resp. $f:U\rightarrow \rho _{/R}$ for some
dense open subset $U\subset \mathbf{G}_{/R}$), $q\in \mathbf{Q}^{-}\left(
R\right) $ and $x\in U\left( R\right) $. Then we set $\mathrm{Ind}_{\mathbf{Q%
}^{-}}^{\mathbf{G}}\left( \rho \right) :=\mathrm{C}^{\mathrm{alg}}\left( 
\mathbf{G},\rho \right) \left[ \mathbf{G}\right] ^{\mathbf{Q}^{-}}$, $%
\mathrm{Ind}_{\mathbf{Q}^{-}}^{\mathbf{G}}\left( \rho \right) \left[ U\right]
:=\mathrm{C}^{\mathrm{alg}}\left( \mathbf{G},\rho \right) \left[ U\right] ^{%
\mathbf{Q}^{-}}$ and $\mathrm{Ind}_{\mathbf{Q}^{-}}^{\mathbf{G}}\left( \rho
\right) \left[ \eta \right] :=\mathrm{C}^{\mathrm{alg}}\left( \mathbf{G}%
,\rho \right) \left[ \eta \right] ^{h_{\mathbf{Q}^{-}}}$. We remark that $%
\mathbf{Q}\left( R\right) $ (resp. $\mathbf{G}\left( R\right) $) acts by
right translation $\left( gf\right) \left( x\right) :=f\left( xg\right) $ on 
$\mathrm{Ind}_{\mathbf{Q}^{-}}^{\mathbf{G}}\left( \rho \right) \left[ U%
\right] $ (resp. $\mathrm{Ind}_{\mathbf{Q}^{-}}^{\mathbf{G}}\left( \rho
\right) \left[ \eta \right] $ and $\mathrm{Ind}_{\mathbf{Q}^{-}}^{\mathbf{G}%
}\left( \rho \right) $). Indeed, by differentiating the $\mathbf{G}$-action,
we get a monomorphism $\mathrm{Ind}_{\mathbf{Q}^{-}}^{\mathbf{G}}\left( \rho
\right) \subset \mathrm{Ind}_{\mathbf{Q}^{-}}^{\mathbf{G}}\left( \rho
\right) \left[ U\right] $ in $\mathrm{Rep}_{alg}\left( \mathfrak{g},\mathbf{Q%
}\right) $ and $\mathrm{Ind}_{\mathbf{Q}^{-}}^{\mathbf{G}}\left( \rho
\right) \left[ \eta \right] $ is both a $\mathfrak{g}$-module and $\mathbf{Q}
$-module in such a way that the morphism from $\mathrm{Ind}_{\mathbf{Q}%
^{-}}^{\mathbf{G}}\left( \rho \right) \left[ U\right] $ to $\mathrm{Ind}_{%
\mathbf{Q}^{-}}^{\mathbf{G}}\left( \rho \right) \left[ \eta \right] $ is
equivariant (thanks to $\left( \text{\ref{Representations F1 (g,Q)}}\right) $%
).

Because $\mathbf{Q}^{-}\backslash \mathbf{G}$ is proper, it is also true
that $\mathrm{Ind}_{\mathbf{Q}^{-}}^{\mathbf{G}}\left( \rho \right) \in 
\mathrm{R}$\textrm{ep}$_{f}\left( \mathbf{G}\right) $ when $\Bbbk $\ is a
field (by interpreting the elements of $\mathrm{Ind}_{\mathbf{Q}^{-}}^{%
\mathbf{G}}\left( \rho \right) \left( R\right) $ as global section of a
suitable vector bundle on $\mathbf{Q}^{-}\backslash \mathbf{G}$, cfr. $%
\left( \text{\ref{Representations F Inv}}\right) $). The multiplication map
identifies $\mathbf{U}^{-}\times \mathbf{M}\times \mathbf{U}$ with an open
subscheme $\mathbf{Y}_{\mathbf{G},\mathbf{Q}}=\mathbf{Y}\subset \mathbf{G}$
(see \cite[Theorem 4.1.7 4.]{Cnr14}) and, hence, the above discussion
applies with $U=\mathbf{Y}$.

\begin{remark}
\label{Representations R Action}The decomposition%
\begin{equation}
\mathbf{U}^{-}\times \mathbf{M}\times \mathbf{U}\overset{\sim }{\mathbf{%
\rightarrow }}\mathbf{Y}  \label{Representations F BigCell}
\end{equation}%
implies that the restriction map yields an isomorphism (for every $\rho \in 
\mathrm{Rep}_{f}\left( \mathbf{Q}^{-}\right) $)%
\begin{equation}
\mathrm{Ind}_{\mathbf{Q}^{-}}^{\mathbf{G}}\left( \rho \right) \left[ \mathbf{%
Y}\right] \simeq \mathrm{C}^{\mathrm{alg}}\left( \mathbf{G},\rho \right) %
\left[ U\right] ^{h_{\mathbf{Q}^{-}}}\overset{\sim }{\longrightarrow }%
\mathrm{C}^{\mathrm{alg}}\left( \mathbf{U},\rho \right)
\label{Representations F Iso1}
\end{equation}%
which is evidently $\mathbf{U}$-equivariant if we let $\mathbf{U}$ act on
the right hand side via the formula $\left( uf\right) \left( y\right)
:=f\left( yu\right) $. In fact, the $\mathbf{Q}$-action on $\mathbf{Y}$
given by right multiplication descend to a right action $\ast $ on $\mathbf{Q%
}^{-}\backslash \mathbf{Y}\overset{\sim }{\mathbf{\rightarrow }}\mathbf{U}$:
because for every $mu\in \mathbf{Q}=\mathbf{M}\ltimes \mathbf{U}$ and $y\in 
\mathbf{U}$ we have that $ymu$ is equivalent to $m^{-1}ymu$ and $m^{-1}ym\in 
\mathbf{U}$, we see that%
\begin{equation}
y\ast mu=m^{-1}ymu\text{.}  \label{Representations F Act}
\end{equation}%
It follows that the $\mathbf{Q}$-action on the right hand side of $\left( 
\text{\ref{Representations F Iso1}}\right) $ is given by%
\begin{equation*}
\left( muf\right) \left( y\right) =\rho \left( m\right) f\left(
m^{-1}ymu\right) \text{ if }mu\in \mathbf{Q}=\mathbf{M}\ltimes \mathbf{U}%
\text{ and }y\in \mathbf{U}\text{.}
\end{equation*}%
There is a canonical monomorphism%
\begin{equation}
\rho \longrightarrow \mathrm{Ind}_{\mathbf{Q}^{-}}^{\mathbf{G}}\left( \rho
\right) \left[ \mathbf{Y}\right]  \label{Representations F WIncl}
\end{equation}%
which is $\mathbf{Q}$-equivariant: under the $\mathbf{Q}$-equivariant
identification $\left( \text{\ref{Representations F Iso1}}\right) $, it
sends $w\in \rho $ to the constant function $w\left( u\right) =w$, hence $%
w\left( y\right) =w\left( u_{y}^{-}m_{y}u_{y}^{+}\right) =m_{y}w$ for every $%
y\in \mathbf{Y}$ written in the form $y=u_{y}^{-}m_{y}u_{y}^{+}$.
\end{remark}

\begin{proof}
The only non-obvious fact is that the canonical morphism%
\begin{equation}
\mathrm{C}^{\mathrm{alg}}\left( \mathbf{G},\rho \right) \left[ \mathbf{Y}%
\right] ^{\mathbf{Q}^{-}}\longrightarrow \mathrm{C}^{\mathrm{alg}}\left( 
\mathbf{G},\rho \right) \left[ \mathbf{Y}\right] ^{h_{\mathbf{Q}^{-}}}
\label{Representations R Action F1}
\end{equation}%
is an isomorphism of functors. To this end, we first note that the
decomposition $\left( \text{\ref{Representations F BigCell}}\right) $\ first
implies that%
\begin{equation}
\mathrm{C}^{\mathrm{alg}}\left( \mathbf{G},\rho \right) \left[ \mathbf{Y}%
\right] ^{h_{\mathbf{Q}^{-}}}\overset{\sim }{\longrightarrow }\mathrm{C}^{%
\mathrm{alg}}\left( \mathbf{U},\rho \right) \text{.}
\label{Representations R Action F2}
\end{equation}%
In particular, we deduce that $\mathrm{C}^{\mathrm{alg}}\left( \mathbf{G}%
,\rho \right) \left[ \mathbf{Y}\right] ^{h_{\mathbf{Q}^{-}}}\in \underline{%
\mathrm{Mod}}_{alg}$ (thanks to Lemma \ref{Representations L K(X)}). But
then, in order to show that $\left( \text{\ref{Representations R Action F1}}%
\right) $ is an isomorphism, it suffices to show that this holds at the
level of the underlying $\Bbbk $-modules. In view of $\left( \text{\ref%
{Representations F Inv}}\right) $ (applied to the affine group scheme $%
\mathbf{Q}^{-}$), this follows from $\left( \text{\ref{Representations R
Action F2}}\right) $ taking the $\Bbbk $-points.
\end{proof}

We deduce from $\left( \text{\ref{Representations F Iso1}}\right) $ and $%
\left( \text{\ref{Representations F1 (g,Q)}}\right) $\ that we have, indeed,%
\begin{equation*}
\mathrm{Ind}_{\mathbf{Q}^{-}}^{\mathbf{G}}\left( \rho \right) \subset 
\mathrm{Ind}_{\mathbf{Q}^{-}}^{\mathbf{G}}\left( \rho \right) \left[ \mathbf{%
Y}\right] \hookrightarrow \mathrm{Ind}_{\mathbf{Q}^{-}}^{\mathbf{G}}\left(
\rho \right) \left[ \eta \right]
\end{equation*}%
(and that the $R$-points of $\mathrm{Ind}_{\mathbf{Q}^{-}}^{\mathbf{G}%
}\left( \rho \right) \left[ U\right] $ can be explicitly described, using
the fixed point functor, as the set of those morphism of $R$-scheme $%
f:U_{/R}\rightarrow \rho _{/R}$ such that $f\left( qx\right) =\rho \left(
q\right) f\left( x\right) $ holds for every $q\in \mathbf{Q}^{-}\left(
R^{\prime }\right) $, every $x\in U_{/R}\left( R^{\prime }\right) =U\left(
R^{\prime }\right) $ and every $R$-algebra $R^{\prime }$).

\begin{remark}
\label{Representations R IndEx}If we view $\mathrm{Rep}_{f}\left( \mathbf{Q}%
^{-}\right) $ as an exact subcategory of $\mathrm{Rep}_{alg}\left( \mathbf{Q}%
^{-}\right) $, then the formation of $\mathrm{Ind}_{\mathbf{Q}^{-}}^{\mathbf{%
G}}\left( \rho \right) \left[ \mathbf{Y}\right] $ is exact in $\rho $.
\end{remark}

\begin{proof}
In view of $\left( \text{\ref{Representations F Iso1}}\right) $, this is a
consequence of the identification $\mathrm{C}^{\mathrm{alg}}\left( \mathbf{U}%
,\rho \right) \simeq \mathcal{O}\left( \mathbf{U}\right) \otimes _{\Bbbk
}\rho $ in $\underline{\mathrm{Mod}}$, which holds because $\rho \in 
\underline{\mathrm{Mod}}_{f}$, and the flatness of $\mathbf{U}$ over $\Bbbk $%
.
\end{proof}

\bigskip

We set%
\begin{equation*}
V_{\rho }:=\mathrm{Ind}_{\mathbf{Q}^{-}}^{\mathbf{G}}\left( \rho \right) %
\left[ \mathbf{Y}\right] \text{, for }\rho \in \mathrm{Rep}_{f}\left( 
\mathbf{M}\right) \text{ (using }\left( \text{\ref{Representations F Res}}%
\right) \text{).}
\end{equation*}%
When $\mathbf{Q}=\mathbf{B}_{\mathbf{G}}$ is a Borel subgroup, we have $%
\mathbf{M}=\mathbf{T}$, a maximal torus of $\mathbf{G}$, and we let $X^{\ast
}\left( \mathbf{T}\right) :=\underline{Hom}_{gr}\left( \mathbf{T},\mathbf{G}%
_{m}\right) $ be the character group of $\mathbf{T}$. Choosing $\mathbf{B}$
in such a way that $\mathbf{Q\subset B}$ (as it is possible up to
conjugation on $\mathbf{B}$), one has $\mathbf{T\subset M}$. We define the
set of dominant weights $X_{\mathbf{G},+}\subset X^{\ast }\left( \mathbf{T}%
\right) $ (resp. $X_{\mathbf{M},+}\subset X^{\ast }\left( \mathbf{T}\right) $%
) as the subset of those $\lambda \in X^{\ast }\left( \mathbf{T}\right) $
such that $\mathrm{Ind}_{\mathbf{B}_{\mathbf{G}}^{-}}^{\mathbf{G}}\left(
\lambda \right) \neq 0$ (resp. $\mathrm{Ind}_{\mathbf{B}_{\mathbf{M}}^{-}}^{%
\mathbf{M}}\left( \lambda \right) \neq 0$). Then $X_{\mathbf{G},+}\subset X_{%
\mathbf{M},+}$ and, when $\Bbbk $ is a characteristic zero field, by the
theory of highest weight vectors, there is a bijection between $X_{\mathbf{G}%
,+}$ (resp. $X_{\mathbf{M},+}$) and (the isomorphism classes of) the
irreducible objects of $\mathrm{R}$\textrm{ep}$_{f}\left( \mathbf{G}\right) $
(resp. $\mathrm{R}$\textrm{ep}$_{f}\left( \mathbf{M}\right) $): in this
case, we set%
\begin{equation*}
L_{\lambda }:=\mathrm{Ind}_{\mathbf{B}_{\mathbf{G}}^{-}}^{\mathbf{G}}\left(
\lambda \right) \text{ and }W_{\lambda }:=\mathrm{Ind}_{\mathbf{B}_{\mathbf{M%
}}^{-}}^{\mathbf{M}}\left( \lambda \right) \text{.}
\end{equation*}%
Let us write $W_{\mathbf{G}}:=\frac{N_{\mathbf{G}}\left( \mathbf{T}\right) }{%
\mathbf{T}}$ (resp. $W_{\mathbf{M}}:=\frac{N_{\mathbf{M}}\left( \mathbf{T}%
\right) }{\mathbf{T}}$) for the Weyl group of $\mathbf{G}$ (resp. $\mathbf{M}
$) and let $W^{\mathbf{M}}\subset W_{\mathbf{G}}$ be the subset of those $%
w\in W_{\mathbf{G}}$ such that $w\left( \lambda \right) \in X_{\mathbf{M},+}$
for every $\lambda \in X_{\mathbf{G},+}$ (which can be shown to be the set
of minimal length right coset representatives for $W_{\mathbf{M}}\backslash
W_{\mathbf{G}}$). We also set $w\cdot \lambda :=w\left( \lambda +\rho
\right) -\rho $, where $\rho $ is half the sum of the positive roots.

\begin{proposition}
\label{Representations P BGG}Suppose that $\Bbbk $ is a characteristic zero
field, that $\lambda \in X_{\mathbf{G},+}\subset X_{\mathbf{M},+}$ and let $%
l $ be the greatest length of any element in $W^{\mathbf{M}}$. Then there is
an exact complex (dual BGG-complex) of algebraic $\left( \mathfrak{g},%
\mathbf{Q}\right) $-modules in $\mathrm{Rep}_{alg}\left( \mathfrak{g},%
\mathbf{Q}\right) $:%
\begin{eqnarray}
&&0\longrightarrow L_{\lambda }\longrightarrow \mathrm{Ind}_{\mathbf{Q}%
^{-}}^{\mathbf{G}}\left( W_{\lambda }\right) \left[ \mathbf{Y}\right] 
\overset{d}{\longrightarrow }\tbigoplus\nolimits_{w\in W^{\mathbf{M}%
}:l\left( w\right) =1}\mathrm{Ind}_{\mathbf{Q}^{-}}^{\mathbf{G}}\left(
W_{w\cdot \lambda }\right) \left[ \mathbf{Y}\right] \longrightarrow ... 
\notag \\
&&\text{ \ \ \ \ \ \ \ \ \ }...\longrightarrow \tbigoplus\nolimits_{w\in W^{%
\mathbf{M}}:l\left( w\right) =i}\mathrm{Ind}_{\mathbf{Q}^{-}}^{\mathbf{G}%
}\left( W_{w\cdot \lambda }\right) \left[ \mathbf{Y}\right] \longrightarrow
....\longrightarrow \tbigoplus\nolimits_{w\in W^{\mathbf{M}}:l\left(
w\right) =l}\mathrm{Ind}_{\mathbf{Q}^{-}}^{\mathbf{G}}\left( W_{w\cdot
\lambda }\right) \left[ \mathbf{Y}\right] \longrightarrow 0\text{.}
\label{Representations P BGG F BGG}
\end{eqnarray}%
Moreover, the inclusion $\left( \text{\ref{Representations F WIncl}}\right) $%
\ of $W_{\lambda }$ in $\mathrm{Ind}_{\mathbf{Q}^{-}}^{\mathbf{G}}\left(
W_{\lambda }\right) \left[ \mathbf{Y}\right] =:V_{\lambda }$ factors through
the inclusion of $L_{\lambda }$ in $V_{\lambda }$ given by the above exact
sequence:%
\begin{equation}
W_{\lambda }\hookrightarrow L_{\lambda }\hookrightarrow V_{\lambda }\text{.}
\label{Representations P BGG F Incl}
\end{equation}
\end{proposition}

\begin{proof}
Although the result is surely well known to the experts, let us briefly
recall how to deduce it from the (more standard non) dual BGG complex. First
of all we remark that, by definition of the exactness, one is essentially
reduced to the assertion at the level of the underlying $\Bbbk $-modules
and, hence, we first concentrate on them. The key observation is that there
is an isomorphism of $\mathfrak{g}_{\Bbbk }$-modules%
\begin{equation}
\mathrm{Ind}_{\mathbf{Q}^{-}}^{\mathbf{G}}\left( W_{\mu }\right) \left[ 
\mathbf{Y}\right] _{\Bbbk }\overset{\sim }{\rightarrow }\left( U\left( 
\mathfrak{g}_{\Bbbk }\right) \otimes _{U\left( \mathfrak{q}_{\Bbbk
}^{-}\right) }W_{\mu ,\Bbbk }^{\vee }\right) ^{\vee }
\label{Representations P BGG F1}
\end{equation}%
where the duality $\left( -\right) ^{\vee }$, exchanging the $\mathcal{O}^{%
\mathfrak{q}_{\Bbbk }^{-}}$ and the $\mathcal{O}^{\mathfrak{q}_{\Bbbk }}$
BGG categories, is defined considering the sum of the isotypic components
under the action of the Lie algebra $\mathfrak{m}_{\Bbbk }$\ of $\mathbf{M}$
(hence, it agrees with the usual duality on finite dimensional $\mathfrak{g}%
_{\Bbbk }$-modules in the BGG categories - here, for the $\mathfrak{g}%
_{\Bbbk }$-actions, we follow the convetions of \cite[\S 3]{Jo11} and not
those of \cite[\S 3.2 and \S 9.3 Proposition]{Hu08}, implying that $\mathcal{%
O}^{\mathfrak{q}_{\Bbbk }^{-}}$ and the $\mathcal{O}^{\mathfrak{q}_{\Bbbk }}$
are exchanged). Indeed, because $\mathbf{U}$ is unipotent, one has an
isomorphism of $\mathfrak{u}_{\Bbbk }$-modules%
\begin{equation}
\mathrm{C}^{\mathrm{alg}}\left( \mathbf{U},V\right) _{\Bbbk }\simeq
Hom_{\Bbbk }\left( U\left( \mathfrak{u}_{\Bbbk }\right) ,V_{\Bbbk }\right) ^{%
\mathfrak{u}_{\Bbbk }^{\infty }}  \label{Representations P BGG F2}
\end{equation}%
for every $V\in \underline{\mathrm{Mod}}_{f}$, where $\left( -\right) ^{%
\mathfrak{u}^{\infty }}$ denotes the subspace those $v$ such that there
exists some $k\in \mathbb{N}$ and $X_{1},...,X_{k}\in \mathfrak{u}_{\Bbbk }$
such that $X_{1}...X_{k}v=0$, obtained sending $f$ to the homomorphism
sending $u\in U\left( \mathfrak{u}_{\Bbbk }\right) $ to $\left( uf\right)
\left( 1\right) $ (see \cite[Lemma 2.5.3]{Em}). When $W\in \mathrm{R}$%
\textrm{ep}$_{f}\left( \mathbf{M}\right) $, it can be enhanced to an
isomorphism of $\mathfrak{g}_{\Bbbk }$-modules%
\begin{equation}
\mathrm{Ind}_{\mathbf{Q}^{-}}^{\mathbf{G}}\left( W\right) \left[ \mathbf{Y}%
\right] _{\Bbbk }\simeq Hom_{U\left( \mathfrak{q}_{\Bbbk }^{-}\right)
}\left( U\left( \mathfrak{g}_{\Bbbk }\right) ,W_{\Bbbk }\right) ^{\mathfrak{u%
}_{\Bbbk }^{\infty }}  \label{Representations P BGG F3}
\end{equation}%
using $\left( \text{\ref{Representations F Iso1}}\right) $ and the
identification of the right hand side of $\left( \text{\ref{Representations
P BGG F2}}\right) $ with the right hand side of $\left( \text{\ref%
{Representations P BGG F3}}\right) $ provided by Poincar\'{e}%
--Birkhoff--Witt Theorem. Also, for every finite dimensional $\mathfrak{q}%
_{\Bbbk }^{-}$-module $W_{\Bbbk }$, one has an isomorphism of $\mathfrak{g}%
_{\Bbbk }$-modules%
\begin{equation}
Hom_{U\left( \mathfrak{q}_{\Bbbk }^{-}\right) }\left( U\left( \mathfrak{g}%
_{\Bbbk }\right) ,W_{\Bbbk }\right) \simeq \left( U\left( \mathfrak{g}%
_{\Bbbk }\right) \otimes _{U\left( \mathfrak{q}_{\Bbbk }^{-}\right)
}W_{\Bbbk }^{\vee }\right) ^{\ast }\text{,}  \label{Representations P BGG F4}
\end{equation}%
where we write $\left( -\right) ^{\ast }$ for the usual duality (see \cite[%
Ch. 5, 5.5.4 Proposition]{Di77}, applied to the dual $W_{\Bbbk }^{\vee }$ of 
$W_{\Bbbk }$, $W_{\Bbbk }$ being of finite dimension). Then $\left( \text{%
\ref{Representations P BGG F1}}\right) $ is obtained noticing that $\left(
-\right) ^{\ast ,\mathfrak{u}_{\Bbbk }^{\infty }}=\left( -\right) ^{\vee }$
on the $\mathcal{O}^{\mathfrak{q}_{\Bbbk }^{-}}$ BGG category (cfr. \cite[%
Lemma 6]{Jo11}). This beautiful argument is taken from O. T. R. Jones (which
considered the Borel case): we have included it as it will play a
significant role later. Because $\left( -\right) ^{\vee }$ establishes an
antiequivalence between the $\mathcal{O}^{\mathfrak{q}_{\Bbbk }^{-}}$ and
the $\mathcal{O}^{\mathfrak{q}_{\Bbbk }}$ BGG categories (cfr. \cite[Lemma 2]%
{Jo11}), the claimed exact sequence $\left( \text{\ref{Representations P BGG
F BGG}}\right) $ is obtained, at the level of the underlying modules, by
dualizing \cite[\S 9.16, Theorem]{Hu08}. Because $\Bbbk $ is a field, one
gets in this way an exact sequence of $\mathfrak{g}_{R}$-modules for any $%
\Bbbk $-algebra $R$ applying the exact functor $R\otimes _{\Bbbk }-$ and the
claim follows noticing that the morphisms are $\mathbf{Q}$-equivariant, as
it follows from \cite[Part I, \S 7.10 $\left( 1\right) $ and 7.16 Lemma]%
{Jan07}, which applies because $\mathbf{Q}$ is infinitesimally flat and
integral since it is smooth and connected (cfr. \cite[Part I, \S 7.17]{Jan07}%
).
\end{proof}

\bigskip

In order to work integrally some caution is needed. Suppose, until the end
of \S \ref{S Representations BGG}, that $\Bbbk $ is a Dedekind domain with
characteristic zero fraction field $\Bbbk _{fr}$ and let $?\in \left\{ 
\mathbf{M},\mathbf{Q},\mathbf{G},\left( \mathfrak{g},\mathbf{Q}\right)
\right\} $. For every algebraic module $V$\ in $\mathrm{Rep}_{alg}\left(
?\right) $\ whose underlying $\Bbbk $-module is torsion free, the base
changed functor $V_{/\Bbbk _{fr}}$ (defined via $V_{/\Bbbk _{fr}}\left(
R\right) :=V\left( R\right) $ on $\Bbbk _{fr}$-algebras) belongs to $\mathrm{%
Rep}_{alg}\left( ?_{/\Bbbk _{fr}}\right) $ and the underlying $\Bbbk $%
-module of $V$ functorially injects into that of $V_{/\Bbbk _{fr}}$ (which
is $\Bbbk _{fr}\otimes _{\Bbbk }V$) and spans it over $\Bbbk _{fr}$. In the
opposite direction, if $V_{fr}$ is in $\mathrm{Rep}_{alg}\left( ?_{/\Bbbk
_{fr}}\right) $, any object $V$ of $\mathrm{Rep}_{alg}\left( ?\right) $
whose underlying $\Bbbk $-module is torsion free and such that $V_{/\Bbbk
_{fr}}\simeq V_{fr}$ will be called a $?$-model of $V$.

\begin{remark}
\label{Representations R IndModel}Suppose that $\rho $ in $\mathrm{Rep}%
_{f}\left( \mathbf{Q}^{-}\right) $ is a $\mathbf{Q}^{-}$-model of $\rho
_{fr} $ in $\mathrm{Rep}_{f}\left( \mathbf{Q}_{/\Bbbk _{fr}}^{-}\right) $.
Then $\mathrm{Ind}_{\mathbf{Q}^{-}}^{\mathbf{G}}\left( \rho \right) \left[ 
\mathbf{Y}\right] $ is a $\left( \mathfrak{g},\mathbf{Q}\right) $-model of $%
\mathrm{Ind}_{\mathbf{Q}_{/\Bbbk _{fr}}^{-}}^{\mathbf{G}_{/\Bbbk
_{fr}}}\left( \rho _{fr}\right) \left[ \mathbf{Y}_{/\Bbbk _{fr}}\right] $.
\end{remark}

\begin{proof}
The result follows from the fact that, as noticed in the proof of Remark \ref%
{Representations R IndEx}, $\mathrm{Ind}_{\mathbf{Q}^{-}}^{\mathbf{G}}\left(
\rho \right) \left[ \mathbf{Y}\right] $ is identified with $\mathcal{O}%
\left( \mathbf{U}\right) \otimes _{\Bbbk }\rho $ in $\underline{\mathrm{Mod}}
$ and, similarly, for the base changed objects.
\end{proof}

Fix any $\mathbf{M}$-model $W_{\lambda }\in \mathrm{Rep}_{f}\left( \mathbf{M}%
\right) $ of $W_{\lambda _{/\Bbbk _{fr}}}:=\mathrm{Ind}_{\mathbf{B}_{\mathbf{%
M}_{/\Bbbk _{fr}}}^{-}}^{\mathbf{M}_{/\Bbbk _{fr}}}\left( \lambda _{/\Bbbk
_{fr}}\right) $: because $\Bbbk $ is a Dedekind domain it exists (see \cite[%
Part I, \S 10.4 Lemma]{Jan07}) and, writing $Dist\left( \mathbf{M}\right) $
for the distribution algebra of $\mathbf{M}$ (see \cite[Part I, \S 7.1]%
{Jan07}), one could be more canonical choosing $W_{\lambda }:=Dist\left( 
\mathbf{M}\right) v_{\lambda _{/\Bbbk _{fr}}}$ for a highest weight vector $%
v_{\lambda _{/\Bbbk _{fr}}}$ of $W_{\lambda _{/\Bbbk _{fr}}}$ (see \cite[%
Part II, \S 8.3]{Jan07}). It then follows that $V_{\lambda }:=\mathrm{Ind}_{%
\mathbf{Q}^{-}}^{\mathbf{G}}\left( W_{\lambda }\right) \left[ \mathbf{Y}%
\right] $ is a $\left( \mathfrak{g},\mathbf{Q}\right) $-model of $V_{\lambda
_{/\Bbbk _{fr}}}:=\mathrm{Ind}_{\mathbf{Q}_{/\Bbbk _{fr}}^{-}}^{\mathbf{G}%
_{/\Bbbk _{fr}}}\left( W_{\lambda _{/\Bbbk _{fr}}}\right) \left[ \mathbf{Y}%
_{/\Bbbk _{fr}}\right] $ (by Remark \ref{Representations R IndModel}). In
order to appropriately define $L_{\lambda }$ as the inverse image of $%
L_{\lambda _{/\Bbbk _{fr}}}:=\mathrm{Ind}_{\mathbf{B}_{\mathbf{G}_{/\Bbbk
_{fr}}}^{-}}^{\mathbf{G}_{/\Bbbk _{fr}}}\left( \lambda _{/\Bbbk
_{fr}}\right) $ in $V_{\lambda }$ taken in $\mathrm{Rep}_{alg}\left( 
\mathfrak{g},\mathbf{Q}\right) $, some explanation is needed (cfr. \cite[%
Part I, \S 10.1]{Jan07}).

We start with the easy observation that every $V$ in $\underline{\mathrm{Mod}%
}\left( \Bbbk _{fr}\right) $ gives rise to $V_{\mid \Bbbk }$\ in $\underline{%
\mathrm{Mod}}=\underline{\mathrm{Mod}}\left( \Bbbk \right) $ via $V_{\mid
\Bbbk }\left( R\right) :=V\left( \Bbbk _{fr}\otimes _{\Bbbk }R\right) $ on $%
\Bbbk $-algebras\ and $V_{\mid \Bbbk }\in \underline{\mathrm{Mod}}_{alg}$ if 
$V\in \underline{\mathrm{Mod}}_{alg}\left( \Bbbk _{fr}\right) $, the
underlying $\Bbbk $-module being $V_{\Bbbk _{fr}}$ viewed as a $\Bbbk $%
-module. If $V$ is in $\mathrm{Rep}_{alg}\left( \mathfrak{g}_{/\Bbbk
_{fr}}\right) $ (resp. $\mathrm{Rep}_{alg}\left( \mathbf{Q}_{/\Bbbk
_{fr}}\right) $), then $V_{\mid \Bbbk }$ is in $\mathrm{Rep}_{alg}\left( 
\mathfrak{g}\right) $ (resp. $\mathrm{Rep}_{alg}\left( \mathbf{Q}\right) $)
just because the source of $\mathfrak{g}_{R}\rightarrow \mathfrak{g}_{\Bbbk
_{fr}\otimes _{\Bbbk }R}$ (resp. $\mathbf{Q}\left( R\right) \rightarrow 
\mathbf{Q}\left( \Bbbk _{fr}\otimes _{\Bbbk }R\right) =\mathbf{Q}_{/\Bbbk
_{fr}}\left( \Bbbk _{fr}\otimes _{\Bbbk }R\right) $) acts on $V_{\mid \Bbbk
}\left( R\right) =V\left( \Bbbk _{fr}\otimes _{\Bbbk }R\right) $ for every $%
\Bbbk $-algebra $R$ (and the underlying $\mathcal{O}\left( \mathbf{Q}\right) 
$-comodule is given by the $\mathcal{O}\left( \mathbf{Q}_{/\Bbbk
_{fr}}\right) $-comodule structure followed by the identification $\mathcal{O%
}\left( \mathbf{Q}_{/\Bbbk _{fr}}\right) \otimes _{\Bbbk _{fr}}V_{\Bbbk
_{fr}}=\mathcal{O}\left( \mathbf{Q}\right) \otimes _{\Bbbk _{fr}}V_{\Bbbk
_{fr}}$). We deduce that $V_{\mid \Bbbk }$ is in $\mathrm{Rep}_{alg}\left( 
\mathfrak{g},\mathbf{Q}\right) $ when $V_{\mid \Bbbk }$ is in $\mathrm{Rep}%
_{alg}\left( \mathfrak{g}_{/\Bbbk _{fr}},\mathbf{Q}_{/\Bbbk _{fr}}\right) $.
Finally, if $V$ is in $\mathrm{Rep}_{alg}\left( \mathfrak{g},\mathbf{Q}%
\right) $, the canonical morphism $V\rightarrow V_{/\Bbbk _{fr}\mid \Bbbk }$
is in $\mathrm{Rep}_{alg}\left( \mathfrak{g},\mathbf{Q}\right) $. Hence,
because $\mathrm{Rep}_{alg}\left( \mathfrak{g},\mathbf{Q}\right) $ is an
abelian category, it makes sense to form in $\mathrm{Rep}_{alg}\left( 
\mathfrak{g},\mathbf{Q}\right) $\ the following cartesian diagram%
\begin{equation*}
\begin{array}{ccc}
L_{\lambda } & \longrightarrow & V_{\lambda } \\ 
\downarrow &  & \downarrow \\ 
L_{\lambda _{/\Bbbk _{fr}}\mid \Bbbk } & \longrightarrow & V_{\lambda /\Bbbk
_{fr}\mid \Bbbk }\text{.}%
\end{array}%
\end{equation*}%
Concretely, at the level of underlying $\Bbbk $-modules we form the
cartesian diagram which defines $L_{\lambda ,\Bbbk }$ and then $L_{\lambda
}\left( R\right) =R\otimes _{\Bbbk }L_{\lambda ,\Bbbk }$, that we know to
comes equipped with a $\left( \mathfrak{g},\mathbf{Q}\right) $-module
structure; we have that $L_{\lambda }\left( R\right) $ equals the $R$-points
of the inverse image of $L_{\lambda _{/\Bbbk _{fr}}}$ in $V_{\lambda }$ as a
functor for every $\Bbbk $-flat algebra $R$ (hence every $R$, when $\Bbbk $
is field).

\begin{lemma}
\label{Representations L Model}We have, indeed, that $L_{\lambda }$ is a $%
\left( \mathfrak{g},\mathbf{Q}\right) $-model of $L_{\lambda _{/\Bbbk
_{fr}}} $ in $\mathrm{Rep}_{f}\left( \mathfrak{g},\mathbf{Q}\right) $.
Furthermore, if $v_{\lambda _{/\Bbbk _{fr}}}$ is a highest weight vector $%
v_{\lambda _{/\Bbbk _{fr}}}$ of $L_{\lambda _{/\Bbbk _{fr}}}$ contained in $%
L_{\lambda } $ (which exists), then we have%
\begin{equation*}
L_{\lambda ,\mathfrak{g}}:=U\left( \mathfrak{u}^{-}\right) v_{\lambda
_{/\Bbbk _{fr}}}\subset L_{\lambda }\subset Dist\left( \mathbf{U}^{-}\right)
v_{\lambda _{/\Bbbk _{fr}}}=:L_{\lambda ,\mathbf{G}}\text{,}
\end{equation*}%
and $L_{\lambda ,\mathbf{G}}=Dist\left( \mathbf{G}\right) v_{\lambda
_{/\Bbbk _{fr}}}$ (resp. $L_{\lambda ,\mathfrak{g}}$) is a $\mathbf{G}$%
-model (resp. $\mathfrak{g}$-model)\ of $L_{\lambda _{/\Bbbk _{fr}}}$ in $%
\mathrm{Rep}_{f}\left( \mathbf{G}\right) $ (resp. $\mathrm{Rep}_{f}\left( 
\mathfrak{g}\right) $).
\end{lemma}

\begin{proof}
In order to prove the first assertion, we are just left to check that the
underlying $\Bbbk $-module of $L_{\lambda }$ is a lattice in $L_{\lambda
_{/\Bbbk _{fr}}}$. Using the identifications $V_{\lambda }\left( R\right) =%
\mathcal{O}\left( \mathbf{U}_{/R}\right) \otimes _{\Bbbk }W_{\lambda ,R}$
and $V_{\lambda }\left( R_{\Bbbk _{fr}}\right) =\Bbbk _{fr}\otimes _{\Bbbk }%
\mathcal{O}\left( \mathbf{U}_{/R}\right) \otimes _{\Bbbk }W_{\lambda ,R}$
(see the proof of Remark \ref{Representations R IndModel}), we have the to
consider the following cartesian diagram:%
\begin{equation*}
\begin{array}{ccc}
L_{\lambda ,\Bbbk } & \longrightarrow & \mathcal{O}\left( \mathbf{U}\right)
\otimes _{\Bbbk }W_{\lambda ,\Bbbk } \\ 
\downarrow &  & \downarrow \\ 
L_{\lambda _{/\Bbbk _{fr}},\Bbbk _{fr}} & \longrightarrow & \Bbbk
_{fr}\otimes _{\Bbbk }\mathcal{O}\left( \mathbf{U}\right) \otimes _{\Bbbk
}W_{\lambda ,\Bbbk }\text{.}%
\end{array}%
\end{equation*}%
Because the upper arrow is injective by definition and the right vertical
arrow is injective, we deduce that the left vertical arrow is also injective
(so that all the arrows are injective). Since $\mathbf{U\simeq G}_{a}^{d}$
as scheme, we can find a basis $\mathcal{B}$ of $\mathcal{O}\left( \mathbf{U}%
\right) $\ as a $\Bbbk $-module that can be written as an increasing nested
union of finite subsets $\mathcal{B}_{r}\subset \mathcal{B}$ indexed by $%
r\in \mathbb{N}$. Setting $\mathcal{O}_{r}:=\sum_{b\in \mathcal{B}_{r}}\Bbbk 
\mathcal{B}_{r}$, we see that $\mathcal{O}\left( \mathbf{U}\right) \otimes
_{\Bbbk }W_{\lambda ,\Bbbk }$ is an increasing nested union of the $\mathcal{%
O}_{r}\otimes _{\Bbbk }W_{\lambda ,\Bbbk }$'s and $\Bbbk _{fr}\otimes
_{\Bbbk }\mathcal{O}\left( \mathbf{U}\right) \otimes _{\Bbbk }W_{\lambda
,\Bbbk }$ a union of the $\Bbbk _{fr}\otimes _{\Bbbk }\mathcal{O}_{r}\otimes
_{\Bbbk }W_{\lambda ,\Bbbk }$'s. Because $L_{\lambda _{/\Bbbk _{fr}},\Bbbk
_{fr}}$ is finite dimensional, it is contained in some $\Bbbk _{fr}\otimes
_{\Bbbk }\mathcal{O}_{r}\otimes _{\Bbbk }W_{\lambda ,\Bbbk }$. But using the
basis $\mathcal{B}$ one sees that $\mathcal{O}_{r}\otimes _{\Bbbk
}W_{\lambda ,\Bbbk }$ is the inverse image of $\Bbbk _{fr}\otimes _{\Bbbk }%
\mathcal{O}_{r}\otimes _{\Bbbk }W_{\lambda ,\Bbbk }$ in $\mathcal{O}\left( 
\mathbf{U}\right) \otimes _{\Bbbk }W_{\lambda ,\Bbbk }$. It follows that, by
definition, $L_{\lambda ,\Bbbk }$ is contained in $\mathcal{O}_{r}\otimes
_{\Bbbk }W_{\lambda ,\Bbbk }$ and, hence, it is a finitely generated $\Bbbk $%
-module. In order to show that $L_{\lambda }$ generates $L_{\lambda _{/\Bbbk
_{fr}},\Bbbk _{fr}}$ as a $\Bbbk _{fr}$-module, we first note that, by
definition of $L_{\lambda }$ as the inverse image of $L_{\lambda _{/\Bbbk
_{fr}}\mid \Bbbk }$ in $V_{\lambda }$ (thought in the category $\underline{%
\mathrm{Mod}}_{alg}$), it contains $W_{\lambda }$ because this contained in $%
V_{\lambda }$ (by means of $\left( \text{\ref{Representations F WIncl}}%
\right) $) and mapped to $W_{\lambda _{/\Bbbk _{fr}}}\subset L_{\lambda
_{/\Bbbk _{fr}}}$. Because $W_{\lambda }$ is a model of $W_{\lambda _{/\Bbbk
_{fr}}}$, it contains a highest weight vector: as this vector as weight $%
\lambda _{/\Bbbk _{fr}}$, by highest weight theory\ it is also a highest
weight vector of $L_{\lambda _{/\Bbbk _{fr}}}$ and the $\Bbbk _{fr}$-span $%
\Bbbk _{fr}L_{\lambda ,\Bbbk }\subset L_{\lambda _{/\Bbbk _{fr}},\Bbbk
_{fr}} $\ of $L_{\lambda ,\Bbbk }$ equals $L_{\lambda ,\Bbbk }$, being a $%
\mathfrak{g}_{\Bbbk _{fr}}$-submodule.

Having checked that $L_{\lambda ,\Bbbk }$ is a lattice in $L_{\lambda
_{/\Bbbk _{fr}},\Bbbk _{fr}}$, we can take a highest weight vector $%
v_{\lambda _{/\Bbbk _{fr}}}\in L_{\lambda }\cap L_{\lambda _{/\Bbbk _{fr}}}$%
. Then the inclusions and the other assertions can be deduced from the fact
that $Dist\left( \mathbf{G}\right) =Dist\left( \mathbf{U}^{-}\right) \otimes
_{\Bbbk }Dist\left( \mathbf{Q}\right) $ (see \cite[Part II, \S 1.12 $\left(
2\right) $]{Jan07} for the equality) is an algebra over $U\left( \mathfrak{g}%
\right) =U\left( \mathfrak{u}^{-}\right) \otimes _{\Bbbk }U\left( \mathfrak{q%
}\right) $ (the equality holds by Poincar\'{e}--Birkhoff--Witt Theorem) and $%
Dist\left( \mathbf{G}\right) v_{\lambda _{/\Bbbk _{fr}}}$ is a $\mathbf{G}$%
-model of $L_{\lambda _{/\Bbbk _{fr}}}$ in $\mathrm{Rep}_{f}\left( \mathbf{G}%
\right) $ thanks to \cite[Part II, \S 8.3]{Jan07}.
\end{proof}

We can now prove the following result.

\begin{theorem}
\label{Representations T BGG}We can find $\mathbf{M}$-models $W_{w\cdot
\lambda }\in \mathrm{Rep}_{f}\left( \mathbf{M}\right) $ of $W_{w\cdot
\lambda _{/\Bbbk _{fr}}}$ for every $w\in W^{\mathbf{M}}$ such that $l\left(
w\right) =1$ for which there is a left exact sequence in $\mathrm{Rep}%
_{alg}\left( \mathfrak{g},\mathbf{Q}\right) $%
\begin{equation}
0\longrightarrow L_{\lambda }\longrightarrow \mathrm{Ind}_{\mathbf{Q}^{-}}^{%
\mathbf{G}}\left( W_{\lambda }\right) \left[ \mathbf{Y}\right] \overset{d}{%
\longrightarrow }\tbigoplus\nolimits_{w\in W^{\mathbf{M}}:l\left( w\right)
=1}\mathrm{Ind}_{\mathbf{Q}^{-}}^{\mathbf{G}}\left( W_{w\cdot \lambda
}\right) \left[ \mathbf{Y}\right]  \label{Representations T BGG F BGG}
\end{equation}%
which is a $\left( \mathfrak{g},\mathbf{Q}\right) $-model of the first three
term of the dual BGG-complex $\left( \text{\ref{Representations P BGG F BGG}}%
\right) $. Moreover, the inclusion $\left( \text{\ref{Representations F
WIncl}}\right) $\ of $W_{\lambda }$ in $\mathrm{Ind}_{\mathbf{Q}^{-}}^{%
\mathbf{G}}\left( W_{\lambda }\right) \left[ \mathbf{Y}\right] =:V_{\lambda
} $ factors through the inclusion of $L_{\lambda }$ in $V_{\lambda }$ given
by the above exact sequence:%
\begin{equation}
W_{\lambda }\hookrightarrow L_{\lambda }\hookrightarrow V_{\lambda }\text{,}
\label{Representations T BGG F Incl}
\end{equation}%
which is a model of $\left( \text{\ref{Representations P BGG F Incl}}\right) 
$.
\end{theorem}

\begin{proof}
We have already remarked in the proof of Lemma \ref{Representations L Model}
that we have $\left( \text{\ref{Representations T BGG F Incl}}\right) $,
which is a model of $\left( \text{\ref{Representations P BGG F Incl}}\right) 
$ in view of Lemma \ref{Representations L Model} and the discussion
preceding it. By definition, $L_{\lambda }$ is a subobject of $\mathrm{Ind}_{%
\mathbf{Q}^{-}}^{\mathbf{G}}\left( W_{\lambda }\right) \left[ \mathbf{Y}%
\right] $ in $\mathrm{Rep}_{alg}\left( \mathfrak{g},\mathbf{Q}\right) $ and,
hence, we only need to prolong this inclusion to a left exact sequence
furnishing a $\left( \mathfrak{g},\mathbf{Q}\right) $-model of $\left( \text{%
\ref{Representations P BGG F BGG}}\right) $. To this end we choose, for
every $w\in W^{\mathbf{M}}$ such that $l\left( w\right) =1$ a model $%
W_{w\cdot \lambda }\in \mathrm{Rep}_{f}\left( \mathbf{M}\right) $ of $%
W_{w\cdot \lambda _{/\Bbbk _{fr}}}$ (see \cite[Part I, \S 10.4 Lemma]{Jan07}%
). Then Remark \ref{Representations R IndModel} tells us that $\mathrm{Ind}_{%
\mathbf{Q}^{-}}^{\mathbf{G}}\left( W_{w\cdot \lambda }\right) \left[ \mathbf{%
Y}\right] $ is a $\left( \mathfrak{g},\mathbf{Q}\right) $-model of $%
V_{w\cdot \lambda _{/\Bbbk _{fr}}}$. Indeed, it follows from the proof of
this remark that, we also have $\mathrm{Ind}_{\mathbf{Q}^{-}}^{\mathbf{G}%
}\left( \pi _{w}W_{w\cdot \lambda }\right) \left[ \mathbf{Y}\right] =\pi _{w}%
\mathrm{Ind}_{\mathbf{Q}^{-}}^{\mathbf{G}}\left( W_{w\cdot \lambda }\right) %
\left[ \mathbf{Y}\right] $ in $V_{w\cdot \lambda _{/\Bbbk _{fr}}}$ for every 
$\pi _{w}\in \Bbbk $. Then%
\begin{equation*}
C:=\tbigoplus\nolimits_{w\in W^{\mathbf{M}}:l\left( w\right) =1}\mathrm{Ind}%
_{\mathbf{Q}^{-}}^{\mathbf{G}}\left( \pi _{w}W_{w\cdot \lambda }\right) %
\left[ \mathbf{Y}\right] =\tbigoplus\nolimits_{w\in W^{\mathbf{M}}:l\left(
w\right) =1}\pi _{w}\mathrm{Ind}_{\mathbf{Q}^{-}}^{\mathbf{G}}\left(
W_{w\cdot \lambda }\right) \left[ \mathbf{Y}\right]
\end{equation*}%
is a model of $C_{/\Bbbk _{fr}}:=\tbigoplus\nolimits_{w\in W^{\mathbf{M}%
}:l\left( w\right) =1}V_{w\cdot \lambda _{/\Bbbk _{fr}}}$ for every family $%
\left\{ \pi _{w}:l\left( w\right) =1\right\} $ and we will show that a
suitable choice of such a family gives a model of $C_{/\Bbbk _{fr}}$
containing the image of $\mathrm{Ind}_{\mathbf{Q}^{-}}^{\mathbf{G}}\left(
W_{\lambda }\right) \left[ \mathbf{Y}\right] $ in $C_{/\Bbbk _{fr}}$ via the
morphism $d$ appearing in $\left( \text{\ref{Representations P BGG F BGG}}%
\right) $. This fact, together with the Remark at the end of the proof, will
give the result.

To this end, let us first define integral BGG categories $\mathcal{O}^{%
\mathfrak{q}_{\Bbbk }^{-}}$ and $\mathcal{O}^{\mathfrak{q}_{\Bbbk }}$
naively as in \cite[\S 9.3]{Hu08}, but over the ring $\Bbbk $ rather than a
field and replacing condition $\left( \mathcal{O}^{\mathfrak{p}}2\right) $
of loc.cit. by the requirement that, viewed as a $U\left( \mathfrak{m}%
_{\Bbbk }\right) $-module, an object $M$ of this category should be the
direct sum of $U\left( \mathfrak{m}_{\Bbbk }\right) $-modules $M_{i}$ that
are finitely generated and projective as $\Bbbk $-modules and such that $%
\Bbbk _{fr}\otimes _{\Bbbk }M_{i}$ is a simple $U\left( \mathfrak{m}_{\Bbbk
}\right) $-module. Define a duality (exchanging the categories $\mathcal{O}^{%
\mathfrak{q}_{\Bbbk }^{-}}$ and $\mathcal{O}^{\mathfrak{q}_{\Bbbk }}$) by
setting $M^{\vee }:=\bigoplus\nolimits_{i}M_{i}^{\ast }$ where $\left(
-\right) ^{\ast }$ denotes the $\Bbbk $-dual. Then we claim that $\left( 
\text{\ref{Representations P BGG F1}}\right) $ generalizes to an inclusion%
\begin{equation}
\mathrm{Ind}_{\mathbf{Q}^{-}}^{\mathbf{G}}\left( W_{\mu }\right) \left[ 
\mathbf{Y}\right] _{\Bbbk }\hookrightarrow \left( U\left( \mathfrak{g}%
_{\Bbbk }\right) \otimes _{U\left( \mathfrak{q}_{\Bbbk }^{-}\right) }W_{\mu
,\Bbbk }^{\vee }\right) ^{\vee }\text{.}  \label{Representations T BGG F1}
\end{equation}%
Indeed, $\left( \text{\ref{Representations P BGG F2}}\right) $ is still
injective (because $\mathcal{O}\left( \mathbf{U}\right) $ does not have $%
\mathbb{Z}$-torsion, since $\mathbf{U}$ is flat over $\Bbbk $ which is $%
\mathbb{Z}$-torsion free) and the proof of \cite[Ch. 5, 5.5.4 Proposition]%
{Di77} shows that $\left( \text{\ref{Representations P BGG F4}}\right) $ is
an isomorphism as long as we assume that $W_{\Bbbk }^{\vee }$ is finitely
generated and projective over $\Bbbk $ (which is true for $W_{\Bbbk }=W_{\mu
,\Bbbk }$). Then we note that $U\left( \mathfrak{g}_{\Bbbk }\right) \otimes
_{U\left( \mathfrak{q}_{\Bbbk }^{-}\right) }W_{\Bbbk ,\mu }^{\vee }=U\left( 
\mathfrak{u}_{\Bbbk }\right) \otimes _{\Bbbk }W_{\Bbbk ,\mu }^{\vee }$
belongs to the integral BGG category $\mathcal{O}^{\mathfrak{q}_{\Bbbk
}^{-}} $ and that the proof of the fact that $\left( -\right) ^{\ast ,%
\mathfrak{u}_{\Bbbk }^{\infty }}=\left( -\right) ^{\vee }$ on the $\mathcal{O%
}^{\mathfrak{q}_{\Bbbk }^{-}}$ BGG category given in \cite[Lemma 6]{Jo11}
still works over $\Bbbk $ (or can be deduced from it by extension of the
scalars). Hence, we have the following commutative diagram, where the right
vertical arrow takes a $\Bbbk $-linear morphism from $X:=U\left( \mathfrak{g}%
_{\Bbbk }\right) \otimes _{U\left( \mathfrak{q}_{\Bbbk }^{-}\right) }W_{\mu
,\Bbbk }^{\vee }$ to $\Bbbk $ of the source to its linear extension to $%
\Bbbk _{fr}$-linear morphism from $U\left( \mathfrak{g}_{\Bbbk _{fr}}\right)
\otimes _{U\left( \mathfrak{q}_{\Bbbk _{fr}}^{-}\right) }W_{\mu ,\Bbbk
_{fr}}^{\vee }=\Bbbk _{fr}\otimes _{\Bbbk }X$ to $\Bbbk _{fr}$ (the equality
in view of $U\left( \mathfrak{g}_{?}\right) \otimes _{U\left( \mathfrak{q}%
_{?}^{-}\right) }W_{?}^{\vee }=U\left( \mathfrak{u}_{?}\right) \otimes
_{?}W_{?}^{\vee }$):%
\begin{equation*}
\begin{array}{ccc}
\mathrm{Ind}_{\mathbf{Q}^{-}}^{\mathbf{G}}\left( W_{\mu }\right) \left[ 
\mathbf{Y}\right] _{\Bbbk } & \hookrightarrow & \left( U\left( \mathfrak{g}%
_{\Bbbk }\right) \otimes _{U\left( \mathfrak{q}_{\Bbbk }^{-}\right) }W_{\mu
,\Bbbk }^{\vee }\right) ^{\vee } \\ 
\cap &  & \cap \\ 
\mathrm{Ind}_{\mathbf{Q}^{-}}^{\mathbf{G}}\left( W_{\mu }\right) \left[ 
\mathbf{Y}\right] _{\Bbbk _{fr}} & \overset{\sim }{\rightarrow } & \left(
U\left( \mathfrak{g}_{\Bbbk _{fr}}\right) \otimes _{U\left( \mathfrak{q}%
_{\Bbbk _{fr}}^{-}\right) }W_{\mu ,\Bbbk _{fr}}^{\vee }\right) ^{\vee }\text{%
.}%
\end{array}%
\end{equation*}%
It follows that we can canonically identify $\left( U\left( \mathfrak{g}%
_{\Bbbk }\right) \otimes _{U\left( \mathfrak{q}_{\Bbbk }^{-}\right) }W_{\mu
,\Bbbk }^{\vee }\right) ^{\vee }$ with a $\Bbbk $-submodule of $\mathrm{Ind}%
_{\mathbf{Q}^{-}}^{\mathbf{G}}\left( W_{\mu }\right) \left[ \mathbf{Y}\right]
_{\Bbbk _{fr}}$ containing $\mathrm{Ind}_{\mathbf{Q}^{-}}^{\mathbf{G}}\left(
W_{\mu }\right) \left[ \mathbf{Y}\right] _{\Bbbk }$. Hence, we can just
check that there is a family $\left\{ \pi _{w}:l\left( w\right) =1\right\} $
such that the resulting $C$ has underlying $\Bbbk $-module containing the
image of $\left( U\left( \mathfrak{g}_{\Bbbk }\right) \otimes _{U\left( 
\mathfrak{q}_{\Bbbk }^{-}\right) }W_{\mu ,\Bbbk }^{\vee }\right) ^{\vee }$
for $\mu =\lambda $. But we have that $\left( U\left( \mathfrak{g}_{\Bbbk
}\right) \otimes _{U\left( \mathfrak{q}_{\Bbbk }^{-}\right) }W_{\mu ,\Bbbk
}^{\vee }\right) ^{\vee }$ contains a $\Bbbk $-submodule isomorphic to $%
W_{\mu ,\Bbbk }^{\vee }$ such that $U\left( \mathfrak{u}_{\Bbbk }^{-}\right)
\cdot W_{\mu ,\Bbbk }^{\vee }=\left( U\left( \mathfrak{g}_{\Bbbk }\right)
\otimes _{U\left( \mathfrak{q}_{\Bbbk }^{-}\right) }W_{\mu ,\Bbbk }^{\vee
}\right) ^{\vee }$ (because $U\left( \mathfrak{g}_{\Bbbk }\right) \otimes
_{U\left( \mathfrak{q}_{\Bbbk }^{-}\right) }W_{\mu ,\Bbbk }^{\vee }=U\left( 
\mathfrak{u}_{\Bbbk }\right) \otimes _{\Bbbk }W_{\mu ,\Bbbk }^{\vee }$) and $%
W_{\mu ,\Bbbk }^{\vee }$ is a finitely generated $\Bbbk $-module: it follows
that there is a choice of the family $\left\{ \pi _{w}:l\left( w\right)
=1\right\} $ such that the underlying $\Bbbk $-module of $C$ contains this
copy of $W_{\mu ,\Bbbk }^{\vee }$, implying that it contains $U\left( 
\mathfrak{u}_{\Bbbk }\right) \cdot W_{\Bbbk }^{\vee }$ and, hence, proving
the result.

\textbf{Remark.} \textit{Suppose that }$\mathbf{H}$\textit{\ is a flat
affine group scheme over a domain }$\Bbbk $\textit{\ with fraction field }$%
\Bbbk _{fr}$\textit{\ and that }$V_{1},V_{2}$\textit{\ in }$\mathrm{Rep}%
_{alg}\left( \mathbf{H}\right) $\textit{\ are such that }$V_{2,\Bbbk }$%
\textit{\ is }$\Bbbk $\textit{-torsion free. If }$f\in Hom_{\Bbbk }\left(
V_{1,\Bbbk },V_{2,\Bbbk }\right) $\textit{\ is such that its (injective, by
torsion freeness of }$V_{2,\Bbbk }$\textit{) image in }$Hom_{\Bbbk
_{fr}}\left( V_{1,\Bbbk _{fr}},V_{2,\Bbbk _{fr}}\right) $\textit{\ belongs
to }$\mathrm{Rep}_{alg}\left( \mathbf{H}_{/\Bbbk _{fr}}\right) $\textit{,
then }$f$\textit{\ itself belongs to }$\mathrm{Rep}_{alg}\left( \mathbf{H}%
\right) $\textit{.}

\textit{Proof.} Writing $\rho _{i}^{\#}$ for the comodule structure of $%
V_{i} $, we have to show that the equality $\rho _{2}^{\#}\circ f=f\otimes
_{\Bbbk }1\circ \rho _{1}^{\#}$ holds in $Hom_{\Bbbk }\left( V_{1,\Bbbk
},V_{2,\Bbbk }\otimes _{\Bbbk }\mathcal{O}\left( \mathbf{H}\right) \right) $%
. Because $V_{2,\Bbbk }\otimes _{\Bbbk }\mathcal{O}\left( \mathbf{H}\right) $
is $\Bbbk $-torsion free, the equality can be checked in $Hom_{\Bbbk
_{fr}}\left( V_{1,\Bbbk _{fr}},V_{2,\Bbbk _{fr}}\otimes _{\Bbbk _{fr}}%
\mathcal{O}\left( \mathbf{Q}\right) \right) $, where it holds by assumption.
\end{proof}

\begin{remark}
\label{Representations R Model}By definition, if $\mathrm{Ind}_{\mathbf{Q}%
^{-}}^{\mathbf{G}}\left( W_{\lambda }\right) \left[ \mathbf{Y}\right] $
happens to be a $Dist\left( \mathbf{G}\right) $-module, then $L_{\lambda
}=Dist\left( \mathbf{G}\right) v_{\lambda _{/\Bbbk _{fr}}}$ is also $%
Dist\left( \mathbf{G}\right) $-module and, hence, an object of $\mathrm{Rep}%
_{f}\left( \mathbf{G}\right) $ because (the underlying $\Bbbk $-module of) $%
L_{\lambda }$ is a lattice in (the underlying $\Bbbk _{fr}$-module of) the
finite dimensional representation $L_{\lambda _{/\Bbbk _{fr}}}$\ of $\mathbf{%
G}_{/\Bbbk _{fr}}$ (see \cite[Part II, \S 8.3]{Jan07}). Furthermore, because 
$Dist\left( \mathbf{G}\right) =Dist\left( \mathbf{U}^{-}\right) \otimes
_{\Bbbk }Dist\left( \mathbf{Q}\right) $ (see \cite[Part II, \S 1.12 $\left(
2\right) $]{Jan07} for the equality) and $\mathrm{Ind}_{\mathbf{Q}^{-}}^{%
\mathbf{G}}\left( W_{\lambda }\right) \left[ \mathbf{Y}\right] $ is a $%
\left( \mathfrak{g},\mathbf{Q}\right) $-module, it is a $Dist\left( \mathbf{G%
}\right) $-module if and only if it is a $Dist\left( \mathbf{U}^{-}\right) $%
-module. We will use this observation in order to verify that $L_{\lambda
}=L_{\lambda ,\mathbf{G}}$ in case $\mathbf{G}=\mathbf{GSp}_{2g}$ (see
Proposition \ref{Representations P Model} below).
\end{remark}

\subsection{The degree filtration on the induced modules in the symplectic
case}

It will be convenient to introduce the following assumption.

\begin{axiom}
\label{Representations Ass1}There is an isomorphism of pointed $\Bbbk $%
-schemes%
\begin{equation*}
u:\mathbf{G}_{a}^{d}\overset{\sim }{\longrightarrow }\mathbf{U}
\end{equation*}%
with the property that the action of $\mathbf{Q}$ on $\mathbf{G}_{a}^{d}$
obtained by transport from its action $\left( \text{\ref{Representations F
Act}}\right) $\ on $\mathbf{U}$ is given by affine transformations and its
restriction to $\mathbf{M}\subset \mathbf{Q}$ is given by linear
transformations.
\end{axiom}

Assumption \ref{Representations Ass1} allows us to define $\mathbf{Q}$%
-subrepresentations\ (resp. $\mathbf{M}$-stable direct summands) of $\mathrm{%
Ind}_{\mathbf{Q}^{-}}^{\mathbf{G}}\left( \rho \right) \left[ \mathbf{Y}%
\right] $ as follows (for $\rho \in \mathrm{Rep}_{f}\left( \mathbf{Q}%
^{-}\right) $): for every $r\in \mathbb{N}$, we let $\mathrm{Ind}_{\mathbf{Q}%
^{-}}^{\mathbf{G}}\left( \rho \right) \left[ \mathbf{Y}\right] _{\leq
r}\subset \mathrm{Ind}_{\mathbf{Q}^{-}}^{\mathbf{G}}\left( \rho \right) %
\left[ \mathbf{Y}\right] $ (resp. $\mathrm{Ind}_{\mathbf{Q}^{-}}^{\mathbf{G}%
}\left( \rho \right) \left[ \mathbf{Y}\right] _{=r}\subset \mathrm{Ind}_{%
\mathbf{Q}^{-}}^{\mathbf{G}}\left( \rho \right) \left[ \mathbf{Y}\right] $)
be the $\mathbf{Q}$-subrepresentation (resp. $\mathbf{M}$-subrepresentation)%
\ which corresponds to the $\mathbf{Q}$-subrepresentation $\mathrm{C}^{%
\mathrm{alg}}\left( \mathbf{U},\rho \right) _{\leq r}\subset \mathrm{C}^{%
\mathrm{alg}}\left( \mathbf{U},\rho \right) $ (resp. $\mathrm{C}^{\mathrm{alg%
}}\left( \mathbf{U},\rho \right) _{\leq r}\subset \mathrm{C}^{\mathrm{alg}%
}\left( \mathbf{U},\rho \right) $) of $\rho $-valued polynomials that are of
degree $\leq r$ (resp. $=r$) with respect to the coordinates induced by $%
\mathbf{G}_{a}^{d}\overset{\sim }{\longrightarrow }\mathbf{U}$. Then $%
\left\{ \mathrm{Ind}_{\mathbf{Q}^{-}}^{\mathbf{G}}\left( \rho \right) \left[ 
\mathbf{Y}\right] _{\leq r}\right\} $ is an increasing exhaustive filtration
of $\mathrm{Ind}_{\mathbf{Q}^{-}}^{\mathbf{G}}\left( \rho \right) \left[ 
\mathbf{Y}\right] $\ by $\mathbf{Q}$-subrepresentations\ and 
\begin{equation}
\mathrm{Ind}_{\mathbf{Q}^{-}}^{\mathbf{G}}\left( \rho \right) \left[ \mathbf{%
Y}\right] =\tbigoplus\nolimits_{r\in \mathbb{N}}\mathrm{Ind}_{\mathbf{Q}%
^{-}}^{\mathbf{G}}\left( \rho \right) \left[ \mathbf{Y}\right] _{=r}\text{
as an }\mathbf{M}\text{-representations.}
\label{Representations F DirSumDec}
\end{equation}

For every $r_{1},r_{2},r\in \mathbb{N}\cup \left\{ \phi \right\} $, consider
the natural $\mathbf{Q}$-equivariant multiplication map%
\begin{equation*}
\cdot :\mathrm{Ind}_{\mathbf{Q}^{-}}^{\mathbf{G}}\left( \rho _{1}\right) %
\left[ \mathbf{Y}\right] _{\leq r_{1}}\otimes \mathrm{Ind}_{\mathbf{Q}^{-}}^{%
\mathbf{G}}\left( \rho _{2}\right) \left[ \mathbf{Y}\right] _{\leq
r_{2}}\longrightarrow \mathrm{Ind}_{\mathbf{Q}^{-}}^{\mathbf{G}}\left( \rho
_{1}\otimes \rho _{2}\right) \left[ \mathbf{Y}\right] _{\leq
r_{1}+r_{2}},\left( f_{1}\cdot f_{2}\right) \left( y\right) :=f_{1}\left(
y\right) \otimes f_{2}\left( y\right) \text{.}
\end{equation*}%
When $\rho \in \mathrm{Rep}_{f}\left( \mathbf{Q}^{-}\right) $, using the
isomorphism $\mathrm{Ind}_{\mathbf{Q}^{-}}^{\mathbf{G}}\left( \rho \right) %
\left[ \mathbf{Y}\right] \simeq \mathrm{C}^{\mathrm{alg}}\left( \mathbf{G}%
_{a}^{d},\rho \right) $ induced by $\left( \text{\ref{Representations F Iso1}%
}\right) $ and $\mathbf{G}_{a}^{d}\overset{\sim }{\longrightarrow }\mathbf{U}
$, we see that it induces, for every $r\in \mathbb{N}$, an isomorphism of $%
\mathbf{Q}$-representations%
\begin{equation}
\rho \otimes \mathrm{Sym}^{r}\left( \mathrm{Ind}_{\mathbf{Q}^{-}}^{\mathbf{G}%
}\left( \mathrm{1}\right) \left[ \mathbf{Y}\right] _{\leq 1}\right) \overset{%
\sim }{\longrightarrow }\mathrm{Ind}_{\mathbf{Q}^{-}}^{\mathbf{G}}\left(
\rho \right) \left[ \mathbf{Y}\right] _{\leq r}\text{.}
\label{Representations F IndSym}
\end{equation}

\bigskip

Suppose that $V$ is a finite free $\Bbbk $-module of rank $2g$, endowed with
a symplectic pairing $\psi $ (meaning that it is perfect and alternating)\
and define $\mathbf{GSp}\left( V,\psi \right) $ via%
\begin{equation*}
\mathbf{GSp}\left( V,\psi \right) \left( R\right) :=\left\{ \gamma \in 
\mathbf{GL}\left( R\otimes V\right) :\exists \nu \left( \gamma \right) \in
R^{\times }\text{ s.t. }\psi \left( \gamma v_{1},\gamma v_{2}\right) =\nu
\left( \gamma \right) \psi \left( v_{1},v_{2}\right) ,\forall v_{1},v_{2}\in
R\otimes V\right\} \text{.}
\end{equation*}%
We remark that $\nu \left( \gamma \right) $\ is well defined and, indeed,
using Yoneda's lemma once checks that the rule sending $\gamma $ to $\nu
\left( \gamma \right) $ defines a morphism of group schemes from $\mathbf{GSp%
}\left( V,\psi \right) $ to $\mathbf{G}_{m}$: we let $\mathbf{Sp}\left(
V,\psi \right) $ be the kernel. We also assume that $W\subset V$ is a
maximal isotropic $\Bbbk $-submodule such that $\frac{V}{W}$\ is locally
free (it always exists, when $\Bbbk $ is local, see \cite[Lemma 1.5]{Fo23})%
\footnote{%
In geometric language, $W$ is a subbundle of $V$ over $\mathfrak{Spec}\left(
K\right) $. This extra generality will be needed only later in order to
consider symplectic basis in coherent cohomology.}: then there exists a
maximal isotropic $\Bbbk $-submodule $\overline{W}$ of $V$ such that $%
V=W\oplus \overline{W}$ and every basis $\mathcal{B}_{W}:=\left\{
w_{1},...,w_{g}\right\} $ uniquely extend to a basis $\mathcal{B}_{\overline{%
W}}:=\left\{ \overline{w}_{1},...,\overline{w}_{g}\right\} $ of $\overline{W}
$ such that $\mathcal{B}:=\mathcal{B}_{W}\cup \mathcal{B}_{\overline{W}}$ is
symplectic-Hodge basis of $V$, where symplectic means that $\psi \left(
w_{i},w_{j}\right) =\psi \left( \overline{w}_{i},\overline{w}_{j}\right) =0$
and $\psi \left( w_{i},\overline{w}_{j}\right) =\delta _{i,j}$ for $i,j\in
\left\{ 1,...,g\right\} $ and Hodge refers to the fact that the first $g$%
-elements belongs to $W$ (see \cite[Proposition 1.9 $\left( 2\right) $]{Fo23}%
). We fix a symplectic-Hodge basis $\mathcal{B}$, so that the first $g$%
-vectors span $W$ (or equivalently, we fix a maximal isotropic subspace $%
\overline{W}$ which is a complement of $W$ and a basis $\mathcal{B}_{W}$ of $%
W$). Sometime, we write $w_{i}:=\overline{w}_{g+i}$ for $i\in \left\{
1,...,g\right\} $. We let $\mathbf{Q}_{W}\subset \mathbf{GSp}\left( V,\psi
\right) $ be the stabilizer of $W\subset V$ in $\mathbf{GSp}\left( V,\psi
\right) $. It is a parabolic subgroup, called a Siegel parabolic subgroup,
and we write $\mathbf{U}_{W}$ (resp. $\mathbf{M}_{W}$) for its unipotent
radical (resp. its Levi component), which is identified with the set of
those elements of $\mathbf{Q}_{W}$ inducing the identity on $W$. The choice
of $\overline{W}$ allows us to view $\mathbf{M}_{W}$ as the subgroup of
those elements of $\mathbf{GSp}\left( V,\psi \right) $ preserving the
decomposition $V=W\oplus \overline{W}$, thus realizing $\mathbf{Q}_{W}$ as a
semidirect product $\mathbf{U}_{W}\rtimes \mathbf{M}_{W}$\ and then the
corresponding opposite parabolic subgroup $\mathbf{Q}_{W}^{-}$ is identified
with the stabilized $\mathbf{Q}_{\overline{W}}$ of $\overline{W}$.

\begin{example}
\label{Representations E GSp}Suppose that $\left( V,\psi \right) =\left(
\Bbbk ^{2g},\psi _{g}\right) $, where the elements of $\Bbbk ^{2g}$ are
viewed as column vectors and $\psi _{g}\left( v_{1},v_{2}\right)
=v_{1}^{t}\left( 
\begin{array}{cc}
0 & 1_{g} \\ 
-1_{g} & 0%
\end{array}%
\right) v_{2}$. Then we write $\mathbf{GSp}_{2g}:=\mathbf{GSp}\left( \Bbbk
^{2g},\psi _{g}\right) $ and $\mathbf{Sp}_{2g}:=\mathbf{Sp}\left( \Bbbk
^{2g},\psi _{g}\right) $. In this case, we can take $W\subset V$ to be the
inclusion $\Bbbk ^{g}\subset \Bbbk ^{2g}$ given by the first $g$%
-coordinates. Then the standard basis $\mathcal{B}=\left\{
e_{1},...,e_{g},e_{g+1},...,e_{2g}\right\} $ is the symplectic-Hodge basis
which corresponds to the standard basis $\mathcal{B}_{W}=\left\{
e_{1},...,e_{g}\right\} $ of $W$ and the complement $\overline{W}$ provided
by the inclusion $\Bbbk ^{g}\subset \Bbbk ^{2g}$ given by the second set of $%
g$-coordinates. Setting $\mathbf{Q}_{g}:=\mathbf{Q}_{\Bbbk ^{g}}$, we see
that%
\begin{equation*}
\mathbf{Q}_{g}=\left\{ \left( 
\begin{array}{cc}
a & b \\ 
0 & d%
\end{array}%
\right) \in \mathbf{GL}_{2g}:d=a^{-t}\nu ,\nu \in \mathbf{G}%
_{m},ab^{t}=ba^{t}\right\} \text{.}
\end{equation*}%
Let us write $\mathbf{S}_{g}$ for the functor of symmetric matrices,
regarded as an additive group scheme. Then $\mathbf{Q}_{g}=\mathbf{M}%
_{g}\ltimes \mathbf{U}_{g}$ (where $\mathbf{M}_{g}:=\mathbf{M}_{\Bbbk ^{g}}$
and $\mathbf{U}_{g}:=\mathbf{U}_{\Bbbk ^{g}}$) and the above description of $%
\mathbf{Q}_{g}$\ shows that there are isomorphism of group schemes%
\begin{equation}
\mathbf{GL}_{g}\times \mathbf{G}_{m}\overset{\sim }{\longrightarrow }\mathbf{%
M}_{g}\text{ and }\mathbf{S}_{g}\overset{\sim }{\longrightarrow }\mathbf{U}%
_{g}  \label{Representations E GSp F}
\end{equation}%
by means of the rule sending $\left( a,\nu \right) $ to $\left( 
\begin{array}{cc}
a & 0_{g} \\ 
0_{g} & a^{-t}\nu%
\end{array}%
\right) $ and, respectively, $Y$ to $\left( 
\begin{array}{cc}
1_{g} & Y \\ 
0 & 1_{g}%
\end{array}%
\right) $. Similarly, we also have $\mathbf{S}_{g}\overset{\sim }{%
\longrightarrow }\mathbf{U}_{g}^{-}$ sending $Y$ to $\left( 
\begin{array}{cc}
1_{g} & 0 \\ 
Y & 1_{g}%
\end{array}%
\right) $.

The symplectic-Hodge basis $\mathcal{B}$ of a more general $V$\ yields an
identification $\theta :\left( \Bbbk ^{2g},\psi _{g}\right) \overset{\sim }{%
\rightarrow }\left( V,\psi \right) $ sending $e_{i}$ to $w_{i}$ and, hence,
an isomorphism $c_{\theta }:\mathbf{GSp}_{2g}\overset{\sim }{\rightarrow }%
\mathbf{GSp}\left( V,\psi \right) $ (restricting to $\mathbf{Sp}_{2g}\overset%
{\sim }{\rightarrow }\mathbf{Sp}\left( V,\psi \right) $) given by $c_{\theta
}\left( \gamma \right) =\theta \gamma \theta ^{-1}$ identifying $\mathbf{Q}%
_{g}\overset{\sim }{\rightarrow }\mathbf{Q}_{W}$. When $V=\Bbbk ^{2g}$, then 
$\psi $ is given by a matrix in $\mathbf{M}_{2g}\left( \Bbbk \right) $\ and $%
\theta \in \mathbf{GL}_{2g}\left( \Bbbk \right) $ by a matrix $\theta $ such
that $\theta ^{t}\psi \theta =\psi _{g}$.
\end{example}

\bigskip

We suppose from now on that $\mathbf{G=GSp}\left( V,\psi \right) $, we write 
$\mathbf{Q}\subset \mathbf{G}$ for the stabilizer of $W\subset V$ in $%
\mathbf{G}$ and omit the subscript $W$ in the notation for the group schemes
introduced just before Example \ref{Representations E GSp}. If $\mathbf{H}%
\subset \mathbf{G}$, we also set $\mathbf{H}^{\circ }:=\mathbf{H}\cap 
\mathbf{Sp}\left( V,\psi \right) $. It follows from Example \ref%
{Representations E GSp} that there are isomorphism (the composition $%
c_{\theta }\circ \left( \text{\ref{Representations E GSp F}}\right) $):%
\begin{equation}
\mathbf{GL}_{g}\times \mathbf{G}_{m}\overset{\sim }{\longrightarrow }\mathbf{%
M}\text{ and }\mathbf{S}_{g}\overset{\sim }{\longrightarrow }\mathbf{U}\text{%
.}  \label{Representations F Levi}
\end{equation}

\begin{lemma}
\label{Representations L1'}If we take $u$ the composition of the isomorphism
of group schemes $\mathbf{G}_{a}^{d}\simeq \mathbf{S}_{g}$ with $d=\frac{%
g\left( g+1\right) }{2}$\ obtained by the matrix coordinates of $\mathbf{S}%
_{g}$\ followed by the above identification $\mathbf{S}_{g}\overset{\sim }{%
\rightarrow }\mathbf{U}$ given by $\left( \text{\ref{Representations F Levi}}%
\right) $, then Assumption \ref{Representations Ass1} holds. In particular,
the isomorphism $\left( \text{\ref{Representations F IndSym}}\right) $ is
force for every $\rho \in \mathrm{Rep}_{f}\left( \mathbf{Q}^{-}\right) $.
\end{lemma}

\begin{proof}
We can assume $\mathbf{G}=\mathbf{GSp}_{2g}$. For future reference, we
remark that then $\mathbf{Y}$ is the set of those $\left( 
\begin{array}{cc}
a & b \\ 
c & d%
\end{array}%
\right) \in \mathbf{GSp}_{2g}$ such that $\mathrm{de}$\textrm{t}$\left(
a\right) \neq 0$ and $\left( \text{\ref{Representations F BigCell}}\right) $
is given by%
\begin{equation}
\left( 
\begin{array}{cc}
a & b \\ 
c & d%
\end{array}%
\right) =\left( 
\begin{array}{cc}
1 & 0 \\ 
ca^{-1} & 1%
\end{array}%
\right) \left( 
\begin{array}{cc}
a & 0 \\ 
0 & d-ca^{-1}b%
\end{array}%
\right) \left( 
\begin{array}{cc}
1 & a^{-1}b \\ 
0 & 1%
\end{array}%
\right)  \label{Representations L1 F0}
\end{equation}%
In particular, for every $\left( \gamma ,Y\right) \in \mathbf{G}\times 
\mathbf{S}_{g}$ with $\gamma =\left( 
\begin{array}{cc}
a & b \\ 
c & d%
\end{array}%
\right) $ such that $a+Yc\in \mathbf{GL}_{g}$ (a Zariski open condition),
once checks that $\left( \text{\ref{Representations L1 F0}}\right) $ gives%
\begin{equation}
\left( 
\begin{array}{cc}
1 & Y \\ 
0 & 1%
\end{array}%
\right) \gamma =\left( 
\begin{array}{cc}
1 & 0 \\ 
c\left( a+Yc\right) ^{-1} & 1%
\end{array}%
\right) \left( 
\begin{array}{cc}
a+Yc & 0 \\ 
0 & d-c\left( a+Yc\right) ^{-1}\left( b+Yd\right)%
\end{array}%
\right) \left( 
\begin{array}{cc}
1 & \left( a+Yc\right) ^{-1}\left( b+Yd\right) \\ 
0 & 1%
\end{array}%
\right) \text{.}  \label{Representations L1 F1}
\end{equation}

Specializing $\left( \text{\ref{Representations L1 F1}}\right) $ to the $c=0$%
\ case shows that, under the identification $\mathbf{S}_{g}\overset{\sim }{%
\rightarrow }\mathbf{U}_{g}\overset{\sim }{\mathbf{\leftarrow }}\mathbf{Q}%
_{g}^{-}\backslash \mathbf{Y}$, the $\ast $-action induced by the right
multiplication by an element of $\mathbf{Q}_{g}$ on $\mathbf{Y}$ (cfr. $%
\left( \text{\ref{Representations F Act}}\right) $) is given by the
following formula verifying formula Assumption \ref{Representations Ass1}:%
\begin{equation}
Y\gamma :=a^{-1}\left( b+Yd\right) \text{ if }\gamma =\left( 
\begin{array}{cc}
a & b \\ 
0 & d%
\end{array}%
\right) \in \mathbf{Q}_{g}\text{.}  \label{Representations L1 F Act}
\end{equation}
\end{proof}

\bigskip

If $\rho \in \mathrm{R}$\textrm{ep}$_{f}\left( \mathbf{GL}_{g}\times \mathbf{%
G}_{m}\right) $, then $\mathbf{GSp}_{2g}$ acts on $\mathrm{C}^{\mathrm{alg}%
}\left( \mathbf{S}_{g},\rho \right) \left[ \eta \right] $ by means of the
rule%
\begin{equation*}
\left( \gamma f\right) \left( Y\right) :=\rho \left( a+Yc,\nu \left( \gamma
\right) \right) f\left( \left( a+Yc\right) ^{-1}\left( b+Yd\right) \right) 
\text{ for every }\gamma =\left( 
\begin{array}{cc}
a & b \\ 
c & d%
\end{array}%
\right) \in \mathbf{G}\text{.}
\end{equation*}%
(In order to see that the above formula defines an action on $\mathrm{C}^{%
\mathrm{alg}}\left( \mathbf{S}_{g},\rho \right) \left[ \eta \right] $,
suffices to check that, setting $j\left( \gamma ,Y\right) :=\left( a+Yc,\nu
\left( \gamma \right) \right) $ and $Y\gamma :=\left( a+Yc\right)
^{-1}\left( b+Yd\right) $, one has $j\left( \gamma _{2}\gamma _{1},Y\right)
=j\left( \gamma _{2},Y\right) j\left( \gamma _{1},Y\gamma _{2}\right) $ and $%
Y\left( \gamma _{2}\gamma _{1}\right) =\left( Y\gamma _{2}\right) \gamma
_{1} $ on the Zariski open subset of those $\left( \gamma ,Y\right) $ in $%
\mathbf{G}\times \mathbf{S}_{g}$ such that, if $\gamma =\left( 
\begin{array}{cc}
a & b \\ 
c & d%
\end{array}%
\right) $, then $\mathrm{de}$\textrm{t}$\left( a\right) \neq 0$, $\mathrm{de}
$\textrm{t}$\left( a+Yc\right) \neq 0$ and $\mathrm{de}$\textrm{t}$\left(
Y\right) \neq 0$. This claim follows looking at the decomposition $\left( 
\text{\ref{Representations L1 F1}}\right) $, noticing that $\left( 
\begin{array}{cc}
1 & Y \\ 
0 & 1%
\end{array}%
\right) \left( \gamma _{2}\gamma _{1}\right) =\left( \left( 
\begin{array}{cc}
1 & Y \\ 
0 & 1%
\end{array}%
\right) \gamma _{2}\right) \gamma _{1}$). Using the first isomorphism of $%
\left( \text{\ref{Representations F Levi}}\right) $ and then $\left( \text{%
\ref{Representations F Res}}\right) $, we can regard $\rho $ as a
representation of $\mathbf{Q}^{-}$\ and then form $\mathrm{Ind}_{\mathbf{Q}%
^{-}}^{\mathbf{G}}\left( \rho \right) \left[ U\right] $ for $U=\mathbf{Y}=%
\mathbf{Y}_{\mathbf{G},\mathbf{Q}}$ or $U=\eta $.

\begin{lemma}
\label{Representations L1}The restriction morphism induced by $\mathbf{U}%
\subset \mathbf{G}$ and the second isomorphism of $\left( \text{\ref%
{Representations F Levi}}\right) $ yield isomorphisms%
\begin{equation*}
\mathrm{Ind}_{\mathbf{Q}^{-}}^{\mathbf{G}}\left( \rho \right) \left[ \eta %
\right] \overset{\sim }{\longrightarrow }\mathrm{C}^{\mathrm{alg}}\left( 
\mathbf{U},\rho \right) \left[ \eta \right] \overset{\sim }{\longrightarrow }%
\mathrm{C}^{\mathrm{alg}}\left( \mathbf{S}_{g},\rho \right) \left[ \eta %
\right] \text{,}
\end{equation*}%
where the first identification extends $\left( \text{\ref{Representations F
Iso1}}\right) $. Furthermore, the above composition has the following
properties.

\begin{itemize}
\item[$\left( 1\right) $] It is $\mathbf{G}\overset{\sim }{\rightarrow }%
\mathbf{GSp}_{2g}$-equivariant (the isomorphism being induced by $c_{\theta
}^{-1}$).

\item[$\left( 2\right) $] It identifies $\mathrm{Ind}_{\mathbf{Q}^{-}}^{%
\mathbf{G}}\left( \rho \right) \left[ \mathbf{Y}\right] $ with $\mathrm{C}^{%
\mathrm{alg}}\left( \mathbf{S}_{g},\rho \right) $, implying by $\left(
1\right) $ that we have an identification of $\left( \mathfrak{g},\mathbf{Q}%
\right) \overset{\sim }{\rightarrow }\left( \mathfrak{gsp}_{2g},\mathbf{Q}%
_{g}\right) $-modules (the isomorphism being induced by $c_{\theta }^{-1}$):%
\begin{equation*}
\mathrm{Ind}_{\mathbf{Q}^{-}}^{\mathbf{G}}\left( \rho \right) \left[ \mathbf{%
Y}\right] \overset{\sim }{\longrightarrow }\mathrm{C}^{\mathrm{alg}}\left( 
\mathbf{U},\rho \right) \overset{\sim }{\longrightarrow }\mathrm{C}^{\mathrm{%
alg}}\left( \mathbf{S}_{g},\rho \right) \text{.}
\end{equation*}

\item[$\left( 3\right) $] If $X\in \mathfrak{u}^{-}$, then $X\mathrm{Ind}_{%
\mathbf{Q}^{-}}^{\mathbf{G}}\left( \rho \right) \left[ \mathbf{Y}\right]
_{=0}=0$ and $X\mathrm{Ind}_{\mathbf{Q}^{-}}^{\mathbf{G}}\left( \rho \right) %
\left[ \mathbf{Y}\right] _{=r}\subset \mathrm{Ind}_{\mathbf{Q}^{-}}^{\mathbf{%
G}}\left( \rho \right) \left[ \mathbf{Y}\right] _{=r+1}$ for every $r\in 
\mathbb{N}_{\geq 1}$.
\end{itemize}
\end{lemma}

\begin{proof}
As usual, we can assume that $\mathbf{G}=\mathbf{GSp}_{2g}$. The isomorphism 
$\mathrm{C}^{\mathrm{alg}}\left( \mathbf{U},\rho \right) \left[ \eta \right]
\simeq \mathrm{C}^{\mathrm{alg}}\left( \mathbf{S}_{g},\rho \right) \left[
\eta \right] $ comes from $\left( \text{\ref{Representations F Levi}}\right) 
$. The analogue of $\left( \text{\ref{Representations F Iso1}}\right) $ for
the rational functions can be checked using the fact that $\mathbf{Q}%
_{/R}^{-}\times V\overset{\sim }{\rightarrow }\mathbf{Q}_{/R}^{-}V\subset 
\mathbf{Y}_{/R}$ is an open subset of $\mathbf{G}_{/R}$ for every open
subset $V\subset \mathbf{U}_{/R}$ (thanks to $\left( \text{\ref%
{Representations F BigCell}}\right) $). The equivariance in $\left( 1\right) 
$ uses $\left( \text{\ref{Representations L1 F1}}\right) $ and then, in view
of $\left( \text{\ref{Representations F Iso1}}\right) $, $\left( 2\right) $
follows. Then, in order to check $\left( 3\right) $, according to Lemma \ref%
{Representations L1'} suffices to show that $X\mathrm{C}^{\mathrm{alg}%
}\left( \mathbf{S}_{g},\rho \right) _{=r}\subset \mathrm{C}^{\mathrm{alg}%
}\left( \mathbf{S}_{g},\rho \right) _{=r+1}$, where $\mathrm{C}^{\mathrm{alg}%
}\left( \mathbf{S}_{g},\rho \right) _{=r}\subset \mathrm{C}^{\mathrm{alg}%
}\left( \mathbf{S}_{g},\rho \right) $ denote the $\mathbf{M}$%
-subrepresentation of $\rho $-valued polynomials that are of degree $=r$.
According to $\left( 1\right) $, $\mathrm{C}^{\mathrm{alg}}\left( \mathbf{S}%
_{g},\rho \right) \left[ \eta \right] $ is isomorphic to the $\mathbf{G}$%
-module $\mathrm{Ind}_{\mathbf{Q}^{-}}^{\mathbf{G}}\left( \rho \right) \left[
\eta \right] =\mathrm{C}^{\mathrm{alg}}\left( \mathbf{G},\rho \right) \left[
\eta \right] ^{h_{\mathbf{Q}^{-}}}$, which is a subobject of $\left( \sigma
,V\right) :=\mathrm{C}^{\mathrm{alg}}\left( \mathbf{G},\rho \right) \left[
\eta \right] \in \mathrm{R}$\textrm{ep}$_{gd}\left( \mathbf{G}\right) $. We
can therefore use $\left( \text{\ref{Representations F dr}}\right) $ and the
resulting inclusion of $\mathrm{C}^{\mathrm{alg}}\left( \mathbf{S}_{g},\rho
\right) $ in $V$ (induced by the restriction obtained from $\mathbf{S}_{g}%
\overset{\sim }{\rightarrow }\mathbf{U\subset G}$) in order to compute the $%
\mathfrak{g}$-action. To this end, we write $X\in \mathfrak{u}^{-}$ in the
form $X=\left( 
\begin{array}{cc}
0 & 0 \\ 
x & 0%
\end{array}%
\right) $ with $x\in \mathbf{S}_{g}$, so that $1+\varepsilon X=\left( 
\begin{array}{cc}
1 & 0 \\ 
\varepsilon x & 1%
\end{array}%
\right) $. Then, noticing that $\rho \left( 1+\varepsilon Yx\right) =1$
because $\rho \in \mathrm{R}$\textrm{ep}$_{f}\left( \mathbf{M}\right) $ and $%
\left( 1+\varepsilon Yx\right) ^{-1}=1-\varepsilon Yx$, we see that%
\begin{eqnarray}
&&\left( \left( 1+\varepsilon X\right) f\right) \left( Y\right) =\rho \left(
1+\varepsilon Yx\right) f\left( \left( 1+\varepsilon Yx\right) ^{-1}Y\right)
=f\left( Y-\varepsilon YxY\right) \text{,}  \notag \\
&&\text{so that }\left( \text{\ref{Representations F dr}}\right) \text{\
yields }\left( Xf\right) \left( Y\right) =\frac{f\left( Y-\varepsilon
YxY\right) -f\left( Y\right) }{\varepsilon }\text{.}
\label{Representations L1 F3}
\end{eqnarray}%
Because the underlying $\Bbbk $-module of $\mathrm{C}^{\mathrm{alg}}\left( 
\mathbf{S}_{g},\rho \right) _{=r}\subset V$ is freely generated by the
expressions $Y^{\mathbf{i}}:=\tprod\nolimits_{i\leq j}Y_{i,j}^{\mathbf{i}%
_{i,j}}$ for $\mathbf{i}=\left( \mathbf{i}_{i,j}\right) _{i\leq j}\in 
\mathbb{N}^{d}$ such that $\left\vert \mathbf{i}\right\vert
:=\tsum\nolimits_{i\leq j}\mathbf{i}_{i,j}=r$, suffices to check that $Xf\in 
\mathrm{C}^{\mathrm{alg}}\left( \mathbf{S}_{g},\rho \right) _{=r+1}$ for
every $f=Y^{\mathbf{i}}$ (where $Y_{i,j}$ is viewed in $V$ as the
restriction of the appropriate coordinate function of $\mathbf{G}$). It
follows from $\left( \text{\ref{Representations L1 F3}}\right) $ that we
have $X\mathrm{Ind}_{\mathbf{Q}^{-}}^{\mathbf{G}}\left( \rho \right) \left[ 
\mathbf{Y}\right] _{=0}=0$ and, because $\frac{Y-\varepsilon YxY-Y}{%
\varepsilon }=-YxY$, taking $f=Y_{i,j}$ in $\left( \text{\ref%
{Representations L1 F3}}\right) $ proves that $X\mathrm{Ind}_{\mathbf{Q}%
^{-}}^{\mathbf{G}}\left( \rho \right) \left[ \mathbf{Y}\right] _{=1}\subset 
\mathrm{Ind}_{\mathbf{Q}^{-}}^{\mathbf{G}}\left( \rho \right) \left[ \mathbf{%
Y}\right] _{=2}$. We now assume that $r\geq 2$ and apply the Leibnitz rule
to the derivation $X$ to $Y^{\mathbf{i}}$ with $\left\vert \mathbf{i}%
\right\vert =r$:%
\begin{equation}
X\left( Y^{\mathbf{i}}\right) =\tsum\nolimits_{\left( i,j\right) }\left( 
\mathbf{i}_{i,j}X\left( Y_{i,j}\right) Y_{i,j}^{\mathbf{i}%
_{i,j}-1}\tprod\nolimits_{\left( i_{o},j_{o}\right) \neq \left( i,j\right)
}Y_{i_{o},j_{o}}^{\mathbf{i}_{i_{o},j_{o}}}\right) \text{.}
\label{Representations L1 F5}
\end{equation}%
Because $X\mathrm{Ind}_{\mathbf{Q}^{-}}^{\mathbf{G}}\left( \rho \right) %
\left[ \mathbf{Y}\right] _{=1}\subset \mathrm{Ind}_{\mathbf{Q}^{-}}^{\mathbf{%
G}}\left( \rho \right) \left[ \mathbf{Y}\right] _{=2}$, we see that $X\left(
Y_{i,j}\right) \in \mathrm{Ind}_{\mathbf{Q}^{-}}^{\mathbf{G}}\left( \rho
\right) \left[ \mathbf{Y}\right] _{=2}$. Since%
\begin{equation*}
Y_{i,j}^{\mathbf{i}_{i,j}-1}\tprod\nolimits_{\left( i_{o},j_{o}\right) \neq
\left( i,j\right) }Y_{i_{o},j_{o}}^{\mathbf{i}_{i_{o},j_{o}}}\in \mathrm{Ind}%
_{\mathbf{Q}^{-}}^{\mathbf{G}}\left( \rho \right) \left[ \mathbf{Y}\right]
_{=r-1}\text{,}
\end{equation*}%
$\left( \text{\ref{Representations L1 F5}}\right) $ proves that $X\left( Y^{%
\mathbf{i}}\right) \in \mathrm{Ind}_{\mathbf{Q}^{-}}^{\mathbf{G}}\left( \rho
\right) \left[ \mathbf{Y}\right] _{=r+1}$.
\end{proof}

In view of the fact that one can attach sheaves to (algebraic) $\left( 
\mathfrak{g},\mathbf{Q}\right) $-modules, the following fact is not so
important for our purposes.

\begin{proposition}
\label{Representations P Model}In our symplectic setting, we know that $%
L_{\lambda }=L_{\lambda ,\mathbf{G}}$ is in $\mathrm{Rep}_{f}\left( \mathbf{G%
}\right) $.
\end{proposition}

\begin{proof}
As usual, we can assume that $\mathbf{G}=\mathbf{GSp}_{2g}$. According to
Remark \ref{Representations R Model}, it suffices to check that $\mathrm{Ind}%
_{\mathbf{Q}^{-}}^{\mathbf{G}}\left( W_{\lambda }\right) \left[ \mathbf{Y}%
\right] $ is a $Dist\left( \mathbf{U}^{-}\right) $-module. For every $1\leq
k\leq l\leq g$, let us write $\partial _{k,l}$ for the unique matrix in $%
\mathfrak{u}^{-}$ whose lower left $g$-by-$g$ entry is the symmetric matrix $%
x_{k,l}$\ whose upper triangular part has all zero except a $1$ in $\left(
k,l\right) $-entry. Similarly as in the proof of Lemma \ref{Representations
L1} $\left( 3\right) $, if $\mathbf{k}=\left( \mathbf{k}_{k,l}\right)
_{k\leq l}$\ write $\partial ^{\mathbf{k}}:=\tprod\nolimits_{k\leq
l}\partial _{k,l}^{\mathbf{k}_{k,l}}$ (the ordering being not relevant,
since $\mathfrak{u}^{-}$ is commutative) and set $\mathbf{k}!\mathbf{:=}%
\tprod\nolimits_{k\leq l}\mathbf{k}_{k,l}!$: then $Dist\left( \mathbf{G}%
\right) $ is freely generated as a $\Bbbk $-module by the expressions $\frac{%
\partial ^{\mathbf{k}}}{\mathbf{k}!}$ (cfr. \cite[Part II, \S 1.12
discussion showing $\left( 4\right) $]{Jan07}, although here we use the
coordinates from $\mathbf{S}_{g}\overset{\sim }{\rightarrow }\mathbf{U}^{-}$
rather than the simple roots). Because, using the multi-index notation $%
\mathbf{i}=\left( \mathbf{i}_{i,j}\right) _{i\leq j}$ for the $Y$'s
variables, the expressions $Y^{\mathbf{i}}:=\tprod\nolimits_{i\leq
j}Y_{i,j}^{\mathbf{i}_{i,j}}$ are a basis of the underlying $\Bbbk $-module
of $\mathrm{C}^{\mathrm{alg}}\left( \mathbf{S}_{g},\rho \right) $, we have
to check that $\frac{\partial ^{\mathbf{k}}}{\mathbf{k}!}\left( Y^{\mathbf{i}%
}\right) \in \mathrm{C}^{\mathrm{alg}}\left( \mathbf{S}_{g},\rho \right)
_{\Bbbk }$ for every $\mathbf{i}$ and $\mathbf{k}$. Using the formula $%
Y_{i,j}x_{k,l}Y_{i,j}=\delta _{k,l}^{i,j}Y_{k,l}^{2}$ for the Kronecker
delta function $\delta _{k,l}^{i,j}$ which is non-zero if and only if $%
\left( i,j\right) =\left( k,l\right) $, one verifies that $\partial
_{k,l}\left( Y_{i,j}\right) =-\delta _{k,l}^{i,j}Y_{k,l}^{2}$. Using $\left( 
\text{\ref{Representations L1 F5}}\right) $ and then an inductive argument
proves that $\partial ^{\mathbf{k}}\left( Y^{\mathbf{i}}\right) =0$ if there
exists $\left( k,l\right) $ such that $\mathbf{k}_{k,l}>0$ but $\mathbf{i}%
_{k,l}=0$ and, otherwise, $\partial ^{\mathbf{k}}\left( Y^{\mathbf{i}%
}\right) =-\mathbf{k}!Y^{\mathbf{i}+\mathbf{k}}$. The claim follows.
\end{proof}

\subsection{The isomorphism $\mathrm{Sym}_{W}^{2,\vee }\left( 1\right) 
\protect\overset{\sim }{\longrightarrow }\mathfrak{u}^{-}$}

Suppose first that $\mathbf{G}=\mathbf{GSp}_{2g}$. If $g\in \mathbb{N}$, we
let \textrm{Std}$_{g}\in \mathrm{R}$\textrm{ep}$_{f}\left( \mathbf{GL}%
_{g}\right) $ be the representation of $\mathbf{GL}_{g}$ given by column
vectors of length $g$ on which $\mathbf{GL}_{g}$ acts from the left by
matrix multiplication: we write $e_{i}\in \mathrm{Std}_{g}$ for the column
vector whose entries are all zero except for the $i$-row which equals $1$.
We then identify \textrm{Std}$_{g}^{\vee }$ with the row vectors on which $%
\mathbf{GL}_{g}$ acts from the right again by matrix multiplication (and
regard it as a left $\mathbf{GL}_{g}$ as usual by means of $ax:=xa^{-1}$, if
needed): the evaluation pairing $\left\langle -,-\right\rangle $ corresponds
to the matrix multiplication (of a row by a column of the same length) and
satisfies $\left\langle la,x\right\rangle =\left\langle l,ax\right\rangle $
for every $l\in \mathrm{Std}_{g}^{\vee },x\in \mathrm{Std}_{g}$ and $a\in 
\mathbf{GL}_{g}$, so that it gives rise to a morphism of (left or right)
representations \textrm{Std}$_{g}^{\vee }\otimes $\textrm{Std}$%
_{g}\rightarrow \mathrm{1}$. We remark that there is an isomorphism (here
and below $\mathbf{M}_{g}$ denotes the $g$-by-$g$ matrix functor, not to be
confused with the Levi):%
\begin{equation}
\varsigma :\left( \mathrm{Std}_{g}\otimes \mathrm{Std}_{g}\right) ^{\vee }%
\overset{\sim }{\longrightarrow }\mathbf{M}_{g}\text{ via }\varsigma \left(
b\right) :=\left( b\left( e_{i},e_{j}\right) \right) \text{ such that }%
\varsigma \left( ba\right) =a^{t}\varsigma \left( b\right) a\text{.}
\label{Representations F StdDualTens}
\end{equation}%
The symmetric quotient $\mathrm{Sym}_{g}^{2}:=\mathrm{Sym}^{2}\left( \mathrm{%
Std}_{g}\right) $ of $\mathrm{Std}_{g}^{\otimes 2}$ yields a monomorphism $%
\mathrm{Sym}_{g}^{2,\vee }\hookrightarrow \left( \mathrm{Std}_{g}^{\otimes
2}\right) ^{\vee }$ and $\left( \text{\ref{Representations F StdDualTens}}%
\right) $ restricts to an isomorphism%
\begin{equation}
\varsigma :\mathrm{Sym}_{g}^{2,\vee }\overset{\sim }{\longrightarrow }%
\mathbf{S}_{g}\text{ such that }\varsigma \left( ba\right) =a^{t}\varsigma
\left( b\right) a\text{.}  \label{Representations F StdDual}
\end{equation}%
We remark that $\mathbf{M}$ acts by conjugation on $\mathbf{U}^{-}$ and,
hence, differentiating this action yields an action \textrm{Ad}\ of $\mathbf{%
M}$ on $\mathfrak{u}^{-}$. We view the representations of $\mathbf{GL}_{g}$
as representations of $\mathbf{M\simeq GL}_{g}\times \mathbf{G}_{m}$ by
restriction along the projection onto the $\mathbf{GL}_{g}$-component.

\begin{lemma}
\label{Representations L2}The identification $\varsigma :\mathrm{Sym}%
_{g}^{2,\vee }\left( 1\right) \overset{\sim }{\rightarrow }\mathbf{S}_{g}$
and the isomorphism $\mathbf{S}_{g}\overset{\sim }{\longrightarrow }%
\mathfrak{u}^{-}$ mapping $x$ to $\left( 
\begin{array}{cc}
0 & 0 \\ 
x & 0%
\end{array}%
\right) $ are $\mathbf{GL}_{g}\times \mathbf{G}_{m}$-equivariant.
\end{lemma}

\begin{proof}
Left to the reader.
\end{proof}

For a more general $\mathbf{G}=\mathbf{GSp}\left( V,\psi \right) $, applying 
$\theta :\left( \Bbbk ^{2g},\psi _{g}\right) \overset{\sim }{\rightarrow }%
\left( V,\psi \right) $ and $c_{\theta }:\mathbf{GSp}_{2g}\overset{\sim }{%
\rightarrow }\mathbf{GSp}\left( V,\psi \right) $ one gets, setting $\mathrm{%
Sym}_{W}^{2}:=\mathrm{Sym}^{2}\left( W\right) $, an $\mathbf{M}$-equivariant
isomorphism $\mathrm{Sym}_{W}^{2,\vee }\left( 1\right) \overset{\sim }{%
\rightarrow }\mathfrak{u}^{-}$ (cfr. the end of Example \ref{Representations
E GSp}).

\subsection{The isomorphism $\mathrm{J}^{\prime }\protect\overset{\sim }{%
\longrightarrow }\mathrm{Ind}_{\mathbf{Q}^{-}}^{\mathbf{G}}\left( \mathrm{1}%
\right) \left[ \mathbf{Y}\right] _{\leq 1}^{\vee }$}

Suppose first that $\mathbf{G}=\mathbf{GSp}_{2g}$. We remark that there is
an isomorphism \textrm{Std}$_{g}\overset{\sim }{\rightarrow }$\textrm{Std}$%
_{g}^{\vee }$ mapping $x$ to $x^{t}$, satisfying $\left( ax\right)
^{t}=x^{t}a^{t}$ when $a\in \mathbf{GL}_{g}$. It induces a morphism \textrm{%
Std}$_{g}^{\otimes 2}\overset{\sim }{\rightarrow }$\textrm{Std}$_{g}^{\vee
\otimes 2}$ satisfying the same equivariance property. Then we identify 
\textrm{Std}$_{g}^{\vee \otimes 2}\overset{\sim }{\rightarrow }\left( 
\mathrm{Std}_{g}^{\otimes 2}\right) ^{\vee }$ mapping $l_{1}\otimes l_{2}$
to $\lambda _{l_{1}\otimes l_{2}}$ such that $\lambda _{l_{1}\otimes
l_{2}}\left( v_{1}\otimes _{\mathbb{Z}}v_{2}\right) =l_{1}\left(
v_{1}\right) l_{2}\left( v_{2}\right) $. This is $\mathbf{GL}_{g}$%
-equivariant under the natural right action. Finally, we identify $\left( 
\mathrm{Std}_{g}^{\otimes 2}\right) ^{\vee }$ with $\mathbf{M}_{g}$ via $%
\left( \text{\ref{Representations F StdDualTens}}\right) $, thus getting%
\begin{equation}
\mathrm{Std}_{g}\otimes \mathrm{Std}_{g}\overset{\sim }{\longrightarrow }%
\mathbf{M}_{g}\text{ such that }\tau \left( ax\right) =a\tau \left( x\right)
a^{t}\text{.}  \label{Representations F StdTens}
\end{equation}%
We remark that, if $M$ is a module over a ring $R$, then we always have a
canonical morphism \textrm{Sym}$_{R}^{2}\left( M\right) \rightarrow M\otimes
_{R}M$ which is induced by the endomorphism of $M\otimes _{R}M$ sending $%
m_{1}\otimes _{R}m_{2}$ to $m_{1}\otimes _{R}m_{2}+m_{2}\otimes _{R}m_{1}$:
when $M$ is a free $R$-module, it is injective. Let us write $\mathbf{S}%
_{g}^{even}\subset \mathbf{S}_{g}$ for the subfunctor of those matrices
whose diagonal entries are multiplies of $2$. The map sending $x_{1}\otimes
x_{2}$ to $x_{2}\otimes x_{1}$ in $\mathrm{Std}_{g}^{\otimes 2}$\
corresponds to the transposition in $\mathbf{M}_{g}$ up to the isomorphism $%
\left( \text{\ref{Representations F StdTens}}\right) $ and $\mathbf{S}%
_{g}^{even}\subset \mathbf{M}_{g}$ is the subfunctor of those matrices of
the form $X+X^{t}$ with $X\in \mathbf{M}_{g}$. It follows that $\left( \text{%
\ref{Representations F StdTens}}\right) \ $restricts on $\mathrm{Sym}%
_{g}^{2}\rightarrow \mathrm{Std}_{g}^{\otimes 2}$ to an isomorphism%
\begin{equation}
\tau :\mathrm{Sym}_{g}^{2}\overset{\sim }{\longrightarrow }\mathbf{S}%
_{g}^{even}\text{ such that }\tau \left( ax\right) =a\tau \left( x\right)
a^{t}\text{.}  \label{Representations F Std}
\end{equation}%
We can therefore define a morphism of $\mathbf{GL}_{g}$-representations%
\begin{equation}
\mathrm{Sym}_{g}^{2}\longrightarrow \left( \mathbf{S}_{g}\right) ^{\vee \vee
}=\mathrm{C}^{\mathrm{alg}}\left( \mathbf{S}_{g},\mathrm{1}\right)
_{=1}^{\vee }\simeq \mathrm{Ind}_{\mathbf{Q}^{-}}^{\mathbf{G}}\left( \mathrm{%
1}\right) \left[ \mathbf{Y}\right] _{=1}^{\vee }
\label{Representations F StdInd}
\end{equation}%
mapping $x\in \mathrm{Sym}_{g}^{2}$ to the evaluation at $\tau \left(
x\right) \in \mathbf{S}_{g}^{even}\subset \mathbf{S}_{g}$: it is an
isomorphism when $2\in \Bbbk ^{\times }$ because then $\mathbf{S}_{g}^{even}=%
\mathbf{S}_{g}$ and $\left( \text{\ref{Representations F StdInd}}\right) $
is the biduality morphism $\mathrm{Sym}_{g}^{2}\rightarrow \left( \mathrm{Sym%
}_{g}^{2}\right) ^{\vee \vee }$, up to the identification $\tau $.

Consider the $\mathbf{Q}_{g}$-equivariant exact sequence%
\begin{equation}
0\longrightarrow \mathrm{Std}_{g}\longrightarrow \mathrm{Std}%
_{2g}\longrightarrow \frac{\mathrm{Std}_{2g}}{\mathrm{Std}_{g}}%
\longrightarrow 0\text{.}  \label{Representations D StdEx}
\end{equation}%
Because $\mathrm{Std}_{g}$ is isotropic, then natural symplectic pairing $%
\psi _{g}$ induces the morphism of $\mathbf{Q}$-representations%
\begin{equation}
\frac{\mathrm{Std}_{2g}}{\mathrm{Std}_{g}}\otimes \mathrm{Std}%
_{g}\longrightarrow \mathrm{1}\left( 1\right) \text{, which gives }\mathrm{%
Std}_{g}\left( -1\right) \overset{\sim }{\longrightarrow }\left( \frac{%
\mathrm{Std}_{2g}}{\mathrm{Std}_{g}}\right) ^{\vee }\text{.}
\label{Representations D StdPair}
\end{equation}%
Consider the following diagram of $\mathbf{Q}$-representations\ that we are
going to describe:%
\begin{equation}
\begin{array}{ccccccc}
0\longrightarrow & \mathrm{Std}_{g}\otimes \mathrm{Std}_{g}\left( -1\right)
& \longrightarrow & \mathrm{Std}_{2g}\otimes \mathrm{Std}_{g}\left( -1\right)
& \longrightarrow & \frac{\mathrm{Std}_{2g}}{\mathrm{Std}_{g}}\otimes 
\mathrm{Std}_{g}\left( -1\right) & \longrightarrow 0 \\ 
& \parallel &  & \uparrow &  & \uparrow &  \\ 
0\longrightarrow & \mathrm{Std}_{g}\otimes \mathrm{Std}_{g}\left( -1\right)
& \longrightarrow & \mathrm{J}^{\prime \prime } & \longrightarrow & \mathrm{1%
} & \longrightarrow 0 \\ 
& \downarrow &  & \downarrow &  & \parallel &  \\ 
0\longrightarrow & \mathrm{Sym}_{g}^{2}\left( -1\right) & \longrightarrow & 
\mathrm{J}^{\prime } & \longrightarrow & \mathrm{1} & \longrightarrow 0\text{%
.}%
\end{array}
\label{Representations D1}
\end{equation}%
The first row is obtained applying $-\otimes \mathrm{Std}_{g}\left(
-1\right) $ to the exact sequence $\left( \text{\ref{Representations D StdEx}%
}\right) $. Then dualize the evaluation pairing $\left( \frac{\mathrm{Std}%
_{2g}}{\mathrm{Std}_{g}}\right) ^{\vee }\otimes \frac{\mathrm{Std}_{2g}}{%
\mathrm{Std}_{g}}\rightarrow \mathrm{1}$ to get, using $\left( \text{\ref%
{Representations D StdPair}}\right) $, the morphism%
\begin{equation}
\mathrm{1}\longrightarrow \left( \frac{\mathrm{Std}_{2g}}{\mathrm{Std}_{g}}%
\right) ^{\vee \vee }\otimes \left( \frac{\mathrm{Std}_{2g}}{\mathrm{Std}_{g}%
}\right) ^{\vee }\overset{\sim }{\longleftarrow }\frac{\mathrm{Std}_{2g}}{%
\mathrm{Std}_{g}}\otimes \mathrm{Std}_{g}\left( -1\right) \text{.}
\label{Representations D1 Cas}
\end{equation}%
We pull-back the first row via $\left( \text{\ref{Representations D1 Cas}}%
\right) $ to get the second row. Finally, we push-out the second row via the
canonical quotient morphism from $\mathrm{Std}_{g}^{\otimes 2}\left(
-1\right) $ to $\mathrm{Sym}_{g}^{2}\left( -1\right) $ in order to get the
third row. By construction, the diagram $\left( \text{\ref{Representations
D1}}\right) $ is commutative with exact rows.

\begin{proposition}
\label{Representations P Spl}If we suppose that $2\in \Bbbk ^{\times }$,
then there is an isomorphism of $\mathbf{Q}$-representations%
\begin{equation*}
\mathrm{J}^{\prime }\overset{\sim }{\longrightarrow }\mathrm{J}^{\vee }\text{%
, where }\mathrm{J}:=\mathrm{Ind}_{\mathbf{Q}^{-}}^{\mathbf{G}}\left( 
\mathrm{1}\right) \left[ \mathbf{Y}\right] _{\leq 1}\text{.}
\end{equation*}%
It is uniquely determined as a morphism of $\mathbf{Q}^{\circ }$%
-representations up to non-zero scalar factors and its restriction to $%
\mathrm{Sym}_{g}^{2}\left( -1\right) $ agrees with $\left( \text{\ref%
{Representations F StdInd}}\right) $ up to non-zero scalar factors. In
particular, the third row of $\left( \text{\ref{Representations D1}}\right) $
is uniquely up to non-zero scalar factors identified with the short exact
sequence of $\mathbf{Q}$-representations obtained dualizing the exact
sequence of $\mathbf{Q}$-representations%
\begin{equation}
0\longrightarrow \mathrm{1}\longrightarrow \mathrm{J}\longrightarrow \frac{%
\mathrm{J}}{\mathrm{1}}\longrightarrow 0\text{.}
\label{Representations P Spl Claim}
\end{equation}
\end{proposition}

\begin{proof}
\emph{Existence. }Let us write $\left\langle -,-\right\rangle $ for the
symplectic pairing $\psi _{g}$. Let us write $\left\langle -,-\right\rangle $
for the symplectic pairing $\psi _{g}$. As a basis of $\mathrm{Std}_{2g}$ we
may take the symplectic-Hodge basis $\left\{ e_{1},...,e_{2g}\right\} $ and
we set $f_{i}=e_{g+i}$, so that $\left\langle e_{i},f_{j}\right\rangle
=\delta _{i,j}$. Then every element $v$ of $\mathrm{Std}_{2g}\otimes \mathrm{%
Std}_{g}\left( -1\right) $ can be uniquely written in the form%
\begin{equation*}
v=\widetilde{x}\left( v\right) +\tsum\nolimits_{i,j=1}^{g}\lambda
_{i,j}\left( v\right) f_{i}\otimes e_{j}\text{, where }\widetilde{x}\left(
v\right) \in \mathrm{Std}_{g}\otimes \mathrm{Std}_{g}\left( -1\right) \text{
and }\lambda _{i,j}\left( v\right) \in \mathbf{G}_{a}\text{.}
\end{equation*}%
Because the dual of $e_{i}$ under the pairing $\left( \text{\ref%
{Representations D StdPair}}\right) $ induced by $\psi _{g}$ is the image $%
\overline{f}_{i}$ of $f_{i}$ in $\frac{\mathrm{Std}_{2g}}{\mathrm{Std}_{g}}$%
, we see that $\lambda \in \mathrm{1}$ maps via $\left( \text{\ref%
{Representations D1 Cas}}\right) $ to%
\begin{equation*}
\lambda \cdot \tsum\nolimits_{i=1}^{g}\overline{f}_{i}\otimes
e_{i}=\tsum\nolimits_{i=1}^{g}\lambda \overline{f}_{i}\otimes e_{i}\text{.}
\end{equation*}%
Since the span $\mathrm{Std}_{g}^{\prime }$ of $\left\{ f_{i}\otimes
e_{j}\right\} _{i,j=1,...,g}$ maps isomorphically onto $\frac{\mathrm{Std}%
_{2g}}{\mathrm{Std}_{g}}\otimes \mathrm{Std}_{g}\left( -1\right) $, it
follows that%
\begin{eqnarray*}
&&\mathrm{J}^{\prime \prime }=\left\{ \left( v,\lambda \right) \in \mathrm{%
Std}_{2g}\otimes \mathrm{Std}_{g}\left( -1\right) \times \mathrm{1}:\lambda
_{i,j}\left( v\right) =\lambda \delta _{i,j}\right\} \\
&&\overset{\sim }{\rightarrow }\left\{ v\in \mathrm{Std}_{2g}\otimes \mathrm{%
Std}_{g}\left( -1\right) :v=\widetilde{x}\left( v\right) +\lambda \left(
v\right) \tsum\nolimits_{i=1}^{g}f_{i}\otimes e_{i}\right\}
\end{eqnarray*}%
In order to obtain the identification between the third row of $\left( \text{%
\ref{Representations D1}}\right) $ and the dual of $\left( \text{\ref%
{Representations P Spl Claim}}\right) $ we have to define a $\mathbf{Q}$%
-equivariant morphism%
\begin{equation*}
\tciLaplace :\mathrm{J}^{\prime \prime }\longrightarrow \mathrm{C}^{\mathrm{%
alg}}\left( \mathbf{S}_{g},\rho \right) _{\leq 1}^{\vee }
\end{equation*}%
whose restriction to $\mathrm{Std}_{g}\otimes \mathrm{Std}_{g}\left(
-1\right) $ induces the isomorphism $\left( \text{\ref{Representations F
StdInd}}\right) $ and with the property that the induced morphism $\mathrm{1}%
\rightarrow \mathrm{1}$ between the third terms of the third row of $\left( 
\text{\ref{Representations D1}}\right) $ and the dual of $\left( \text{\ref%
{Representations P Spl Claim}}\right) $ is an isomorphism.

We claim that, writing $x\left( v\right) \in \mathrm{Sym}_{g}^{2}\left(
-1\right) $ for the image of $\widetilde{x}\left( v\right) \in \mathrm{Std}%
_{g}^{\otimes 2}\left( -1\right) $ (so that $\tau \left( x\left( v\right)
\right) \in \mathbf{S}_{g}^{even}\subset \mathbf{S}_{g}$), the rule%
\begin{equation*}
\tciLaplace \left( v\right) \left( f\right) :=f_{1}\left( \tau \left(
x\left( v\right) \right) \right) -2\lambda \left( v\right) f_{0}\text{,
where }f_{0}=f\left( 0\right) \text{ and }f_{1}=f-f_{0}\text{,}
\end{equation*}%
works. First of all, because the evaluation at a point and the association
mapping $f$ to $f_{1}$ are linear, $\tciLaplace \left( v\right) \in \mathrm{C%
}^{\mathrm{alg}}\left( \mathbf{S}_{g},\rho \right) _{\leq 1}^{\vee }$. Also,
because $f_{1}$ and the associations mapping $v$ to $\tau \left( x\left(
v\right) \right) $ and $v$ to $\lambda \left( v\right) $ are linear, $%
\tciLaplace $ is linear. By definition, $\tciLaplace $ induces the
isomorphism $\left( \text{\ref{Representations F StdInd}}\right) $ when
restricted to $\mathrm{Std}_{g}\otimes \mathrm{Std}_{g}\left( -1\right) $
and maps $c:=\tsum\nolimits_{i=1}^{g}f_{i}\otimes e_{i}$ to $-2ev_{0}$ for
the evaluation at zero map $ev_{0}\in \mathrm{C}^{\mathrm{alg}}\left( 
\mathbf{S}_{g},\rho \right) _{\leq 1}^{\vee }$, so that it also induces $-2:%
\mathrm{1}\rightarrow \mathrm{1}$, which is an isomorphism because $2\in
\Bbbk ^{\times }$. Let us check the $\mathbf{Q}$-equivariance of $%
\tciLaplace $, i.e. that $\left( \gamma \tciLaplace \left( v\right) \right)
\left( f\right) =\tciLaplace \left( \gamma v\right) \left( f\right) $, where%
\begin{eqnarray}
&&\left( \gamma \tciLaplace \left( v\right) \right) \left( f\right)
=\tciLaplace \left( v\right) \left( \gamma ^{-1}f\right) =\left( \gamma
^{-1}f\right) _{1}\left( \tau \left( x\left( v\right) \right) \right)
-2\lambda \left( v\right) \left( \gamma ^{-1}f\right) _{0}\text{,}  \notag \\
&&\text{ }\tciLaplace \left( \gamma v\right) \left( f\right) =f_{1}\left(
\tau \left( x\left( \gamma v\right) \right) \right) -2\lambda \left( \gamma
v\right) f_{0}\text{.}  \label{Representations P Spl F0}
\end{eqnarray}

Let us define $c^{\prime }:=\tsum\nolimits_{i=1}^{g}e_{i}\otimes f_{i}$ in $%
\mathrm{Std}_{g}\otimes \mathrm{Std}_{2g}\left( -1\right) $. Because the
dual of $f_{i}$ with respect to $\psi _{g}$ is $-e_{i}$, the element $%
c-c^{\prime }\in \mathrm{Std}_{2g}\otimes \mathrm{Std}_{2g}\left( -1\right) $
is the image of $1\in \mathrm{1}$ under the dual of the symplectic pairing $%
\psi _{g}$ and, hence, it is $\mathbf{G}$-invariant. We deduce that%
\begin{equation*}
\gamma c-c=\gamma c^{\prime }-c^{\prime }\text{ for every }\gamma \in 
\mathbf{G}\text{.}
\end{equation*}%
We remark that $c\in \mathrm{Std}_{g}^{\prime }\otimes \mathrm{Std}%
_{g}\left( -1\right) $ and $c^{\prime }\in \mathrm{Std}_{g}\otimes \mathrm{%
Std}_{g}^{\prime }\left( -1\right) $: if we take $\gamma \in \mathbf{M}$ we
see that we also have $\gamma c\in \mathrm{Std}_{g}^{\prime }\otimes \mathrm{%
Std}_{g}\left( -1\right) $ and $\gamma c^{\prime }\in \mathrm{Std}%
_{g}\otimes \mathrm{Std}_{g}^{\prime }\left( -1\right) $. Because $\mathrm{%
Std}_{g}\cap \mathrm{Std}_{g}^{\prime }=0$, the above equation implies that%
\begin{equation}
\gamma c-c=\gamma c^{\prime }-c^{\prime }=0\text{ for every }\gamma \in 
\mathbf{M}\text{, so that }\gamma c=c\text{ in }\mathrm{Std}_{2g}\otimes 
\mathrm{Std}_{g}\left( -1\right) \text{.}  \label{Representations P Spl F1}
\end{equation}%
Suppose now that $\gamma =\left( 
\begin{array}{cc}
a & b \\ 
0 & d%
\end{array}%
\right) $ so that $\gamma _{\mathbf{M}}:=\left( 
\begin{array}{cc}
a & 0 \\ 
0 & d%
\end{array}%
\right) \in \mathbf{M}$ (see $\left( \text{\ref{Representations L1 F0}}%
\right) $). Then\ we have, using the symbols $\gamma $ and $\gamma _{\mathbf{%
M}}$ for the action in $\mathrm{Std}_{2g}\otimes \mathrm{Std}_{g}\left(
-1\right) $ and writing $a\otimes a$, $b\otimes a$ and $d\otimes a$ for the
tensor product of the matrix multiplications:%
\begin{eqnarray*}
\gamma v &=&\nu \left( \gamma \right) ^{-1}a\otimes a\left( \widetilde{x}%
\left( v\right) \right) +\lambda \left( v\right) \tsum\nolimits_{i}^{g}\nu
\left( \gamma \right) ^{-1}be_{i}\otimes ae_{i}+\lambda \left( v\right)
\tsum\nolimits_{i}^{g}\nu \left( \gamma \right) ^{-1}df_{i}\otimes ae_{i} \\
&=&\nu \left( \gamma \right) ^{-1}a\otimes a\left( \widetilde{x}\left(
v\right) \right) +\lambda \left( v\right) \tsum\nolimits_{i}^{g}\nu \left(
\gamma \right) ^{-1}be_{i}\otimes ae_{i}+\lambda \left( v\right) \gamma _{%
\mathbf{M}}c \\
&&\overset{\left( \text{\ref{Representations P Spl F1}}\right) }{=}\nu
\left( \gamma \right) ^{-1}a\otimes a\left( \widetilde{x}\left( v\right)
\right) +\lambda \left( v\right) \tsum\nolimits_{i}^{g}\nu \left( \gamma
\right) ^{-1}be_{i}\otimes ae_{i}+\lambda \left( v\right) c\text{.}
\end{eqnarray*}%
We deduce that%
\begin{eqnarray}
\widetilde{x}\left( \gamma v\right) &=&\nu \left( \gamma \right) ^{-1}\left(
a\otimes a\left( \widetilde{x}\left( v\right) \right) +\lambda \left(
v\right) \tsum\nolimits_{i}^{g}b\otimes a\left( e_{i}\otimes e_{i}\right)
\right) \text{,}  \notag \\
\lambda \left( \gamma v\right) &=&\lambda \left( v\right) \text{.}
\label{Representations P Spl F2}
\end{eqnarray}

We now make the observation that $\left( \text{\ref{Representations F
StdDualTens}}\right) $ satisfies, more generally, the equivariance property $%
\varsigma \left( b\cdot a_{1}\otimes a_{2}\right) =a_{1}^{t}\varsigma \left(
b\right) a_{2}$ for every $a_{1},a_{2}\in \mathbf{GL}_{g}$, implying that $%
\tau $ satisfies the rule%
\begin{equation*}
\tau \left( a_{1}\otimes a_{2}\cdot x\right) =a_{1}\tau \left( x\right)
a_{2}^{t}\text{.}
\end{equation*}%
Since $\tau \left( e_{i}\otimes e_{i}\right) $ is the matrix whose entries
are all zero except for the $\left( i,i\right) $-entry which equals $2$,
then it follows from $\left( \text{\ref{Representations P Spl F2}}\right) $
that we have%
\begin{eqnarray}
\tau \left( x\left( \gamma v\right) \right) &=&\nu \left( \gamma \right)
^{-1}\left( a\tau \left( x\left( v\right) \right) a^{t}+\lambda \left(
v\right) \tsum\nolimits_{i}^{g}b\tau \left( e_{i}\otimes e_{i}\right)
a^{t}\right)  \notag \\
&=&\nu \left( \gamma \right) ^{-1}\left( a\tau \left( x\left( v\right)
\right) a^{t}+2\lambda \left( v\right) ba^{t}\right) \text{.}
\label{Representations P Spl F3}
\end{eqnarray}

Suppose now that $b=0$ and let us prove the $\mathbf{M}$-equivariance of $%
\tciLaplace $. Because $\mathbf{M}$ acts by means of linear transformations
on $\mathbf{S}_{g}$\ in view of Lemma \ref{Representations L1'}, it
preserves degrees and we see that $\left( \gamma ^{-1}f\right) _{0}=f\left(
0\gamma ^{-1}\right) =f_{0}$ and $\left( \gamma ^{-1}f\right) _{1}\left(
Y\right) =f_{1}\left( Y\gamma ^{-1}\right) $. Taking into account $\left( 
\text{\ref{Representations L1 F Act}}\right) $ and the fact that $%
d=a^{-t}\nu \left( \gamma \right) $, we find%
\begin{equation*}
\left( \gamma ^{-1}f\right) _{0}=f_{0}\text{ and }\left( \gamma
^{-1}f\right) _{1}\left( Y\right) =f_{1}\left( aYd^{-1}\right) =\nu \left(
\gamma \right) ^{-1}f_{1}\left( aYa^{t}\right) \text{.}
\end{equation*}%
Then the first equation of $\left( \text{\ref{Representations P Spl F0}}%
\right) $ becomes%
\begin{equation*}
\left( \gamma \tciLaplace \left( v\right) \right) \left( f\right) =\nu
\left( \gamma \right) ^{-1}f_{1}\left( a\tau \left( x\left( v\right) \right)
a^{t}\right) -2\lambda \left( v\right) f_{0}\text{.}
\end{equation*}%
In view of $\left( \text{\ref{Representations P Spl F2}}\right) $ and $%
\left( \text{\ref{Representations P Spl F3}}\right) $, the above left hand
side equals the left hand side of the second equation of $\left( \text{\ref%
{Representations P Spl F0}}\right) $. The $\mathbf{M}$-equivariance follows.

Suppose now that $a=d=1$ and let us prove the $\mathbf{U}$-equivariance of $%
\tciLaplace $. If $f=f_{0}+\tsum\nolimits_{i,j}\alpha _{i,j}Y_{i,j}$, then
according to $\left( \text{\ref{Representations L1 F Act}}\right) $ we have $%
Y_{i,j}\gamma ^{-1}=-b+Y_{i,j}$, so that $Y_{i,j}$ maps to $-b_{i,j}+Y_{i,j}$
and we have%
\begin{eqnarray*}
&&\left( \gamma ^{-1}f\right) \left( Y\right) =f\left( Y\gamma ^{-1}\right)
=f_{0}-\tsum\nolimits_{i,j}\alpha _{i,j}b_{i,j}+\tsum\nolimits_{i,j}\alpha
_{i,j}Y_{i,j}=f\left( -b\right) +f_{1}\text{,} \\
&&\text{so that }\left( \gamma ^{-1}f\right) _{0}=f\left( -b\right) \text{
and }\left( \gamma ^{-1}f\right) _{1}=f_{1}\text{.}
\end{eqnarray*}%
Then the first equation of $\left( \text{\ref{Representations P Spl F0}}%
\right) $ becomes%
\begin{equation*}
\left( \gamma \tciLaplace \left( v\right) \right) \left( f\right)
=f_{1}\left( \tau \left( x\left( v\right) \right) \right) -2\lambda \left(
v\right) f\left( -b\right) \text{.}
\end{equation*}%
On the other hand, in view of $\left( \text{\ref{Representations P Spl F2}}%
\right) $ and $\left( \text{\ref{Representations P Spl F3}}\right) $, the
second equation of $\left( \text{\ref{Representations P Spl F0}}\right) $
becomes%
\begin{eqnarray*}
\tciLaplace \left( \gamma v\right) \left( f\right) &=&f_{1}\left( \tau
\left( x\left( v\right) \right) +2\lambda \left( v\right) b\right) -2\lambda
\left( v\right) f_{0} \\
&=&f_{1}\left( \tau \left( x\left( v\right) \right) \right) -2\lambda \left(
v\right) f_{1}\left( -b\right) +2\lambda \left( v\right) f_{0} \\
&=&f_{1}\left( \tau \left( x\left( v\right) \right) \right) -2\lambda \left(
v\right) f\left( -b\right) \text{.}
\end{eqnarray*}%
The $\mathbf{U}$-equivariance follows and, with it, the claimed $\mathbf{Q}$%
-equivariance.

\emph{Uniqueness.}$\ $Suppose that $\varphi :\mathrm{J}^{\prime }\overset{%
\sim }{\rightarrow }\mathrm{J}^{\vee }$ is an isomorphism of $\mathbf{Q}%
^{\circ }$-representations. The restriction of the third row of $\left( 
\text{\ref{Representations D1}}\right) $ to the reductive group $\mathbf{GL}%
_{g}$ splits as the direct sum of two non-isomorphic irreducible
representations. The same is true for the dual of $\left( \text{\ref%
{Representations P Spl Claim}}\right) $, in view of the $\mathbf{GL}_{g}$%
-equivariant isomorphism $\left( \text{\ref{Representations F StdInd}}%
\right) $. Hence $\varphi $ is of the form $\varphi =\mu _{1}\varphi
_{1}\oplus \mu _{2}\varphi _{2}$, where $\varphi _{1}$ is $\left( \text{\ref%
{Representations F StdInd}}\right) $, $\varphi _{2}=1_{\mathrm{1}}$ and $\mu
_{i}$ is a non-zero scalar factor for $i=1,2$. In order to make $\varphi
_{2} $ more explicit, let us write $ev_{0}\in \mathrm{J}^{\vee }$ for the
evaluation at $0$: because $\mathbf{M}$ acts by means of linear
transformations on $\mathbf{S}_{g}$\ in view of Lemma \ref{Representations
L1'}, we see that $ev_{0}$ is $\mathbf{GL}_{g}$-invariant and non-zero.
Hence we may assume that $\varphi _{2}\left( 1\right) =ev_{0}$ when we view $%
\mathrm{1}\subset \mathrm{J}^{\vee }$. Let us now write $\varphi ^{\prime }$
for the composition of $\mathrm{J}^{\prime \prime }\rightarrow \mathrm{J}%
^{\prime }$ followed by $\varphi $. The above description of $\varphi _{1}$
and $\varphi _{2}$ implies that%
\begin{equation*}
\varphi ^{\prime }\left( v\right) \left( f\right) =\mu _{1}f_{1}\left( \tau
\left( x\left( v\right) \right) \right) +\mu _{2}\lambda \left( v\right)
f_{0}\text{ for every }f\in \mathrm{C}^{\mathrm{alg}}\left( \mathbf{S}%
_{g},\rho \right) _{\leq 1}\text{.}
\end{equation*}%
Then the proof of the $\mathbf{U}$-equivariance of $\tciLaplace $ in the
existence part shows that we must have $\mu _{2}=-2\mu _{1}$ in order for $%
\varphi ^{\prime }$ being $\mathbf{U}$-equivariant.
\end{proof}

For a more general $\mathbf{G}=\mathbf{GSp}\left( V,\psi \right) $, applying 
$\theta :\left( \Bbbk ^{2g},\psi _{g}\right) \overset{\sim }{\rightarrow }%
\left( V,\psi \right) $ and $c_{\theta }:\mathbf{GSp}_{2g}\overset{\sim }{%
\rightarrow }\mathbf{GSp}\left( V,\psi \right) $ one gets an $\mathbf{M}$%
-equivariant morphism $\mathrm{Sym}_{W}^{2}\rightarrow \mathrm{Ind}_{\mathbf{%
Q}^{-}}^{\mathbf{G}}\left( \mathrm{1}\right) \left[ \mathbf{Y}\right]
_{=1}^{\vee }$ which is an isomorphism when $2\in \Bbbk ^{\times }$. Then
the obvious analogous of Proposition \ref{Representations P Spl} holds,
after replacing $\left( \text{\ref{Representations D1}}\right) $ with its
obvious analogue involving $W\subset V$ rather than $\mathrm{Std}_{g}\subset 
\mathrm{Std}_{2g}$. (Cfr. the end of Example \ref{Representations E GSp}).

\section{\label{S Sheaves}Automorphic sheaves attached to representations of
symplectic groups}

Let us fix $\left( g,N\right) \in \mathbb{N}_{\geq 1}^{2}$ and $d=\left(
d_{1},...,d_{g}\right) \in \mathbb{N}_{\geq 1}^{g}$ such that $\left(
d_{1}...d_{g},N\right) =1$. Define%
\begin{eqnarray*}
K_{d,N} &:&=\mathbf{GSp}_{2g}\left( \mathbb{A}_{f}\right) \cap
\tprod\nolimits_{l\mid d_{1}...d_{g}}\left( 
\begin{array}{cc}
\mathbf{M}_{g}\left( \mathbb{Z}_{l}\right) & \mathbf{M}_{g}\left( \mathbb{Z}%
_{l}\right) \Delta \\ 
\Delta ^{-1}\mathbf{M}_{g}\left( \mathbb{Z}_{l}\right) & \Delta ^{-1}\mathbf{%
M}_{g}\left( \mathbb{Z}_{l}\right) \Delta%
\end{array}%
\right) \\
&&\tprod\nolimits_{l\mid N}\left( 
\begin{array}{cc}
1+N\mathbf{M}_{g}\left( \mathbb{Z}_{l}\right) & N\mathbf{M}_{g}\left( 
\mathbb{Z}_{l}\right) \\ 
N\mathbf{M}_{g}\left( \mathbb{Z}_{l}\right) & 1+N\mathbf{M}_{g}\left( 
\mathbb{Z}_{l}\right)%
\end{array}%
\right) \tprod\nolimits_{l\nmid d_{1}...d_{g}N}\mathbf{M}_{2g}\left( \mathbb{%
Z}_{l}\right) \text{,}
\end{eqnarray*}%
which is an open and compact subgroup of $\mathbf{GSp}_{2g}\left( \mathbb{A}%
_{f}\right) $. Then, writing $\mathbf{GSp}_{2g}^{+}\left( \mathbb{Q}\right)
\subset \mathbf{GSp}_{2g}\left( \mathbb{Q}\right) $ for the subgroup of
those elements having positive simplectic multiplies, one knows that, for
every open and compact subgroup $K\subset \mathbf{GSp}_{2g}\left( \mathbb{A}%
_{f}\right) $ such that $K_{d,N}\subset K\subset K_{d,1}$, the double
quotient space%
\begin{equation*}
\mathbf{GSp}_{2g}^{+}\left( \mathbb{Q}\right) \backslash \left( \mathbf{GSp}%
_{2g}\left( \mathbb{A}_{f}\right) \times \mathcal{H}_{g}\right) /K
\end{equation*}%
admits the structure of a course moduli scheme $Y_{K}$ classifying
quadruples $\left( A,\lambda ,\psi \right) $ where $A$ is an abelian scheme
of dimension $g$, $\lambda :A\rightarrow A^{t}$ is a polarization of type $d$
and $\psi $ is a level $K$-structure compatible with the standard symplectic
pairing for some choice of an $N$-root of unity (see \cite[Ch. VII, \S 1]%
{Deb}). Indeed, one has a fine moduli scheme quasi-projective over $%
\mathfrak{Spec}\left( \mathbb{Z}\right) $ when $N>6^{g}d_{1}...d_{g}\sqrt{g!}
$ which is smooth over $\mathbb{Z}\left[ 1/d_{1}...d_{g}N\right] $\ (see 
\cite[Theorem 7.9]{MFK}) and finer results are available: for example, when $%
g\in \left\{ 1,2\right\} $ suffices to take $N\geq 3$ (see \cite[Corollary
3.3]{Lau05}). The works of Faltings-Chai and Lan provide toroidal
compactifications $X_{K}=X_{K}^{\Sigma }$ of $Y_{K}$ depending on suitable
polyhedral decompositions $\Sigma $.

\begin{remark}
\label{Sheaves R PLevel}If in the definition of $K_{d,N}$ one replaces the
local condition at $l\mid d_{1}...d_{g}$ with the stronger condition of
being in%
\begin{equation*}
\left( 
\begin{array}{cc}
1+\Delta \mathbf{M}_{g}\left( \mathbb{Z}_{l}\right) & \Delta \mathbf{M}%
_{g}\left( \mathbb{Z}_{l}\right) \Delta \\ 
\mathbf{M}_{g}\left( \mathbb{Z}_{l}\right) & 1+\Delta \mathbf{M}_{g}\left( 
\mathbb{Z}_{l}\right)%
\end{array}%
\right) \text{,}
\end{equation*}%
then one gets moduli and fine moduli schemes classifying quadruples $\left(
A,\lambda ,\psi ^{\prime },\psi \right) $ where $\left( A,\lambda ,\psi
\right) $ is as above and $\psi ^{\prime }$ is a suitable symplectic
isomorphism of the kernel of $\lambda $ (see \cite[Ch. VII, Proposition 3.3]%
{Deb}).
\end{remark}

For a more general $\mathbf{G}=\mathbf{GSp}\left( V,\psi \right) $, one
defines the groups $K_{d,N}$ applying the isomorphism $c_{\theta }:\mathbf{%
GSp}_{2g}\overset{\sim }{\rightarrow }\mathbf{GSp}\left( V,\psi \right) $ of
Example \ref{Representations E GSp}\ to those previously defined and, again,
one gets varieties classifying the same kind of moduli problem.

\bigskip

Let us fix a level $K$ as above and simply write $X=X_{K}$, $Y=Y_{K}$ and $%
D:=X-Y$. Let $\pi :A\rightarrow Y$ be the universal abelian scheme (resp. $%
\pi :G\rightarrow X$ the universal semiabelian scheme extending $A$ acting
on the compactification $A_{X}/X$\ of $A/X$, the latter having a normal
crossing divisor $D_{A_{X}}$ lying over $D$, also depending on $\Sigma $
compatible with $\Sigma $ as in \cite[Ch. VI, \S 1, 1.1 Theorem and 1.3
Definition]{FC}) and consider the canonical extension%
\begin{equation*}
\mathcal{H}_{\mathrm{dR}}^{1}:=\mathcal{H}_{\mathrm{dR},X}^{1}:=R^{1}\pi
_{\ast }\left( \Omega _{A_{X}/X}^{\cdot }\left( \mathrm{lo}\text{\textrm{g}}%
\left( D_{A_{X}}/D\right) \right) \right) \subset \left( Y\subset X\right)
_{\ast }\left( R^{1}\pi _{\ast }\left( \Omega _{A/Y}^{\cdot }\right) \right)
\end{equation*}%
of $\mathcal{H}_{\mathrm{dR},Y}^{1}:=R^{1}\pi _{\ast }\left( \Omega
_{A/Y}\right) $ to $X$. Setting $\omega _{\mathrm{dR}}=\omega _{\mathrm{dR}%
,X}:=R^{0}\pi _{\ast }\left( \Omega _{A_{X}/X}^{1}\left( \mathrm{lo}\text{%
\textrm{g}}\left( D_{A_{X}}/D\right) \right) \right) \simeq \omega \left(
G/X\right) $ (the sheaf of invariant differentials of $G$, see \cite[Ch. VI, 
\S 1, 1.1 Theorem]{FC} for the isomorphism) and $Lie_{\mathrm{dR}}^{t}=Lie_{%
\mathrm{dR},X}^{t}:=R^{1}\pi _{\ast }\left( \Omega _{A_{X}/X}^{\cdot }\left( 
\mathrm{lo}\text{\textrm{g}}\left( D_{A_{X}}/D\right) \right) \right) \simeq
\omega \left( G^{t}/X\right) ^{\vee }\left( -1\right) $ (the relative Lie
algebra of the dual semiabelian scheme $G^{t}$, see \cite[Ch. VI, \S 1, 1.1
Theorem]{FC} for the isomorphism, which extends \cite[proof of Lemma 2.3]%
{Fo23} from $Y$ to $X$)\footnote{%
Here (and below in \S \ref{S Sheaves}, when considering the Kodaira-Spencer
isomorphism and other identifications induced by it) the twist by $-1$ is
purely formal. Indeed, all the associations $\rho \mapsto \mathcal{E}_{\rho
} $ from representations to sheaves considered below will not depend on the
twists by the symplectic multiplier $\nu $\ and one could equally well work
with representations of $\mathbf{M}^{\circ }\simeq \mathbf{GL}_{g}$ and $%
\mathbf{G}^{\circ }=\mathbf{Gsp}_{2g}$ (taking into account \cite[pag. 230]%
{FC}, when defining the filtration). However, the further data of the twist
can be introduced in order to keep track of other additional data. For
example, the symplectic multiplier introduced in the definition of
automorphic forms, in order to make the natural association sending a
modular form $f$ to the section $\omega _{f}$ Hecke equivariant. In other
contexts, it may be useful to account for possible Tate twists. In \S \ref{S
Igusa acyclic} below, it will be useful in order to take into account the
twists under the Frobenious denoted by $\varphi $ (defined before Lemma \ref%
{Primitives L URSpl} below). Then, the proof of Lemma \ref{Primitives L2}
shows that, if we denote by $\mathcal{E}\left( 1\right) $ a twist that
indicates that, at the level of sheaves, we multiply the Frobenious of $%
\mathcal{E}$ by $p^{-1}$, the equality $\mathcal{E}_{\rho \left( -1\right) }=%
\mathcal{E}_{\rho }\left( 1\right) $ holds for the representation $\rho
\left( -1\right) :=\rho \otimes \nu ^{-1}$.}, the sheaf $\mathcal{H}_{%
\mathrm{dR}}^{1}$ comes equipped with the Hodge filtration%
\begin{equation*}
0\longrightarrow \omega _{\mathrm{dR}}\longrightarrow \mathcal{H}_{\mathrm{dR%
}}^{1}\longrightarrow Lie_{\mathrm{dR}}^{t}\longrightarrow 0\text{,}
\end{equation*}%
a symplectic $\mathcal{O}_{X}$-bilinear pairing%
\begin{equation}
\left\langle -,-\right\rangle :\mathcal{H}_{\mathrm{dR}}^{1}\otimes _{%
\mathcal{O}_{X}}\mathcal{H}_{\mathrm{dR}}^{1}\longrightarrow \mathcal{O}%
_{X}\left( -1\right)  \label{Sheaves F Sympl}
\end{equation}%
under which $\omega _{\mathrm{dR}}$ is maximal isotropic (arising by duality
from the symplectic pairing of $\mathcal{H}_{\mathrm{dR}}^{1,\vee }$\
obtained by evaluation at the first Chern class $c_{1,\mathrm{dR}}\left( 
\mathcal{L}\right) \in \mathcal{H}_{\mathrm{dR}}^{2}\simeq \wedge ^{2}%
\mathcal{H}_{\mathrm{dR}}^{1}$ of the relatively ample line bundle $\mathcal{%
L}$ giving rise to the polarization) and the Gauss-Manin connection%
\begin{equation*}
\nabla :\mathcal{H}_{\mathrm{dR}}^{1}\longrightarrow \mathcal{H}_{\mathrm{dR}%
}^{1}\otimes _{\mathcal{O}_{X}}\Omega _{X/K}^{1}
\end{equation*}%
(see \cite[\S 2.1, \S 2.2 and the proof of Lemma 2.2]{Fo23} for the results
over $Y$, which extends to $X$ using $G$). We remark that, because $\omega _{%
\mathrm{dR}}$ is maximal isotropic, the symplectic pairing yieds an
isomorphism $Lie_{\mathrm{dR}}^{t}\simeq \omega _{\mathrm{dR}}^{\vee }\left(
1\right) $ which composed with the isomorphism $\omega _{\mathrm{dR}}^{\vee
}\simeq \omega \left( G^{t}/X\right) ^{\vee }$\ induced by the polarization
gives the above identification $Lie_{\mathrm{dR}}^{t}\simeq \omega \left(
G^{t}/X\right) ^{\vee }\left( -1\right) $ (cfr. \cite[proof of Lemma 2.3]%
{Fo23}).

On the other hand, let us consider the algebraic group $\mathbf{GSp}\left(
V,\psi \right) $ acting on $V\simeq \Bbbk ^{2g}$ and fix a maximal isotropic
subspaces $W$ and $\overline{W}$ such that $V=W\perp \overline{W}$. Let us
write $\mathbf{Q}\subset \mathbf{GSp}\left( V,\psi \right) $ (resp. $\mathbf{%
M}\subset \mathbf{GSp}\left( V,\psi \right) $) for the subgroup of those
elements which preserves $W$ (resp. the decomposition $W\perp \overline{W}$%
). Then, setting $\mathcal{O}_{X}^{?}:=\mathcal{O}_{X}\otimes _{\Bbbk }?$,
we see that $\mathcal{O}_{X}^{V}$ comes equipped with a filtration%
\begin{equation*}
0\longrightarrow \mathcal{O}_{X}^{W}\longrightarrow \mathcal{O}%
_{X}^{V}\longrightarrow \mathcal{O}_{X}^{V/W}\longrightarrow 0
\end{equation*}%
and a symplectic pairing $1_{\mathcal{O}_{X}}\otimes _{K}\psi $ under which $%
\mathcal{O}_{X}^{W}$ is maximal isotropic. Furthermore, the natural action
of $\mathbf{Q}$ on $\mathcal{O}_{X}^{V}$ preserves the symplectic pairing
and the filtration.

Writing $\mathcal{T}_{\mathcal{H},\mathrm{dR}}^{\times }:=\mathcal{I}som_{%
\mathcal{O}_{X},ss,Hdg}\left( \mathcal{O}_{X}^{V},\mathcal{H}_{\mathrm{dR}%
}^{1}\right) $ (resp. $\mathcal{T}_{\omega ,\mathrm{dR}}^{\times }:=\mathcal{%
I}som_{\mathcal{O}_{X}}\left( \mathcal{O}_{X}^{W},\omega _{\mathrm{dR}%
}\right) \times \mathbf{G}_{m,X}$) for the Zariski sheaf whose sections over 
$U\subset X$ are the isomorphisms of $\mathcal{O}_{U}$-modules $\mathcal{O}%
_{U}^{V}\overset{\sim }{\rightarrow }\mathcal{H}_{\mathrm{dR},U}^{1}$
respecting the symplectic pairing up to similitude and the filtration (resp.
isomorphisms of $\mathcal{O}_{U}$-modules$\ \mathcal{O}_{U}^{W}\overset{\sim 
}{\rightarrow }\omega _{\mathrm{dR},U}$ times $\mathbf{G}_{m,U}$), we have a
natural $\mathbf{Q}$\textbf{-}action via%
\begin{equation*}
g^{-1}\alpha =\alpha g:=\alpha \circ g\text{ (resp. }g^{-1}\left( \alpha
,\nu \right) =\left( \alpha ,\nu \right) g:=\left( \alpha \circ g,\nu \nu
\left( g\right) \right) \text{).}
\end{equation*}%
Because an isomorphism $\alpha :\mathcal{O}_{U}^{V}\overset{\sim }{%
\rightarrow }\mathcal{H}_{\mathrm{dR},U}^{1}$ naturally induces isomorphisms 
$\mathcal{O}_{U}^{W}\overset{\sim }{\rightarrow }\omega _{\mathrm{dR},U}$
and it has an associated multiplier $\nu _{\alpha }\in \mathcal{O}%
_{X}^{\times }\left( U\right) $ characterized by the equality $\left\langle
\alpha \left( x\right) ,\alpha \left( y\right) \right\rangle =\nu _{\alpha
}\left\langle x,y\right\rangle $, we have a morphism of sheaves of $\mathbf{Q%
}$-sets%
\begin{equation*}
\mathcal{T}_{\mathcal{H},\mathrm{dR}}^{\times }\longrightarrow \mathcal{T}%
_{\omega ,\mathrm{dR}}^{\times }\text{.}
\end{equation*}

\begin{remark}
\label{Sheaves R Sim}Consider the Zariski sheaf $\mathcal{T}_{\mathcal{H},%
\mathrm{dR}}^{\circ \times }:=\mathcal{I}som_{\mathcal{O}_{X},s,Hdg}\left( 
\mathcal{O}_{X}^{V},\mathcal{H}_{\mathrm{dR}}^{1}\right) $ classifying
isomorphisms of $\mathcal{O}_{U}$-modules $\mathcal{O}_{U}^{V}\overset{\sim }%
{\rightarrow }\mathcal{H}_{\mathrm{dR},U}^{1}$ respecting the symplectic
pairing and the filtration, so that $\mathcal{T}_{\mathcal{H},\mathrm{dR}%
}^{\circ \times }\subset \mathcal{T}_{\mathcal{H},\mathrm{dR}}^{\times }$.
Then, to give a section $\alpha $ of the pull-back of $\mathcal{T}_{\mathcal{%
H},\mathrm{dR}}^{\circ \times }$ to $S\rightarrow X$ (a morphism of schemes)
is the same thing as to give an ordered symplectic-Hodge basis $\mathcal{B}$%
\ of $\mathcal{H}_{\mathrm{dR},S}^{1}$: after fixing once and for all a
symplectic-Hodge $\mathcal{B}_{V}$\ of $V$, we associate to the section $%
\alpha $ the ordered basis $\alpha \left( \mathcal{B}_{V}\right) $ and, to
the basis $\mathcal{B}$, we associate the unique isomorphism $\alpha _{%
\mathcal{B}}:\mathcal{O}_{S}^{V}\overset{\sim }{\rightarrow }\mathcal{H}_{%
\mathrm{dR},S}^{1}$ sending the ordered basis $\mathcal{B}_{V}$ to $\mathcal{%
B}$. Then $\mathcal{T}_{\mathcal{H},\mathrm{dR}}^{\circ \times }$ is a $%
\mathbf{Q}^{\circ }$-torsor over $X$ (see \cite[Theorem 4.5 and Lemma 4.6]%
{Fo23}).
\end{remark}

Indeed, one sees that $\mathcal{T}_{\mathcal{H},\mathrm{dR}}^{\times }$ is a 
$\mathbf{Q}$-torsor over $X$, the action of $\mathbf{Q}$ on $\mathcal{T}%
_{\omega ,\mathrm{dR}}^{\times }$ factors through the quotient map $\mathbf{Q%
}\rightarrow \mathbf{M}$ making $\mathcal{T}_{\omega ,\mathrm{dR}}^{\times }$
an $\mathbf{M}$-torsor over $X$ and $\mathcal{T}_{\mathcal{H},\mathrm{dR}%
}^{\times }$ is an $\mathbf{U}$-torsor over $\mathcal{T}_{\omega ,\mathrm{dR}%
}^{\times }$. Hence, $\mathcal{T}_{\mathcal{H},\mathrm{dR}}^{\times }$
admits sections Zariski locally on $X$ and, if $\alpha _{0}:\mathcal{O}%
_{X\mid U}^{V}\overset{\sim }{\rightarrow }\mathcal{H}_{\mathrm{dR},U}^{1}$
locally on $U$, then $\mathbf{Q}_{/U}\overset{\sim }{\rightarrow }\mathcal{T}%
_{\mathcal{H},\mathrm{dR}\mid U}^{\times }$ and we have%
\begin{equation*}
\mathbf{Q}_{/U}\overset{\sim }{\rightarrow }\mathcal{T}_{\mathcal{H},\mathrm{%
dR}\mid U}^{\times }\text{ via }g\mapsto \alpha _{0}\circ g=\alpha _{0}g%
\text{.}
\end{equation*}%
Similarly, locally on $X$ we have%
\begin{equation*}
\mathbf{M}_{/U}\overset{\sim }{\rightarrow }\mathcal{T}_{\omega ,\mathrm{dR}%
\mid U}^{\times }\text{.}
\end{equation*}

\bigskip

We can now define a functor%
\begin{equation*}
\mathcal{E}:\mathrm{Rep}\left( \mathbf{Q}\right) \longrightarrow \mathrm{Mod}%
_{\mathcal{O}_{X}}
\end{equation*}%
from the category $\mathrm{R}$\textrm{ep}$\left( \mathbf{Q}\right) $ to the
category $\mathrm{Mod}_{\mathcal{O}_{X}}$ of $\mathcal{O}_{X}$-modules as
follows (cfr. \cite[\S 2.1]{Liu12} and \cite[\S 7]{Ti12}):%
\begin{equation*}
\mathcal{E}_{\rho }:=\mathcal{H}om_{\mathbf{Q}}\left( \mathcal{T}_{\mathcal{H%
},\mathrm{dR}}^{\times },\mathcal{O}_{X}\otimes _{\Bbbk }\rho \right) \text{,%
}
\end{equation*}%
where $\mathcal{H}om_{\mathbf{Q}}$ denotes the sheaf whose section over $%
U\subset X$ are the morphisms of sheaves of sets $f:\mathcal{T}_{\mathcal{H},%
\mathrm{dR}\mid U}^{\times }\rightarrow \mathcal{O}_{U}\otimes _{\Bbbk }\rho 
$ such that $f\left( \alpha g\right) =g^{-1}f\left( \alpha \right) $ for
every $g\in \mathbf{Q}\left( U\right) $ and every local section $\alpha $
defined over an open subset of $U$. Here we remark that, by definition of $%
\mathrm{R}$\textrm{ep}$\left( \mathbf{Q}\right) $, we know that $\rho $ is a
Zariski sheaf on $\mathfrak{Spec}\left( \Bbbk \right) $\ and $\mathcal{O}%
_{X}\otimes _{\Bbbk }\rho $ denotes its pull-back to $X$.

\begin{remark}
\label{Sheaves R Triv}Suppose that $X_{0}\rightarrow X$ is a morphism of
locally ringed spaces with the property that the pull-back $\mathcal{H}_{%
\mathrm{dR},X_{0}}^{1}$\ of $\mathcal{H}_{\mathrm{dR}}^{1}$ to $X_{0}$ (as
an $\mathcal{O}_{X_{0}}$-module) comes equipped with an isomorphism $\alpha
_{0}:\mathcal{O}_{X_{0}}^{V}\overset{\sim }{\rightarrow }\mathcal{H}_{%
\mathrm{dR},X_{0}}^{1}$. Then $\mathbf{Q}_{/X_{0}}\overset{\sim }{%
\rightarrow }\mathcal{T}_{\mathcal{H},\mathrm{dR}\mid X_{0}}^{\times }$ (via 
$g\mapsto \alpha _{0}g$, as above) and, hence,%
\begin{equation*}
\vartheta _{\alpha _{0}}:\mathcal{E}_{\rho \mid X_{0}}\overset{\sim }{%
\rightarrow }\mathcal{H}om_{\mathbf{Q}}\left( \mathbf{Q}_{/X_{0}},\mathcal{O}%
_{X_{0}}\otimes _{\Bbbk }\rho \right) \overset{\sim }{\rightarrow }\mathcal{O%
}_{X_{0}}\otimes _{\Bbbk }\rho \text{ via }f\mapsto \left[ g\mapsto f\left(
\alpha _{0}g\right) \right] \mapsto f\left( \alpha _{0}\right) \text{.}
\end{equation*}%
The identification depends on the fixed choice of $\alpha _{0}$ and is
canonical in the variable $\rho $: this will give rise to the coherent
trivializations for varying $\rho $.

On the other hand, there is a unique (iso)morphism%
\begin{equation*}
\vartheta _{\mathcal{H},\mathrm{dR}}:\mathcal{E}_{V}\longrightarrow \mathcal{%
H}_{\mathrm{dR}}^{1}
\end{equation*}%
with the property that, if $X_{0}\rightarrow X$ and $\alpha _{0}:\mathcal{O}%
_{X_{0}}^{V}\overset{\sim }{\rightarrow }\mathcal{H}_{\mathrm{dR},X_{0}}^{1}$
are as above, then $\vartheta _{\mathcal{H},\mathrm{dR}\mid X_{0}}\left(
f\right) :=\alpha _{0}\left( f\left( \alpha _{0}\right) \right) $:%
\begin{equation*}
\vartheta _{\mathcal{H},\mathrm{dR}\mid X_{0}}:\mathcal{E}_{V\mid X_{0}}%
\overset{\vartheta _{\alpha _{0}}}{\longrightarrow }\mathcal{O}%
_{X_{0}}\otimes _{\Bbbk }\rho =\mathcal{O}_{X_{0}}^{V}\overset{\alpha _{0}}{%
\longrightarrow }\mathcal{H}_{\mathrm{dR},X_{0}}^{1}\text{.}
\end{equation*}%
Indeed, if $\alpha _{0}^{\prime }=\alpha _{0}g$ is another trivialization
over $X_{0}$, then%
\begin{equation*}
\alpha _{0}^{\prime }\left( f\left( \alpha _{0}^{\prime }\right) \right)
=\alpha _{0}gf\left( \alpha _{0}g\right) =\alpha _{0}f\left( \alpha
_{0}gg^{-1}\right) =\alpha _{0}\left( f\left( \alpha _{0}\right) \right) 
\text{.}
\end{equation*}%
Hence $\vartheta _{\mathcal{H},\mathrm{dR}\mid X_{0}}$ is well defined and,
because $\mathcal{T}_{\mathcal{H},\mathrm{dR}}^{\times }$ is a $\mathbf{Q}$%
-torsor, this implies that the $\vartheta _{\alpha _{0}}$'s obtained from a
suitable Zariski covering of $X$ glue together to give $\vartheta _{\mathcal{%
H},\mathrm{dR}}$.
\end{remark}

We recall that $\mathbf{Q}=\mathbf{U\rtimes M}$, implying that the
representations of $\mathbf{M}$ can be regarded as representations of $%
\mathbf{Q}$ by restriction via the canonical quotient map%
\begin{equation*}
\mathbf{Q}\longrightarrow \mathbf{M}\text{.}
\end{equation*}%
We get in this way a functor $\mathrm{Rep}\left( \mathbf{M}\right)
\rightarrow \mathrm{Rep}\left( \mathbf{Q}\right) $ and, if $\rho \in \mathrm{%
Rep}\left( \mathbf{M}\right) $, we write again $\rho $ for its image in $%
\mathrm{Rep}\left( \mathbf{Q}\right) $ and consider $\mathcal{E}_{\rho }$.
We can give an alternative description of the functor $\mathcal{E}_{\rho }$
in this case as follows. We define%
\begin{equation*}
\mathcal{W}:\mathrm{Rep}\left( \mathbf{M}\right) \longrightarrow \mathrm{Mod}%
_{\mathcal{O}_{X}}
\end{equation*}%
via the rule%
\begin{equation*}
\mathcal{W}_{\rho }:=\mathcal{H}om_{\mathbf{M}}\left( \mathcal{T}_{\omega ,%
\mathrm{dR}}^{\times },\mathcal{O}_{X}\otimes _{K}\rho \right) \text{.}
\end{equation*}

\begin{remark}
\label{Sheaves R Triv'}The analogue of Remark \ref{Sheaves R Triv} holds for
the functor $\mathcal{W}$: we will write $\vartheta _{\alpha _{0}}:\mathcal{W%
}_{\rho \mid X_{0}}\overset{\sim }{\rightarrow }\mathcal{O}_{X_{0}}\otimes
_{\Bbbk }\rho $ and $\vartheta _{\omega ,\mathrm{dR}}:\mathcal{W}_{W}\overset%
{\sim }{\rightarrow }\omega _{\mathrm{dR}}$ for the analogous isomorphisms
(the notations being the same).
\end{remark}

We have%
\begin{equation*}
\mathcal{W}_{\rho }=\mathcal{E}_{\rho }\text{ canonically for }\rho \in 
\mathrm{Rep}\left( \mathbf{M}\right)
\end{equation*}%
via the canonical morphism of sheaves $\mathcal{W}_{\rho }\rightarrow 
\mathcal{E}_{\rho }$ induced by $\mathcal{T}_{\mathcal{H},\mathrm{dR}%
}^{\times }\rightarrow \mathcal{T}_{\omega ,\mathrm{dR}}^{\times }$: since
the isomorphism can be checked locally, this indeed a consequence of Remark %
\ref{Sheaves R Triv} and Remark \ref{Sheaves R Triv'}. Then, we deduce from
Remark \ref{Sheaves R Triv'} that there is a canonical identification%
\begin{equation*}
\mathcal{E}_{W}=\mathcal{W}_{W}=\omega _{\mathrm{dR}}\text{.}
\end{equation*}

\bigskip

Let us write $\mathrm{Rep}_{fl\text{-}alg}\left( \mathfrak{g},\mathbf{Q}%
\right) $ for the full subcategory of those objects of $\mathrm{Rep}%
_{alg}\left( \mathfrak{g},\mathbf{Q}\right) $ whose underlying $\Bbbk $%
-module is flat. We will show below that, if we start with an algebraic $%
\left( \mathfrak{g},\mathbf{Q}\right) $-module $\rho \in \mathrm{Rep}_{fl%
\text{-}alg}\left( \mathfrak{g},\mathbf{Q}\right) $, then we can get more
structure on the underlying sheaf $\mathcal{E}_{\rho }$. Namely, assuming
that $\Bbbk $ is any Dedekind domain containing $\mathbb{Z}\left[
1/d_{1}...d_{g}N\right] $ (this assumption will be in force until the end of 
\S \ref{S Sheaves}), we are going to define a functor%
\begin{equation*}
\mathcal{E}:\mathrm{Rep}_{fl\text{-}alg}\left( \mathfrak{g},\mathbf{Q}%
\right) \longrightarrow \mathrm{FMIC}_{\mathcal{O}_{X}}
\end{equation*}%
from the category $\mathrm{Rep}_{fl\text{-}alg}\left( \mathfrak{g},\mathbf{Q}%
\right) $ to the category of filtered modules endowed with an integrable
connection satisfying the Griffith transversality condition whose underlying 
$\mathcal{O}_{X}$-module is $\mathcal{E}_{\rho }$ as defined above. Let us
write $\mathcal{D}er_{\Bbbk }^{\mathrm{lo}\text{\textrm{g}}}\left( \mathcal{O%
}_{X}\right) :=\mathcal{D}er_{\Bbbk }\left( \mathcal{O}_{X}\right) \left( 
\mathrm{lo}\text{\textrm{g}}\left( D\right) \right) $ for the Zariski sheaf
whose sections over $U\subset X$ are the $\Bbbk $-linear derivations of $%
\mathcal{O}_{U}$\ onto itself (with logarithmic poles at the boundary) and,
for an $\mathcal{O}_{X}$-module $\mathcal{E}$, let $\mathcal{E}nd_{\Bbbk
}\left( \mathcal{E}\right) $ be the sheaf whose sections over $U\subset X$
are the $\Bbbk $-linear endomorphisms of $\mathcal{E}_{\mid U}$: both of
them are $\mathcal{O}_{X}$-modules and sheaves of Lie algebras over $\Bbbk $%
. Note that $\mathcal{D}er_{\Bbbk }^{\mathrm{lo}\text{\textrm{g}}}\left( 
\mathcal{O}_{X}\right) $ is canonically isomorphic, as an $\mathcal{O}_{X}$%
-module, to the sheaf $\mathcal{H}om_{\mathcal{O}_{X}}\left( \Omega
_{X/\Bbbk }^{1}\left( \mathrm{lo}\text{\textrm{g}}\left( D\right) \right) ,%
\mathcal{O}_{X}\right) $ whose sections over $U\subset X$ are the morphisms
of $\mathcal{O}_{U}$-modules from $\Omega _{U/\Bbbk }^{\mathrm{lo}\text{%
\textrm{g},}1}:=\Omega _{U/\Bbbk }^{1}\left( \mathrm{lo}\text{\textrm{g}}%
\left( D\right) \right) $ to $\mathcal{O}_{U}$: it follows that a connection 
$\nabla :\mathcal{E}\longrightarrow \mathcal{E}\otimes _{\mathcal{O}%
_{X}}\Omega _{X/\Bbbk }^{\mathrm{lo}\text{\textrm{g},}1}$ give rise to a
morphism of $\mathcal{O}_{X}$-modules $\nabla $ from $\mathcal{D}er_{\Bbbk
}^{\mathrm{lo}\text{\textrm{g}}}\left( \mathcal{O}_{X}\right) $ to $\mathcal{%
E}nd_{\Bbbk }\left( \mathcal{E}_{\rho }\right) $ sending a section $D$ to $%
\nabla \left( D\right) :=\left( 1\otimes _{\mathcal{O}_{X}}D\right) \circ
\nabla $ and the Leibnitz rule for $\nabla $ implies that%
\begin{equation}
\nabla \left( D\right) \left( ax\right) =D\left( a\right) x+a\nabla \left(
D\right) \left( x\right)  \label{Sheaves F Leib}
\end{equation}%
for sections $D$, $a$ and $x$ of $\mathcal{D}er_{\Bbbk }^{\mathrm{lo}\text{%
\textrm{g}}}\left( \mathcal{O}_{X}\right) $, $\mathcal{O}_{X}$ and,
respectively, $\mathcal{E}$; furthermore, if the connection is integrable,
then the resulting morphism of $\mathcal{O}_{X}$-modules is of Lie algebras.
When $\Omega _{X/\Bbbk }^{\mathrm{lo}\text{\textrm{g},}1}$ is locally free
of finite type, we can go in the opposite direction: to give a morphism of $%
\mathcal{O}_{X}$-modules $\nabla $ from $\mathcal{D}er_{\Bbbk }^{\mathrm{lo}%
\text{\textrm{g}}}\left( \mathcal{O}_{X}\right) $ to $\mathcal{E}nd_{\Bbbk
}\left( \mathcal{E}\right) $ such that $\left( \text{\ref{Sheaves F Leib}}%
\right) $ holds is the same as to give a connection $\nabla $ and, under
this correspondence, the connection is integrable if and only if the
morphism of $\mathcal{O}_{X}$-modules is of Lie algebras.

\begin{remark}
Indeed, without any assumption on $\Bbbk $, we will endow $\mathcal{E}_{\rho
}$ with a morphism of $\mathcal{O}_{X}$-modules$\ \nabla _{\rho }$ from $%
\mathcal{D}er_{\Bbbk }^{\mathrm{lo}\text{\textrm{g}}}\left( \mathcal{O}%
_{X}\right) $ to $\mathcal{E}nd_{\Bbbk }\left( \mathcal{E}_{\rho }\right) $
which is of Lie algebras and such that $\nabla _{\rho }\left( D\right)
\left( \mathrm{Fil}^{r}\mathcal{E}_{\rho }\right) \subset \mathrm{Fil}^{r-1}%
\mathcal{E}_{\rho }$. When $\Bbbk $ contains $\mathbb{Z}\left[
1/d_{1}...d_{g}N\right] $, then $X$ is smooth over $\Bbbk $ and,
consequently, $\Omega _{X/\Bbbk }^{\mathrm{lo}\text{\textrm{g},}1}$ is
locally free of finite type and, hence, $\mathcal{E}$ can be regarded as
taking values in $\mathrm{FMIC}_{\mathcal{O}_{X}}$.
\end{remark}

Fix once and for all an ordered symplectic-Hodge basis $\left\{
w_{1},...,w_{g},w_{g+1},...,w_{2g}\right\} $\ of $V$ (for example, the
standard bsis when $\mathbf{G}=\mathbf{GSp}_{2g}$). We are going to define a
morphism of sheaves%
\begin{equation*}
X:\mathcal{D}er_{\Bbbk }^{\mathrm{lo}\text{\textrm{g}}}\left( \mathcal{O}%
_{X}\right) \times \mathcal{T}_{\mathcal{H},\mathrm{dR}}^{\times
}\longrightarrow \mathcal{O}_{X}\otimes _{\Bbbk }\mathfrak{g}
\end{equation*}%
sending a section $\left( D,\alpha \right) \in \mathcal{D}er_{\Bbbk }^{%
\mathrm{lo}\text{\textrm{g}}}\left( \mathcal{O}_{X}\right) \left( U\right)
\times \mathcal{T}_{\mathcal{H},\mathrm{dR}}^{\times }\left( U\right) $ to $%
X\left( D,\alpha \right) $ defined as follows. Suppose now that $\alpha \in 
\mathcal{T}_{\mathcal{H},\mathrm{dR}}^{\times }\left( U\right) $, set $%
\omega _{i}:=\alpha \left( w_{i}\right) $ and, using the fact that $\left\{
\omega _{i}:i=1,...,2g\right\} $ is an $\mathcal{O}_{U}$-basis of $\mathcal{H%
}_{\mathrm{dR},U}^{1}$, define $X_{g}\left( D,\alpha \right) =\left(
X_{g}\left( D,\alpha \right) _{j,i}\right) \in \mathbf{M}_{2g}\left( 
\mathcal{O}_{X}\left( U\right) \right) =\mathcal{O}_{X}\left( U\right)
\otimes _{\Bbbk }\mathfrak{gl}_{2g}$ by means of the rule%
\begin{equation*}
\nabla \left( D\right) \left( \omega _{i}\right)
=\tsum\nolimits_{j=1}^{2g}X_{g}\left( D,\alpha \right) _{j,i}\omega _{j}%
\text{,}
\end{equation*}%
where here $\nabla $ is the Gauss-Manin connection. The identification $%
\theta :\left( \Bbbk ^{2g},\psi _{g}\right) \overset{\sim }{\rightarrow }%
\left( V,\psi \right) $ of at the end of Example \ref{Representations E GSp}%
\ yields $c_{\theta }:\mathbf{GL}_{2g}\overset{\sim }{\rightarrow }\mathbf{GL%
}\left( V\right) $ and, hence, a morphism $c_{\theta }:\mathfrak{gl}_{2g}%
\overset{\sim }{\rightarrow }\mathfrak{gl}_{V}$: we write $X\left( D,\alpha
\right) \in \mathcal{O}_{X}\left( U\right) \otimes _{\Bbbk }\mathfrak{gl}%
_{V} $ for the image of $X_{g}\left( D,\alpha \right) $.

\begin{lemma}
\label{Sheaves L T1}We have that $X\left( D,\alpha \right) \in \mathcal{O}%
_{X}\left( U\right) \otimes _{\Bbbk }\mathfrak{g}$.
\end{lemma}

\begin{proof}
We (may) assume that $\mathbf{G}=\mathbf{GSp}_{2g}$ and we write $I=\left(
\left\langle w_{i},w_{j}\right\rangle \right) _{i,j}$ for the matrix
representing the symplectic pairing $\psi $ in the basis $\left\{
w_{i}:i=1,...,2g\right\} $. Then we have%
\begin{equation*}
\mathfrak{g}_{R}=\left\{ X\in \mathbf{M}_{2g}\left( R\right) :\exists
\lambda \in R\text{ s.t. }X^{t}I+IX=\lambda I\right\} \text{.}
\end{equation*}%
We remark that, for every section $D$ of $\mathcal{D}er_{\Bbbk }^{\mathrm{lo}%
\text{\textrm{g}}}\left( \mathcal{O}_{X}\right) $ and every section $x,y\in 
\mathcal{H}_{\mathrm{dR}}^{1}$, we have that%
\begin{equation}
D\left( \left\langle x,y\right\rangle \right) =\left\langle \nabla \left(
D\right) \left( x\right) ,y\right\rangle +\left\langle x,\nabla \left(
D\right) \left( y\right) \right\rangle  \label{Sheaves L T1 F1}
\end{equation}%
as it follows form the fact that the first Chern class of the Poincar\'{e}
line bundle is horizontal for the Gauss-Manin connection, since it comes
from a de Rham class of $\mathcal{H}_{\mathrm{dR}}^{2}\left( G_{U}\times
_{U}G_{U}^{t}/\Bbbk \right) $\ (cfr. \cite[$\left( 5.5\right) $]{Fo23}). Let 
$\nu _{\alpha }\in \mathcal{O}_{X}^{\times }\left( U\right) $ be
characterized by the equality $\left\langle \alpha \left( x\right) ,\alpha
\left( y\right) \right\rangle =\nu _{\alpha }\left\langle x,y\right\rangle $%
. Then%
\begin{equation*}
\left\langle \omega _{i},\omega _{j}\right\rangle =\left\langle \alpha
\left( w_{i}\right) ,\alpha \left( w_{j}\right) \right\rangle =\nu _{\alpha
}\left\langle w_{i},w_{j}\right\rangle \text{.}
\end{equation*}%
In particular, we have that%
\begin{eqnarray*}
&&\left\langle \nabla \left( D\right) \left( \omega _{i}\right) ,\omega
_{j}\right\rangle +\left\langle \omega _{i},\nabla \left( D\right) \left(
\omega _{j}\right) \right\rangle =\tsum\nolimits_{k=1}^{2g}X\left( D,\alpha
\right) _{k,i}\left\langle \omega _{k},\omega _{j}\right\rangle
+\tsum\nolimits_{l=1}^{2g}X\left( D,\alpha \right) _{l,j}\left\langle \omega
_{i},\omega _{l}\right\rangle \\
&&\text{ }=\nu _{\alpha }\tsum\nolimits_{k=1}^{2g}X\left( D,\alpha \right)
_{k,i}\left\langle w_{k},w_{j}\right\rangle +\nu _{\alpha
}\tsum\nolimits_{l=1}^{2g}X\left( D,\alpha \right) _{l,j}\left\langle
w_{i},w_{l}\right\rangle =\nu _{\alpha }\left( X\left( D,\alpha \right)
^{t}I\right) _{i,j}+\nu _{\alpha }\left( IX\left( D,\alpha \right) \right)
_{i,j}
\end{eqnarray*}%
and that, because $\left\langle w_{i},w_{j}\right\rangle \in \Bbbk $,%
\begin{equation*}
D\left( \left\langle \omega _{i},\omega _{j}\right\rangle \right) =D\left(
\nu _{\alpha }\right) \left\langle w_{i},w_{j}\right\rangle \text{.}
\end{equation*}%
It then follows from $\left( \text{\ref{Sheaves L T1 F1}}\right) $ applied
to $\left( x,y\right) =\left( \omega _{i},\omega _{j}\right) $ that, setting 
$\lambda _{\alpha }:=\frac{D\left( \nu _{\alpha }\right) }{\nu _{\alpha }}$,
we have%
\begin{equation*}
X\left( D,\alpha \right) ^{t}I+IX\left( D,\alpha \right) =\lambda _{\alpha }I%
\text{.}
\end{equation*}
\end{proof}

If now $\rho \in \mathrm{Rep}_{fl\text{-}alg}\left( \mathfrak{g},\mathbf{Q}%
\right) $ and $f\in \mathcal{E}_{\rho }\left( U\right) =Hom_{\mathbf{Q}%
}\left( \mathcal{T}_{\mathcal{H},\mathrm{dR}\mid U}^{\times },\mathcal{O}%
_{U}\otimes _{\Bbbk }\rho \right) $, we let $\nabla _{\rho }\left( D\right)
\left( f\right) $ be the morphism of sheaves in $Hom\left( \mathcal{T}_{%
\mathcal{H},\mathrm{dR}\mid U}^{\times },\mathcal{O}_{U}\otimes _{\Bbbk
}\rho \right) $ defined by means of the rule%
\begin{equation*}
\nabla _{\rho }\left( D\right) \left( f\right) _{U^{\prime }}\left( \alpha
\right) :=D\otimes _{\Bbbk }1_{\rho }\left( f_{U^{\prime }}\left( \alpha
\right) \right) +X\left( D,\alpha \right) \left( f_{U^{\prime }}\left(
\alpha \right) \right)
\end{equation*}%
for every open subset $U^{\prime }\subset U$ and every $\left( D,\alpha
\right) \in \mathcal{D}er_{\Bbbk }^{\mathrm{lo}\text{\textrm{g}}}\left( 
\mathcal{O}_{X}\right) \left( U^{\prime }\right) \times \mathcal{T}_{%
\mathcal{H},\mathrm{dR}}^{\times }\left( U^{\prime }\right) $. Here, thanks
to Lemma \ref{Sheaves L T1}, we let $X\left( D,\alpha \right) \in \mathcal{O}%
_{X}\left( U^{\prime }\right) \otimes _{\Bbbk }\mathfrak{g}$ acts on $%
f_{U^{\prime }}\left( \alpha \right) \in \mathcal{O}_{X}\left( U^{\prime
}\right) \otimes _{\Bbbk }\rho $ via its Lie algebra action. We ill later
discuss the filtrations: for the moment, we write $\mathrm{MIC}_{\mathcal{O}%
_{X}}$ for the category of modules with an integrable connection and we
prove the following result.

\begin{theorem}
\label{Sheaves T1}Suppose that $\Bbbk $ is any Dedekind domain containing $%
\mathbb{Z}\left[ 1/d_{1}...d_{g}N\right] $ and that $\rho \in \mathrm{Rep}%
_{fl\text{-}alg}\left( \mathfrak{g},\mathbf{Q}\right) $. We have that $%
\nabla _{\rho }\left( D\right) \left( f\right) \in \mathcal{E}_{\rho }\left(
U\right) $ and, in this way, we get $\left( \mathcal{E}_{\rho },\nabla
_{\rho }\right) $ which belongs to (the objects of) $\mathrm{MIC}_{\mathcal{O%
}_{X}}$. This rule defines a functor from $\mathrm{Rep}_{fl\text{-}%
alg}\left( \mathfrak{g},\mathbf{Q}\right) $ to $\mathrm{MIC}_{\mathcal{O}%
_{X}}$ with the property that $\left( \mathcal{E}_{V},\nabla _{V}\right)
=\left( \mathcal{H}_{\mathrm{dR}}^{1},\nabla \right) $ equals the
Gauss-Manin connection, the identification being given by $\vartheta _{%
\mathcal{H},\mathrm{dR}}$ (see Remark \ref{Sheaves R Triv}).
\end{theorem}

\begin{proof}
We (may) assume that $\mathbf{G}=\mathbf{GSp}_{2g}$. When $\Bbbk $ is a
field, the result is proved in \cite[Proposition 2.4]{Liu12} and here we
just indicates why we can assume that $\Bbbk $ is a Dedekind domain. Let us
set $D_{\rho }:=D\otimes _{\Bbbk }1_{\rho }$ and, if $g=\left(
g_{i,j}\right) \in \mathbf{M}_{2g}\left( \mathcal{O}_{X}\left( U^{\prime
}\right) \right) $ for some open subset $U^{\prime }\subset U$, define $%
D\left( g\right) :=\left( D\left( g_{j,i}\right) \right) \in \mathbf{M}%
_{2g}\left( \mathcal{O}_{X}\left( U^{\prime }\right) \right) $. Then the
third of the chain of equalities appearing in \cite[pag. 2448]{Liu12} uses
the following equality%
\begin{equation*}
D_{\rho }\left( gx\right) =g\left( D_{\rho }\left( x\right) \right) +D\left(
g\right) \left( x\right)
\end{equation*}%
for $g\in \mathbf{Q}\left( U^{\prime }\right) $ and $x\in \mathcal{O}%
_{X}\left( U^{\prime }\right) \otimes _{\Bbbk }\rho $ (with our notations,
it needs to be applied to $x=f\left( \alpha g\right) $). Because $\rho \in 
\mathrm{Rep}_{fl\text{-}alg}\left( \mathfrak{g},\mathbf{Q}\right) $ and $%
\Bbbk $ is a Dedekind domain, it follows from Theorem \ref{Representations
T1} that $\rho $ is a colimit in \textrm{Rep}$_{alg}\left( \mathbf{Q}\right) 
$ of representations in \textrm{Rep}$_{f}\left( \mathbf{Q}\right) $, each of
them being constructed from $V=\mathrm{Std}_{2g}$ by forming tensor
products, direct sums, duals, and subquotients. This allow us to check the
equality in the $\rho =V$ case, which is easy. The other calculations are
the same.
\end{proof}

As promised, it remains to discuss the filtration. As explained in \cite[%
pag. 230]{FC}, one can define it by looking at the effect of the action of
the diagonal element $\left( 0,...,0,1,...,1\right) \in \mathfrak{g}$ when
acting on $\rho $: the sum $F_{p}\left( \rho \right) $ of the generalized
eigenspaces for eigenvalues $\geq p$ gives rise to decreasing filtration by $%
\mathbf{Q}$-submodules\ which yields a filtration on $\mathcal{E}_{\rho }$
satisfying the Griffith transversality condition and giving rise to a mixed
Hodge structure when $\rho $ is in \textrm{Rep}$_{f}\left( \mathbf{G}\right) 
$ (see \cite[Ch. VI, \S 5, 5.5 Theorem]{FC}). For our purposes, however, it
suffices to define such a filtration in the case $\rho =\mathrm{Ind}_{%
\mathbf{Q}^{-}}^{\mathbf{G}}\left( \rho _{0}\right) \left[ \mathbf{Y}\right] 
$ for some $\rho _{0}\in \mathrm{R}$\textrm{ep}$_{f}\left( \mathbf{M}\right) 
$. For these specific representations, we may use the $\mathbf{Q}$-module
filtration provided by the (opposite of the) increasing degree filtration $%
\mathrm{Ind}_{\mathbf{Q}^{-}}^{\mathbf{G}}\left( \rho _{0}\right) \left[ 
\mathbf{Y}\right] _{\leq r}$ (cfr. Lemma \ref{Representations L1'}), that we
will check to satisfy Griffith transversality in Lemma \ref{Sheaves L Fil}
below.

\begin{remark}
\label{Sheaves R Fil1}One can check that $F_{-r}\left( \mathrm{Ind}_{\mathbf{%
Q}^{-}}^{\mathbf{G}}\left( \rho _{0}\right) \left[ \mathbf{Y}\right] \right)
=\mathrm{Ind}_{\mathbf{Q}^{-}}^{\mathbf{G}}\left( \rho _{0}\right) \left[ 
\mathbf{Y}\right] _{\leq r}$ implying that, in this way, when $\Bbbk $\ is a
field we get the same filtration as in \cite[Ch. VI, \S 5, 5.5 Theorem]{FC}
by viewing an irreducible representation inside such a kind of induced
module by means of the inclusion $\left( \text{\ref{Representations P BGG F
Incl}}\right) $. We will not need this fact.
\end{remark}

If $\mathcal{C}$ is an exact category (say understood as a full subcategory
of an abelian category with the resulting notion of exactness), let us write 
$\mathrm{Fil}\left( \mathcal{C}\right) $ to denote the subcategory of $%
\mathbb{Z}$-filtered objects with filtered morphisms. The exactness of
sequences $E$\ in \textrm{Fil}$\left( \mathcal{C}\right) $ is defined by
requiring the sequences $E$ and \textrm{Fil}$^{i}\left( E\right) $\ to be
exact when viewed in $\mathcal{C}$ for every $i\in \mathbb{Z}$.

\begin{remark}
\label{Sheaves R Fil2}Also, we note that $\mathrm{Ind}_{\mathbf{Q}^{-}}^{%
\mathbf{G}}\left( \rho _{0}\right) \left[ \mathbf{Y}\right] _{\leq r}$ can
be recursively defined by looking at the successive invariants under the $%
\mathbf{U}$-module action (by Lemma \ref{Representations L1'}). In this way,
we see that the (dual) BGG complex $\left( \text{\ref{Representations P BGG
F BGG}}\right) $\ (when $\Bbbk $ is a field) and its truncated integral
version $\left( \text{\ref{Representations T BGG F BGG}}\right) $ (when $%
\Bbbk $ is a Dedekind domain) can be promoted to exact sequences in $\mathrm{%
Fil}\left( \mathrm{Rep}_{alg}\left( \mathfrak{g},\mathbf{Q}\right) \right) $.
\end{remark}

\begin{proof}
Let us write $F:\mathrm{R}$\textrm{ep}$_{alg}\left( \mathbf{U}\right)
\rightarrow \mathrm{Fil}\left( \mathrm{Rep}_{alg}\left( \mathbf{U}\right)
\right) $ for the functor which takes the filtration recursively defined by
looking at the successive invariants under the $\mathbf{U}$-module action.
Taking the invariants by an affine group scheme is always a left exact
operation: hence, by induction, one deduces that $F$ is a left exact
functor. This proves the assertion for the truncated integral version $%
\left( \text{\ref{Representations T BGG F BGG}}\right) $, which is
everything we need in what follows. One checks that, indeed, for unipotent
groups taking the invariants is an exact operation and, hence, $F$ is exact
from which the assertion for $\left( \text{\ref{Representations P BGG F BGG}}%
\right) $ follows.
\end{proof}

\bigskip

We now discuss the Kodaira-Spencer isomorphism, again assuming that $\Bbbk $
is a Dedekind domain which contains $\mathbb{Z}\left[ 1/d_{1}...d_{g}N\right]
$ (being a Dedekind domain is not needed until Lemma \ref{Sheaves L Fil}
below). Consider the following chain of morphisms of $\mathcal{O}_{X}$%
-modules%
\begin{equation*}
\mathcal{D}er_{\Bbbk }^{\mathrm{lo}\text{\textrm{g}}}\left( \mathcal{O}%
_{X}\right) \longrightarrow \mathcal{E}nd_{\Bbbk }\left( \mathcal{H}_{%
\mathrm{dR}}^{1}\right) \longrightarrow \mathcal{H}om_{\mathcal{O}%
_{X}}\left( \omega _{\mathrm{dR}},Lie_{\mathrm{dR}}^{t}\right) \text{,}
\end{equation*}%
where the first arrow is obtained applying the above discussion to the
Gauss-Manin connection $\left( \mathcal{E},\mathcal{\nabla }\right) =\left( 
\mathcal{H}_{\mathrm{dR}}^{1},\mathcal{\nabla }\right) $ and the second
arrow, which is obtained by precomposition with the restriction to $\omega _{%
\mathrm{dR}}$ and postcomposition with the projection onto $Lie_{\mathrm{dR}%
}^{t}$, takes values in $\mathcal{H}om_{\mathcal{O}_{X}}$ (rather than $%
\mathcal{H}om_{\Bbbk }$) thanks to the Griffiths transversality condition.
We remark that the target is identified with $Lie_{\mathrm{dR}}^{t}\otimes _{%
\mathcal{O}_{X}}\omega _{\mathrm{dR}}^{\vee }\simeq \omega \left(
G^{t}/X\right) ^{\vee }\otimes _{\mathcal{O}_{X}}\omega \left( G/X\right)
^{\vee }\left( -1\right) $ and, hence, we get in this way a morphism from $%
\mathcal{D}er_{\Bbbk }^{\mathrm{lo}\text{\textrm{g}}}\left( \mathcal{O}%
_{X}\right) $ to $\omega \left( G^{t}/X\right) ^{\vee }\otimes _{\mathcal{O}%
_{X}}\omega \left( G/X\right) ^{\vee }\left( -1\right) $ whose dual is the
so called Kodaira-Spencer morphism from $\omega \left( G^{t}/X\right)
\otimes _{\mathcal{O}_{X}}\omega \left( G/X\right) \left( 1\right) $ to $%
\Omega _{X/\Bbbk }^{\mathrm{lo}\text{\textrm{g},}1}$. Because $\Bbbk $
contains $\mathbb{Z}\left[ 1/d_{1}...d_{g}N\right] $, the universal
polarization $\lambda $ yields an isomorphism%
\begin{equation}
\lambda ^{\ast }\otimes 1:\frac{\omega \left( G^{t}/X\right) \otimes _{%
\mathcal{O}_{X}}\omega \left( G/X\right) }{\left( y\otimes \lambda ^{\ast
}\left( z\right) -z\otimes \lambda ^{\ast }\left( y\right) :y,z\in \omega
\left( G^{t}/X\right) \right) }\longrightarrow \mathrm{Sym}^{2}\left( \omega
\left( G/X\right) \right) =:\mathrm{Sym}^{2}\left( \omega _{\mathrm{dR}%
}\right)  \label{Sheaves F Sym}
\end{equation}%
and it is known that the Kodaira-Spencer morphism factors through the target
of the above arrow and induces an isomorphism%
\begin{equation*}
KS_{X}:\mathrm{Sym}^{2}\left( \omega _{\mathrm{dR}}\right) \left( 1\right) 
\overset{\sim }{\longrightarrow }\Omega _{X/\Bbbk }^{\mathrm{lo}\text{%
\textrm{g},}1}
\end{equation*}%
(see \cite[Proposition 2.3.5.2 and Proposition 6.2.5.18]{Lan08} and \cite[%
Ch. III, discussion before 9.3 Lemma and 9.8 Corollary and Ch. IV, 3.1
Proposition $\left( vi\right) $]{FC}). Dually, we have the isomorphism%
\begin{equation*}
KS_{X}^{\vee }:\mathcal{D}er_{\Bbbk }^{\mathrm{lo}\text{\textrm{g}}}\left( 
\mathcal{O}_{X}\right) \overset{\sim }{\longrightarrow }\mathrm{Sym}%
^{2}\left( \omega _{\mathrm{dR}}^{\vee }\right) \left( -1\right)
\end{equation*}%
whose composition with the inclusion from $\mathrm{Sym}^{2}\left( \omega _{%
\mathrm{dR}}^{\vee }\right) $ to $\omega _{\mathrm{dR}}^{\vee }\otimes _{%
\mathcal{O}_{X}}\omega _{\mathrm{dR}}^{\vee }$ (sending the element
represented by $x\otimes y$\ to $x\otimes y+y\otimes x$) is the original
morphism from $\mathcal{D}er_{\Bbbk }^{\mathrm{lo}\text{\textrm{g}}}\left( 
\mathcal{O}_{X}\right) $ to $\omega \left( G/X\right) ^{\vee }\otimes _{%
\mathcal{O}_{X}}\omega \left( G^{t}/X\right) ^{\vee }$ followed by the
isomorphism $\lambda _{\ast }\otimes 1$. Taking into account the canonical
isomorphism $\mathcal{W}_{W}=\omega _{\mathrm{dR}}$ inducing $\mathcal{W}_{%
\mathrm{Sym}^{2}\left( W\right) \left( -1\right) }=\mathrm{Sym}^{2}\left(
\omega _{\mathrm{dR}}\right) \left( 1\right) $ we see that, for every $%
X_{0}\rightarrow X$ and $\alpha _{0}$ as in Remark \ref{Sheaves R Triv'}, we
get%
\begin{equation}
KS_{\alpha _{0}}:=KS_{X}\circ \vartheta _{\alpha _{0}}^{-1}:\mathcal{O}%
_{X_{0}}\otimes _{\Bbbk }\mathrm{Sym}^{2}\left( W\right) \left( -1\right) 
\overset{\sim }{\longrightarrow }\Omega _{X/\Bbbk }^{\mathrm{lo}\text{%
\textrm{g},}1}\text{.}  \label{Sheaves F KS}
\end{equation}%
Dually, we have%
\begin{equation}
KS_{\alpha _{0}}^{\vee }:\mathcal{D}er_{\Bbbk }^{\mathrm{lo}\text{\textrm{g}}%
}\left( \mathcal{O}_{X}\right) \overset{\sim }{\longrightarrow }\mathcal{O}%
_{X_{0}}\otimes _{\Bbbk }\mathrm{Sym}^{2}\left( W^{\vee }\right) \left(
1\right) \text{.}  \label{Sheaves F KS dual}
\end{equation}

\begin{lemma}
\label{Sheaves L Fil}Suppose that $\rho _{0}\in \mathrm{Rep}_{f}\left( 
\mathbf{GL}_{g}\times \mathbf{G}_{m}\right) $ and consider $\rho :=\mathrm{%
Ind}_{\mathbf{Q}^{-}}^{\mathbf{G}}\left( \rho _{0}\right) \left[ \mathbf{Y}%
\right] $, which comes equipped with the $\mathbf{Q}$-module filtration $%
\rho _{\leq r}:=\mathrm{Ind}_{\mathbf{Q}^{-}}^{\mathbf{G}}\left( \rho
_{0}\right) \left[ \mathbf{Y}\right] _{\leq r}$. Setting $F^{r}\left( 
\mathcal{V}_{\rho }\right) :=\mathcal{E}_{\rho _{\leq r}}$ we have $\nabla
_{\lambda }\left( F^{r}\left( \mathcal{V}_{\rho }\right) \right) \subset
F^{r+1}\left( \mathcal{V}_{\rho }\right) $.
\end{lemma}

\begin{proof}
It follows from Lemma \ref{Representations L1} $\left( 3\right) $ that every 
$\partial \in \mathfrak{u}^{-}$ increase the degree by one for every $v\in 
\mathrm{Ind}_{\mathbf{Q}^{-}}^{\mathbf{G}}\left( W_{\lambda }\right) \left[ 
\mathbf{Y}\right] $. On the other hand, because $\mathrm{Ind}_{\mathbf{Q}%
^{-}}^{\mathbf{G}}\left( W_{\lambda }\right) \left[ \mathbf{Y}\right] _{\leq
r}$ is a $\mathbf{Q}$-submodule, we know that every $\partial \in \mathfrak{q%
}$ preserves the filtration. Because $\mathfrak{g}=\mathfrak{u}^{-}\oplus 
\mathfrak{q}$, the result follows from $\left( \text{\ref{Sheaves F KS dual}}%
\right) $ and the definition of $\nabla _{\lambda }=\nabla _{\rho }$ for $%
\rho =\mathrm{Ind}_{\mathbf{Q}^{-}}^{\mathbf{G}}\left( W_{\lambda }\right) %
\left[ \mathbf{Y}\right] $.
\end{proof}

\bigskip

When $\lambda \in X_{\mathbf{G},+}$, applying $\mathcal{E}_{\cdot }$ to the
inclusions $\left( \text{\ref{Representations T BGG F Incl}}\right) $ (and
recalling that $\mathcal{E}_{\rho }=\mathcal{W}_{\rho }$ for every $\rho \in 
\mathrm{R}$\textrm{ep}$\left( \mathbf{M}\right) $ regarded as a $\mathbf{Q}$%
-module by means of the canonical morphism $\mathbf{Q\rightarrow M}$, so
that $\mathcal{E}_{W_{\lambda }}=\mathcal{W}_{W_{\lambda }}=:\mathcal{W}%
_{\lambda }$) yields (setting $\mathcal{L}_{\lambda }:=\mathcal{E}%
_{L_{\lambda }}$ and $\mathcal{V}_{\lambda }:=\mathcal{E}_{\mathrm{Ind}_{%
\mathbf{Q}^{-}}^{\mathbf{G}}\left( W_{\lambda }\right) \left[ \mathbf{Y}%
\right] }$):%
\begin{equation*}
\mathcal{W}_{\lambda }\hookrightarrow \mathcal{L}_{\lambda }\hookrightarrow 
\mathcal{V}_{\lambda }
\end{equation*}%
such that the first inclusion is of sheaves while the latter is of
integrable modules with connection $\nabla _{\lambda }$. Hence we also have
an inclusion of de Rham complexes%
\begin{equation}
\mathrm{dR}\left( \mathcal{L}_{\lambda }\right) \hookrightarrow \mathrm{dR}%
\left( \mathcal{V}_{\lambda }\right) \text{.}  \label{Sheaves F dRCmp}
\end{equation}%
We have defined in Lemma \ref{Sheaves L Fil} a filtration on $\mathcal{V}%
_{\lambda }$. Setting $L_{\lambda ,\leq r}:=L_{\lambda }\cap \mathrm{Ind}_{%
\mathbf{Q}^{-}}^{\mathbf{G}}\left( W_{\lambda }\right) \left[ \mathbf{Y}%
\right] _{\leq r}$ and then $F^{r}\left( \mathcal{L}_{\lambda }\right) :=%
\mathcal{E}_{L_{\lambda ,\leq r}}$ yields a filtration on $\mathcal{L}%
_{\lambda }$ such that $F^{0}\left( \mathcal{L}_{\lambda }\right) =\mathcal{W%
}_{\lambda }$, $\nabla _{\lambda }\left( F^{r}\left( \mathcal{L}_{\lambda
}\right) \right) \subset F^{r+1}\left( \mathcal{L}_{\lambda }\right) $ and
which promotes $\left( \text{\ref{Sheaves F dRCmp}}\right) $ to a morphism
of filtered complexes (as can be checked locally using Remark \ref{Sheaves R
Triv}).

\begin{remark}
\label{Sheaves R Ex}The above discussion requires the flatness of $X$ over $%
\Bbbk $ (which is true because we assume that $\Bbbk $ contains $\mathbb{Z}%
\left[ 1/d_{1}...d_{g}N\right] $). Indeed, as discussion in \S \ref{S
Representations Pre}, the exactness in $\underline{\mathrm{Mod}}_{alg}$
differs from the exactness as a functor; however, an exact sequence in $%
\underline{\mathrm{Mod}}_{alg}$ gives rise to an exact sequence of $R$%
-modules when evaluated at $R$-points for every $\Bbbk $-flat algebra $R$.
Applying this observation to (the ring of functions) of an open affine
subset of $X$ and using Remark \ref{Sheaves R Triv} we deduce that the
functor $\mathcal{E}_{-}$ is exact.
\end{remark}

\section{\label{S q-exp}$q$-expansion and degeneration at the $0$%
-dimensional cusps}

Let suppose that $\phi :\mathbb{Y}\rightarrow \mathbb{X}$ is an injective
morphism with finite cokernel between free $\mathbb{Z}$-modules of finite
rank $g$. Consider the tori $\mathbf{T}_{\mathbb{W}}:=\underline{Hom}%
_{gr}\left( \mathbb{W},\mathbf{G}_{m}\right) $, where $\mathbb{W}\in \left\{ 
\mathbb{X},\mathbb{Y}\right\} $ is regarded as an fppf sheaf over a base
scheme such that $\#\frac{\mathbb{X}}{\phi \left( \mathbb{Y}\right) }$ is
invertible\footnote{%
Strictly speaking the invertibility assumption is not needed in the
following discussion. However, the toroidal compactifications require this
assumption and, hence, the hypothesis allows us to interpret the discussion
below in this framework.}. We remark that the datum of a morphism (of
sheaves) of groups $Q:\mathbb{Y}\otimes _{\mathbb{Z}}\mathbb{X}\rightarrow 
\mathbf{G}_{m}$ is equivalent to that of a morphism $M_{Q}=\left[ \mathbb{Y}%
\rightarrow \mathbf{T}_{\mathbb{X}}\right] $ or, in other words, a ($1$%
-)motive with maximal torus rank $g$ (see \cite[D\'{e}finition 1.2.1]{Str10}%
). The dual motive is $M_{Q}^{t}:=\left[ \mathbb{X}\rightarrow \mathbf{T}_{%
\mathbb{Y}}\right] $, again obtained from $Q$. The datum of $\phi $ gives
rise to a morphism $\phi ^{t}:\mathbf{T}_{\mathbb{X}}\rightarrow \mathbf{T}_{%
\mathbb{Y}}$ and $\phi ^{\ast }\left( Q\right) :=Q\circ \left( 1\mathbb{%
\otimes }_{\mathbb{Z}}\phi \right) $ is symmetric if and only if $\lambda
_{\phi }:=\left( \phi ^{t},\phi \right) :M_{Q}\rightarrow M_{Q}^{t}$ is a
morphism of motives, which is then a polarization because $\phi $ is
injective with finite cokernel (see \cite[D\'{e}finition 1.2.3.3]{Str10}).
One can consider the functor classifying $4$-tuples $\left( M,\lambda
,\alpha ,\beta \right) $ where $M=\left[ Y\rightarrow T\right] $ is a motive
with maximal torus rank $g$, $\lambda =\left( \phi _{Y},\phi _{Y}^{t}\right) 
$ is a polarization and $\alpha :\mathbf{T}_{\mathbb{X}}\overset{\sim }{%
\rightarrow }T$ and $\beta :\mathbb{Y}\overset{\sim }{\rightarrow }Y$ are
isomorphisms such that $\phi $\ equals the composition%
\begin{equation*}
\mathbb{Y}\overset{\beta }{\longrightarrow }Y\overset{\phi _{Y}}{%
\longrightarrow }\underline{Hom}_{gr}\left( T,\mathbf{G}_{m}\right) \overset{%
\underline{Hom}_{gr}\left( \alpha ,\mathbf{G}_{m}\right) }{\longrightarrow }%
\underline{Hom}_{gr}\left( \mathbf{T}_{\mathbb{X}},\mathbf{G}_{m}\right) =%
\mathbb{X}\text{.}
\end{equation*}%
Indeed the above discussion shows that, if%
\begin{equation*}
Q_{\phi }:\mathbb{Y}\otimes _{\mathbb{Z}}\mathbb{X}\longrightarrow \frac{%
\mathbb{Y}\otimes _{\mathbb{Z}}\mathbb{X}}{\left( y\mathbb{\otimes }_{%
\mathbb{Z}}\phi \left( z\right) -z\mathbb{\otimes }_{\mathbb{Z}}\phi \left(
y\right) :y,z\in \mathbb{Y}\right) }=:\mathbb{S}_{\phi }
\end{equation*}%
is the canonical quotient map, then $\mathbf{E}_{\phi }:=\underline{Hom}%
_{gr}\left( \mathbb{S}_{\phi },\mathbf{G}_{m}\right) $ classifies these
quadruples: a point of $\mathbf{E}_{\phi }$ defined over a scheme is a
morphism (of sheaves) of groups $x:\mathbb{S}_{\phi }\rightarrow \mathbf{G}%
_{m}$ (over that scheme), which gives rise (and is equivalent) to $%
Q_{x}:=x\circ Q_{\phi }$ to which the above discussion applies giving to the
equivalent datum of the $4$-tuple $\left( M_{Q_{x}},\lambda _{\phi },1_{%
\mathbf{T}_{\mathbb{X}}},1_{\mathbb{Y}}\right) $\ (cfr. \cite[D\'{e}finition
1.4.3 and Proposition 1.4.4]{Str10}). In particular $\mathbf{E}_{\phi }$,
comes equipped with a universal motive $\mathbf{M}_{\phi }\rightarrow 
\mathbf{E}_{\phi }$ with maximal torus rank $g$.

\begin{remark}
\label{q-exp R1 KS}Using the universal trivializations $\alpha ^{univ}$ and $%
\beta ^{univ}$\ to change $\mathbf{M}_{\phi }$ by a (unique) isomorphism, we
may assume that $\mathbf{M}_{\phi }=\left[ \mathbb{Y}_{\mathbf{E}_{\phi
}}\rightarrow \mathbf{T}_{\mathbb{X},\mathbf{E}_{\phi }}\right] $ and $%
\mathbf{M}_{\phi }^{t}=\left[ \mathbb{X}_{\mathbf{E}_{\phi }}\rightarrow 
\mathbf{T}_{\mathbb{Y},\mathbf{E}_{\phi }}\right] $ (and the universal
trivialization are given by the identity). We have $\mathcal{O}_{\mathbf{E}%
_{\phi }}\left( \mathbf{E}_{\phi }\right) =\mathbb{Z}\left[ \mathbb{S}_{\phi
}\right] $ and, if $\omega _{\mathrm{dR}\text{,}\mathbf{T}_{\mathbb{X}}/%
\mathbf{E}_{\phi }}$ denotes the sheaf of $\mathcal{O}_{\mathbf{E}_{\phi }}$%
-modules given by the invariant differentials, then $\mathcal{O}_{\mathbf{E}%
_{\phi }}\otimes _{\mathbb{Z}}\mathbb{X}\overset{\sim }{\rightarrow }\omega
_{\mathrm{dR}\text{,}\mathbf{T}_{\mathbb{X}}/\mathbf{E}_{\phi }}$ by $%
\mathcal{O}_{\mathbf{E}_{\phi }}$-linear extension of the rule sending $s\in 
\mathbb{X}$ to $\mathrm{dlog}\left( s\right) $ (which is a priori a section
of the pushfoward ot $\Omega _{\mathbf{T}_{\mathbb{X}}/\mathbf{E}_{\phi
}}^{1}$ to $\mathbf{E}_{\phi }$, easily checked to be invariant). Similarly, 
$\mathcal{O}_{\mathbf{E}_{\phi }}\otimes _{\mathbb{Z}}\mathbb{Y}\overset{%
\sim }{\rightarrow }\omega _{\mathrm{dR}\text{,}\mathbf{T}_{\mathbb{Y}}/%
\mathbf{E}_{\phi }}$. Using the universal vectorial extensions one can
define, for every motive $M$, a Kodaira-Spencer map $KS_{M}$ which extends
the definition for abelian schemes (see \cite[Ch. III, discussion after
Proposition 9.2]{FC}). According to \cite[Ch. III, discussion before Lemma
9.3]{FC}, in the specific case of a motive having maximal torus rank such as 
$\mathbf{M}_{\phi }$, this morphism%
\begin{equation*}
KS_{\mathbf{M}_{\phi }}:\omega _{\mathrm{dR}\text{,}\mathbf{T}_{\mathbb{X},%
\mathbf{E}_{\phi }}/\mathbf{E}_{\phi }}\otimes _{\mathcal{O}_{\mathbf{E}%
_{\phi }}}\omega _{\mathrm{dR}\text{,}\mathbf{T}_{\mathbb{Y}}/\mathbf{E}%
_{\phi }}\left( -1\right) \longrightarrow \Omega _{\mathbf{E}_{\phi }}^{1}
\end{equation*}%
admits the following simple description. Identifying the source with $%
\mathcal{O}_{\mathbf{E}_{\phi }}\otimes _{\mathbb{Z}}\left( \mathbb{Y}%
\otimes _{\mathbb{Z}}\mathbb{X}\right) $, it is the $\mathcal{O}_{\mathbf{E}%
_{\phi }}$-linear extension of%
\begin{equation}
\mathbb{Y}\otimes _{\mathbb{Z}}\mathbb{X}\overset{Q_{\phi }}{\longrightarrow 
}\mathbb{S}_{\phi }\subset \mathcal{O}_{\mathbf{E}_{\phi }}^{\times }\overset%
{-\mathrm{dlog}}{\longrightarrow }\Omega _{\mathbf{E}_{\phi }}^{1}\text{.}
\label{q-exp R1 KS F}
\end{equation}
\end{remark}

We remark that the cocharacter group $\mathbb{B}_{\phi }:=Hom_{\mathbb{Z}%
}\left( \mathbb{S}_{\phi },\mathbb{Z}\right) $ of $\mathbf{E}_{\phi }$ is a
free $\mathbb{Z}$-module of rank $d_{g}:=\frac{g\left( g+1\right) }{2}$. We
let $\mathbb{C}_{\phi }\subset Hom_{\mathbb{R}}\left( S_{\phi },\mathbb{R}%
\right) $ be the closed cone consisting of those bilinear forms $Q:\mathbb{Y}%
\times \mathbb{X}\rightarrow \mathbb{R}$ such that $\phi ^{\ast }\left(
Q\right) :=Q\circ \left( 1\mathbb{\otimes }_{\mathbb{Z}}\phi \right) $ has
rational kernel (i.e. admits a basis in $\mathbb{Q\otimes }_{\mathbb{Z}}%
\mathbb{Y}$, equivalently the analogous condition holds for the two kernels
of $Q$) and is positive semidefinite definite. In other words, $\mathbb{C}%
_{\phi }$ is the inverse image of the closed cone of positive semidefinite
definite bilinear forms $\mathbb{C}_{1_{\mathbb{Y}}}\subset Hom_{\mathbb{R}%
}\left( \mathbb{S}_{1_{\mathbb{Y}}},\mathbb{R}\right) $ with rational kernel
under the natural map $\mathbb{B}_{\phi }\rightarrow \mathbb{B}_{1_{\mathbb{Y%
}}}$ sending $b$ to $\phi ^{\ast }\left( b\right) :=b\circ \left( 1\otimes
\phi \right) $. Consider the set $GL_{\phi }$ of couples $\left( \gamma
_{Y},\gamma _{X}\right) \in GL\left( \mathbb{Y}\right) \times GL\left( 
\mathbb{X}\right) $ such that $\gamma _{\mathbb{X}}\circ \phi =\phi \circ
\gamma _{\mathbb{Y}}$. It is easily checked to be a subgroup\footnote{%
Note also that, because $\phi $ is injective (resp. $\phi $ has finite
cokernel), $\gamma _{\mathbb{Y}}$ (resp. $\gamma _{\mathbb{X}}$)\ is
uniquely determined by $\gamma _{\mathbb{X}}$ (resp. $\gamma _{\mathbb{Y}}$)
and, hence, the projection onto the second (resp. first) componet is
injective and identifies $GL\left( \phi \right) $ with a subgroup of those
elements $\gamma $\ of $GL\left( \mathbb{X}\right) $ (resp. $GL\left( 
\mathbb{Y}\right) $) such that $\gamma \left( \phi \left( \mathbb{Y}\right)
\right) \subset \phi \left( \mathbb{Y}\right) $ (resp. such that the scalar
extension $\gamma _{\mathbb{Y},\mathbb{Q}}\in GL\left( \mathbb{Y}_{\mathbb{Q}%
}\right) $ of $\gamma _{\mathbb{Y}}$\ preserves $\mathbb{X}$).}. Using the
integral structure provided by $\mathbb{B}_{\phi }$, one can extend the
notions of smooth $GL\left( \mathbb{Y}\right) $-admissible polyhedral cone
decomposition of $\mathbb{C}_{1_{\mathbb{Y}}}$ and their related notions
(see \cite[Ch. IV, \S 2]{FC}) to smooth $GL_{\phi }$-admissible polyhedral
decompositions\ $\Sigma $\ and similar results hold (see \cite[pag. 99]{FC}
and \cite[\S 6.1]{Lan08})\footnote{%
We remark that the morphism $1\mathbb{\otimes }_{\mathbb{Z}}\phi $ yields a
morphism $\phi _{\ast }:\mathbb{S}_{1_{\mathbb{Y}}}\rightarrow \mathbb{S}%
_{\phi }$ such that $\phi _{\ast }\circ Q_{1_{\mathbb{Y}}}=Q_{\phi }\circ
\left( 1\mathbb{\otimes }_{\mathbb{Z}}\phi \right) =:\phi ^{\ast }\left(
Q_{\phi }\right) $ and, dually, we have $\mathbb{B}_{\phi }\rightarrow 
\mathbb{B}_{1_{\mathbb{Y}}}$ sending $b$ to $\phi ^{\ast }\left( b\right)
:=b\circ \left( 1\otimes \phi \right) $. These relatioships can be used in
order to reduce some of the statement below to the case $\phi =1_{Y}$.}. In
particular, if $\sigma \in \Sigma $ is a cone (required to be open in the
smallest subspace it generates, contained in $\mathbb{C}_{\phi }$ and
smooth), then we get a torus embedding%
\begin{equation*}
\mathbf{E}_{\phi }\hookrightarrow \mathbf{E}_{\phi ,\sigma }\text{,}
\end{equation*}%
where $\mathbf{E}_{\phi ,\sigma }$ is the spectrum of the monoid algebra $%
R_{\phi ,\sigma }^{\prime }:=\mathbb{Z}\left[ \mathbb{S}_{\phi }\cap \sigma
^{\vee }\right] $, if $\sigma ^{\vee }\subset \mathbb{R\otimes }_{\mathbb{Z}}%
\mathbb{S}_{\phi }$ is the set of those $s\in \mathbb{S}_{\phi }$ such that $%
\left\langle s,b\right\rangle \geq 0$ for every $b\in \sigma $. Let $I_{\phi
,\sigma }\subset R_{\phi ,\sigma }^{\prime }$ be the ideal generated by
those $s\in \mathbb{S}_{\phi }$ such that $\left\langle s,b\right\rangle >0$
for every $b\in \sigma $ and define $R_{\phi ,\sigma }$ to be the $I_{\phi
,\sigma }$-adic completion of $R_{\phi ,\sigma }^{\prime }$. Actually, this
give the \textquotedblleft correct definition\textquotedblright\ only when $%
\sigma $ is contained in the interior of $\mathbb{C}_{\phi }$: in this case,
the bilinear form $Q_{\phi }$ satisfies the positivity condition that $\phi
^{\ast }\left( Q_{\phi }\right) \left( y,y\right) \in I_{\phi ,\sigma }$ for
every $y\in \mathbb{Y}-\left\{ 0\right\} $\footnote{%
Indeed, with reference to the previous footnote, we have $\phi ^{\ast
}\left( Q_{\phi }\right) =\phi _{\ast }\circ Q_{1_{\mathbb{Y}}}$, implying
that $\phi ^{\ast }\left( Q_{\phi }\right) \left( y,y\right) =\phi _{\ast
}\left( Q_{1_{\mathbb{Y}}}\left( y,y\right) \right) $ and, hence, $%
\left\langle \phi ^{\ast }\left( Q_{\phi }\right) \left( y,y\right)
,b\right\rangle =\left\langle Q_{1_{\mathbb{Y}}}\left( y,y\right)
,b\right\rangle $. By definition, we have $\mathbb{C}_{\phi }=\left( \phi
^{\ast }\right) ^{-1}\left( \mathbb{C}_{1_{\mathbb{Y}}}\right) $ and the
boundary of $\mathbb{C}_{\phi }$ is defined by the condition that $\phi
^{\ast }\left( b\right) \in \mathbb{C}_{1_{\mathbb{Y}}}$ is not definite.
Because $\sigma $ is contained in the interior of $\mathbb{C}_{\phi }$ we
deduce that, for every $b\in \sigma $, the form $b$ is positive definite,
which means that $\left\langle Q_{1_{\mathbb{Y}}}\left( y,y\right)
,b\right\rangle =b\left( y,y\right) >0$. (Cfr. \cite[pag. 103-104]{FC} or 
\cite[\S 2.1]{Str10}).}. Then, as we are going to explain, Mumford's
construction provides a semiabelian scheme $\mathbf{G}_{\phi ,\sigma
}\rightarrow \mathfrak{Spec}\left( R_{\phi ,\sigma }\right) $ whose formal
completion along (the closed subscheme defined by) $I_{\phi ,\sigma }$ is
canonically isomorphic to the $I_{\phi ,\sigma }$-adic completion of $%
\mathbf{T}_{\mathbb{X},\mathbf{E}_{\phi ,\sigma }}$ and whose generic fiber
is a polarized abelian scheme.

We recall Mumford's construction $\mathrm{M}_{pol}$ in the special case of
motives of maximal torus rank. Suppose that $S$ is an affine normal and
excellent scheme which is complete with respect to a reduced closed
subscheme $S_{0}$ defined by an ideal $I$\ and set $U:=S-S_{0}$: then one
consider the category \textrm{Mum}$_{U,S,S_{0},pol,+}^{tor}$\ of $4$-tuples $%
\left( T,Y,M,\lambda \right) $ such that $T$ is a torus, $Y$ is an \'{e}tale
sheaf which is locally constant with values in a free $\mathbb{Z}$-module of
finite rang, $M=\left[ Y_{U}\rightarrow T_{U}\right] $ is a ($1$-)motive (of
maximal torus rank) and $\lambda :M\rightarrow M^{t}$ is a polarization
satisfying the following positivity condition (see \cite[D\'{e}finition
1.3.3.1]{Str10} and \cite[pag. 103-104]{FC}). Suppose first that $Y$ and $X:=%
\underline{Hom}_{gr}\left( T,\mathbf{G}_{m}\right) $ are constant. Then, as
remarked above, $M=\left[ Y_{U}\rightarrow T_{U}\right] $ is equivalent to $%
Q:Y\otimes _{\mathbb{Z}}X\rightarrow \mathbf{G}_{m,U}$ and $\phi ^{\ast
}\left( Q\right) $ on $U$-points yields $\phi ^{\ast }\left( Q\right)
:Y\otimes _{\mathbb{Z}}Y\rightarrow \mathcal{O}_{U}^{\times }\left( U\right) 
$: we require that $\phi ^{\ast }\left( Q\right) \left( y,y\right) \in I$
for every $y\in Y-\left\{ 0\right\} $. In general, we require this condition 
\'{e}tale locally. On the other hand, one can consider the category \textrm{%
Deg}$_{U,S,S_{0},pol}^{tor}$ consisting of semi-abelian $S$-schemes that are
polarized and abelian over $U$ and whose reduction modulo $I$ is a torus.
Then $\mathrm{M}_{pol}$ is a functor from \textrm{Mum}$%
_{U,S,S_{0},pol,+}^{tor}$ to \textrm{Deg}$_{U,S,S_{0},pol}^{tor}$ realizing
an equivalence (see \cite[Th\'{e}or\`{e}me 1.3.3.3]{Str10} and \cite[Ch.
III, \S 7]{FC}): if $G=\mathrm{M}_{pol}\left( T,Y,M,\lambda \right) $, then
the formal completion along $S_{0}$ of $G\rightarrow S$ is canonically
isomorphic to the $I$-adic completion of $T$. Now take $\left( S,I\right)
=\left( \mathfrak{Spec}\left( R_{\phi ,\sigma }\right) ,R_{\phi ,\sigma
}I_{\phi ,\sigma }\right) $. because $I_{\phi ,\sigma }$ is the ideal which
defines $\mathbf{E}_{\phi ,\sigma }-\mathbf{E}_{\phi }$ (see \cite[pag.
101-102]{FC}), the open immersion $U\hookrightarrow S$ is the pull-back of $%
\mathbf{E}_{\phi }\hookrightarrow \mathbf{E}_{\phi ,\sigma }$ via $%
S\rightarrow \mathbf{E}_{\phi ,\sigma }$\ and, hence, we can pull-back the
polarized motive $\left( \mathbf{M}_{\phi },\lambda _{\phi }\right) $ over $%
\mathbf{E}_{\phi }$\ to $\left( \mathbf{M}_{\phi ,U},\lambda _{\phi
,U}\right) $ over $U$. On the other hand, since $\mathbb{Y}$ and $\mathbb{X}$
are constant, we can extend $\mathbf{T}_{\mathbb{X}}$ and $\mathbb{Y}$\ to $%
S $. Hence, we can define%
\begin{equation*}
\left( \mathbf{G}_{\phi ,\sigma }/\mathfrak{Spec}\left( R_{\phi ,\sigma
}\right) ,\lambda _{\phi ,\sigma }\right) :=\mathrm{M}_{pol}\left( \mathbf{T}%
_{\mathbb{X},S},\mathbb{Y}_{S},\mathbf{M}_{\phi ,U},\lambda _{\phi
,U}\right) \in \mathrm{Deg}_{U,S,S_{0},pol}^{tor}\text{.}
\end{equation*}

\bigskip

We can consider more general level structures as follows. Suppose that $%
f:M\rightarrow M^{\prime }$ is a morphism between motives, viewed as
complexes (of fppf sheaves) concentrated in degrees $\left[ -1,0\right] $:
then \textrm{ker}$\left( f\right) :=H^{-1}\left( \mathrm{Cone}\left(
f\right) \right) $; when $f$ is an isogeny, \textrm{ker}$\left( f\right) $
is a finite flat group scheme which comes equipped with a filtration (see 
\cite[\S 1.2.4]{Str10}). In the special case where $f=n$ and $M=\left[
Y\rightarrow T\right] $ has maximal torus rank $g$, we write $M\left[ n%
\right] :=$\textrm{ker}$\left( f\right) $ and everything reduces to the
connected-\'{e}tale exact sequence%
\begin{equation*}
0\longrightarrow T\left[ n\right] \longrightarrow M\left[ n\right]
\longrightarrow \frac{Y}{nY}\longrightarrow 0\text{.}
\end{equation*}%
Furthermore, if $\left( T,Y,M,\lambda \right) \in $\textrm{Mum}$%
_{U,S,S_{0},pol,+}^{tor}$ and $G:=\mathrm{M}_{pol}\left( T,Y,M,\lambda
\right) $, then $M\left[ n\right] $ is canonically identified with $G_{U}%
\left[ n\right] $ (see \cite[Th\'{e}or\`{e}me 1.3.3.3]{Str10} and \cite[Ch.
III, \S 7]{FC}) and, hence (applying the above discussion to $M=\left[
Y_{U}\rightarrow T_{U}\right] $), we see that there is a canonical exact
sequence%
\begin{equation*}
0\longrightarrow T_{U}\left[ n\right] \longrightarrow G_{U}\left[ n\right]
\longrightarrow \frac{Y}{nY}\longrightarrow 0\text{.}
\end{equation*}%
Also, $M^{t}\left[ n\right] $ is the Cartier dual of $M\left[ n\right] $
and, hence, $G_{U}^{t}\left[ n\right] $ is the Cartier dual of $G_{U}\left[ n%
\right] $. In particular, when $\lambda =\left( \phi _{Y},\phi
_{Y}^{t}\right) $ is such that $\#\frac{X}{\phi \left( Y\right) }$ is prime
to $n$ (where $X:=\underline{Hom}_{gr}\left( T,\mathbf{G}_{m}\right) $),
then $\lambda $ yields an isomorphism $M\left[ n\right] \overset{\sim }{%
\rightarrow }M^{t}\left[ n\right] $.

We apply this discussion in two ways as follows.

\begin{itemize}
\item[$\left( i\right) $] Suppose that $n=N\in \mathbb{N}_{\geq 1}$ is prime
to $\#\frac{\mathbb{X}}{\phi \left( \mathbb{Y}\right) }$ and that $N\cdot \#%
\frac{\mathbb{X}}{\phi \left( \mathbb{Y}\right) }$ is invertible. Then we
can replace the $4$-tuples considered above by $5$-tuples $\left( M,\lambda
,\alpha ,\beta ,\gamma \right) $ in which we have added the datum of a
symplectic isomorphism (of fppf sheaves) $\gamma :\frac{\mathbb{Z}^{2g}}{N%
\mathbb{Z}^{2g}}\overset{\sim }{\rightarrow }M\left[ N\right] $. Then this
moduli problem is relatively representable over the $4$-tuples by a finite 
\'{e}tale morphism, namely the torus $\mathbf{E}_{\phi ,N}:=\underline{Hom}%
_{gr}\left( N^{-1}\mathbb{S}_{\phi },\mathbf{G}_{m}\right) \rightarrow 
\mathbf{E}_{\phi }$. More generally, suppose that $K\subset \mathbf{GSp}%
_{2g}\left( \mathbb{A}_{f}\right) $ is any open compact subgroup as at the
beginning of \S \ref{S Sheaves}, i.e. we have $K_{d,N}\subset K\subset
K_{d,1}$. Then one can consider $K$-level structures and one gets a torus $%
\mathbf{E}_{\phi ,K}$ \'{e}tale over $\mathbf{E}_{\phi }$ with character
group $\mathbb{S}_{\phi ,K}\supset \mathbb{S}_{\phi }$ and $\mathbb{B}_{\phi
,K}:=Hom_{\mathbb{Z}}\left( \mathbb{S}_{\phi ,K},\mathbb{Z}\right) $ can be
used as an integral structure in order to define $GL_{\phi ,K}$-invariant
polyhedral decompositions. More precisely, one should consider conical
complexes obtained by considering several $\gamma $'s\ (cfr. \cite[pag. 129]%
{FC}): however, if one content itself to work with a single geometrically
connected component of the $K$-level Shimura variety (determined by the
choice of a primitive $N$-root of unity), then one can fix a single $\gamma $
and work with smooth $GL_{\phi }$-admissible polyhedral decompositions (cfr. 
\cite[pag. 126]{FC}). We get $\mathbf{M}_{\phi }$, a ring $R_{\phi ,K,\sigma
}$ and $\left( \mathbf{G}_{\phi ,K,\sigma }/\mathfrak{Spec}\left( R_{\phi
,K,\sigma }\right) ,\lambda _{\phi ,K,\sigma }\right) $ and the fact that $M%
\left[ n\right] $ is canonically identified with $G_{U}\left[ n\right] $
under Mumford construction yields a canonical $K$-level structure $\gamma
_{\phi ,K,\sigma ,U}$\ on $\mathbf{G}_{\phi ,K,\sigma ,U}$ where $U$ is
defined similalry as above.

\item[$\left( ii\right) $] Suppose that $n=p^{m}$ is a prime such that $%
p\nmid N\cdot \#\frac{\mathbb{X}}{\phi \left( \mathbb{Y}\right) }$ and that $%
N\cdot \#\frac{\mathbb{X}}{\phi \left( \mathbb{Y}\right) }$ is invertible.
Then taking the inverse limit of the canonical isomorphisms $\frac{\mathbb{X}%
}{p^{m}\mathbb{X}}\simeq \mathbf{G}_{\phi ,K,\sigma ,U}^{t}\left[ p^{m}%
\right] ^{\mathrm{\acute{e}t}}$ yields a canonical isomorphism%
\begin{equation*}
\psi _{\phi ,K,\sigma ,U}:\mathbb{Z}_{p}\otimes _{\mathbb{Z}}\mathbb{X}%
\overset{\sim }{\rightarrow }T\mathbf{G}_{\phi ,K,\sigma ,U}^{t}\left[
p^{\infty }\right] ^{\mathrm{\acute{e}t}}\text{.}
\end{equation*}%
In particular, we see that $x_{\phi ,K,\sigma ,U}:=\left( \mathbf{G}_{\phi
,K,\sigma ,U},\lambda _{\phi ,K,\sigma },\gamma _{\phi ,K,\sigma ,U}\right) $
is ordinary at $p$ (meaning that $\#\mathbf{G}_{\phi ,K,\sigma ,u}\left[ p%
\right] \left( k\left( u\right) \right) \geq g$ for every geometric point $u=%
\mathfrak{Spec}\left( k\left( u\right) \right) $ of $U$)\ and $\left(
x_{\phi ,K,\sigma ,U},\psi _{\phi ,K,\sigma ,U}\right) $ defines a point of
the Igusa tower classifying (over the $K$-level Shimura variety)
isomorphisms between $\mathbb{Z}_{p}\otimes _{\mathbb{Z}}\mathbb{X}$ and $%
TA^{t}\left[ p^{\infty }\right] ^{\mathrm{\acute{e}t}}$. Equivalently, for
every fixed isomorphism $\psi _{o}:\mathbb{Z}_{p}^{g}\simeq \mathbb{Z}%
_{p}\otimes _{\mathbb{Z}}\mathbb{X}$, we can lift $x_{\phi ,K,\sigma ,U}$ to
a point $\left( x_{\phi ,K,\sigma ,U},\psi _{\phi ,K,\sigma ,U}\circ \psi
_{o}\right) $ of the standard Igusa tower.
\end{itemize}

\bigskip

Let now $U\subset S=\mathfrak{Spec}\left( R_{\phi ,K,\sigma }\right) $ be as
in $\left( i\right) $\ above. It follows from Mumford's construction that
there is a canonical morphism from $\mathbf{T}_{\mathbb{X},U}$ to $\mathbf{G}%
_{\phi ,K,\sigma ,U}$ (the \textquotedblleft quotient by $\mathbb{Y}$%
\textquotedblright ) inducing an isomorphism $\omega _{\mathrm{dR}\text{,}%
\mathbf{G}_{\phi ,K,\sigma ,U/U}}\overset{\sim }{\mathbf{\rightarrow }}%
\omega _{\mathrm{dR}\text{,}\mathbf{T}_{\mathbb{X}}/U}$ and similarly for
the dual semi-abelian scheme $\mathbf{G}_{\phi ,K,\sigma }^{t}$ (which is
related to $\mathbf{M}_{\phi ,K}^{t}$). As a consequence of Remark \ref%
{q-exp R1 KS}, we find canonical isomorphisms%
\begin{equation}
\alpha _{\phi ,K,\sigma }^{t}:\mathcal{O}_{U}\otimes _{\mathbb{Z}}\mathbb{Y}%
\overset{\sim }{\rightarrow }\omega _{\mathrm{dR}\text{,}\mathbf{G}_{\phi
,K,\sigma ,U/U}^{t}}\text{ and }\alpha _{\phi ,K,\sigma }:\mathcal{O}%
_{U}\otimes _{\mathbb{Z}}\mathbb{X}\overset{\sim }{\rightarrow }\omega _{%
\mathrm{dR}\text{,}\mathbf{G}_{\phi ,K,\sigma ,U/U}}\text{.}
\label{q-exp F1}
\end{equation}

\begin{proposition}
\label{q-exp P1 KS}Up to the identifications provided by $\left( \text{\ref%
{q-exp F1}}\right) $, we have that $KS_{\mathbf{G}_{\phi ,K,\sigma ,U/U}}$
equals the $\mathcal{O}_{U}$-linear extension (of the analogue) of $\left( 
\text{\ref{q-exp R1 KS F}}\right) $ (with $\mathbb{S}_{\phi }$ replaced by $%
\mathbb{S}_{\phi ,K}$ and $\mathbf{E}_{\phi }$ by $\mathbf{E}_{\phi ,K}$).
\end{proposition}

\begin{proof}
In view of Remark \ref{q-exp R1 KS}, the result follows from \cite[Ch. III,
9.4 Theorem]{FC}.
\end{proof}

\bigskip

We now focus on the Siegel modular variety $X=X_{K}$ which is the
compactification of the moduli space classifying abelian schemes of
dimension $g$, polarizations of type $d$ and $K$-level structure, where $K$
is as at the beginning of \S \ref{S Sheaves}, i.e. we have $K_{d,N}\subset
K\subset K_{d,1}$ and we always work over a ring $\Bbbk $ containing $%
\mathbb{Z}\left[ 1/d_{1}...d_{g}N,\zeta _{N}\right] $, $\zeta _{N}$ being a
primitive $N$-root of unity: we then take $\phi =\Delta :\mathbb{Z}%
^{g}\rightarrow \mathbb{Z}^{g}$ where $\Delta $ is the diagonal matrix with
diagonal given by a vector $d=\left( d_{1},...,d_{g}\right) $ (indeed, by
the Elementary Divisor Theorem, we can always find and isomorphism of $\phi $
with such a $\Delta $).

For each geometrically connected component $\infty $\ of $X$, we get a cusp%
\begin{equation*}
x_{\infty }:X_{\infty }=\mathfrak{Spec}\left( R_{\phi ,K,\sigma }\right)
\longrightarrow X
\end{equation*}%
equipped with $\left( \mathbf{G}_{\phi ,K,\sigma }/\mathfrak{Spec}\left(
R_{\phi ,K,\sigma }\right) ,\lambda _{\phi ,K,\sigma },\gamma _{\phi
,K,\sigma ,U}^{\infty }\right) $ as above and a canonical trivialization $%
\alpha _{\phi ,K,\sigma }$ (see $\left( \text{\ref{q-exp F1}}\right) $). We
remark that $R_{\phi ,K,\sigma }$ and $R_{U}:=\mathcal{O}_{U}\left( U\right) 
$ are both contained in $\Bbbk \left[ \left[ N^{-1}\mathbb{S}_{\phi }\right] %
\right] $, that the morphism $\phi \otimes 1$ sends $N^{-1}\mathbb{S}_{\phi
} $ injectively into $\mathrm{Sym}^{2}\left( \mathbb{X}\right) =\mathrm{Sym}%
_{g,\mathbb{Z}}^{2}$ and that, thanks to $\left( \text{\ref{Representations
F Std}}\right) $, we have $N^{-1}\mathrm{Sym}_{g,\mathbb{Z}}^{2}\overset{%
\sim }{\rightarrow }N^{-1}\mathbf{S}_{g,\mathbb{Z}}^{even}$. It follows that
there are inclusions%
\begin{equation*}
R_{\phi ,K,\sigma }\subset R_{U}\subset \Bbbk \left[ \left[ N^{-1}\mathbb{S}%
_{\phi }\right] \right] \hookrightarrow \Bbbk \left[ \left[ N^{-1}\mathrm{Sym%
}_{g,\mathbb{Z}}^{2}\right] \right] \simeq \Bbbk \left[ \left[ N^{-1}\mathbf{%
S}_{g,\mathbb{Z}}^{even}\right] \right] \subset \Bbbk \left[ \left[ N^{-1}%
\mathbf{S}_{g,\mathbb{Z}}\right] \right] \text{,}
\end{equation*}%
Evaluating $f\in H^{0}\left( Y,\mathcal{W}_{\rho }\right) $ at $\left(
x_{\infty },\alpha _{\phi ,K,\sigma }\right) $ we get a $q$-expansion%
\begin{equation}
f\left( x_{\infty },\alpha _{\phi ,K,\sigma }\right) =\tsum\nolimits_{s\in 
\mathrm{Sym}_{g,\mathbb{Z}}^{2}}a_{f}\left( s\right)
q^{N^{-1}s}=\tsum\nolimits_{\beta \in \mathbf{S}_{g,\mathbb{Z}%
}^{even}}a_{f}\left( \beta \right) q^{N^{-1}\beta }\in R_{U}\otimes _{\Bbbk
}\rho \text{,}  \label{q-exp F Def}
\end{equation}%
where we adopt the convention of writing $q^{N^{-1}s}$ (resp. $%
q^{N^{-1}\beta }$) for the element $N^{-1}s$ (resp. $N^{-1}\beta $) viewed
in the corresponding group ring. One has that $f\in H^{0}\left( X,\mathcal{W}%
_{\rho }\right) $ if and only if the $q$-expansion is supported at
semipositive definite $\beta $'s and the Koecher principle implies that,
when $g\geq 2$, then $H^{0}\left( Y,\mathcal{W}_{\rho }\right) =H^{0}\left(
X,\mathcal{W}_{\rho }\right) $. For every $i\leq j$, let us write $%
\underline{\beta }_{ij}$ for the symmetric matrix whose unique upper
triangular non-zero entry is $1$\ at position $\left( i,j\right) $ and set $%
q_{ij}:=q^{\underline{\beta }_{ij}}$. Then%
\begin{equation*}
q^{N^{-1}\beta }=\tprod\nolimits_{1\leq i\leq j\leq g}q_{ij}^{N^{-1}\beta
_{ij}}
\end{equation*}%
if $\beta \in \mathbf{S}_{g,\mathbb{Z}}$ has $\underline{\beta }_{ij}$%
-component $\beta _{ij}\in \mathbb{Z}$. Recall the Kodaira-Spencer
isomorphism $KS_{\alpha _{\phi ,K,\sigma }}\ $of $\left( \text{\ref{Sheaves
F KS}}\right) $.

\begin{corollary}
\label{q-exp C1 KS}We have $KS_{\alpha _{\phi ,K,\sigma }}\left(
e_{i}e_{j}\right) =\left( 1+\delta _{ij}\right) \frac{dq_{ij}}{q_{ij}}$.
\end{corollary}

\begin{proof}
Note that $\left( \text{\ref{Representations F Std}}\right) $ sends $%
e_{i}e_{j}\in \mathrm{Sym}_{g,\mathbb{Z}}^{2}$ to $\underline{\beta }_{ij}$
when $i\neq j$ and $2\underline{\beta }_{ij}$ when $i=j$. Hence, the result
is just a consequence of Proposition \ref{q-exp P1 KS} and the definition of 
$KS_{\alpha _{0}}:=KS_{X}\circ \vartheta _{\alpha _{0}}^{-1}$.
\end{proof}

\section{\label{S Depletion and integration}$p$-depletions and $p$-adic
integration over the Igusa tower}

In this section, we let $X=X_{K}$ be the Siegel modular variety which is the
compactification of the moduli space $Y=Y_{K}$\ classifying abelian schemes
of dimension $g$, polarizations of type $d$ and $K$-level structure, where $%
K $ is as at the beginning of \S \ref{S Sheaves}, i.e. we have $%
K_{d,N}\subset K\subset K_{d,1}$. We work over the ring of integers $\Bbbk $
of a finite extension of $\mathbb{Q}_{p}$ such that $Y$ is defined over it
(we can always take $\Bbbk =\mathcal{O}_{\mathbb{Q}_{p}\left( \zeta
_{N}\right) }$ for a primitive $N$-root of unity $\zeta _{N}$), where $%
p\nmid d_{1}...d_{g}N $. We let $X^{\mathrm{ord}}\subset X$ (resp. $Y^{%
\mathrm{ord}}\subset Y$) be the open subscheme obtained as the complement of
the non-ordinary locus $X_{\Bbbk _{p}}$ (resp. $Y_{\Bbbk _{p}}$) for the
residue field $\Bbbk _{p}$ of $\Bbbk $, so that $X^{\mathrm{ord}}=X\left[ 1/H%
\right] $ (resp. $Y^{\mathrm{ord}}=Y\left[ 1/H\right] $) for a suitable $p$%
-power $H$ of the Hasse invariant (at $p$). Write $\mathfrak{X}^{\mathrm{ord}%
}\subset \mathfrak{X}$ (resp. $\mathfrak{Y}^{\mathrm{ord}}\subset \mathfrak{Y%
}$) for the open immersion obtained by considering the formal completion of $%
X^{\mathrm{ord}}\subset X$ (resp. $Y^{\mathrm{ord}}\subset Y$). Over the
formal scheme $\mathfrak{Z}^{\mathrm{ord}}$ for $\mathfrak{Z}\in \left\{ 
\mathfrak{X,Y}\right\} $ there is a Igusa tower $\mathfrak{I}_{\mathfrak{Z}%
}\rightarrow \mathfrak{Z}^{\mathrm{ord}}$ classifying isomorphisms of pro%
\'{e}tale sheaves $\mathbb{Z}_{p}^{g}\overset{\sim }{\rightarrow }TB^{t}%
\left[ p^{\infty }\right] ^{\mathrm{\acute{e}t}}$\ over $\mathfrak{Z}^{%
\mathrm{ord}}$ where $B=G$ (resp. $B=A$) for $\mathfrak{Z}=\mathfrak{X}$
(resp. $\mathfrak{Z}=\mathfrak{Y}$). The Hodge-Tate map yields a canonical
isomorphism of pro\'{e}tale sheaves $\mathcal{O}_{\mathfrak{Z}^{\mathrm{ord}%
}}\otimes _{\mathbb{Z}}TB^{t}\left[ p^{\infty }\right] ^{\mathrm{\acute{e}t}}%
\overset{\sim }{\rightarrow }\omega _{\mathfrak{Z}^{\mathrm{ord}},\mathrm{dR}%
}$ (cfr. \cite[\S 3.3]{Kz78}), where $\omega _{\mathfrak{Z}^{\mathrm{ord}},%
\mathrm{dR}}=\omega _{\mathfrak{B}/\mathfrak{Z}^{\mathrm{ord}},\mathrm{dR}}$
denotes the relative dual Lie algebra and $\mathfrak{B}/\mathfrak{Z}^{%
\mathrm{ord}}$ is the $p$-adic completion of $B_{Z^{\mathrm{ord}}}/Z^{%
\mathrm{ord}}=G_{X^{\mathrm{ord}}}/X^{\mathrm{ord}}$ (resp. $B_{Z^{\mathrm{%
ord}}}/Z^{\mathrm{ord}}=A_{Y^{\mathrm{ord}}}/Y^{\mathrm{ord}}$) for $%
\mathfrak{Z}=\mathfrak{X}$ (resp. $\mathfrak{Z}=\mathfrak{Y}$). Hence, we
get a canonical trivialization%
\begin{equation*}
\alpha _{\mathfrak{I}_{\mathfrak{Z}}}:\mathcal{O}_{\mathfrak{I}_{\mathfrak{Z}%
}}^{g}\overset{\sim }{\rightarrow }\omega _{\mathfrak{I}_{\mathfrak{Z}},%
\mathrm{dR}}\text{.}
\end{equation*}%
Let $X_{\infty }\rightarrow X$ be a cusp and let $\mathfrak{X}_{\infty
}\rightarrow \mathfrak{X}$ be the $p$-adic formal completion (hence $%
X_{\infty }$ and $\mathfrak{X}_{\infty }$ have the same global sections, but
they are different as topological spaces). As explained in \S \ref{S q-exp} $%
\left( ii\right) $, the morphism $\mathfrak{X}_{\infty }\rightarrow 
\mathfrak{X}$ canonically lift to $\mathfrak{X}_{\infty }\rightarrow 
\mathfrak{I}_{\mathfrak{X}}=:\mathfrak{I}$. Thanks to the irreduciblity of
the Igusa tower, the $q$-expansion principle also holds for the sections of
the Igusa tower, up to considering a cusp $x_{\infty }:\mathfrak{X}_{\infty
}\rightarrow \mathfrak{I}$\ for each geometrically connected component of $%
\mathfrak{X}^{\mathrm{ord}}$ (see \cite[Corollary 8.17]{Hi04}): it implies
that, as it is for $\mathfrak{X}^{\mathrm{ord}}$,\ a section of a sheaf that
is a direct sum of locally free finite $\mathcal{O}_{X_{0}}$-modules is
uniquely determined by its pull-backs to the cusps $x_{\infty }$.

\begin{notation}
We will write $X_{0}\rightarrow X$ for one between the morphisms (of locally
ringed spaces) between $\mathfrak{X}^{\mathrm{ord}}\rightarrow X$, $%
\mathfrak{I}\rightarrow X$, $X_{\infty }\rightarrow X$ or $\mathfrak{X}%
_{\infty }\rightarrow X$ and, for $X_{0}\in \left\{ \mathfrak{I},X_{\infty },%
\mathfrak{X}_{\infty }\right\} $, we let $\alpha _{0}:\mathcal{O}_{X_{0}}^{g}%
\overset{\sim }{\rightarrow }\omega _{X_{0},\mathrm{dR}}$ be the
corresponding trivialization, namely $\alpha _{0}=\alpha _{\mathfrak{I}}$
over $\mathfrak{I}$, $\alpha _{0}=\alpha _{\infty }:=\alpha _{\phi ,K,\sigma
}$ over $X_{\infty }$\ (see $\left( \text{\ref{q-exp F1}}\right) $) or the
trivialization induced by $\alpha _{\infty }$ over $\mathfrak{X}_{\infty }$,
denoted again $\alpha _{\infty }$: there will be no fear of confusion
because these trivializations are compatible under the pull-back with
respect to the natural morphisms between the various $X_{0}$'s described
above. If $\mathcal{F}$ is a sheaf on $X$, we will simply write $H^{0}\left(
X_{0},\mathcal{F}\right) =\mathcal{F}\left( X_{0}\right) $ for the global
section of the pull-back of $\mathcal{F}$ to $X_{0}$. Finally, we will
explicit mention when a result is specific for a subset of $\left\{ 
\mathfrak{X}^{\mathrm{ord}},\mathfrak{I},X_{\infty },\mathfrak{X}_{\infty
}\right\} $.
\end{notation}

We will work with the (partially) compactified spaces, but everything we
will prove below also holds for the respective non-compactified spaces.
However, we will use the abbreviation $\Omega _{X_{0}/\Bbbk }^{p}$ for $%
\Omega _{X_{0}/\Bbbk }^{\mathrm{log},p}$ and similarly for the derivations.

\begin{remark}
For psicological reasons, we will work with $\mathbf{G}=\mathbf{GSp}_{2g}$.
For a more general $\mathbf{G}=\mathbf{GSp}\left( V,\psi \right) $, suffices
to replace everywhere the standard symplectic basis with its image in $V$
under the isomorphism $\theta $ of Example \ref{Representations E GSp} and
then to replace $\partial ,\partial _{i}\in \mathfrak{gsp}_{2g}$ with their
image in $\mathfrak{g}$ under the isomorphism $c_{\theta }:\mathfrak{gsp}%
_{2g}\overset{\sim }{\rightarrow }\mathfrak{g}$ (cfr. the definition of $%
X\left( D,\alpha \right) $ before Lemma \ref{Sheaves L T1}).
\end{remark}

\subsection{\label{S Igusa de Rham}Description of the de Rham complex}

For every representation $\rho \in \mathrm{Rep}\left( \mathbf{GL}_{g}\times 
\mathbf{G}_{m}\right) $, we have (see Remark \ref{Sheaves R Triv'})%
\begin{equation}
\vartheta _{\alpha _{0}}:\mathcal{W}_{\rho }\left( X_{0}\right) \overset{%
\sim }{\longrightarrow }\mathcal{O}_{X_{0}}\left( X_{0}\right) \otimes
_{\Bbbk }\rho =:R_{0}\otimes _{\Bbbk }\rho \text{.}  \label{de Rham F Triv1}
\end{equation}%
We view the objects of $\mathrm{Rep}\left( \mathbf{GL}_{g}\right) $ in $%
\mathrm{Rep}\left( \mathbf{GL}_{g}\times \mathbf{G}_{m}\right) $ by
restriction with respect to the projection from $\mathbf{GL}_{g}\times 
\mathbf{G}_{m}$ to $\mathbf{GL}_{g}$. Let us write $e_{i}$ for the standard
basis vector in $\mathrm{Std}_{g}$ whose only non zero component is $1$ in
the $i$-position, where $\mathrm{Std}_{g}$ is the standard representation of 
$\mathbf{GL}_{g}$ given by column vectors on which $\mathbf{GL}_{g}$ acts
from the left, and define $\delta _{i}:=\alpha _{0}\left( e_{i}\right) \in
\omega _{X_{0},\mathrm{dR}}$. Fix a bijection $\left\{ \left( i,j\right)
:1\leq i\leq j\leq g\right\} \simeq \left\{ 1,...,d_{g}\right\} $ such that $%
\left( i,i\right) $ corresponds to $i$. The Kodaira-Spencer isomorphism
allows us to define, for every $X_{0}\in \left\{ \mathfrak{I},X_{\infty },%
\mathfrak{X}_{\infty }\right\} $, the canonical differentials%
\begin{equation}
\omega _{i}:=KS\left( \delta _{i}^{2}\right) \in \Omega _{X_{0}/\Bbbk }^{1}%
\text{ for }i=1,...,g\text{ and }\omega _{i}:=KS\left( \delta _{k}\delta
_{l}\right) \in \Omega _{X_{0}/\Bbbk }^{1}\text{ for }g<i\leq d_{g}\text{,}
\label{de Rham F TrivKS}
\end{equation}%
where $\left( k,l\right) $ corresponds to $i$ under our fixed bijection.

Recall that we view $q$-expansions as taking values in $\Bbbk \left[ \left[
N^{-1}\mathbf{S}_{g,\mathbb{Z}}\right] \right] $ and that we write $%
q^{N^{-1}\beta }$ for $N^{-1}\beta \in N^{-1}\mathbf{S}_{g,\mathbb{Z}}$
viewed in this ring. Also, if $i\leq j$, we defined $\underline{\beta }_{ij}$
to be the symmetric matrix whose unique upper triangular non-zero entry is $%
1 $\ at position $\left( i,j\right) $ and, if $\beta \in \mathbf{S}_{g,%
\mathbb{Z}}$, we defined $\beta _{ij}\in \mathbb{Z}$ as being the $%
\underline{\beta }_{ij}$-component of $\beta $. Similarly as above, it will
be convenient to set $\beta _{i}:=\beta _{ii}$ (resp. $q_{i}:=q_{ij}$) for $%
i=1,...,g$ and $\beta _{i}:=\beta _{kl}$ (resp. $q_{i}:=q_{kl}$)\ for $i>g$
where $\left( k,l\right) $ corresponds to $i$.

\begin{remark}
\label{q-exp R KS}We can reformulate Corollary \ref{q-exp C1 KS} by saying
that, if $X_{0}\in \left\{ X_{\infty },\mathfrak{X}_{\infty }\right\} $,
then $\omega _{i}=2\frac{dq_{i}}{q_{i}}$ for $i=1,...,g$ and $\omega _{i}=%
\frac{dq_{i}}{q_{i}}$ for $i>g$. Since%
\begin{equation*}
\nabla ^{1}\left( \frac{dq_{i}}{q_{i}}\right) =d\left( \frac{1}{q_{i}}%
\right) \wedge dq_{i}=-\frac{1}{q_{i}^{2}}dq_{i}\wedge dq_{i}=0\text{,}
\end{equation*}%
we have $d^{1}\left( \omega _{i}\right) =0$.
\end{remark}

The following result is a restatement of Remark \ref{q-exp R KS}.

\begin{lemma}
\label{q-exp L KS}Suppose that $X_{0}\in \left\{ X_{\infty },\mathfrak{X}%
_{\infty }\right\} $ and consider the $\Bbbk $-linear derivations $\theta
_{i}:R_{0}\rightarrow R_{0}$ defined by the rule $\theta _{i}\left( q^{\beta
}\right) =\frac{1}{2}\beta _{i}q^{\beta }$ for $i=1,...,g$ and $\theta
_{i}\left( q^{\beta }\right) =\beta _{i}q^{\beta }$ otherwise for $i>g$.
Then $\left\{ \theta _{i}:i=1,...,d_{g}\right\} $ is the dual basis of $%
\left\{ \omega _{i}:i=1,...,d_{g}\right\} $.
\end{lemma}

\bigskip

Of course, we remark that the definition of $\theta _{i}$ (as the dual of $%
\omega _{i}$) makes sense over the Igusa tower and defines%
\begin{equation*}
\theta _{i}:H^{0}\left( \mathfrak{I},\mathcal{O}_{\mathfrak{I}}\right)
\longrightarrow H^{0}\left( \mathfrak{I},\mathcal{O}_{\mathfrak{I}}\right) 
\text{.}
\end{equation*}%
Choose any polynomial%
\begin{equation*}
P\in \mathbb{Z}_{p}\left[ T_{i}:i=1,...,d_{g}\right]
\end{equation*}%
and, setting $\mathbf{\theta }=\left( \theta _{1},...,\theta _{d_{g}}\right) 
$, consider the differential operator%
\begin{equation*}
\theta _{P}:=P\left( \mathbf{\theta }\right) :H^{0}\left( \mathfrak{I},%
\mathcal{O}_{\mathfrak{I}}\right) \longrightarrow H^{0}\left( \mathfrak{I},%
\mathcal{O}_{\mathfrak{I}}\right) \text{.}
\end{equation*}%
According to Lemma \ref{q-exp L KS}, we have%
\begin{equation*}
\theta _{P}\left( f\right) \left( q\right) =\tsum\nolimits_{\beta \in 
\mathbf{S}_{g,\mathbb{Z}}^{even}}a_{f}\left( \beta \right) P^{\prime }\left(
\beta \right) q^{\beta }
\end{equation*}%
where $P^{\prime }\left( T_{1},...,T_{g},T_{g+1},...\right) =P\left(
2^{-1}T_{1},...,2^{-1}T_{g},T_{g+1},...\right) $ and $f\left( q\right) $ is
as in $\left( \text{\ref{q-exp F Def}}\right) $. In particular, choosing any
sequence $\left\{ s_{n}\right\} _{n\in \mathbb{N}}\subset \mathbb{N}$ such
that $s_{n}\rightarrow +\infty $ in the archimedean topology but converging
to zero $p$-adically and considering the product of the $p$-adic topologies
on the coefficients of the $q$-expansions, we see that the rule%
\begin{equation*}
e^{\left[ P\right] }\left( f\left( q\right) \right) :=\lim_{n\rightarrow
+\infty }\left( \theta _{P}^{s_{n}}\left( f\right) \left( q\right) \right)
=\tsum\nolimits_{\beta :p\nmid P^{\prime }\left( \beta \right) }a_{f}\left(
\beta \right) q^{\beta }
\end{equation*}%
defines a $q$-expansion which does not depend on the choice of the sequence $%
\left\{ s_{n}\right\} _{n\in \mathbb{N}}$. Indeed, it follows from the $q$%
-expansion principle that we can define an idempotent%
\begin{equation*}
e^{\left[ P\right] }:H^{0}\left( \mathfrak{I},\mathcal{O}_{\mathfrak{I}%
}\right) \longrightarrow H^{0}\left( \mathfrak{I},\mathcal{O}_{\mathfrak{I}%
}\right) \text{ via }e^{\left[ P\right] }\left( f\left( q\right) \right)
:=\lim_{n\rightarrow +\infty }\theta _{P}^{s_{n}}\left( f\right)
\end{equation*}%
whose description in terms of $q$-expansion is as above: $e^{\left[ P\right]
}\left( f\right) \left( q\right) =e^{\left[ P\right] }\left( f\left(
q\right) \right) $. We set $H^{0}\left( X_{0},\mathcal{O}_{X_{0}}\right) ^{%
\left[ P\right] }:=e^{\left[ P\right] }H^{0}\left( X_{0},\mathcal{O}%
_{X_{0}}\right) $ and, more generally, if $\rho $ is a representation of $%
\mathbf{GL}_{g}\times \mathbf{G}_{m}$,%
\begin{equation*}
H^{0}\left( X_{0},\mathcal{W}_{\rho }\right) ^{\left[ P\right] }:=e^{\left[ P%
\right] }H^{0}\left( X_{0},\mathcal{O}_{X_{0}}\right) \otimes _{\Bbbk }\rho
=:e^{\left[ P\right] }\left( R_{0}\right) \otimes _{\Bbbk }\rho \text{.}
\end{equation*}%
(where we use the canonical trivialization $\left( \text{\ref{de Rham F
Triv1}}\right) $). As a consequence of the formula $e^{\left[ P\right]
}\left( f\right) \left( q\right) =e^{\left[ P\right] }\left( f\left(
q\right) \right) $, we deduce the following fact.

\begin{lemma}
\label{de Rham L0}If $X_{0}\in \left\{ \mathfrak{I},X_{\infty },\mathfrak{X}%
_{\infty }\right\} $, the operator $\theta _{P}$ restricts to an invertible
operator on $H^{0}\left( X_{0},\mathcal{W}_{\rho }\right) ^{\left[ P\right]
} $.
\end{lemma}

\bigskip

We assume, until the end of \S \ref{S Igusa de Rham}, that $X_{0}\in \left\{ 
\mathfrak{I},X_{\infty },\mathfrak{X}_{\infty }\right\} $. The Gauss-Manin
connection $\nabla $ allows us to define%
\begin{equation*}
\eta _{i}:=\nabla \left( \theta _{i}\right) \left( \delta _{i}\right) \in
H^{0}\left( X_{0},\mathcal{H}_{\mathrm{dR},X_{0}}^{1}\right) \text{ for }%
i=1,...,g
\end{equation*}%
and, indeed, in this way the trivialization $\alpha _{0}:\mathcal{O}%
_{X_{0}}^{g}\overset{\sim }{\rightarrow }\omega _{X_{0}}$ can be extended to
a trivialization $\alpha _{\mathrm{dR},0}:\mathcal{O}_{X_{0}}^{2g}\overset{%
\sim }{\rightarrow }\mathcal{H}_{\mathrm{dR},X_{0}}^{1}$ in $\mathcal{T}_{%
\mathrm{dR},X_{0}}^{\times }$ (and, indeed, in the $\mathbf{Q}^{\circ }$%
-torsor $\mathcal{T}_{\mathcal{H},\mathrm{dR}}^{\circ \times }\subset 
\mathcal{T}_{\mathrm{dR},X_{0}}^{\times }$ of Remark \ref{Sheaves R Sim},\
see the proof of Lemma \ref{de Rham L1} below).

\begin{remark}
\label{de Rham R DefEta}Another equivalent definition $\eta _{i}^{\prime }$\
of $\eta _{i}$ making evident that the basis $\left\{ \delta _{i},\eta
_{j}^{\prime }\right\} _{i,j=1}^{g}$ yields an isomorphism $\alpha _{\mathrm{%
dR},0}:\mathcal{O}_{X_{0}}^{2g}\overset{\sim }{\rightarrow }\mathcal{H}_{%
\mathrm{dR},X_{0}}^{1}$ in $\mathcal{T}_{\mathcal{H},\mathrm{dR}}^{\circ
\times }$, not relying on the proof of Lemma \ref{de Rham L1} (hence, on
Fonseca's calculation \cite[Theorem 6.4]{Fo23}, as explained below), is
obtained as follows. According to \cite[Proposition 1.9 $\left( 2\right) $]%
{Fo23}, to give an ordered symplectic-Hodge basis of $\mathcal{H}_{\mathrm{dR%
},X_{0}}^{1}$ is the same thing as to give an ordered basis of $\omega _{%
\mathrm{dR},X_{0}}$ and a maximal isotropic $\mathcal{O}_{X_{0}}^{2g}$%
-subspace $\overline{\omega }_{\mathrm{dR},X_{0}}\subset \mathcal{H}_{%
\mathrm{dR},X_{0}}^{1}$ such that $\mathcal{H}_{\mathrm{dR}%
,X_{0}}^{1}=\omega _{\mathrm{dR},X_{0}}\oplus \overline{\omega }_{\mathrm{dR}%
,X_{0}}$: every ordered basis $\left\{ w_{i}\right\} _{i=1}^{g}$\ of $\omega
_{\mathrm{dR},X_{0}}$ uniquely extend to an ordered symplectic-Hodge basis $%
\left\{ w_{i},\overline{w}_{i}\right\} _{i=1}^{g}$\ of $\mathcal{H}_{\mathrm{%
dR},X_{0}}^{1}$ with the property that $\overline{w}_{i}$ is the dual of $%
w_{i}$ and $\left\{ \overline{w}_{i}\right\} _{i=1}^{g}\subset \overline{%
\omega }_{\mathrm{dR},X_{0}}$. Thanks to the unit root splitting $\left( 
\text{\ref{Primitives F URSpl}}\right) $\ recalled before Lemma \ref%
{Primitives L URSpl} below, we can take $\overline{\omega }_{\mathrm{dR}%
,X_{0}}=\mathcal{H}_{\mathrm{dR},X_{0}}^{1,\varphi =1}$: indeed, it is
isotropic because the symplectic pairing $\left\langle -,-\right\rangle $ on 
$\mathcal{H}_{\mathrm{dR},X_{0}}^{1}$ satisfies the equality $\left\langle
\varphi \left( x\right) ,\varphi \left( y\right) \right\rangle
=p\left\langle x,y\right\rangle $ (as it follows from the fact that it
arises by duality from the symplectic pairing of $\mathcal{H}_{\mathrm{dR}%
}^{1,\vee }$\ obtained by evaluation at the first Chern class $c_{1,\mathrm{%
dR}}\left( \mathcal{L}\right) \in \mathcal{H}_{\mathrm{dR}}^{2,\varphi =p}$
of the relatively ample line bundle $\mathcal{L}$ giving rise to the
polarization); hence, it is maximal isotropic in view of the unit root
splitting. It then follows that we can uniquely extend $\left\{ \delta
_{i}\right\} _{i=1}^{g}$ to an ordered symplectic-Hodge basis $\left\{
\delta _{i},\eta _{j}^{\prime }\right\} _{i,j=1}^{g}$ of $\mathcal{H}_{%
\mathrm{dR},X_{0}}^{1}$ with the property that $\eta _{i}^{\prime }$ is the
dual of $\delta _{i}$ and $\left\{ \eta _{i}^{\prime }\right\}
_{i=1}^{g}\subset \mathcal{H}_{\mathrm{dR},X_{0}}^{1,\varphi =1}$. From this
point of view, the fact that $\left\{ \delta _{i},\eta _{j}\right\}
_{i,j=1}^{g}$ is another symplectic-Hodge basis extending $\left\{ \delta
_{i}\right\} _{i=1}^{g}$\ such that $\eta _{i}$ is the dual of $\delta _{i}$%
\ (by Fonseca's calculation) and $\left\{ \eta _{i}\right\}
_{i=1}^{g}\subset \mathcal{H}_{\mathrm{dR},X_{0}}^{1,\varphi =1}$ (as proved
in Lemma \ref{Primitives L URSpl} below) implies that $\left\{ \eta
_{j}^{\prime }\right\} =\left\{ \eta _{j}\right\} $.
\end{remark}

For every $\rho \in \mathrm{Rep}_{fl\text{-}alg}\left( \mathfrak{g},\mathbf{Q%
}\right) $, consider the $\nabla _{\rho }$-operator%
\begin{equation*}
\nabla _{\rho }:H^{0}\left( X_{0},\mathcal{E}_{\rho }\right) \longrightarrow
H^{0}\left( X_{0},\mathcal{E}_{\rho }\otimes _{\mathcal{O}_{X}}\Omega
_{X/\Bbbk }^{1}\right) \text{.}
\end{equation*}%
The choice of the trivialization $\alpha _{\mathrm{dR},0}$ induces%
\begin{equation}
\vartheta _{\alpha _{\mathrm{dR},0}}:\mathcal{E}_{\rho }\left( X_{0}\right) 
\overset{\sim }{\longrightarrow }\mathcal{O}_{X}\left( X_{0}\right) \otimes
_{\Bbbk }\rho =R_{0}\otimes _{\Bbbk }\rho \text{.}  \label{de Rham F Triv2}
\end{equation}%
By definition of $\nabla _{\rho }\left( \theta _{i}\right) $ and the fact
that $\theta _{i}$ is the dual of $\omega _{i}$, we see that $f\otimes
_{\Bbbk }v\in R_{0}\otimes _{\Bbbk }\rho $ is sent to%
\begin{equation*}
\nabla _{\rho }\left( f\otimes _{\Bbbk }v\right)
=\tsum\nolimits_{i=1}^{d_{g}}\nabla _{\rho }\left( \theta _{i}\right) \left(
f\otimes _{\Bbbk }v\right) \omega _{i}=\tsum\nolimits_{i=1}^{d_{g}}\left(
\theta _{i}\left( f\right) \otimes _{\Bbbk }v+X\left( \theta _{i},\alpha
_{0}\right) \left( f\otimes _{\Bbbk }v\right) \right) \otimes _{R_{0}}\omega
_{i}\text{.}
\end{equation*}%
Consider the elements%
\begin{equation*}
\partial _{ij}\in \mathfrak{u}^{-}\simeq \mathrm{Sym}^{2}\left( \mathrm{Std}%
_{g}\right) ^{\vee }
\end{equation*}%
corresponding to $e_{i}^{t}e_{j}^{t}$ under the isomorphism of Lemma \ref%
{Representations L2}: hence, in $\mathfrak{u}^{-}$ the matrix $\partial
_{ij} $ is characterized by the fact that its lower left $g$-by-$g$ entry is
the symmetric matrix whose upper triangular part has all zero except a $1$
in $\left( i,j\right) $-entry. As usual, set $\partial _{i}:=\partial _{ii}$
for $i=1,...,g$ and $\partial _{i}:=\partial _{kl}$ for $i>g$, where $\left(
k,l\right) $ corresponds to $i$ under our fixed bijection.

\begin{lemma}
\label{de Rham L1}Suppose that $X_{0}\in \left\{ X_{\infty },\mathfrak{X}%
_{\infty }\right\} $. For every%
\begin{equation*}
\partial \in \mathfrak{u}^{-}\simeq \mathrm{Sym}^{2}\left( \mathrm{Std}%
_{g}\right) ^{\vee }\left( 1\right) \overset{\left( \text{\ref{Sheaves F KS
dual}}\right) }{\hookrightarrow }\mathcal{D}er_{\Bbbk }\left( \mathcal{O}%
_{X_{0}}\right)
\end{equation*}%
with image $\theta _{\partial }\in \mathcal{D}er_{\Bbbk }\left( \mathcal{O}%
_{X_{0}}\right) $, we have $X\left( \theta _{\partial },\alpha _{\infty
}\right) =1\otimes _{\Bbbk }\partial $ on $R_{0}\otimes _{\Bbbk }\rho $ and,
furthermore, $X\left( \theta _{i},\alpha _{\infty }\right) =1\otimes _{\Bbbk
}\partial _{i}$.
\end{lemma}

\begin{proof}
The claim is reduced to verification that $\nabla \eta _{i}=0$, $\nabla
\left( \theta _{ij}\right) \left( \delta _{i}\right) =\eta _{j}$, $\nabla
\left( \theta _{ij}\right) \left( \delta _{j}\right) =\eta _{i}$ and $\nabla
\left( \theta _{ij}\right) \left( \delta _{k}\right) =0$ for every $i,j,k\in
\left\{ 1,...,g\right\} $ with $k\notin \left\{ i,j\right\} $, where $\theta
_{ij}:R_{0}\rightarrow R_{0}$ is defined by the rule $\theta _{ii}\left(
q^{\beta }\right) =\frac{1}{2}\beta _{ii}q^{\beta }$ and $\theta _{kl}\left(
q^{\beta }\right) =\beta _{kl}q^{\beta }$ if $k\neq l$ (so that $\theta
_{i}=\theta _{ii}$ for $i=1,...,g$ and $\theta _{i}:=\theta _{kl}$ for $i>g$
where $\left( k,l\right) $ corresponds to $i$)\footnote{%
Indeed, we remark that the relations $\nabla \left( \theta _{ij}\right)
\left( \delta _{i}\right) =\eta _{j}$, $\nabla \left( \theta _{ij}\right)
\left( \delta _{j}\right) =\eta _{i}$ and $\nabla \left( \theta _{ij}\right)
\left( \delta _{k}\right) =0$ for every $i,j,k\in \left\{ 1,...,g\right\} $
with $k\notin \left\{ i,j\right\} $ also follows from Lemma \ref{q-exp L KS}%
, in view of the definition of the Kodaira-Spencer morphism recalled in \S %
\ref{S Sheaves}.}. For principally polarized abelian schemes, this is the
content of \cite[Theorem 6.4]{Fo23}, taking into account \cite[Proposition
5.17 and Remark 5.18]{Fo23}, which also proves the fact that setting $\eta
_{i}:=\nabla \left( \theta _{i}\right) \left( \delta _{i}\right) $ extends
the trivialization $\mathcal{O}_{X_{0}}^{g}\overset{\sim }{\rightarrow }%
\omega _{X_{0}}$ to $\alpha _{\mathrm{dR},0}:\mathcal{O}_{X_{0}}^{2g}\overset%
{\sim }{\rightarrow }\mathcal{H}_{\mathrm{dR},X_{0}}^{1}$ in $\mathcal{T}_{%
\mathcal{H},\mathrm{dR}}^{\circ \times }$, as claimed above. Let us explain
how these claims extend to the more general PEL setting. As explained in 
\cite[\S 11.4]{Fo23}, these results follow from their complex analytic
counterparts, namely \cite[Theorem 11.2]{Fo23} and \cite[Proposition 11.7]%
{Fo23} and the same reduction holds in the PEL setting. Suppose we are
considering abelian schemes with polarization of type $d=\left(
d_{1},...,d_{g}\right) $ and let $\Delta $ be the diagonal matrix with
diagonal given by the vector $d$. Then in \cite[Example 9.8, Example 10.4
and Proposition 10.5]{Fo23} we have simply to replace the lattice $\mathbb{Z}%
+\tau \mathbb{Z}\subset \mathbb{C}^{g}$ with the lattice $\mathbb{\Delta Z}%
^{g}+\tau \mathbb{Z}^{g}\subset \mathbb{C}^{g}$, the Hermitian metric is
again given by the matrix \textrm{Im}$\left( \tau \right) ^{-1}$ and a $d$%
-integral symplectic basis of $R^{1}p_{g\ast }\left( \mathbb{Z}\right) $ is
given by the sections $\gamma _{i}\left( \tau \right) :=d_{i}e_{i}$ (rather
than $\gamma _{i}\left( \tau \right) :=e_{i}$) and $\delta _{i}\left( \tau
\right) :=\tau e_{i}$, where we define a $d$-integral symplectic basis $%
\left\{ \gamma _{1},...,\gamma _{g},\delta _{1},...,\delta _{g}\right\} $ by
the condition that $\left\langle \gamma _{i},\gamma _{j}\right\rangle
=\left\langle \delta _{i},\delta _{j}\right\rangle $ and $\left\langle
\gamma _{i},\delta _{j}\right\rangle =\Delta _{ij}$ for $i,j\in \left\{
1,...,g\right\} $ (see \cite[Ch. VI, proof of Th\'{e}or\`{e}me\ 1.3 and Ch.
VII, \S 1]{Deb}; here $\delta _{i}\left( \tau \right) $ has nothing to do
with the differentials $\delta _{i}$ considered above). Then all the results
of loc.cit. hold unchanged: the only point that involves a priori a
modification is the calculation of $\tint\nolimits_{\gamma _{i}}\eta
_{k}^{ij}$ in the proof of \cite[Proposition 11.3]{Fo23}, which is however
again zero. (Let us also remark that, since we assume that $d$ is inverted,
the notion of symplectic basis in de Rham cohomology holds unchanged as in 
\cite[Definition 2.5]{Fo23} and they Zariski locally exist by \cite[Lemma
1.5 and Proposition 1.9]{Fo23}). Furthermore, one can also deduce the result
directly from the principally polarized case as follows. Suppose that $%
\varphi _{\mathcal{L}}$ is a polarization with kernel $K_{\mathcal{L}}$ \'{e}%
tale locally isomorphic to $H_{d}^{2}$ with $H_{d}=\frac{\mathbb{Z}}{d_{1}%
\mathbb{Z}}\times ...\frac{\mathbb{Z}}{d_{g}\mathbb{Z}}$ and non-degenerate
symplectic pairing $e_{\mathcal{L}}$ (see \cite[Ch. VI, \S 4.4]{Deb} for its
definition from the complex point of view). Then one can consider the
functor classifying triples $\left( A,\varphi _{\mathcal{L}},H\right) $
where $\left( A,\varphi _{\mathcal{L}}\right) $ is a polarized abelian
scheme of type $d$ and $H$ is a totally isotropic subgroup of $K_{\mathcal{L}%
}$ \'{e}tale locally isomorphic to $H_{d}$ (cfr. Remark \ref{Sheaves R
PLevel}). This functor is the source of the following two morphisms. The
former maps $\left( A,\varphi _{\mathcal{L}},H\right) $ to $\left( A,\varphi
_{\mathcal{L}}\right) $ and the second maps it is the principally polarized
abelian scheme $\left( A/H,\overline{\varphi }_{\mathcal{L}}\right) $, where 
$\overline{\varphi }_{\mathcal{L}}$ denotes the induced polarization. As
these two morphisms are finite \'{e}tale, the result in the general PEL
setting follows from the principally polarized case.
\end{proof}

For every $p\in \mathbb{N}$, write $I_{p}\subset \left\{ 1,...,n\right\}
^{p} $ for the set of multi-indexes $\mathbf{i}=\left(
i_{1},...,i_{p}\right) $ such that $i_{1}<...<i_{p}$ (we understand that $%
I_{0}:=\left\{ \phi \right\} $). For every $\mathbf{i}\in I_{p}$, we write $%
\left\{ \mathbf{i}\right\} $ for the set $\left\{ i_{1},...,i_{p}\right\} $
(we understand that $\left\{ \phi \right\} =\phi $). With these notations,
we define%
\begin{equation*}
\omega _{\mathbf{i}}:=\omega _{i_{1}}\wedge ...\wedge \omega _{i_{p}}\in
\Omega _{X_{0}/\Bbbk }^{p}\left( X_{0}\right) \text{.}
\end{equation*}%
Then $\left( \text{\ref{de Rham F Triv2}}\right) $ yields%
\begin{equation}
H^{0}\left( X_{0},\mathcal{E}_{\rho }\otimes _{\mathcal{O}_{X}}\Omega
_{X/\Bbbk }^{p}\right) =\tbigoplus\nolimits_{\mathbf{i}\in
I_{p}}R_{0}\otimes _{\Bbbk }\rho \otimes _{R_{\infty }}\omega _{\mathbf{i}}%
\text{.}  \label{de Rham F Triv3}
\end{equation}

\begin{lemma}
\label{de Rham L2}If $X_{0}=X_{\infty }$ and%
\begin{equation*}
\omega =\tsum\nolimits_{\mathbf{i}\in I_{p}}f_{\mathbf{i}}\otimes _{\Bbbk
}v_{\mathbf{i}}\otimes _{R_{\infty }}\omega _{\mathbf{i}}\text{ and }\nabla
_{\rho }^{p}\left( \omega \right) =\tsum\nolimits_{\mathbf{j}\in
I_{p+1}}\nabla _{\rho }^{p}\left( \omega \right) _{\mathbf{j}}\otimes
_{R_{\infty }}\omega _{\mathbf{j}}\text{,}
\end{equation*}%
writing $\mathbf{j}=\left( j_{1},...,j_{p+1}\right) $ we have%
\begin{equation*}
\nabla _{\rho }^{p}\left( \omega \right) _{\mathbf{j}}=\tsum%
\nolimits_{k=1}^{p+1}\left( -1\right) ^{k+1}\left( \theta _{j_{k}}\left(
f_{(j_{1},...,\widehat{j_{k}},...,j_{p+1})}\right) \otimes _{\Bbbk
}v_{(j_{1},...,\widehat{j_{k}},...,j_{p+1})}+f_{(j_{1},...,\widehat{j_{k}}%
,...,j_{p+1})}\otimes _{\Bbbk }\partial _{j_{k}}v_{(j_{1},...,\widehat{j_{k}}%
,...,j_{p+1})}\right) \text{.}
\end{equation*}
\end{lemma}

\begin{proof}
Since $d^{1}\left( \omega _{i}\right) =0$ (see Remark \ref{q-exp R KS}), the
result can be deduce from the above formula for $\nabla _{\rho }$ and Lemma %
\ref{de Rham L1}. Indeed, an efficient way to implement the Leibnitz rule is
to identify the de Rham complex with an appropriate Koszul complex, as
explained after Proposition \ref{Kos Int PdR} below.
\end{proof}

Regarding $\rho \otimes \wedge ^{p}\mathrm{Sym}_{g}^{2}$ as a $\mathbf{Q}$%
-module, the trivialization $\left( \text{\ref{de Rham F Triv2}}\right) $
yields%
\begin{equation}
H^{0}\left( X_{0},\mathcal{E}_{\rho }\otimes _{\mathcal{O}_{X}}\Omega
_{X/\Bbbk }^{p}\right) \overset{\sim }{\longrightarrow }R_{0}\otimes _{\Bbbk
}\rho \otimes _{\Bbbk }\wedge ^{p}\mathrm{Sym}_{g}^{2}\simeq
\tbigoplus\nolimits_{\mathbf{i}\in I_{p}}R_{0}\otimes _{\Bbbk }\rho \otimes
_{R_{\infty }}\omega _{\mathbf{i}}\text{,}  \label{de Rham F Triv4}
\end{equation}%
where, taking into account $\left( \text{\ref{de Rham F TrivKS}}\right) $,
the above isomorphism $\left( \text{\ref{de Rham F Triv4}}\right) \simeq
\left( \text{\ref{de Rham F Triv3}}\right) $ is given by the Kodaira-Spencer
isomorphism $\mathcal{W}_{\mathrm{Sym}_{g}^{2}\left( -1\right) }\simeq
\Omega _{X_{0}/\Bbbk }^{1}$ inducing $\mathcal{W}_{\wedge ^{p}\mathrm{Sym}%
_{g}^{p}}\simeq \Omega _{X_{0}/\Bbbk }^{p}$. Using $\left( \text{\ref{de
Rham F Triv4}}\right) $, we can define $H^{0}\left( X_{0},\mathcal{E}_{\rho
}\otimes _{\mathcal{O}_{X}}\Omega _{X/\Bbbk }^{p}\right) ^{\left[ P\right] }$
by means of the following isomorphism%
\begin{equation*}
H^{0}\left( X_{0},\mathcal{E}_{\rho }\otimes _{\mathcal{O}_{X}}\Omega
_{X/\Bbbk }^{p}\right) ^{\left[ P\right] }\overset{\sim }{\longrightarrow }%
e^{\left[ P\right] }\left( R_{0}\right) \otimes _{\Bbbk }\rho \otimes
_{\Bbbk }\wedge ^{p}\mathrm{Sym}_{g}^{2}=\tbigoplus\nolimits_{\mathbf{i}\in
I_{p}}e^{\left[ P\right] }\left( R_{0}\right) \otimes _{\Bbbk }\rho \otimes
_{R_{\infty }}\omega _{\mathbf{i}}\text{.}
\end{equation*}

\begin{corollary}
\label{de Rham C1}We have that $H^{0}\left( X_{0},\mathcal{E}_{\rho }\otimes
_{\mathcal{O}_{X}}\Omega _{X/\Bbbk }^{p}\right) ^{\left[ P\right] }$ is a
subcomplex.
\end{corollary}

\begin{proof}
This is a consequence of Lemma \ref{de Rham L2} and the fact that $\theta
_{P}$ operates on $e^{\left[ P\right] }\left( R_{0}\right) \otimes _{\Bbbk
}\rho =H^{0}\left( X_{0},\mathcal{W}_{\rho }\right) ^{\left[ P\right] }$.
\end{proof}

\subsection{\label{S Igusa acyclic}Acyclicity of $p$-depleted de Rham
complexes}

Consider the following diagram that we are going to describe\footnote{%
As explained in the first footnote of \S \ref{S Sheaves} and clarified in
the proof of Lemma \ref{Primitives L2} below, the twists introduced in $%
\left( \text{\ref{Primitives D1}}\right) $ and the ones previously
considered in \S \ref{S Sheaves} are inserted in order to make everything
Frobenious equivariant when restricting to $X_{0}$. Also, the proof of Lemma %
\ref{Primitives L2} below shows that one has to impose the equality $%
\mathcal{E}_{\rho \left( -1\right) }=\mathcal{E}_{\rho }\left( 1\right) $ if 
$\rho \left( -1\right) :=\rho \otimes \nu ^{-1}$ for the symplectic
multiplier $\nu $, where the Frobenious of $\mathcal{E}_{\rho }\left(
1\right) $ is obtained multiplying the Frobenious of $\mathcal{E}_{\rho }$
by $p^{-1}$.}:%
\begin{equation}
\begin{array}{ccccccc}
0\longrightarrow & \omega _{\mathrm{dR},X_{0}}\otimes _{\mathcal{O}%
_{X_{0}}}\omega _{\mathrm{dR},X_{0}}\left( 1\right) & \longrightarrow & 
\omega _{\mathrm{dR},X_{0}}\otimes _{\mathcal{O}_{X_{0}}}\mathcal{H}_{%
\mathrm{dR},X_{0}}^{1}\left( 1\right) & \longrightarrow & \omega _{\mathrm{dR%
},X_{0}}\otimes _{\mathcal{O}_{X_{0}}}Lie_{\mathrm{dR},X_{0}}^{t}\left(
1\right) & \longrightarrow 0 \\ 
& \parallel &  & \uparrow &  & \uparrow &  \\ 
0\longrightarrow & \omega _{\mathrm{dR},X_{0}}\otimes _{\mathcal{O}%
_{X_{0}}}\omega _{\mathrm{dR},X_{0}}\left( 1\right) & \longrightarrow & 
\mathcal{J}^{\prime \prime } & \longrightarrow & \mathcal{O}_{X_{0}} & 
\longrightarrow 0 \\ 
& \downarrow &  & \downarrow &  & \parallel &  \\ 
0\longrightarrow & \mathrm{Sym}^{2}\left( \omega _{\mathrm{dR},X_{0}}\right)
\left( 1\right) & \longrightarrow & \mathcal{J}^{\prime } & \longrightarrow
& \mathcal{O}_{X_{0}} & \longrightarrow 0\text{.}%
\end{array}
\label{Primitives D1}
\end{equation}%
The first row is obtained applying $\omega _{\mathrm{dR},X_{0}}\left(
1\right) \otimes _{\mathcal{O}_{X_{0}}}-$ to the exact sequence which gives
the Hodge filtration of $\mathcal{H}_{\mathrm{dR},X_{0}}^{1}$. Then, we
identify $\omega _{\mathrm{dR},X_{0}}\simeq Lie_{\mathrm{dR},X_{0}}^{t,\vee
}\left( -1\right) $ using the symplectic pairing (cfr. \S \ref{S Sheaves}),
consider the canonical evaluation pairing $Lie_{\mathrm{dR}%
,X_{0}}^{t}\otimes _{\mathcal{O}_{X_{0}}}Lie_{\mathrm{dR},X_{0}}^{t,\vee
}\rightarrow \mathcal{O}_{X_{0}}$ and pull-back the first row via the
resulting $\mathcal{O}_{X_{0}}$-dual morphism $\mathcal{O}%
_{X_{0}}\rightarrow \omega _{\mathrm{dR},X_{0}}\otimes _{\mathcal{O}%
_{X_{0}}}Lie_{\mathrm{dR},X_{0}}^{t}\left( 1\right) $ to get the second row.
Finally, we push-out the second row via the canonical morphism $\omega _{%
\mathrm{dR},X_{0}}\otimes _{\mathcal{O}_{X_{0}}}\omega _{\mathrm{dR}%
,X_{0}}\rightarrow \mathrm{Sym}^{2}\left( \omega _{\mathrm{dR},X_{0}}\right) 
$ in order to get the third row. By construction the diagram is commutative
with exact rows. The following result follows from Proposition \ref%
{Representations P Spl} and the discussion preceding it (see $\left( \text{%
\ref{Representations D1}}\right) $).

\begin{lemma}
\label{Primitives L1}Taking the dual exact sequence of the third row of the
above diagram yields the exact sequence%
\begin{equation}
0\longrightarrow \mathcal{O}_{X_{0}}\longrightarrow \mathcal{J}%
\longrightarrow \frac{\mathcal{J}}{\mathcal{O}_{X_{0}}}\longrightarrow 0
\label{Primitives L1 Spl Claim}
\end{equation}%
associated to $\mathcal{E}_{\cdot }$ to $\left( \text{\ref{Representations P
Spl Claim}}\right) $.
\end{lemma}

\bigskip

Applying $\mathcal{E}_{\cdot }$ to $\left( \text{\ref{Representations F
IndSym}}\right) $ and recalling that $\mathcal{E}_{\cdot }$ restricts to $%
\mathcal{W}_{\cdot }$ on $\mathrm{Rep}\left( \mathbf{M}\right) $ via $\left( 
\text{\ref{Representations F Res}}\right) $, we see that there is an
isomorphism%
\begin{equation}
\mathcal{W}_{\rho }\otimes _{\mathcal{O}_{X}}\mathrm{Sym}^{r}\left( \mathcal{%
J}\right) \overset{\sim }{\longrightarrow }\mathcal{V}_{\rho }^{\leq r}
\label{Primitives F IndSym}
\end{equation}%
functorial in $\rho \in \mathrm{Rep}_{f}\left( \mathbf{M}\right) $. We will
now assume that a section $p_{0}:\mathcal{H}_{\mathrm{dR},X_{0}}^{1}%
\rightarrow \omega _{\mathrm{dR},X_{0}}$ of the inclusion $\omega _{\mathrm{%
dR},X_{0}}\rightarrow \mathcal{H}_{\mathrm{dR},X_{0}}^{1}$ has been fixed.
By construction it induces a splitting of all the rows appearing in $\left( 
\text{\ref{Primitives D1}}\right) $ and then, thanks to Lemma \ref%
{Primitives L1}, we get a splitting%
\begin{equation}
\mathcal{J\simeq O}_{X_{0}}\oplus \mathrm{Sym}^{2}\left( \omega _{\mathrm{dR}%
,X_{0}}\right) ^{\vee }\left( 1\right) \simeq \mathcal{O}_{X_{0}}\oplus 
\mathcal{W}_{\mathrm{Sym}_{g}^{2,\vee }\left( -1\right) }
\label{Primitives F Spl}
\end{equation}%
of the exact sequence $\left( \text{\ref{Primitives L1 Spl Claim}}\right) $.
It will be convenient to set $\rho _{i}:=\rho \otimes \mathrm{Sym}^{i}\left( 
\mathrm{Sym}_{g}^{2,\vee }\right) \left( -i\right) $, so that $\left( \text{%
\ref{Primitives F IndSym}}\right) $\ together with $\left( \text{\ref%
{Primitives F Spl}}\right) $\ yields the decompositions functorial in $\rho
\in \mathrm{Rep}_{f}\left( \mathbf{M}\right) $%
\begin{equation}
\mathcal{V}_{\rho }\simeq \tbigoplus\nolimits_{i=0}^{+\infty }\mathcal{W}%
_{\rho _{i}}\text{ and }\mathcal{V}_{\rho }\otimes _{\mathcal{O}%
_{X_{0}}}\Omega _{X_{0}/\Bbbk }^{p}\simeq \tbigoplus\nolimits_{i=0}^{+\infty
}\mathcal{W}_{\rho _{i}}\otimes _{\mathcal{O}_{X_{0}}}\Omega _{X_{0}/\Bbbk
}^{p}\simeq \tbigoplus\nolimits_{i=0}^{+\infty }\mathcal{W}_{\rho
_{i}\otimes \wedge ^{p}(\mathrm{Sym}_{g}^{2}\left( -1\right) )}\text{,}
\label{Primitives FDec}
\end{equation}%
where the second identifications are induced by the Kodaira-Spences
isomorphism $\mathcal{W}_{\mathrm{Sym}_{g}^{2}\left( -1\right) }\simeq
\Omega _{X_{0}/\Bbbk }^{1}$ inducing $\mathcal{W}_{\wedge ^{p}(\mathrm{Sym}%
_{g}^{2}\left( -1\right) )}\simeq \Omega _{X_{0}/\Bbbk }^{p}$.

In particular, thanks to \cite[Theorem 4.1]{Kz73} and \cite[Lemma 4.2.1]%
{Kz78}, we can consider the unit root splitting%
\begin{equation}
\mathcal{H}_{\mathrm{dR},X_{0}}^{1}=\mathcal{H}_{\mathrm{dR}%
,X_{0}}^{1,\varphi =p}\oplus \mathcal{H}_{\mathrm{dR},X_{0}}^{1,\varphi =1}%
\text{, where }\mathcal{H}_{\mathrm{dR},X_{0}}^{1,\varphi =p}=\omega _{%
\mathrm{dR},X_{0}}\text{,}  \label{Primitives F URSpl}
\end{equation}%
$\varphi $ is the $\sigma $-linear morphism induced by the morphism of $X^{%
\mathrm{ord}}$ onto itself sending $G$ to the quotient by the connected part
of $G\left[ p\right] $ and $\mathcal{H}_{\mathrm{dR},X_{0}}^{1,\varphi
=p^{r}}$ is the submodule where the $\sigma $-linear Frobenious $\varphi $
acts with slope $r$. Indeed, as proved in Lemma \ref{Primitives L URSpl}
below (see also \cite[Lemma 4.2.1]{Kz78}), one has that $\mathcal{H}_{%
\mathrm{dR},X_{0}}^{1,\varphi =p^{r}}$ is spanned as an $\mathcal{O}_{X_{0}}$%
-module by a $\mathbb{Z}_{p}$-module free of rank $g$ on which $\varphi
=p^{r}$.

\begin{lemma}
\label{Primitives L URSpl}We have that $\left\{ \eta _{i}\right\}
_{i=1}^{g}\subset \mathcal{H}_{\mathrm{dR},X_{0}}^{1,\varphi =1}$ and,
indeed, $\varphi \left( \delta _{i}\right) =p\delta _{i}$ and $\varphi
\left( \eta _{i}\right) =\eta _{i}$ for $i=1,...,g$.
\end{lemma}

\begin{proof}
The fact that $\varphi \left( \delta _{i}\right) =p\delta _{i}$ follows from 
$\left( \text{\ref{q-exp F1}}\right) $ and, then, the relation $\varphi
\left( \eta _{i}\right) =\eta _{i}$ will follow if we can check that%
\begin{equation*}
\varphi \circ \nabla \left( \theta _{i}\right) =p^{-1}\cdot \nabla \left(
\theta _{i}\right) \circ \varphi \text{.}
\end{equation*}%
By definition, we have $\nabla \left( \theta _{i}\right) :=\left( 1\otimes 
\widetilde{\theta }_{i}\right) \circ \nabla $, where $\widetilde{\theta }%
_{i}:\Omega _{X_{0}/\Bbbk }^{1}\rightarrow \mathcal{O}_{X_{0}}$ denotes the $%
\mathcal{O}_{X_{0}}$-linear morphism associated to the derivation $\theta
_{i}:\mathcal{O}_{X_{0}}\rightarrow \mathcal{O}_{X_{0}}$. But according to
Lemma \ref{q-exp L KS} (where $\widetilde{\theta }_{i}$ is abusively denoted
again $\theta _{i}$), we have that $\widetilde{\theta }_{i}\left( \omega
_{j}\right) =\delta _{ij}$, where $\omega _{i}$ satisfies $\varphi \left(
\omega _{i}\right) =p\omega _{i}$ (because it is a multiple of $\frac{dq_{i}%
}{q_{i}}$ by an element of $\Bbbk $ and $\varphi $ operates on $q$-expansion
by sending $q^{\beta }$ to $q^{p\beta }$, see \cite[pag. 247, pa262 and Ch.
VII, 4.2 Theorem and 4.4 Proposition]{FC}). One deduces from this fact that $%
1\otimes \widetilde{\theta }_{i}\circ \left( \varphi \otimes \varphi \right)
=p\cdot \varphi \circ \left( 1\otimes \widetilde{\theta }_{i}\right) $ on $%
\mathcal{H}_{\mathrm{dR},X_{0}}^{1}\otimes _{\mathcal{O}_{X_{0}}}\Omega
_{X_{0}/\Bbbk }^{1}$. Because by functoriality of the Gauss-Manin
connection, $\varphi $ satisfies the relation $\left( \varphi \otimes
\varphi \right) \circ \nabla =\nabla \circ \varphi $, the claimed relation
follows.
\end{proof}

\begin{lemma}
\label{Primitives L2}If $p_{0}=p^{\mathrm{ord}}$ is the unit root splitting,
then under the trivializations $\mathcal{V}_{\rho }\simeq \mathcal{O}%
_{X_{0}}\otimes _{\Bbbk }\mathrm{Ind}_{\mathbf{Q}^{-}}^{\mathbf{GSp}%
_{2g}}\left( \rho \right) \left[ \mathbf{Y}\right] $ and $\mathcal{W}_{\rho
_{i}}\simeq \mathcal{O}_{X_{0}}\otimes _{\Bbbk }\rho _{i}$ given by $\left( 
\text{\ref{de Rham F Triv2}}\right) $, the first decomposition in $\left( 
\text{\ref{Primitives FDec}}\right) $ corresponds to $\mathcal{O}%
_{X_{0}}\otimes _{\Bbbk }\left( \text{\ref{Representations F DirSumDec}}%
\right) $.
\end{lemma}

\begin{proof}
Because the first decomposition in $\left( \text{\ref{Primitives FDec}}%
\right) $ (resp. $\mathcal{O}_{X_{0}}\otimes _{\Bbbk }\left( \text{\ref%
{Representations F DirSumDec}}\right) $) is induced by $\left( \text{\ref%
{Primitives F IndSym}}\right) $ (resp. $\left( \text{\ref{Representations F
IndSym}}\right) $) and the splitting $\left( \text{\ref{Primitives F Spl}}%
\right) $ (resp. the splitting given by the graduation on $\mathrm{Ind}_{%
\mathbf{Q}^{-}}^{\mathbf{GSp}_{2g}}\left( \rho \right) \left[ \mathbf{Y}%
\right] _{\leq 1}$), suffices to check that the splitting $\left( \text{\ref%
{Primitives F Spl}}\right) $ equals $\mathcal{O}_{X_{0}}\otimes _{\Bbbk }-$
applied to the splitting given by the graduation on $\mathrm{Ind}_{\mathbf{Q}%
^{-}}^{\mathbf{GSp}_{2g}}\left( \rho \right) \left[ \mathbf{Y}\right] _{\leq
1}$. By definition of $\alpha _{\mathrm{dR},0}$, the ordered
symplectic-Hodge basis provided by the standard basis $\left\{
e_{i},,f_{j}\right\} _{i,j=1}^{g}$ is sent via $\alpha _{\mathrm{dR},0}$ to $%
\left\{ \delta _{i},\eta _{j}\right\} _{i,j=1}^{g}$ (see the discussion
before Remark \ref{de Rham R DefEta}). Hence, it follows from Lemma \ref%
{Primitives L URSpl} that the Frobenious $\varphi $ on the left hand side of%
\begin{equation*}
\vartheta _{\alpha _{\mathrm{dR},0}}=\alpha _{\mathrm{dR},0}^{-1}:\mathcal{H}%
_{\mathrm{dR},X_{0}}^{1}\simeq \mathcal{O}_{X_{0}}\otimes _{\Bbbk }V
\end{equation*}%
is given on the right hand side of this identification by the Frobenious
element%
\begin{equation*}
m_{0}:=\left( 
\begin{array}{cc}
p\cdot 1_{g} & 0_{g} \\ 
0_{g} & 1_{g}%
\end{array}%
\right) \in \mathbf{M}
\end{equation*}%
We deduce that $\left( \text{\ref{Primitives F URSpl}}\right) $ is induced
by the decomposition of $V$ as an $\left\langle m_{0}\right\rangle $-module.
It also follows from Remark \ref{de Rham R DefEta} that the symplectic
pairing $\left( \text{\ref{Sheaves F Sympl}}\right) $ is equivariant (with
the convention that the twist by $-1$ means that the Frobenious is
multiplied by $p$): it then follows that all the arrows appearing in $\left( 
\text{\ref{Primitives D1}}\right) $ are $\varphi $-equivariant. On the other
hand, because $m_{0}\in \mathbf{Q}$, we also know that all the arrows
appearing in $\left( \text{\ref{Representations D1}}\right) $ are $m_{0}$%
-equivariant. Because $\vartheta _{\alpha _{\mathrm{dR},0}}$ identifies $%
\left( \text{\ref{Primitives D1}}\right) $ with $\mathcal{O}_{X_{0}}\otimes
_{\Bbbk }\left( \text{\ref{Representations D1}}\right) $ (in view of Lemma %
\ref{Primitives L1}, with the conventions that $\mathcal{E}_{\rho \left(
-1\right) }=\mathcal{E}_{\rho }\left( 1\right) $ if $\rho \left( -1\right)
:=\rho \otimes \nu ^{-1}$ for the symplectic multiplier $\nu $), we deduce
that $\left( \text{\ref{Primitives F Spl}}\right) $ is also induced by the
decomposition of $\mathrm{J}:=\mathrm{Ind}_{\mathbf{Q}^{-}}^{\mathbf{G}%
}\left( \mathrm{1}\right) \left[ \mathbf{Y}\right] _{\leq 1}$ as an $%
\left\langle m_{0}\right\rangle $-module. But this decomposition coincides
with the splitting given by the graduation on $\mathrm{Ind}_{\mathbf{Q}%
^{-}}^{\mathbf{GSp}_{2g}}\left( \rho \right) \left[ \mathbf{Y}\right] _{\leq
1}$, as it follows from Lemma \ref{Representations L1} $\left( 2\right) $,
which gives $\mathrm{Ind}_{\mathbf{Q}^{-}}^{\mathbf{GSp}_{2g}}\left( \rho
\right) \left[ \mathbf{Y}\right] _{\leq 1}^{m_{0}=p^{-1}}=\mathrm{Ind}_{%
\mathbf{Q}^{-}}^{\mathbf{GSp}_{2g}}\left( \rho \right) \left[ \mathbf{Y}%
\right] _{=1}$ and $\mathrm{Ind}_{\mathbf{Q}^{-}}^{\mathbf{GSp}_{2g}}\left(
\rho \right) \left[ \mathbf{Y}\right] _{\leq 1}^{m_{0}=1}=\mathrm{Ind}_{%
\mathbf{Q}^{-}}^{\mathbf{GSp}_{2g}}\left( \rho \right) \left[ \mathbf{Y}%
\right] _{=0}$.
\end{proof}

\bigskip

We assume, until the end of \S \ref{S Igusa de Rham}, that $X_{0}\in \left\{ 
\mathfrak{I},X_{\infty },\mathfrak{X}_{\infty }\right\} $ and that $p_{0}=p^{%
\mathrm{ord}}$ is the unit root splitting. The trivialization $\left( \text{%
\ref{de Rham F Triv2}}\right) $ yields%
\begin{equation*}
H^{0}\left( X_{0},\mathcal{W}_{\rho _{i}}\otimes _{\mathcal{O}_{X}}\Omega
_{X/\Bbbk }^{p}\right) \overset{\sim }{\longrightarrow }R_{0}\otimes _{\Bbbk
}\rho _{i}\otimes _{\Bbbk }\wedge ^{p}\mathrm{Sym}_{g}^{2}\simeq
\tbigoplus\nolimits_{\mathbf{i}\in I_{p}}R_{0}\otimes _{\Bbbk }\rho
_{i}\otimes _{R_{\infty }}\omega _{\mathbf{i}}\text{,}
\end{equation*}%
where, taking into account $\left( \text{\ref{de Rham F TrivKS}}\right) $,
again the above isomorphism is given by the Kodaira-Spencer isomorphism $%
\mathcal{W}_{\mathrm{Sym}_{g}^{2}\left( -1\right) }\simeq \Omega
_{X_{0}/\Bbbk }^{1}$ inducing $\mathcal{W}_{\wedge ^{p}\mathrm{Sym}%
_{g}^{2}}\simeq \Omega _{X_{0}/\Bbbk }^{p}$. Thanks to Lemma \ref{Primitives
L2} we deduce that, if we define $H^{0}\left( X_{0},\mathcal{W}_{\rho
_{i}}\otimes _{\mathcal{O}_{X}}\Omega _{X/\Bbbk }^{\cdot }\right) ^{\left[ P%
\right] }$ by means of the isomorphism%
\begin{equation}
H^{0}\left( X_{0},\mathcal{W}_{\rho _{i}}\otimes _{\mathcal{O}_{X}}\Omega
_{X/\Bbbk }^{\cdot }\right) ^{\left[ P\right] }\overset{\sim }{%
\longrightarrow }e^{\left[ P\right] }\left( R_{0}\right) \otimes _{\Bbbk
}\rho _{i}\otimes _{\Bbbk }\wedge ^{p}\mathrm{Sym}_{g}^{2}\simeq
\tbigoplus\nolimits_{\mathbf{i}\in I_{p}}e^{\left[ P\right] }\left(
R_{0}\right) \otimes _{\Bbbk }\rho _{i}\otimes _{R_{\infty }}\omega _{%
\mathbf{i}}\text{,}  \label{de Rham F Triv5depl}
\end{equation}%
then $\left( \text{\ref{Primitives FDec}}\right) $ (and the fact that the $p$%
-depletion $\left( -\right) ^{\left[ P\right] }$ is can be defined by
evaluating the sections at $X_{\infty }$) yields%
\begin{equation*}
H^{0}\left( X_{0},\mathcal{V}_{\rho }\otimes _{\mathcal{O}_{X}}\Omega
_{X/\Bbbk }^{p}\right) ^{\left[ P\right] }\simeq
\tbigoplus\nolimits_{i=0}^{+\infty }H^{0}\left( X_{0},\mathcal{W}_{\rho
_{i}}\otimes _{\mathcal{O}_{X}}\Omega _{X/\Bbbk }^{p}\right) ^{\left[ P%
\right] }\text{.}
\end{equation*}%
We have%
\begin{equation*}
\nabla ^{p}:H^{0}\left( X_{0},\mathcal{V}_{\rho }\otimes _{\mathcal{O}%
_{X}}\Omega _{X/\Bbbk }^{p}\right) \longrightarrow H^{0}\left( X_{0},%
\mathcal{V}_{\rho }\otimes _{\mathcal{O}_{X}}\Omega _{X/\Bbbk }^{p+1}\right) 
\text{ for }p\in \left\{ 0,...,d_{g}-1\right\} \text{.}
\end{equation*}%
Thanks to Griffiths' tranversality condition (see Lemma \ref{Sheaves L Fil})
and the identification $\mathrm{gr}_{i}\left( \mathcal{V}_{\rho }\right)
\simeq \mathcal{W}_{\rho _{i}}$, it induces%
\begin{eqnarray*}
&&\Delta _{i}^{p}:H^{0}\left( X_{0},\mathcal{W}_{\rho _{i}}\otimes _{%
\mathcal{O}_{X}}\Omega _{X/\Bbbk }^{p}\right) \simeq H^{0}\left( X_{0},%
\mathrm{gr}_{i}\left( \mathcal{V}_{\rho }\right) \otimes _{\mathcal{O}%
_{X}}\Omega _{X/\Bbbk }^{p}\right) \\
&&\overset{\mathrm{gr}_{i}\left( \nabla ^{p}\right) }{\longrightarrow }%
H^{0}\left( X_{0},\mathrm{gr}_{i+1}\left( \mathcal{V}_{\rho }\right) \otimes
_{\mathcal{O}_{X}}\Omega _{X/\Bbbk }^{p+1}\right) \simeq H^{0}\left( X_{0},%
\mathcal{W}_{\rho _{i+1}}\otimes _{\mathcal{O}_{X}}\Omega _{X/\Bbbk
}^{p+1}\right) \text{.}
\end{eqnarray*}

\begin{lemma}
\label{Primitives LKey}Suppose that%
\begin{equation*}
F\in H^{0}\left( X_{0},\mathcal{W}_{\rho _{i}}\otimes _{\mathcal{O}%
_{X}}\Omega _{X/\Bbbk }^{p}\right) ^{\left[ P\right] }\overset{\left( \text{%
\ref{Primitives FDec}}\right) }{\hookrightarrow }H^{0}\left( X_{0},\mathcal{V%
}_{\rho }\otimes _{\mathcal{O}_{X}}\Omega _{X/\Bbbk }^{p}\right) ^{\left[ P%
\right] }\text{.}
\end{equation*}%
Then%
\begin{eqnarray*}
\nabla ^{p}\left( F\right) &=&\Theta ^{p}\left( F\right) +\Delta
_{i}^{p}\left( F\right) \in H^{0}\left( X_{0},\mathcal{W}_{\rho _{i}}\otimes
_{\mathcal{O}_{X}}\Omega _{X/\Bbbk }^{p+1}\right) ^{\left[ P\right] }\oplus
H^{0}\left( X_{0},\mathcal{W}_{\rho _{i+1}}\otimes _{\mathcal{O}_{X}}\Omega
_{X/\Bbbk }^{p+1}\right) ^{\left[ P\right] } \\
&&\overset{\left( \text{\ref{Primitives FDec}}\right) }{\hookrightarrow }%
H^{0}\left( X_{0},\mathcal{V}_{\rho }\otimes _{\mathcal{O}_{X}}\Omega
_{X/\Bbbk }^{p+1}\right) ^{\left[ P\right] }\text{,}
\end{eqnarray*}%
where $\Theta ^{p}\left( F\right) \in H^{0}\left( X_{0},\mathcal{W}_{\rho
_{i}}\otimes _{\mathcal{O}_{X}}\Omega _{X/\Bbbk }^{p+1}\right) ^{\left[ P%
\right] }$ and $\Delta _{i}^{p}\left( F\right) \in H^{0}\left( X_{0},%
\mathcal{W}_{\rho _{i+1}}\otimes _{\mathcal{O}_{X}}\Omega _{X/\Bbbk
}^{p+1}\right) ^{\left[ P\right] }$ admits the following description. If%
\begin{equation*}
F=\tsum\nolimits_{\mathbf{i}\in I_{p}}F_{\mathbf{i}}\otimes _{\Bbbk }v_{%
\mathbf{i}}\otimes _{R_{0}}\omega _{\mathbf{i}}\text{ and }\Psi \left(
F\right) =\tsum\nolimits_{\mathbf{j}\in I_{p+1}}\Psi \left( F\right) _{%
\mathbf{j}}\otimes _{R_{0}}\omega _{\mathbf{j}}
\end{equation*}%
for $\Psi \in \left\{ \Theta ^{p},\Delta _{i}^{p}\right\} $, writing $%
\mathbf{j}=\left( j_{1},...,j_{p+1}\right) $ we have%
\begin{eqnarray*}
\Theta ^{p}\left( F\right) _{\mathbf{j}}\left( q\right)
&=&\tsum\nolimits_{k=1}^{p+1}\left( -1\right) ^{k+1}\theta _{j_{k}}\left(
F_{(j_{1},...,\widehat{j_{k}},...,j_{p+1})}\right) \otimes _{\Bbbk
}v_{(j_{1},...,\widehat{j_{k}},...,j_{p+1})}\text{ and} \\
\Delta _{i}^{p}\left( F\right) _{\mathbf{j}}\left( q\right)
&=&\tsum\nolimits_{k=1}^{p+1}\left( -1\right) ^{k+1}F_{(j_{1},...,\widehat{%
j_{k}},...,j_{p+1})}\otimes _{\Bbbk }\partial _{j_{k}}v_{(j_{1},...,\widehat{%
j_{k}},...,j_{p+1})}\text{.}
\end{eqnarray*}
\end{lemma}

\begin{proof}
Taking into account Lemma \ref{Primitives L2}, Lemma \ref{Representations L1}
$\left( 3\right) $ tells us that $\theta _{i}\ $raises the total $Y$-degree
by one for $i=1,...,d_{g}$. Hence Lemma \ref{de Rham L2} with ($\rho $ taken
to be $\mathrm{Ind}_{\mathbf{Q}^{-}}^{\mathbf{G}}\left( \rho \right) \left[ 
\mathbf{Y}\right] $) gives the claim.
\end{proof}

\bigskip

The following lemma tells us that the theta operators defined in Lemma \ref%
{Primitives LKey} give rise to complexes.

\begin{lemma}
\label{Primitives L3}If $p\in \left\{ 1,...,d_{g}\right\} $, then $\Theta
^{p}\left( \Theta ^{p-1}\left( F\right) \right) =0$ for every $F\in
H^{0}\left( X_{0},\mathcal{W}_{\rho _{i}}\otimes _{\mathcal{O}_{X}}\Omega
_{X/\Bbbk }^{p-1}\right) ^{\left[ P\right] }$: hence $H^{0}\left( X_{0},%
\mathcal{W}_{\rho _{i}}\otimes _{\mathcal{O}_{X}}\Omega _{X/\Bbbk }^{\cdot
}\right) ^{\left[ P\right] }$ is a complex. Indeed, as a complex, we have%
\begin{equation*}
\left( H^{0}\left( X_{0},\mathcal{W}_{\rho _{i}}\otimes _{\mathcal{O}%
_{X}}\Omega _{X/\Bbbk }^{\cdot }\right) ^{\left[ P\right] },\Theta ^{\cdot
}\right) \simeq \left( \rho _{i}\otimes _{\Bbbk }H^{0}\left( X_{0},\Omega
_{X/\Bbbk }^{\cdot }\right) ^{\left[ P\right] },1\otimes _{\Bbbk }d^{\cdot
}\right)
\end{equation*}
\end{lemma}

\begin{proof}
Indeed, a comparison between the explicit formulas for $\nabla ^{p}\left(
F\right) $ and $\Theta ^{p}\left( F\right) $ appearing in Lemma \ref{de Rham
L2} (with $\rho $ of loc.cit. taken to be the trivial representation) and,
respectively, Lemma \ref{Primitives LKey} (with $\rho $ our given
representation) shows that the graded module $\rho _{i}\otimes _{\Bbbk
}H^{0}\left( X_{0},\Omega _{X/\Bbbk }^{\cdot }\right) ^{\left[ P\right] }$
with its degree one morphism $1_{\rho _{i}}\otimes _{\Bbbk }d^{\cdot }$ is
identified with $H^{0}\left( X_{0},\mathcal{W}_{\rho _{i}}\otimes _{\mathcal{%
O}_{X}}\Omega _{X/\Bbbk }^{\cdot }\right) ^{\left[ P\right] }$ with its
degree one morphism $\Theta ^{\cdot }$. The claimed isomorphism follows and,
with it, the fact that $H^{0}\left( X_{0},\mathcal{W}_{\rho _{i}}\otimes _{%
\mathcal{O}_{X}}\Omega _{X/\Bbbk }^{\cdot }\right) ^{\left[ P\right] }$ is a
complex.
\end{proof}

\bigskip

We are now going to show certain acyclicity properties of $H^{0}\left( X_{0},%
\mathcal{W}_{\rho _{i}}\otimes _{\mathcal{O}_{X}}\Omega _{X/\Bbbk }^{\cdot
}\right) ^{\left[ P\right] }$ under the assumption that $P\left( 0\right) =0$%
. It then follows that we can write%
\begin{equation*}
P=\tsum\nolimits_{k=1}^{d_{g}}T_{k}P_{k}^{\ast }
\end{equation*}%
for suitable polynomials $P_{k}^{\ast }\in \mathbb{Z}_{p}\left[
T_{i}:i=1,...,d_{g}\right] $. Using $\left( \text{\ref{de Rham F Triv5depl}}%
\right) $, we define%
\begin{equation*}
\Theta _{p-1}^{-1}:H^{0}\left( X_{\infty },\mathcal{W}_{\rho _{i}}\otimes _{%
\mathcal{O}_{X}}\Omega _{X/\Bbbk }^{p}\right) \longrightarrow H^{0}\left(
X_{\infty },\mathcal{W}_{\rho _{i}}\otimes _{\mathcal{O}_{X}}\Omega
_{X/\Bbbk }^{p-1}\right)
\end{equation*}%
as follows. For every $\mathbf{i}\in I_{p}$, we write $\left\{ \mathbf{i}%
\right\} $ for the set $\left\{ i_{1},...,i_{p}\right\} $ (we understand
that $\left\{ \phi \right\} =\phi $). If $\mathbf{i}\in I_{p}$ and $k\in
\left\{ 1,...,n\right\} $, we set $k\wedge \mathbf{i}=\phi $ in case $k\in
\left\{ \mathbf{i}\right\} $ and, otherwise, we write $k\wedge \mathbf{i}$
for the unique element of $I_{p+1}$ such that $\left\{ k\wedge \mathbf{i}%
\right\} =\left\{ k\right\} \cup \left\{ \mathbf{i}\right\} $. Define $%
\varepsilon _{\mathbf{i}}\left( k\right) :=\#\left\{ i\in \left\{ \mathbf{i}%
\right\} :i<k\right\} $. If%
\begin{equation*}
\omega =\tsum\nolimits_{\mathbf{i}\in I_{p}}f_{\mathbf{i}}\otimes _{\Bbbk
}v_{\mathbf{i}}\otimes _{R_{0}}\omega _{\mathbf{i}}\text{ and }\Theta
_{p-1}^{-1}\left( \omega \right) =\tsum\nolimits_{\mathbf{j}\in
I_{p-1}}\Theta _{p-1}^{-1}\left( \omega \right) _{\mathbf{j}}\otimes
_{R_{0}}\omega _{\mathbf{j}}\text{,}
\end{equation*}%
then we have, by definition,%
\begin{equation}
\Theta _{p-1}^{-1}\left( \omega \right) _{\mathbf{j}}=\tsum\nolimits_{k\in
\left\{ 1,...,n\right\} -\left\{ \mathbf{j}\right\} }\left( -1\right)
^{\varepsilon _{k\wedge \mathbf{j}}\left( k\right) }P_{k}^{\ast }\left( 
\mathbf{\theta }\right) \left( \theta _{P}^{-1}\left( f_{k\wedge \mathbf{j}%
}\right) \right) \otimes _{\Bbbk }v_{\mathbf{i}}\text{.}
\label{Primitives F Integration}
\end{equation}

\begin{proposition}
\label{Primitives PKey}Suppose that $P\left( 0\right) =0$. If $p\in \left\{
1,...,d_{g}\right\} $ and%
\begin{equation*}
f\in H^{0}\left( X_{0},\mathcal{W}_{\rho _{i}}\otimes _{\mathcal{O}%
_{X}}\Omega _{X/\Bbbk }^{p}\right) ^{\left[ P\right] }\hookrightarrow
H^{0}\left( X_{0},\mathcal{V}_{\rho }\otimes _{\mathcal{O}_{X}}\Omega
_{X/\Bbbk }^{p}\right) ^{\left[ P\right] }
\end{equation*}%
is such that, $\Theta ^{p}\left( f\right) =0$ (we understand this condition
as an empty condition when $p=d_{g}$),\ then $f=\Theta ^{p-1}\left( \Theta
_{p-1}^{-1}\left( f\right) \right) $. In particular, $H^{0}\left( X_{0},%
\mathcal{W}_{\rho _{i}}\otimes _{\mathcal{O}_{X}}\Omega _{X/\Bbbk }^{\cdot
}\right) ^{\left[ P\right] }$ is acyclic in degree $p\in \left\{
1,...,d_{g}\right\} $.
\end{proposition}

\begin{proof}
This will be proved after Proposition \ref{Kos Int PdR} below.
\end{proof}

\bigskip

\begin{theorem}
\label{Primitives TKey}Suppose that $P\left( 0\right) =0$. For $p\in \left\{
1,...,d_{g}\right\} $, suppose that $f=\tsum\nolimits_{i=0}^{+\infty }f_{i}$
with%
\begin{equation*}
f_{i}\in H^{0}\left( X_{0},\mathcal{W}_{\rho _{i}}\otimes _{\mathcal{O}%
_{X}}\Omega _{X/\Bbbk }^{p}\right) ^{\left[ P\right] }\hookrightarrow
H^{0}\left( X_{0},\mathcal{V}_{\rho }\otimes _{\mathcal{O}_{X}}\Omega
_{X/\Bbbk }^{p}\right) ^{\left[ P\right] }
\end{equation*}%
is such that $\nabla ^{p}\left( f\right) =0$. Define%
\begin{equation*}
F_{0}:=\Theta _{p-1}^{-1}\left( f_{0}\right) \in H^{0}\left( X_{0},\mathcal{W%
}_{\rho _{0}}\otimes _{\mathcal{O}_{X}}\Omega _{X/\Bbbk }^{p-1}\right) ^{%
\left[ P\right] }\text{.}
\end{equation*}%
Then $\Delta _{0}^{p-1}\left( F_{0}\right) \in H^{0}\left( X_{0},\mathcal{W}%
_{\rho _{1}}\otimes _{\mathcal{O}_{X}}\Omega _{X/\Bbbk }^{p}\right) ^{\left[
P\right] }$ and, having defined inductively%
\begin{equation*}
F_{k}\in H^{0}\left( X_{0},\mathcal{W}_{\rho _{k}}\otimes _{\mathcal{O}%
_{X}}\Omega _{X/\Bbbk }^{p-1}\right) ^{\left[ P\right] }\text{ such that }%
\Delta _{k}^{p-1}\left( F_{k}\right) \in H^{0}\left( X_{0},\mathcal{W}_{\rho
_{k+1}}\otimes _{\mathcal{O}_{X}}\Omega _{X/\Bbbk }^{p}\right) ^{\left[ P%
\right] }
\end{equation*}%
for every $k=0,...,i-1$ with $i\geq 1$, setting%
\begin{equation*}
F_{i}:=\Theta _{p-1}^{-1}\left( f_{i}-\Delta _{i-1}^{p-1}\left(
F_{i-1}\right) \right) \in H^{0}\left( X_{0},\mathcal{W}_{\rho _{i}}\otimes
_{\mathcal{O}_{X}}\Omega _{X/\Bbbk }^{p-1}\right) ^{\left[ P\right] }
\end{equation*}%
we have $\Delta _{i}^{p-1}\left( F_{i}\right) \in H^{0}\left( X_{0},\mathcal{%
W}_{\rho _{i+1}}\otimes _{\mathcal{O}_{X}}\Omega _{X/\Bbbk }^{p}\right) ^{%
\left[ P\right] }$. Furthermore, assuming that there is $r\in \mathbb{N}$
such that $\Delta _{r}^{p-1}\left( F_{r}\right) =0$ and $f_{i}=0$ for every $%
i\geq r+1$, we have $F_{i}=0$ for every $i\geq r+1$ and we can define%
\begin{equation*}
F:=\tsum\nolimits_{i=0}^{+\infty }F_{i}\in
\tbigoplus\nolimits_{i=0}^{+\infty }H^{0}\left( X_{0},\mathcal{W}_{\rho
_{i}}\otimes _{\mathcal{O}_{X}}\Omega _{X/\Bbbk }^{p-1}\right) ^{\left[ P%
\right] }\overset{\left( \text{\ref{Primitives FDec}}\right) }{\simeq }%
H^{0}\left( X_{0},\mathcal{V}_{\rho }\otimes _{\mathcal{O}_{X}}\Omega
_{X/\Bbbk }^{p-1}\right) ^{\left[ P\right] }\text{.}
\end{equation*}%
Then%
\begin{equation*}
\nabla ^{p-1}\left( F\right) =f\text{.}
\end{equation*}
\end{theorem}

\begin{proof}
The fact that $\Delta _{0}^{p-1}\left( F_{0}\right) \in H^{0}\left( X_{0},%
\mathcal{W}_{\rho _{1}}\otimes _{\mathcal{O}_{X}}\Omega _{X/\Bbbk
}^{p}\right) ^{\left[ P\right] }$ and then $\Delta _{i}^{p-1}\left(
F_{i}\right) \in H^{0}\left( X_{0},\mathcal{W}_{\rho _{i+1}}\otimes _{%
\mathcal{O}_{X}}\Omega _{X/\Bbbk }^{p}\right) ^{\left[ P\right] }$ for every 
$i\geq 1$ follows from the inductive definition and the fact that $%
F_{i-1}\in H^{0}\left( X_{0},\mathcal{W}_{\rho _{0}}\otimes _{\mathcal{O}%
_{X}}\Omega _{X/\Bbbk }^{p-1}\right) ^{\left[ P\right] }$ in view of Lemma %
\ref{Primitives LKey}. In particular, the family $\left\{ F_{i}\right\}
_{i\in \mathbb{N}}$ is defined.

We now claim, by induction on $i\in \mathbb{N}$, that%
\begin{equation}
\Theta ^{p}\left( f_{i}-\Delta _{i-1}^{p-1}\left( F_{i-1}\right) \right) =0%
\text{,}  \label{Primitives TKey F1}
\end{equation}%
where we set $F_{-1}=0$. The fact that $\nabla ^{p}\left( f\right) =0$
implies, thanks to Lemma \ref{Primitives LKey}, that its $i$-component%
\begin{equation}
\Theta ^{p}\left( f_{i}\right) +\Delta _{i-1}^{p}\left( f_{i-1}\right) =0
\label{Primitives TKey F2}
\end{equation}%
for every $i\geq 0$, where we set $f_{-1}=0$. This proves $\left( \text{\ref%
{Primitives TKey F1}}\right) $ when $i=0$ and, assuming $\left( \text{\ref%
{Primitives TKey F1}}\right) $ for $i-1$, we find%
\begin{eqnarray*}
\Theta ^{p}\left( f_{i}-\Delta _{i-1}^{p-1}\left( F_{i-1}\right) \right)
&=&\Theta ^{p}\left( f_{i}\right) -\Theta ^{p}\left( \Delta
_{i-1}^{p-1}\left( F_{i-1}\right) \right) \\
&&\overset{\left( A\right) }{=}\Theta ^{p}\left( f_{i}\right) +\Theta
^{p}\left( \Theta ^{p-1}\left( F_{i}\right) \right) +\Delta _{i-1}^{p}\left(
\Theta ^{p-1}\left( F_{i-1}\right) \right) +\Delta _{i-1}^{p}\left( \Delta
_{i-2}^{p-1}\left( F_{i-2}\right) \right) \\
&&\overset{\left( B\right) }{=}\Theta ^{p}\left( f_{i}\right) +\Theta
^{p}\left( \Theta ^{p-1}\left( F_{i}\right) \right) +\Delta _{i-1}^{p}\left(
f_{i-1}-\Delta _{i-2}^{p-1}\left( F_{i-2}\right) \right) +\Delta
_{i-1}^{p}\left( \Delta _{i-2}^{p-1}\left( F_{i-2}\right) \right) \\
&&\overset{\left( C\right) }{=}\Theta ^{p}\left( \Theta ^{p-1}\left(
F_{i}\right) \right) -\Delta _{i-1}^{p}\left( \Delta _{i-2}^{p-1}\left(
F_{i-2}\right) \right) +\Delta _{i-1}^{p}\left( \Delta _{i-2}^{p-1}\left(
F_{i-2}\right) \right) \\
&=&\Theta ^{p}\left( \Theta ^{p-1}\left( F_{i}\right) \right) \text{,}
\end{eqnarray*}%
where we are going to explain the identities. The equality $\left( A\right) $
follows from the relation $\nabla ^{p}\circ \nabla ^{p-1}=0$, in view of the
decomposition provided by Lemma \ref{Primitives LKey}, $\left( B\right) $
follows the definition $F_{i-1}=\Theta _{p-1}^{-1}\left( f_{i-1}-\Delta
_{i-2}^{p-1}\left( F_{i-2}\right) \right) $ and the fact that we can apply
Proposition \ref{Primitives PKey} to $f_{i-1}-\Delta _{i-2}^{p-1}\left(
F_{i-2}\right) $ thanks to the induction and $\left( C\right) =\left( \text{%
\ref{Primitives TKey F2}}\right) $. Thanks to Lemma \ref{Primitives L3} we
deduce the claimed $\left( \text{\ref{Primitives TKey F1}}\right) $.

Suppose now that $\Delta _{r}^{p-1}\left( F_{r}\right) =0$ and $f_{i}=0$ for
every $i\geq r+1$. Then $F_{r+1}:=\Theta _{p-1}^{-1}\left( f_{r+1}-\Delta
_{r}^{p-1}\left( F_{r}\right) \right) =0$ and the inductive definition of $%
F_{i}$ shows that $F_{i}=0$ for every $i\geq r+1$, so that $F$ is defined.
According to Lemma \ref{Primitives LKey}, the inductive definition of $F_{i}$
and Proposition \ref{Primitives PKey} (which applies to $f_{i}-\Delta
_{i-1}^{p-1}\left( F_{i-1}\right) $, thanks to $\left( \text{\ref{Primitives
TKey F1}}\right) $):%
\begin{eqnarray*}
\nabla ^{p-1}\left( F\right) &=&\tsum\nolimits_{i=0}^{r}\Theta ^{p-1}\left(
F_{i}\right) +\Delta _{i}^{p-1}\left( F_{i}\right) =\Theta ^{p-1}\left(
F_{0}\right) +\tsum\nolimits_{i=1}^{r}\left( \Theta ^{p-1}\left(
F_{i}\right) +\Delta _{i-1}^{p-1}\left( F_{i-1}\right) \right) +\Delta
_{r}^{p-1}\left( F_{r}\right) \\
&=&\Theta ^{p-1}\left( \Theta _{p-1}^{-1}\left( f_{0}\right) \right)
+\tsum\nolimits_{i=1}^{r}\left( \Theta ^{p-1}\left( \Theta _{p-1}^{-1}\left(
f_{i}-\Delta _{i-1}^{p-1}\left( F_{i-1}\right) \right) \right) +\Delta
_{i-1}^{p-1}\left( F_{i-1}\right) \right) +\Delta _{r}^{p-1}\left(
F_{r}\right) \\
&=&f_{0}+\tsum\nolimits_{i=1}^{r}\left( f_{i}-\Delta _{i-1}^{p-1}\left(
F_{i-1}\right) +\Delta _{i-1}^{p-1}\left( F_{i-1}\right) \right) +\Delta
_{r}^{p-1}\left( F_{r}\right) =f\text{,}
\end{eqnarray*}%
where in the last equalities we have used the fact that $\Delta
_{r}^{p-1}\left( F_{r}\right) =0$ and $f_{i}=0$ for every $i\geq r+1$.
\end{proof}

\bigskip

Suppose now that $\lambda \in X_{\mathbf{G},+}$ and recall the inclusion of
filtered complexes%
\begin{equation*}
\mathrm{dR}\left( \mathcal{L}_{\lambda }\right) \hookrightarrow \mathrm{dR}%
\left( \mathcal{V}_{\lambda }\right) \text{.}
\end{equation*}%
Suppose that $\mathcal{C}$ is an exact category and let \textrm{Fil}$\left( 
\mathcal{C}\right) $ be the associated filtered category (we will consider
increasing filtrations). By a splitting of an object $M\in \mathrm{Fil}%
\left( \mathcal{C}\right) $ we mean an isomorphism $s_{M}:M\overset{\sim }{%
\rightarrow }$\textrm{gr}$\left( M\right) $ in $\mathcal{C}$ such that,
writing $\sigma _{M}$ for the inverse, then $\mathrm{Fil}_{i}\left( M\right)
=\sigma _{M}\left( \tbigoplus\nolimits_{j\in \mathbb{N}}\mathrm{gr}%
^{i-j}\left( M\right) \right) $.

\begin{lemma}
\label{Primitives L Fil(L)}The isomorphisms $\left( \text{\ref{Primitives
FDec}}\right) $ are splittings and they yield, by restriction, the splittings%
\begin{equation*}
\mathcal{L}_{\lambda }\otimes _{\mathcal{O}_{X_{0}}}\Omega _{X_{0}/\Bbbk
}^{p}\simeq \tbigoplus\nolimits_{i=0}^{+\infty }\mathrm{gr}_{i}\left( 
\mathcal{L}_{\lambda }\right) \otimes _{\mathcal{O}_{X_{0}}}\Omega
_{X_{0}/\Bbbk }^{p}\text{ for }p=1,...,d_{g}\text{.}
\end{equation*}
\end{lemma}

\begin{proof}
Because $\left( \text{\ref{Representations F DirSumDec}}\right) $ is a
splitting, the fact that the isomorphisms $\left( \text{\ref{Primitives FDec}%
}\right) $ are splittings follows from Lemma \ref{Primitives L2}. Let use
write $\rho _{\lambda ,i}$ for $\rho _{i}$ appearing in $\left( \text{\ref%
{Primitives FDec}}\right) $ in order to emphasize the dependence on the
choice of $\lambda $. Let us consider the left exact sequence obtained from
the truncated integral (dual) BGG complex of Theorem \ref{Representations T
BGG} applying the functor $\mathcal{E}_{-}$ (which is exact, cfr. Remark \ref%
{Sheaves R Ex}): it follows from Remark \ref{Sheaves R Fil2} that we get the
following diagram with exact rows, the former in \textrm{Fil}$\left( 
\mathcal{C}\right) $ with $\mathcal{C}$ the category of $\mathcal{O}_{X_{0}}$%
-modules%
\begin{equation}
\begin{array}{cccccc}
0\longrightarrow & \mathcal{L}_{\lambda }\otimes _{\mathcal{O}%
_{X_{0}}}\Omega _{X_{0}/\Bbbk }^{p} & \longrightarrow & \mathcal{V}_{\lambda
}\otimes _{\mathcal{O}_{X_{0}}}\Omega _{X_{0}/\Bbbk }^{p} & \longrightarrow
& \tbigoplus\nolimits_{w\in W^{\mathbf{M}}:l\left( w\right) =1}\mathcal{V}%
_{w\cdot \lambda }\otimes _{\mathcal{O}_{X_{0}}}\Omega _{X_{0}/\Bbbk }^{p}
\\ 
&  &  & \Vert &  & \Vert \\ 
0\longrightarrow & \tbigoplus\nolimits_{i=0}^{+\infty }\mathrm{gr}_{i}\left( 
\mathcal{L}_{\lambda }\right) \otimes _{\mathcal{O}_{X_{0}}}\Omega
_{X_{0}/\Bbbk }^{p} & \longrightarrow & \tbigoplus\nolimits_{i=0}^{+\infty }%
\mathcal{W}_{\rho _{\lambda ,i}}\otimes _{\mathcal{O}_{X_{0}}}\Omega
_{X_{0}/\Bbbk }^{p} & \longrightarrow & \tbigoplus\nolimits_{w\in W^{\mathbf{%
M}}:l\left( w\right) =1}\left( \tbigoplus\nolimits_{i=0}^{+\infty }\mathcal{W%
}_{\rho _{w\cdot \lambda ,i}}\otimes _{\mathcal{O}_{X_{0}}}\Omega
_{X_{0}/\Bbbk }^{p}\right) \text{.}%
\end{array}
\label{Primitives L Fil(L) F1}
\end{equation}%
Indeed, the square is commutative because the identifications $\left( \text{%
\ref{Primitives FDec}}\right) $ are induced by $\left( \text{\ref%
{Representations F DirSumDec}}\right) $ for $\lambda $ and each $w\cdot
\lambda $ and the proof of Lemma \ref{Primitives L2} shows that these
decompositions $\left( \text{\ref{Representations F DirSumDec}}\right) $ are
obtained by simply regarding the representations as $\left\langle
m_{0}\right\rangle $-modules and noticing that they admit a decomposition
into isotypic components. It follows that we can complete the diagram $%
\left( \text{\ref{Primitives L Fil(L) F1}}\right) $ with an arrow making the
resulting diagram commutative and, hence, giving the claim.
\end{proof}

Define the $\left[ P\right] $-depleted subcomplex $H^{0}\left( X_{0},%
\mathcal{L}_{\lambda }\otimes _{\mathcal{O}_{X}}\Omega _{X/\Bbbk }^{\cdot
}\right) ^{\left[ P\right] }$ of $H^{0}\left( X_{0},\mathcal{L}_{\lambda
}\otimes _{\mathcal{O}_{X}}\Omega _{X/\Bbbk }^{\cdot }\right) $ by taking
the intersection with the subcomplex $H^{0}\left( X_{0},\mathcal{V}_{\lambda
}\otimes _{\mathcal{O}_{X}}\Omega _{X/\Bbbk }^{\cdot }\right) ^{\left[ P%
\right] }$ of $H^{0}\left( X_{0},\mathcal{V}_{\lambda }\otimes _{\mathcal{O}%
_{X}}\Omega _{X/\Bbbk }^{\cdot }\right) $. We can now prove the following
result.

\begin{corollary}
\label{Primitives CKey}Suppose that $P\left( 0\right) =0$. For $p\in \left\{
1,...,d_{g}\right\} $ and every $f\in H^{0}\left( X_{0},\mathcal{L}_{\lambda
}\otimes _{\mathcal{O}_{X}}\Omega _{X/\Bbbk }^{p}\right) ^{\left[ P\right] }$
such that $\nabla ^{p}\left( f\right) =0$, define $F_{i}$ inductively as in
Proposition \ref{Primitives PKey}. Then%
\begin{equation*}
F_{i}\in H^{0}\left( X_{0},\mathrm{gr}_{i}\left( \mathcal{L}_{\lambda
}\right) \otimes _{\mathcal{O}_{X}}\Omega _{X/\Bbbk }^{p-1}\right) ^{\left[ P%
\right] }\hookrightarrow H^{0}\left( X_{0},\mathrm{gr}_{i}\left( \mathcal{V}%
_{\lambda }\right) \otimes _{\mathcal{O}_{X}}\Omega _{X/\Bbbk }^{p-1}\right)
^{\left[ P\right] }
\end{equation*}%
for every $i$, there is $r\in \mathbb{N}$ such that $\Delta _{r}^{p-1}\left(
F_{r}\right) =0$ and $f_{i}=0$ for every $i\geq r+1$ and, with the notations
of Proposition \ref{Primitives PKey},%
\begin{equation*}
F\in H^{0}\left( X_{0},\mathcal{L}_{\lambda }\otimes _{\mathcal{O}%
_{X}}\Omega _{X/\Bbbk }^{p-1}\right) ^{\left[ P\right] }\hookrightarrow
H^{0}\left( X_{0},\mathcal{V}_{\lambda }\otimes _{\mathcal{O}_{X}}\Omega
_{X/\Bbbk }^{p-1}\right) ^{\left[ P\right] }
\end{equation*}%
is such that $\nabla ^{p-1}\left( F\right) =f$. In particular, the complex $%
H^{0}\left( X_{0},\mathcal{L}_{\lambda }\otimes _{\mathcal{O}_{X}}\Omega
_{X/\Bbbk }^{\cdot }\right) ^{\left[ P\right] }$ is acyclic in degree $p\in
\left\{ 1,...,d_{g}\right\} $.
\end{corollary}

\begin{proof}
Suppose that we can show that $F_{i}\in H^{0}\left( X_{0},\mathrm{gr}%
_{i}\left( \mathcal{L}_{\lambda }\right) \otimes _{\mathcal{O}_{X}}\Omega
_{X/\Bbbk }^{p-1}\right) $ for every $i$. Because $\mathcal{L}_{\lambda }$
has finite rank and the filtration $\mathrm{Fil}_{i}\left( \mathcal{V}%
_{\lambda }\right) =\bigoplus_{j\in \mathbb{N}}\mathrm{gr}_{i-j}\left( 
\mathcal{V}_{\lambda }\right) $ of $\mathcal{V}_{\lambda }$ strictly
increases the rank, we have%
\begin{equation*}
H^{0}\left( X_{0},\mathrm{gr}_{r+1}\left( \mathcal{V}_{\lambda }\right)
\otimes _{\mathcal{O}_{X}}\Omega _{X/\Bbbk }^{\cdot }\right) \cap
H^{0}\left( X_{0},\mathcal{L}_{\lambda }\otimes _{\mathcal{O}_{X}}\Omega
_{X/\Bbbk }^{\cdot }\right) =0
\end{equation*}

and thus $\Delta _{r}^{p-1}\left( F_{r}\right) =0$ for some $r\in \mathbb{N}$%
: taking $r$ large enough we may even assume that $f_{i}=0$ for every $i\geq
r+1$, where we write%
\begin{equation}
f=\tsum\nolimits_{i=0}^{+\infty }f_{i}\text{ in }H^{0}\left( X_{0},\mathcal{L%
}_{\lambda }\otimes _{\mathcal{O}_{X}}\Omega _{X/\Bbbk }^{p}\right)
=\tbigoplus\nolimits_{i=0}^{+\infty }H^{0}\left( X_{0},\mathrm{gr}_{i}\left( 
\mathcal{L}_{\lambda }\right) \otimes _{\mathcal{O}_{X}}\Omega _{X/\Bbbk
}^{p}\right)  \label{Primitives CKey F1}
\end{equation}%
(as guaranteed by Lemma \ref{Primitives L Fil(L)}). Then we have $F\in
H^{0}\left( X_{0},\mathcal{L}_{\lambda }\otimes _{\mathcal{O}_{X}}\Omega
_{X/\Bbbk }^{p-1}\right) $ because $F_{i}\in H^{0}\left( X_{0},\mathrm{gr}%
_{i}\left( \mathcal{L}_{\lambda }\right) \otimes _{\mathcal{O}_{X}}\Omega
_{X/\Bbbk }^{p-1}\right) $.

It remains to explain why we have $F_{i}\in H^{0}\left( X_{0},\mathrm{gr}%
_{i}\left( \mathcal{L}_{\lambda }\right) \otimes _{\mathcal{O}_{X}}\Omega
_{X/\Bbbk }^{p-1}\right) $. But indeed, we first note that it follows from $%
\left( \text{\ref{Primitives F Integration}}\right) $ that $\Theta
_{p-1}^{-1}\left( g\right) \in H^{0}\left( X_{0},\mathrm{gr}_{i}\left( 
\mathcal{L}_{\lambda }\right) \otimes _{\mathcal{O}_{X}}\Omega _{X/\Bbbk
}^{p-1}\right) ^{\left[ P\right] }$ if $g\in H^{0}\left( X_{0},\mathrm{gr}%
_{i}\left( \mathcal{L}_{\lambda }\right) \otimes _{\mathcal{O}_{X}}\Omega
_{X/\Bbbk }^{p}\right) ^{\left[ P\right] }$ since we have $v_{\mathbf{i}}\in
L_{\lambda }$ in the right hand side of $\left( \text{\ref{Primitives F
Integration}}\right) $. Because $L_{\lambda }\subset V_{\lambda }$ is a $%
\left( \mathfrak{g},\mathbf{Q}\right) $-submodule, we also see from the
expression of $\Delta _{i}^{p-1}\left( G\right) \left( q\right) $ appearing
in Lemma \ref{Primitives LKey} that we have $\Delta _{i}^{p-1}\left(
G\right) \in H^{0}\left( X_{0},\mathrm{gr}_{i+1}\left( \mathcal{L}_{\lambda
}\right) \otimes _{\mathcal{O}_{X}}\Omega _{X/\Bbbk }^{p}\right) $ if $G\in
H^{0}\left( X_{0},\mathrm{gr}_{i}\left( \mathcal{L}_{\lambda }\right)
\otimes _{\mathcal{O}_{X}}\Omega _{X/\Bbbk }^{p-1}\right) $. Taking into
account $\left( \text{\ref{Primitives CKey F1}}\right) $, the inductive
definition of $F_{i}$ now proves that $F_{i}\in H^{0}\left( X_{0},\mathrm{gr}%
_{i}\left( \mathcal{L}_{\lambda }\right) \otimes _{\mathcal{O}_{X}}\Omega
_{X/\Bbbk }^{p-1}\right) $, as wanted.
\end{proof}

\subsubsection{Proof of Proposition \protect\ref{Primitives PKey}}

Suppose that $M$ is an $R$-module over a (commutative and unitary) ring $R$\
and that $\mathbf{\varphi }=\left( \varphi _{1},...,\varphi _{n}\right) \in
End_{R}\left( M\right) ^{n}$ is an ordered family of $R$-linear maps. Then,
we define a graded $R$-module $K^{\cdot }\left( \mathbf{\varphi }\right)
=K^{\cdot }\left( \varphi _{1},...,\varphi _{n}\right) $\ concentrated in
degree $\left[ 0,n\right] $ equipped with a degree $1$ morphism as follows.
For every $p\in \mathbb{N}$, write $I_{p}\subset \left\{ 1,...,n\right\}
^{p} $ for the set of multi-indexes $\mathbf{i}=\left(
i_{1},...,i_{p}\right) $ such that $i_{1}<...<i_{p}$ (we understand that $%
I_{0}:=\left\{ \phi \right\} $). For every $\mathbf{i}\in I_{p}$, we write $%
\left\{ \mathbf{i}\right\} $ for the set $\left\{ i_{1},...,i_{p}\right\} $
(we understand that $\left\{ \phi \right\} =\phi $\footnote{%
We hope this notation will note create confusion with the set theoretic
notation acoording to which $\left\{ \phi \right\} \neq \phi $. We want $%
\left\{ k\right\} \cup \left\{ \mathbf{i}\right\} $ to be $\left\{ k\right\} 
$ when $\mathbf{i}=\phi $ in the following definition.}). If $\mathbf{i}\in
I_{p}$ and $k\in \left\{ 1,...,n\right\} $, we set $k\wedge \mathbf{i}=\phi $
in case $k\in \left\{ \mathbf{i}\right\} $ and, otherwise, we write $k\wedge 
\mathbf{i}$ for the unique element of $I_{p+1}$ such that $\left\{ k\wedge 
\mathbf{i}\right\} =\left\{ k\right\} \cup \left\{ \mathbf{i}\right\} $.
Define $\varepsilon _{\mathbf{i}}\left( k\right) :=\#\left\{ i\in \left\{ 
\mathbf{i}\right\} :i<k\right\} $.

\begin{itemize}
\item We set $K^{p}\left( \mathbf{\varphi }\right) :=M\otimes _{R}\wedge
_{R}^{p}\left( R^{n}\right) =\tbigoplus\nolimits_{\mathbf{i}\in I_{p}}M$. If 
$\mathbf{i}\in I_{p}$, we write $e_{\mathbf{i}}:=e_{i_{1}}\wedge ...\wedge
e_{i_{p}}$ where $e_{i}\in R^{n}$ denotes the $i$-canonical vector: if $m\in
M$, then $m\otimes _{R}e_{\mathbf{i}}$ corresponds to the family whose
components are all zero except the $\mathbf{i}$-component, which equals $m$.
We also adopt the convention that $e_{\phi }=0$ (cfr. the displayed equality
below).

\item We let $d^{p}=d_{\mathbf{\varphi }}^{p}:K^{p}\left( \mathbf{\varphi }%
\right) \rightarrow K^{p+1}\left( \mathbf{\varphi }\right) $ be the direct
sum of the $R$-linear morphisms%
\begin{equation}
d^{p}\left( m\otimes _{R}e_{\mathbf{i}}\right)
=\tsum\nolimits_{k=1}^{n}\varphi _{k}\left( m\right) \otimes _{R}e_{k}\wedge
e_{\mathbf{i}}=\tsum\nolimits_{k=1}^{n}\left( -1\right) ^{\varepsilon _{%
\mathbf{i}}\left( k\right) }\varphi _{k}\left( m\right) \otimes
_{R}e_{k\wedge \mathbf{i}}\text{.}  \label{Kos F1}
\end{equation}
\end{itemize}

Let us make the following useful remarks.

\begin{itemize}
\item[$\left( i\right) $] If $\mathbf{j}\in I_{p+1}$, then we see that the $%
\mathbf{j}$-component\ $d^{p}\left( m\otimes _{R}e_{\mathbf{i}}\right) _{%
\mathbf{j}}$\ of $d^{p}\left( m\otimes _{R}e_{\mathbf{i}}\right) $ (for $%
\mathbf{i}\in I_{p}$ as above) is zero except in case $\left\{ \mathbf{i}%
\right\} \subset \left\{ \mathbf{j}\right\} $. Furthermore, in this case, if 
$\mathbf{j}=\left( j_{1},...,j_{p+1}\right) $, then we have $\mathbf{i=}%
\left( j_{1},...,\widehat{j_{k}},...,j_{p+1}\right) $ for a unique $j_{k}\in
\left\{ \mathbf{j}\right\} $ and $d^{p}\left( m\otimes _{R}e_{\mathbf{i}%
}\right) _{\mathbf{j}}=\left( -1\right) ^{k+1}\varphi _{j_{k}}\left(
m\right) \otimes _{R}e_{\mathbf{j}}$ (note that $\varepsilon _{\mathbf{i}%
}\left( j_{k}\right) =k-1$). We deduce that, if%
\begin{equation*}
m=\tsum\nolimits_{\mathbf{i}\in I_{p}}m_{\mathbf{i}}\otimes _{R}e_{\mathbf{i}%
}\in K^{p}\left( \mathbf{\varphi }\right) \text{ and }d^{p}\left( m\right)
=\tsum\nolimits_{\mathbf{j}\in I_{p+1}}d^{p}\left( m\right) _{\mathbf{j}%
}\otimes _{R}e_{\mathbf{j}}\text{,}
\end{equation*}%
then%
\begin{equation}
d^{p}\left( m\right) _{\mathbf{j}}=\tsum\nolimits_{k=1}^{p+1}\left(
-1\right) ^{k+1}\varphi _{j_{k}}\left( m_{\left( j_{1},...,\widehat{j_{k}}%
,...,j_{p+1}\right) }\right) \text{ if }\mathbf{j}=\left(
j_{1},...,j_{p+1}\right) \text{.}  \label{Kos F1'}
\end{equation}

\item[$\left( ii\right) $] The degree one morphism $d^{\cdot }$\ makes the
graded $R$-module $K^{\cdot }\left( \mathbf{\varphi }\right) $ a complex if
and only if $d^{0}\circ d^{1}=0$ if and only if $\varphi _{i}\varphi
_{j}=\varphi _{j}\varphi _{i}$ for every $i,j\in \left\{ 1,...,n\right\} $:
this can be deduced from $\left( \text{\ref{Kos F1}}\right) $. In this case,
we say that $\left\{ \varphi _{1},...,\varphi _{n}\right\} $ is a commuting
family and that $K^{\cdot }\left( \mathbf{\varphi }\right) $ is the dual\
Koszul complex associated to it.
\end{itemize}

\bigskip

We now define a degree $-1$ morphism on $K^{\cdot }\left( \mathbf{\varphi }%
\right) $ as follows. For every $p\in \mathbb{N}$, every $\mathbf{i}\in
I_{p} $ and every $k\in \left\{ 1,...,n\right\} $, we set $\mathbf{i}%
_{k}=\phi $ if $k\notin \left\{ \mathbf{i}\right\} $ and, otherwise, we
write $\mathbf{i}_{k}$ for the unique element of $I_{p-1}$ such that $%
\left\{ \mathbf{i}_{k}\right\} =\left\{ \mathbf{i}\right\} -\left\{
k\right\} $.

\begin{itemize}
\item We let $\partial ^{p}=\partial _{\mathbf{\varphi }}^{p}:K^{p}\left( 
\mathbf{\varphi }\right) \rightarrow K^{p-1}\left( \mathbf{\varphi }\right) $
be the direct sum of the $R$-linear morphisms%
\begin{equation}
\partial ^{p}\left( m\otimes _{R}e_{\mathbf{i}}\right)
=\tsum\nolimits_{k=1}^{n}\left( -1\right) ^{\varepsilon _{\mathbf{i}}\left(
k\right) }\varphi _{k}\left( m\right) \otimes _{R}e_{\mathbf{i}_{k}}\text{.}
\label{Kos F1 dual}
\end{equation}
\end{itemize}

Let us make the following useful remarks.

\begin{itemize}
\item[$\left( iii\right) $] If $\mathbf{j}\in I_{p-1}$, then we see that the 
$\mathbf{j}$-component\ $\partial ^{p}\left( m\otimes _{R}e_{\mathbf{i}%
}\right) _{\mathbf{j}}$\ of $\partial ^{p}\left( m\otimes _{R}e_{\mathbf{i}%
}\right) $ (for $\mathbf{i}\in I_{p}$ as above) is zero except in case $%
\left\{ \mathbf{j}\right\} \subset \left\{ \mathbf{i}\right\} $.
Furthermore, in this case, if $\mathbf{j}=\left( j_{1},...,j_{p-1}\right) $,
then we have $\mathbf{i}=k\wedge \mathbf{j}=\left(
j_{1},...,k,...,j_{p-1}\right) $ for a unique $k\in \left\{ \mathbf{i}%
\right\} $ and $\partial ^{p}\left( m\otimes _{R}e_{\mathbf{i}}\right) _{%
\mathbf{j}}=\left( -1\right) ^{\varepsilon _{\mathbf{i}}\left( k\right)
}\varphi _{k}\left( m\right) \otimes _{R}e_{\mathbf{j}}$. We deduce that, if%
\begin{equation*}
m=\tsum\nolimits_{\mathbf{i}\in I_{p}}m_{\mathbf{i}}\otimes _{R}e_{\mathbf{i}%
}\in K^{p}\left( \mathbf{\varphi }\right) \text{ and }\partial ^{p}\left(
m\right) =\tsum\nolimits_{\mathbf{j}\in I_{p-1}}\partial ^{p}\left( m\right)
_{\mathbf{j}}\otimes _{R}e_{\mathbf{j}}\text{,}
\end{equation*}%
then%
\begin{equation}
\partial ^{p}\left( m\right) _{\mathbf{j}}=\tsum\nolimits_{k\in \left\{
1,...,n\right\} -\left\{ \mathbf{j}\right\} }\left( -1\right) ^{\varepsilon
_{k\wedge \mathbf{j}}\left( k\right) }\varphi _{k}\left( m_{k\wedge \mathbf{j%
}}\right) \text{ if }\mathbf{j}=\left( j_{1},...,j_{p-1}\right) \text{.}
\label{Kos F1' dual}
\end{equation}

\item[$\left( iv\right) $] The degree minus one morphism $\partial ^{\cdot }$%
\ makes the graded $R$-module $K^{\cdot }\left( \mathbf{\varphi }\right) $ a
complex if and only if $\partial ^{1}\circ \partial ^{2}=0$ if and only if $%
\varphi _{i}\varphi _{j}=\varphi _{j}\varphi _{i}$ for every $i,j\in \left\{
1,...,n\right\} $: this can be deduced from $\left( \text{\ref{Kos F1' dual}}%
\right) $. In this case, we say that $\left\{ \varphi _{1},...,\varphi
_{n}\right\} $ is a commuting family and that $K^{\cdot }\left( \mathbf{%
\varphi }\right) $ is the Koszul complex associated to it.
\end{itemize}

\bigskip

Let us now fix two ordered families%
\begin{equation*}
\mathbf{\varphi }=\left( \varphi _{1},...,\varphi _{n}\right) ,\mathbf{\psi =%
}\left( \psi _{1},...,\psi _{n}\right) \subset End_{R}\left( M\right) ^{n}
\end{equation*}%
and define two degree $0$ morphisms on $K^{\cdot }\left( M,n\right)
:=K^{\cdot }\left( \mathbf{\varphi }\right) =K^{\cdot }\left( \mathbf{\psi }%
\right) $ (the equality of the underlying graded $R$-modules) as follows.

\begin{itemize}
\item We let $\Delta _{\mathbf{\psi }\cdot \mathbf{\varphi }%
}^{p}:K^{p}\left( M,n\right) \rightarrow K^{p}\left( M,n\right) $ be the
direct sum of the $R$-linear morphisms%
\begin{equation}
\Delta _{\mathbf{\psi }\cdot \mathbf{\varphi }}^{p}\left( m\otimes _{R}e_{%
\mathbf{i}}\right) =\tsum\nolimits_{k\in \left\{ 1,...,n\right\} -\left\{ 
\mathbf{i}\right\} }\psi _{k}\left( \varphi _{k}\left( m\right) \right)
\otimes _{R}e_{\mathbf{i}}+\tsum\nolimits_{i\in \left\{ \mathbf{i}\right\}
}\varphi _{i}\left( \psi _{i}\left( m\right) \right) \otimes _{R}e_{\mathbf{i%
}}  \label{Kos F1 Delta1}
\end{equation}%
and define $\Delta _{\left[ \mathbf{\psi },\mathbf{\varphi }\right]
}^{p}:K^{p}\left( M,n\right) \rightarrow K^{p}\left( M,n\right) $ as the
direct sum of the $R$-linear morphisms%
\begin{equation}
\Delta _{\left[ \mathbf{\psi },\mathbf{\varphi }\right] }^{p}\left( m\otimes
_{R}e_{\mathbf{i}}\right) =\tsum\nolimits_{k\in \left\{ 1,...,n\right\}
-\left\{ \mathbf{i}\right\} ,l\in \left\{ \mathbf{i}\right\} }\left(
-1\right) ^{\varepsilon _{k\wedge \mathbf{j}}\left( k\right) +\varepsilon _{%
\mathbf{j}}\left( l\right) }\left[ \psi _{k},\varphi _{l}\right] \left(
m\right) \otimes _{R}e_{\left( l\wedge \mathbf{i}\right) _{k}}\text{.}
\label{Kos F1 Delta2}
\end{equation}
\end{itemize}

Let us make the following remark.

\begin{itemize}
\item[$\left( v\right) $] If $\mathbf{j}\in I_{p}$, then we see that the $%
\mathbf{j}$-component\ $\Delta _{\left[ \mathbf{\psi },\mathbf{\varphi }%
\right] }^{p}\left( m\otimes _{R}e_{\mathbf{i}}\right) _{\mathbf{j}}$\ of $%
\Delta _{\left[ \mathbf{\psi },\mathbf{\varphi }\right] }^{p}\left( m\otimes
_{R}e_{\mathbf{i}}\right) $ (for $\mathbf{i}\in I_{p}$ as above) is zero
except in case $\left\{ \mathbf{i}\right\} ,\left\{ \mathbf{j}\right\}
\subset \left\{ \mathbf{k}\right\} $ for some $\mathbf{k}\in I_{p+1}$
(taking $\mathbf{k}$ of the form $l\wedge \mathbf{i}$ in the expression $%
\left( \text{\ref{Kos F1 Delta2}}\right) $). Furthermore, in this case, if $%
\mathbf{j}=\left( j_{1},...,j_{p}\right) $, then $\mathbf{k}=k\wedge \mathbf{%
j}=\left( j_{1},...,j_{h_{k}-1},k,j_{h_{k}}...,j_{p}\right) $ for a unique $%
k\in \left\{ \mathbf{k}\right\} -\left\{ \mathbf{j}\right\} $\ (with $h_{k}$
defined by the equality $\varepsilon _{k\wedge \mathbf{j}}\left( k\right)
=h_{k}-1$)\ and $\mathbf{i}=\left( j_{1},...,\widehat{j_{l}}%
,...,j_{h_{k}-1},k,...,j_{p}\right) $ if $l<h_{k}$ or $\mathbf{i}=\left(
j_{1},...,j_{h_{k}-1},k,...,\widehat{j_{l}},...,j_{p}\right) $ if $l\geq
h_{k}$ for a unique $j_{k}\in \left\{ \mathbf{j}\right\} $. In any case, we
have $\mathbf{i}=k\wedge \left( j_{1},...,\widehat{j_{l}},...,j_{p}\right) $%
\ and $\Delta _{\left[ \mathbf{\psi },\mathbf{\varphi }\right] }^{p}\left(
m\otimes _{R}e_{\mathbf{i}}\right) _{\mathbf{j}}=\left( -1\right)
^{\varepsilon _{k\wedge \mathbf{j}}\left( k\right) +l+1}\left[ \psi
_{k},\varphi _{j_{l}}\right] \left( m\right) \otimes _{R}e_{\mathbf{j}}$
(note that $\varepsilon _{\mathbf{j}}\left( j_{l}\right) =l-1$). We deduce
that, if%
\begin{equation*}
m=\tsum\nolimits_{\mathbf{i}\in I_{p}}m_{\mathbf{i}}\otimes _{R}e_{\mathbf{i}%
}\in K^{p}\left( M,n\right) \text{ and }\partial ^{p}\left( m\right)
=\tsum\nolimits_{\mathbf{j}\in I_{p}}\Delta _{\left[ \mathbf{\psi },\mathbf{%
\varphi }\right] }^{p}\left( m\right) _{\mathbf{j}}\otimes _{R}e_{\mathbf{j}}%
\text{,}
\end{equation*}%
then%
\begin{equation}
\Delta _{\left[ \mathbf{\psi },\mathbf{\varphi }\right] }^{p}\left( m\right)
_{\mathbf{j}}=\tsum\nolimits_{k\in \left\{ 1,...,n\right\} -\left\{ \mathbf{j%
}\right\} }\left( \tsum\nolimits_{l=1}^{p}\left( -1\right) ^{\varepsilon
_{k\wedge \mathbf{j}}\left( k\right) +l+1}\left[ \psi _{k},\varphi _{j_{l}}%
\right] \left( m_{k\wedge \left( j_{1},...,\widehat{j_{l}},...,j_{p}\right)
}\right) \right) \text{.}  \label{Kos F1' Delta2}
\end{equation}
\end{itemize}

A somewhat tedious calculation based on $\left( \text{\ref{Kos F1'}}\right) $%
, $\left( \text{\ref{Kos F1' dual}}\right) $ and $\left( \text{\ref{Kos F1'
Delta2}}\right) $ proves the following useful formula.

\begin{lemma}
\label{Kos Int L1}Suppose that $\mathbf{\varphi }$ and $\mathbf{\psi }$ are
as above. Then we have, for every $p\in \left\{ 1,...,n\right\} $, the
following equality:%
\begin{equation*}
\partial _{\mathbf{\psi }}^{p+1}d_{\mathbf{\varphi }}^{p}+d_{\mathbf{\varphi 
}}^{p-1}\partial _{\mathbf{\psi }}^{p}=\Delta _{\mathbf{\psi },\mathbf{%
\varphi }}^{p}+\Delta _{\left[ \mathbf{\psi },\mathbf{\varphi }\right] }^{p}%
\text{.}
\end{equation*}
\end{lemma}

\begin{corollary}
\label{Kos Int L1 C1}Suppose that $p\in \left\{ 1,...,n\right\} $, that $%
\left\{ \varphi _{1},...,\varphi _{n},\psi _{1},...,\psi _{n}\right\}
\subset End_{R}\left( M\right) $ is a commuting family, that $\Delta _{%
\mathbf{\psi }\cdot \mathbf{\varphi }}^{p}$ is invertible and that one
between $\Delta _{\mathbf{\psi },\mathbf{\varphi }}^{p+1}$ or $\Delta _{%
\mathbf{\psi },\mathbf{\varphi }}^{p-1}$ is invertible. Then, $H^{p}\left(
K^{\cdot }\left( M,n\right) ,d_{\mathbf{\varphi }}^{\cdot }\right)
=H^{p}\left( K^{\cdot }\left( M,n\right) ,\partial _{\mathbf{\varphi }%
}^{\cdot }\right) =0$ and a section%
\begin{equation*}
s^{p}:K^{p}\left( M,n\right) \longrightarrow K^{p-1}\left( M,n\right) \text{
(resp. }s^{p}:K^{p}\left( M,n\right) \longrightarrow K^{p+1}\left(
M,n\right) \text{)}
\end{equation*}%
of $d_{\mathbf{\varphi }}^{p-1}$ (resp. $\partial _{\mathbf{\psi }}^{p+1}$)
is obtained as follows. If $\Delta _{\mathbf{\psi },\mathbf{\varphi }}^{p+1}$
(resp. $\Delta _{\mathbf{\psi },\mathbf{\varphi }}^{p-1}$) is invertible,
set $s^{p}:=\partial _{\mathbf{\psi }}^{p}\Delta _{\mathbf{\psi },\mathbf{%
\varphi }}^{p,-1}$ (resp. $s^{p}:=d_{\mathbf{\varphi }}^{p}\Delta _{\mathbf{%
\psi },\mathbf{\varphi }}^{p,-1}$) and, if $\Delta _{\mathbf{\psi },\mathbf{%
\varphi }}^{p-1}$ (resp. $\Delta _{\mathbf{\psi },\mathbf{\varphi }}^{p+1}$)
is invertible, set $s^{p}:=\Delta _{\mathbf{\psi },\mathbf{\varphi }%
}^{p-1,-1}\partial _{\mathbf{\psi }}^{p}$ (resp. $s^{p}:=\Delta _{\mathbf{%
\psi },\mathbf{\varphi }}^{p+1,-1}d_{\mathbf{\varphi }}^{p}$).
\end{corollary}

\begin{remark}
If $\Delta _{\mathbf{\psi },\mathbf{\varphi }}^{p+1}$ and $\Delta _{\mathbf{%
\psi },\mathbf{\varphi }}^{p-1}$ are both invertible, then the two sections
of Corollary \ref{Kos Int L1 C1} are the same.
\end{remark}

\bigskip

Suppose now that $A$ is an $R$-algebra, that $\Omega $ is a free $A$-module
of finite rank $n$\ and that $d:A\rightarrow \Omega $ and $\nabla
:M\rightarrow \Omega \otimes _{A}M$ are a morphism of $R$-modules. Fix an $A$%
-basis $\left\{ \omega _{1},...,\omega _{n}\right\} $ of $\Omega $ and let $%
\left\{ \delta _{1},...,\delta _{n}\right\} \subset Hom_{A}\left( \Omega
,A\right) $ be the dual basis. Then we define $d_{i}:=\delta _{i}\circ
\nabla \in End_{R}\left( M\right) $ and $\nabla _{i}:=\left( \delta
_{i}\otimes _{A}1_{M}\right) \circ \nabla \in End_{R}\left( M\right) $.

\begin{proposition}
\label{Kos Int PdR}There is a unique sequence of morphisms of $R$-modules $%
d^{p}:\wedge _{A}^{p}\Omega \rightarrow \wedge _{A}^{p+1}\Omega $ such that $%
d^{0}=d$ and%
\begin{equation*}
d^{p+q}\left( \omega _{p}\wedge \omega _{q}\right) =d^{p}\left( \omega
_{p}\right) \wedge \omega _{q}+\left( -1\right) ^{p}\omega _{p}\wedge
d^{q}\left( \omega _{q}\right)
\end{equation*}%
for every $p,q\in \mathbb{N}$ such that $p+q\geq 1$, every $\omega _{p}\in
\wedge _{A}^{p}\Omega $ and every $\omega _{q}\in \wedge _{A}^{q}\Omega $\
and there is a unique sequence of morphisms of $R$-modules $\nabla
^{p}:\wedge _{A}^{p}\Omega \otimes _{A}M\rightarrow \wedge _{A}^{p+1}\Omega
\otimes _{A}M$ such that $\nabla ^{0}=\nabla $ and%
\begin{equation*}
\nabla ^{p}\left( \omega \otimes m\right) =d^{p}\left( \omega \right)
\otimes m+\left( -1\right) ^{p}\omega \wedge \nabla \left( m\right)
\end{equation*}%
for every $m\in M$ and $\omega \in \wedge _{A}^{p}\Omega $. The isomorphism $%
A^{n}\overset{\sim }{\rightarrow }\Omega $ induced by the $A$-basis $\left\{
\omega _{1},...,\omega _{n}\right\} $ yields an isomorphism of graded $R$%
-modules $K^{\cdot }\left( \nabla _{1},...,\nabla _{n}\right) \overset{\sim }%
{\rightarrow }\wedge _{A}^{\cdot }\Omega \otimes _{A}M$ such that the degree
one morphism $d^{\cdot }$ on $K^{\cdot }\left( \nabla _{1},...,\nabla
_{n}\right) $ corresponds to $\left\{ \nabla ^{p}\right\} _{p\in \mathbb{N}}$%
: then $\wedge _{A}^{\cdot }\Omega \otimes _{A}M$ is a complex if and only
if $\nabla ^{1}\circ \nabla ^{0}=0$\ if and only if $\left\{ \nabla
_{1},...,\nabla _{n}\right\} $ is a commuting family.
\end{proposition}

\begin{proof}
The isomorphism $A^{n}\overset{\sim }{\rightarrow }\Omega $ induced by the $%
A $-basis $\left\{ \omega _{1},...,\omega _{n}\right\} $ identifies%
\begin{equation*}
K^{p}\left( \nabla _{1},...,\nabla _{n}\right) =M\otimes _{A}\wedge
_{A}^{p}\left( A^{n}\right) \overset{\sim }{\rightarrow }M\otimes _{A}\wedge
_{A}^{p}\Omega \simeq \wedge _{A}^{p}\Omega \otimes _{A}M
\end{equation*}%
and one checks that the degree one differential obtained by transport from $%
K^{\cdot }\left( \nabla _{1},...,\nabla _{n}\right) $ satisfies the required
existence property (taking $M=A$ in order to get $\left\{ d^{p}\right\}
_{p\in \mathbb{N}}$). The uniqueness is clear and, by construction, we have $%
K^{\cdot }\left( \nabla _{1},...,\nabla _{n}\right) \overset{\sim }{%
\rightarrow }\wedge _{A}^{\cdot }\Omega \otimes _{A}M$ with $d$
corresponding to $\left\{ \nabla ^{p}\right\} _{p\in \mathbb{N}}$. Remark $%
\left( ii\right) $ yields the last assertion.
\end{proof}

\bigskip

In the setting before the above proposition, suppose that $d\left(
a_{1}a_{2}\right) =d\left( a_{1}\right) a_{2}+a_{1}d\left( a_{2}\right) $
and $\nabla \left( am\right) =d\left( a\right) \otimes m+a\nabla \left(
m\right) $ for every $a_{1},a_{2},a\in A$ and every $m\in M$: then the
operators $d_{i}$'s and $\nabla _{i}$'s satisfy%
\begin{equation*}
\nabla _{i}\left( am\right) =d_{i}\left( a\right) m+a\nabla _{i}\left(
m\right) \text{ for every }a\in A\text{ and }m\in M\text{.}
\end{equation*}%
In particular, take $\left( A,R\right) =\left( R_{0},K\right) $ and let $%
\Omega $ be the $A$-module of K\"{a}hler differentials $\Omega
_{X_{0}/K}^{1} $. If $M=H^{0}\left( X_{0},\mathcal{E}_{\rho }\right) $, then
formula $\left( \text{\ref{Kos F1'}}\right) $ translates into the formula
claimed in Lemma \ref{de Rham L2}. In order to get Proposition \ref%
{Primitives PKey}, we first remark that, thanks to Lemma \ref{Primitives L3}%
, we deduce the result from the analogous result for $H^{0}\left(
X_{0},\Omega _{X/K}^{\cdot }\right) ^{\left[ P\right] }$. In general, if $%
M^{\prime }\subset M$ is an $R $-submodule such that $\varphi _{i}\left(
M^{\prime }\right) \subset M^{\prime }$ for every $i$ and $\mathbf{\varphi }%
_{\mid M^{\prime }}:=\left( \varphi _{1\mid M^{\prime }},...,\varphi _{n\mid
M^{\prime }}\right) $, then $\left( K^{\cdot }\left( \mathbf{\varphi }_{\mid
M^{\prime }}\right) ,d_{\mathbf{\varphi }_{\mid M^{\prime }}}^{\cdot
}\right) $ is a subcomplex of $\left( K^{\cdot }\left( \mathbf{\varphi }%
\right) ,d_{\mathbf{\varphi }}^{\cdot }\right) $ and similarly for $\partial
_{\mathbf{\psi }}^{\cdot }$ and $\Delta _{\mathbf{\psi }\cdot \mathbf{%
\varphi }}^{\cdot }$: applying this remark to $H^{0}\left( X_{0},\Omega
_{X/K}^{1}\right) ^{\left[ P\right] }\subset H^{0}\left( X_{0},\Omega
_{X/K}^{1}\right) $ and $\left( \varphi _{1},...,\varphi _{n}\right)
:=\left( \theta _{1},...,\theta _{d_{g}}\right) $\ we see, thanks to $\left( 
\text{\ref{de Rham F Triv5depl}}\right) $, that the isomorphism of
Proposition \ref{Kos Int PdR} restricts to and isomorphism $K^{\cdot }\left(
\theta _{1},...,\theta _{d_{g}}\right) \overset{\sim }{\rightarrow }%
H^{0}\left( X_{0},\Omega _{X/K}^{\cdot }\right) ^{\left[ P\right] }$ (where
we do not write the restrictions). If $\left( \psi _{1},...,\psi _{n}\right)
:=\left( P_{1}^{\ast }\left( \mathbf{\theta }\right) ,...,P_{d_{g}}^{\ast
}\left( \mathbf{\theta }\right) \right) $, then we have, by definition, the
equality $\Delta _{\mathbf{\psi }\cdot \mathbf{\varphi }}^{p}=\theta _{P}$,
which is invertible on $H^{0}\left( X_{0},\Omega _{X/K}^{p}\right) ^{\left[ P%
\right] }$ thanks to Lemma \ref{de Rham L0}. The claim follows from
Corollary \ref{Kos Int L1 C1}, noticing that $s^{p}=\Theta _{p-1}^{-1}$ in
our setting.

\end{document}